\RequirePackage{fix-cm}
\documentclass[smallextended,final,numbook]{svjour3}       
\smartqed  
\usepackage{graphicx}
\usepackage{xcolor}
\usepackage{newtxtext,newtxmath}
\usepackage{varioref,prettyref}

\definecolor{darkblue}{rgb}{0,0,0.6}
\definecolor{darkred}{rgb}{0.6,0,0}

\usepackage[unicode=true,bookmarks=false,breaklinks=false,pdfborder={0 0 1},colorlinks=true]{hyperref}
\hypersetup{urlcolor=darkblue,citecolor=darkblue,linkcolor=darkred}

\usepackage{mathrsfs,amsmath,amssymb,amstext,bbm,esint}
\usepackage{pdfsync,subfig}
\usepackage{fullpage}
\PassOptionsToPackage{normalem}{ulem}
\usepackage{ulem}
\newcommand{\indicator}[1]{\mathbbm{1}_{#1}}

\newcommand{\at}[1]{{#1}\rvert}

\newrefformat{prob}{Problem \ref{#1}}
\newrefformat{prop}{Proposition \ref{#1}}
\newrefformat{cor}{Corollary \ref{#1}}
\newrefformat{sec}{Section \ref{#1}}
\newrefformat{subsec}{Section \ref{#1}}
\newrefformat{fig}{Figure \ref{#1}}
\newrefformat{def}{Definition \ref{#1}}
\newrefformat{ex}{Example \ref{#1}}
\newrefformat{rem}{Remark \ref{#1}}
\newrefformat{conj}{Conjecture \ref{#1}}

\journalname{Archive for Rational Mechanics and Analysis}

\begin{document}

\title{Curvature-driven wrinkling of thin elastic shells\thanks{This work was supported
by National Science Foundation Awards DMS-1812831, DMS-1813003, and DMS-2025000, and a University of Michigan Van Loo Postdoctoral Fellowship.}
}

\titlerunning{Curvature-driven wrinkling}        

\author{Ian Tobasco}


\institute{I.\ Tobasco \at Department of Mathematics, Statistics, and Computer Science, University of Illinois at Chicago, Chicago, IL 60607\\ 
Department of Mathematics, University of Michigan, Ann Arbor, MI 48109\\
              \email{itobasco@uic.edu}           
           }

\date{Received: date / Accepted: date}

\maketitle

\begin{abstract}
How much energy does it take to stamp a thin elastic shell flat? Motivated by recent experiments on the wrinkling patterns of floating shells, we develop a rigorous method via $\Gamma$-convergence for answering this question to leading order in the shell's thickness and other small parameters. The observed patterns involve ``ordered'' regions of well-defined wrinkles alongside ``disordered'' regions whose local features are less robust; as little to no tension is applied, the preference for order is not \emph{a priori} clear. Rescaling by the energy of a typical pattern, we derive a limiting variational problem for the effective displacement of the shell. It asks, in a linearized way, to cover up a maximum area with a length-shortening map to the plane. Convex analysis yields a boundary value problem characterizing the accompanying patterns via their defect measures. Partial uniqueness and regularity theorems follow from the method of characteristics on the ordered part of the shell. In this way, we can deduce from the principle of minimum energy the leading order features of stamped elastic shells.
\end{abstract}

\tableofcontents{}


\section{Introduction}

Thin elastic sheets subject to compressive boundary conditions or loads readily take on shapes far from their own. 
Sometimes such changes in shape lead to the development of fine-scale \emph{wrinkling patterns},
indicating the presence of residual strain or excess material that
is ``wrinkled away''. Other times, \emph{crumpling patterns
}occur such as those seen in everyday, crumpled paper sheets \cite{witten2007stress}.
An intriguing and widely open problem is to devise a method for predicting
the features of the often disordered network of creases or ``minimal
ridges'' \cite{lobkovsky1997properties} that forms. This and other
simplified versions of the crumpling problem, which ask for sharp
\emph{a priori} lower bounds on the energy required to crumple \cite{conti2008confining,venkataramani2004lower},
remain far from being understood (despite some striking recent phenomenological
progress identifying as a possible state variable the total length
of the plastically damaged set \cite{gottesman2018state}).

A cousin of the crumpling problem is the stamping one studied here, named after the manufacturing process of the same name. In stamping,
a thin elastic sheet is pressed into a target shape. If the
mid-surface of the sheet embeds isometrically into the target,
the sheet may simply take on the imposed shape. If no such
embedding exists, a pattern can instead appear \cite{hure2012stamping}.
The situation reminds of the isometric embedding theorem of Nash and Kuiper, which guarantees the existence of a sequence of continuously differentiable isometric embeddings converging uniformly to any length-shortening map \cite{kuiper1955C1,nash1954C1}.
However, in this paper we will not be concerned with such ``pure'' isometries, but rather with maps exhibiting small amounts of strain.

Our motivation to study stamping stems from our desire to understand
the patterns that form when a thin elastic shell is placed onto an otherwise planar water bath \cite{aharoni2017smectic,albarran2018curvature}.
The water adheres to the underside of the shell, and capillary and gravitational
forces act to stamp it flat. Stretching forces prefer isometric deformations, while bending forces limit the curvature that results. The authors in \cite{aharoni2017smectic,albarran2018curvature} report on the striking formation of ``wrinkle domains''
made up of sinusoidal oscillations in a piecewise constant or otherwise
slowly-varying direction. A typical floating shell divides into finitely
many domains. At the interfaces are ``walls'', across which the
direction of wrinkling changes rapidly, or ``folds'', wherein material
is lost beneath the surface. The particular arrangement of wrinkles into domains is observed to depend strongly on the initial features of the shell --- namely, its Gaussian curvature and boundary shape --- and the authors wonder about the possibility of designing patterns at will.

The appearance of wrinkle domains in floating elastic shells is
remarkable. It reminds of a key feature of other, more well-studied 
pattern forming systems such as shape memory alloys \cite{bhattacharya2003microstructure},
micromagnets \cite{desimone2006recent}, and liquid crystals \cite{ball2017mathematics}.
The authors in \cite{aharoni2017smectic} highlight in particular
a connection between wrinkles and the layers of a smectic liquid crystal.
They describe a coarse-graining procedure in which the wrinkle direction plays the role of a director field, and the wrinkle peaks and troughs are encoded in the level sets of a phase
field function $\varphi$. Setting an ansatz into the total energy
$E$, the authors extract a coarse-grained or ``effective''
energy $E_{\text{eff}}(\varphi)$. 
Carrying over known results on liquid crystals, the authors make scaling
predictions for various quantities such as the size of a typical domain and the width of its walls. 
It remains unclear, however, whether the
overall layout of the wrinkles, i.e., their particular arrangement into domains,
can be recovered by this approach.  There is reason to doubt it can be done. 
Careful examination of the ansatz in \cite{aharoni2017smectic}
shows it assumes the shell deforms by an oscillatory perturbation
of some leading order deformation which is implicitly defined. As such, it and the corresponding phase field function $\varphi$ may be prohibitively difficult to recover. 

The situation becomes even more complicated as the thickness of the shell tends to zero. 
Forthcoming experiments on ultrathin shells \cite{tobasco2020principles},
having thicknesses several orders of magnitude less than those in
\cite{aharoni2017smectic,albarran2018curvature}, show that sometimes
no coherent wrinkling pattern occurs.
 In particular, spherical caps produce a ``disordered'', crumpling-like
response whose local features are sensitive to perturbation and vary between trials.
Other less-symmetric spherical shells (e.g., triangles cut from
spheres) display a mixed ordered\textendash disordered response,
in which one part remains ordered --- being covered with wrinkle domains --- while another part exhibits
the crumpling-like response. Upon perturbation, the local features of the disordered parts tend to rearrange while their overall layout remains the same. In contrast, the ordered parts remain more or less unchanged. Notably, the opposite response occurs for saddle-shaped shells:
ultrathin negatively curved shells exhibit the same ordered wrinkle domains as do their thicker counterparts. 

The task of determining the features of wrinkled thin elastic
sheets has been the subject of much research. When wrinkles occur in response to applied tensile forces, certain directions are stabilized and one may deduce
the direction of wrinkling from tension-field theory \cite{reissner1938tension,steigmann1990tension,wagner1929ebene}, 
also known as the relaxed energy approach \cite{pipkin1986relaxed,pipkin1993relaxed,pipkin1994relaxed}.
The relaxed energy density $W_\text{rel}$ for a sheet with zero thickness
(a ``membrane'') is a function of its effective strain, which vanishes
on bi-axially compressed states and is otherwise strictly positive.
When applied to the \emph{tension-driven wrinkling} of 
thin elastic sheets \cite{bella2014wrinkles,davidovitch2011prototypical},
one finds that the extent of the un-wrinkled region is determined,
as well as the direction of the wrinkles, by solving a \emph{relaxed
problem} of the form
\begin{equation}
\min_{\Phi_{\text{eff}}}\,\int_{S}W_\text{rel}(D\Phi_{\text{eff}})\,dA \label{eq:relaxed}
\end{equation}
subject to boundary conditions and loads. Here, $\Phi_{\text{eff}}$
denotes the limiting or effective deformation of the mid-sheet $S$ that arises in the vanishing thickness limit. 
A recent focus in tension-driven problems has been on identifying the scaling behavior(s) with respect to thickness (and other parameters)
of the higher order terms in the expansion
\begin{equation}
\min\,E=C_{0}+\text{higher order terms} \label{eq:tensionexpansion}
\end{equation}
as the thickness tends to zero. The constant $C_{0}$ is given by the minimum value of \prettyref{eq:relaxed} and it amounts to the work done at leading order to stretch the sheet. 
Evaluating the higher order terms requires identifying the lengthscale and amplitude of the wrinkles whose existence is implied. In general these quantities can vary throughout the sheet, making their analysis rather involved.
Examples include the ``wrinkling cascades'' seen in uni-axially
compressed floating sheets pulled taught by surface tension \cite{huang2010smooth},
as well as in hanging drapes pulled taught by gravity \cite{bella2017coarsening}.

Wrinkling patterns also occur in situations devoid of strong tensile loads or even lacking them altogether. This is the case for the stamped and floating shells introduced above. The hallmarks of such \emph{curvature-driven wrinkling} are
the presence of geometric incompatibilities driving the patterned response, and a lack of coherence in certain
parameter regimes. The transition from ordered wrinkle domains
in moderately thin floating shells \cite{aharoni2017smectic} to an
ordered wrinkling--disordered crumpling-like response in the ultrathin limit \cite{tobasco2020principles}
is an example of this phenomenon. Other examples include the ordered
``herringbone'' patterns and their disordered ``labyrinthine''
counterparts occurring in bi-axially compressed sheets
on a planar substrate \cite{cai2011periodic,chen2004family,huang2004evolution,huang2005nonlinear},
as well as the hexagonal tiling and labyrinthine patterns occurring in compressed thin elastic spheres
bonded to a spherical core
\cite{stoop2015curvature,terwagne2014smart}. 

In any case where tension fails to dominate, the relaxed
problem \prettyref{eq:relaxed} offers little guidance as to the patterns that occur (we refer to a situation where $C_0=0$). Various authors working on problems for
which surface tension is a small but non-negligible effect have suggested
\cite{paulsen2018optimal,paulsen2017geometry,yao2013planar} that
the shell's response can be determined instead at leading order by solving a \emph{limiting} or \emph{effective
area problem} of the form
\begin{equation}
\max_{\Phi_{\text{eff}}}\,\text{Area}(\Phi_{\text{eff}}(S)). \label{eq:effectiveArea}
\end{equation}
In analogy to \prettyref{eq:tensionexpansion}, the minimum energy is expected to expand as
\begin{equation}
\min\,E= C_{1}\cdot \gamma +\text{higher order terms}\label{eq:min_E_expansion2}
\end{equation}
as the surface tension coefficient $\gamma$ of the exposed interface tends to zero. The shell is therefore predicted to maximize the area it covers at leading order. 
A natural question is regarding constraints: in \cite{yao2013planar} where a flat disc is confined
to a liquid saddle surface, the perimeter of the
sheet is taken to be fixed; in \cite{paulsen2018optimal} where a flat disc wraps a water droplet, the effective deformation $\Phi_{\text{eff}}$ is understood to be a length-shortening map. 

In this paper, we take the first step towards a mathematical analysis of curvature-driven wrinkling. We adopt the viewpoint of energy minimization (even global minimization to simplify) and set ourselves the following tasks: 
prove the validity of an effective area problem such as \prettyref{eq:effectiveArea} for the leading order behavior of (almost) minimizers, 
and deduce from its solutions the patterns that form. We achieve these goals for a class of \emph{weakly curved} or \emph{shallow} shells whose intrinsic geometries are close to flat.
This simplifying assumption facilitates analysis since it allows the use of a geometrically
linear, von Karman-like energy. Geometrically linear models are standard in the literature on elastic pattern formation, though they have yet to enjoy the same level of rigorous derivation from fully nonlinear elasticity as have plate and shell models for finite bending deformations (for a recent review, see \cite{muller2017mathematical}). 
Motivated by a recent suggestion \cite{davidovitch2019geometrically} that
there exists a ``bending-induced'' tension proportional to the geometric mean of the shell's bending modulus $B$ and the substrate stiffness
$K$, we rescale our energy functionals by 
\[
\gamma_{\text{eff}}=2\sqrt{BK}+\gamma\label{eq:effsurfacetension}
\]
and obtain their $\Gamma$-limit as $\gamma_\text{eff}\to 0$ in a topology well-suited to the formation of patterns. A linearized version of \prettyref{eq:effectiveArea} and \prettyref{eq:min_E_expansion2} results, in which $\gamma$ is replaced by the ``effective surface tension'' $\gamma_\text{eff}$. In proving these results, we will not assume that minimizers obey any particular ansatz, or even that they exhibit ordered wrinkle domains. 

This brings us to what may be the most important contribution of this paper: via convex analysis of the limiting, linearized area problem, we derive a new and far-reaching method for proving that almost minimizers must tend towards an ordered--possibly disordered state, one whose ordered part consists of known wrinkle domains, and whose possibly disordered part is left unconstrained. Our method consists of two steps: first, we solve for a set of \emph{stable lines} along which any oscillations (and concentrations) are ruled out; second, we recover the amplitude of the oscillations that do occur via a
 second order linear partial differential equation (PDE) for which the stable lines are characteristic curves. Thus, we have found a way to treat wrinkles as the characteristic curves of a family of differential operators, rather than as the level sets of some unknown phase field function as proposed in \cite{aharoni2017smectic}. The upcoming \prettyref{fig:library_of_patterns} \vpageref{fig:library_of_patterns}
presents various arrangements of our stable lines. We were pleased to learn
that the predicted ordered parts where they exist (shown as striped) compare favorably with the
experiments that motivated our work. Even the leftover, possibly disordered parts (shown in blank)
look to align. A separate paper is currently in preparation, where we plan to report on experimental and numerical tests of our predictions \cite{tobasco2020principles}. 
We turn to introduce the model we use and to state our main results. 

\subsection{Preliminaries}

\prettyref{subsec:weaklycurved_model} introduces a geometrically
linear model of elastic shells.
\prettyref{subsec:non-dimensionalization} passes to its non-dimensional
form and identifies the parameter regime of our results. Finally,
in \prettyref{subsec:Bounded-deformation-maps} we recall some basic
facts about functions of bounded deformation and bounded Hessian. The formal statement of our results is in  \prettyref{subsec:statementofresults}.

\subsubsection{Weakly curved floating shells\label{subsec:weaklycurved_model}}

We consider the model problem of a thin elastic shell floating on an otherwise planar liquid
bath. Let the undeformed mid-shell $S$ be the graph of a function
$p$ over some planar reference domain $\Omega\subset\mathbb{R}^{2}$, i.e.,
\[
S=\left\{ (x_{1},x_{2},p(x)):x\in\Omega\right\} .
\]
Given a deformation $\Phi:S\to\mathbb{R}^{3}$ of the shell, 
introduce its in- and out-of-plane displacements
$u:\Omega\to\mathbb{R}^{2}$ and $w:\Omega\to\mathbb{R}$ according to
\[
\Phi(x_1,x_2,p(x))=(x_{1}+u_{1}(x),x_{2}+u_{2}(x),w(x)),\quad x\in\Omega.
\]
The plane being referenced is that of the undeformed liquid bath.
So long as the shell is weakly curved, meaning that its typical
``slope'' $|\nabla p|\ll1$, its deformation can be modeled as a minimizer of the energy\footnote{We picked up the term ``weakly curved'' from \cite{howell2009applied}. It indicates a family of shells also referred to  as ``shallow'', the deformations of which can be modeled using the Donnel--Mushtari--Vlasov theory \cite{niordson1985shell,ventsel2001thin} or Marguerre's theory of shallow shells \cite{sanders1963nonlinear}. Our stretching and bending terms become the ones from \cite{sanders1963nonlinear} under the substitution $w \to w+p$, and the ones from \cite{niordson1985shell,ventsel2001thin} under the further substitution $u\to u-w\nabla p$.  
}

\begin{equation}\label{eq:energy_dimensional}
\begin{aligned}
E &=\frac{Y}{2}\int_{\Omega}|e(u)+\frac{1}{2}\nabla w\otimes\nabla w-\frac{1}{2}\nabla p\otimes\nabla p|^{2}\,dx+\frac{B}{2}\int_{\Omega}|\nabla\nabla w-\nabla\nabla p|^{2}\,dx\\
 &\quad\qquad +\frac{K}{2}\int_{\Omega}|w|^{2}\,dx+\gamma_{\text{lv}}\left(\int_{\Omega}\frac{1}{2}|\nabla p|^{2}\,dx-\int_{\partial\Omega}u\cdot\hat{\nu}\,ds\right).
 \end{aligned}
 \end{equation} 
The notation $e(u)=\frac{1}{2}(\nabla u+\nabla u^{T})$ stands for the symmetrized gradient of the displacement $u$. We use $x\otimes y$ to denote the outer product of $x$ and $y$, and $\hat{\nu}$ for the outwards-pointing unit normal vector at $\partial\Omega$.  
Our formula for the energy is directly analogous to the one used in \cite{taffetani2017regimes} to study the wrinkling of an internally pressurized spherical shell, as well as the one used in \cite{bella2017wrinkling,davidovitch2019geometrically,hohlfeld2015sheet}
to study the wrinkling of a flat disc on a spherical substrate; it is a geometrically linearized version of the energy used in \cite{aharoni2017smectic} for general floating shells.
Here, to fix ideas, we focus on the setup of a weakly curved shell on a planar liquid substrate,
noting that our analysis can be adapted to the more general setup of a weakly curved shell on a weakly curved substrate. 
Underlying the energy $E$ is a certain ``geometric linearization'' procedure we shall describe. But first,
let us introduce each of the terms in \prettyref{eq:energy_dimensional}. 

The formula \prettyref{eq:energy_dimensional} accounts for the potential
energy of the shell and liquid bath. 
The first two terms are the ``stretching'' and ``bending'' terms. They comprise the elastic energy of the shell.
The parameters $Y=E_{\text{s}}t$
and $B=\frac{1}{12}E_{\text{s}}t^{3}$ are its stretching and bending moduli, where $E_{\text{s}}$ is its Young's modulus and $t$ is its dimensional
thickness. For simplicity, and as it contains the essential mathematical details, we treat the case of an isotropic Hooke's law with 
Poisson ratio $\nu=0$. That is, we take $|\cdot|$ to denote either  
the standard Euclidean or Frobenius matrix norm.  
With this choice, the stretching energy is proportional to the
sum of the squares of the components of the \emph{geometrically linear strain} 
\begin{equation}
\varepsilon=e(u)+\frac{1}{2}\nabla w\otimes\nabla w-\frac{1}{2}\nabla p\otimes\nabla p\label{eq:strain_vK}
\end{equation}
which it prefers to remain small. The bending energy is proportional to the sum of the squares of the
components of $\nabla\nabla w-\nabla\nabla p$. It limits the curvature
that develops. The remaining terms in \prettyref{eq:energy_dimensional}
account for the energy of the liquid bath. The parameter
$K=\rho g$ sets its ``stiffness'' to out-of-plane displacements ($\rho$ is the density of the liquid and $g$ is the gravitational acceleration), while $\gamma_{\text{lv}}$ sets
the strength of the liquid\textendash vapor surface tension 
pulling at the shell's edge. Note in treating only the surface tension of the liquid--vapor interface, we assume  the shell adheres completely to the surface of the bath (see  \cite{hohlfeld2015sheet} for more on this point).

Before non-dimensionalizing, we pause to discuss the fact that \prettyref{eq:energy_dimensional}
 does not report the true energy of the shell and liquid bath, but only approximates it to leading order in a ``geometrically
linear'' setting where
\begin{equation}
|\nabla u|\sim|\nabla w|^{2}\sim|\nabla p|^{2}\ll1.\label{eq:geometrically_linear_regime}
\end{equation}
The use of a more nonlinear model (``geometrically nonlinear'' as
in \cite{aharoni2017smectic} or ``fully nonlinear'' as in \cite{bella2014wrinkles})
would of course yield more accurate results, but would require several
significant mathematical advances beyond the ones achieved here. As
remarked above, we are not the first to make such a simplification
in the study of elastic patterns: other
authors including those of \cite{bella2017wrinkling,davidovitch2019geometrically,hohlfeld2015sheet,taffetani2017regimes}
have used geometrically linear models as well. The picture
that has emerged is that, whereas the quantitative predictions of such
models are only asymptotically correct, their qualitative predictions do often reflect those of a more nonlinear model. So while we expect
the analysis of \prettyref{eq:energy_dimensional} to reveal much
about the experiments that motivated this work, we warn that it
may fail to capture the parts of those experiments that are not weakly curved. The analysis
of general floating shells is the subject of current research.

To illustrate this point further, let us briefly indicate how the geometrically linear energy \prettyref{eq:energy_dimensional}
arises, informally, from a more nonlinear one. We focus
on the stretching term, as the rest can be explained similarly. As in
\cite{aharoni2017smectic,efrati2009elastic}, we note that the (geometrically)
nonlinear stretching energy of the shell is given by
\begin{equation}
E_{\text{stretch}}=\frac{Y}{2}\int_{S}|\varepsilon_{\text{NL}}|_{S}^{2}\,dA\label{eq:NLstretching}
\end{equation}
where $\varepsilon_{\text{NL}}$ is the strain of $\Phi$, $dA$ is the area element of $S$, and $|\cdot|_{S}$
is a suitable matrix norm. Pulling back to $\Omega$, we introduce the deformed and reference
metrics $g=D\Phi^{T}D\Phi$ and $g_{0}=D\Phi_{0}^{T}D\Phi_{0}$ where
$\Phi_{0}(x)=(x_{1},x_{2},p(x))$, and write 
\[
\varepsilon_{\text{NL}}=\frac{1}{2}(g-g_{0}),\quad dA=\sqrt{\det g_{0}}\,dx,\quad\text{and}\quad|\cdot|_{S}=|g_{0}^{-1}\cdot|.
\]
Taylor expanding about the trivial displacements $(u,w)=(0,0)$ and
the trivial shell $p=0$ yields 
\[
\varepsilon_{\text{NL}}=\varepsilon+\text{h.o.t.},\quad dA=dx+\text{h.o.t.},\quad\text{and}\quad|\cdot|_{S}=|\cdot|+\text{h.o.t.}
\]
where we have neglected higher order terms per \prettyref{eq:geometrically_linear_regime}.
Replacing $\varepsilon_{\text{NL}}$, $dA$, and $|\cdot|_{S}$ in \prettyref{eq:NLstretching}  
with their leading order approximations $\varepsilon$, $dx$, and $|\cdot|$ results in the stretching energy from \prettyref{eq:energy_dimensional}.

\subsubsection{Bendability, deformability, and confinement\label{subsec:non-dimensionalization}}

We are interested in the minimizers of the energy \prettyref{eq:energy_dimensional},
and especially in their dependence on its parameters.
Here, we collapse these into three non-dimensional
groups whose inverses are known as ``bendability'', ``deformability'',
and ``confinement''. Similar groups appear whenever elastic, surface tension, and  substrate forces interact, as has been shown in many other works including  \cite{davidovitch2011prototypical,hohlfeld2015sheet,king2012elastic,taffetani2017regimes}. 

Let $L$ be a representative lengthscale of the reference domain $\Omega$,
and let $R$ be a representative radius of curvature for the shell.
Consider the change of variables
\begin{equation}
u(x)=\epsilon^{2}L\hat{u}(\frac{x}{L}),\quad w(x)=\epsilon L\hat{w}(\frac{x}{L}),\quad\text{and}\quad p(x)=\epsilon L\hat{p}(\frac{x}{L})\quad\text{where}\quad\epsilon=\frac{L}{R}.\label{eq:non-dimensionalization}
\end{equation}
Hats denote dimensionless variables defined on the unit domain
$\hat{\Omega}=L^{-1}\Omega$. That the shell is weakly curved corresponds
to taking $\epsilon\ll1$. 
Setting \prettyref{eq:non-dimensionalization} into \prettyref{eq:energy_dimensional}
and changing variables, we find that 
\[
\hat{E}(\hat{u},\hat{w})=\frac{1}{YL^{2}}\frac{1}{\epsilon^{4}}E(u,w)
\]
satisfies
\begin{align*}
\hat{E} & =\frac{1}{2}\int_{\hat{\Omega}}|e(\hat{u})+\frac{1}{2}\nabla\hat{w}\otimes\nabla\hat{w}-\frac{1}{2}\nabla\hat{p}\otimes\nabla\hat{p}|^{2}+\frac{1}{2}\frac{B}{YL^{2}}\frac{1}{\epsilon^{2}}\int_{\hat{\Omega}}|\nabla\nabla\hat{w}-\nabla\nabla\hat{p}|^{2}\\
 & \quad\qquad+\frac{1}{2}\frac{KL^{2}}{Y}\frac{1}{\epsilon^{2}}\int_{\hat{\Omega}}|\hat{w}|^{2}+\frac{\gamma_{\text{lv}}}{Y}\frac{1}{\epsilon^{2}}\left(\int_{\hat{\Omega}}\frac{1}{2}|\nabla\hat{p}|^{2}-\int_{\partial\hat{\Omega}}\hat{u}\cdot\hat{\nu}\right).
\end{align*}
Evidently, minimizing $E$ is equivalent to minimizing $\hat{E}$,
but the latter version has the benefit of collapsing the six original
parameters $Y$, $B$, $K$, $\gamma_{\text{lv}}$, $L$, and $R$ into three non-dimensional groups:
\[
b=\frac{B}{Y L^2}\cdot\frac{R^{2}}{L^{2}}\quad\text{(bendability\ensuremath{^{-1}})},\quad k=\frac{KR^{2}}{Y}\quad\text{(deformability\ensuremath{^{-1}})},\quad\text{and}\quad\gamma=\frac{\gamma_{\text{lv}}}{Y}\cdot\frac{R^{2}}{L^{2}}\quad\text{(confinement\ensuremath{^{-1}})}.
\]
Henceforth, we drop the hats and consider the non-dimensionalized
energy
\begin{equation}\label{eq:non-dim_energy}
\begin{aligned}
E_{b,k,\gamma}(u,w)  &=\frac{1}{2}\int_{\Omega}|e(u)+\frac{1}{2}\nabla w\otimes\nabla w-\frac{1}{2}\nabla p\otimes\nabla p|^{2}\,dx+\frac{b}{2}\int_{\Omega}|\nabla\nabla w-\nabla\nabla p|^{2}\,dx\\
 &\quad\qquad +\frac{k}{2}\int_{\Omega}|w|^{2}\,dx+\gamma\left(\int_{\Omega}\frac{1}{2}|\nabla p|^{2}\,dx-\int_{\partial\Omega}u\cdot\hat{\nu}\,ds\right) 
\end{aligned}
\end{equation}
with its parameters $b,k>0$ and $\gamma\geq0$. 

Having non-dimensionalized, we can now introduce the asymptotic regime
of our results. This paper studies the asymptotics of $E_{b,k,\gamma}$
and its minimizers in any limit $b\to0$, $k\to\infty$, $\gamma\to0$
such that
\begin{equation}
\frac{b}{k},\ \frac{\gamma}{k},\ 2\sqrt{bk}+\gamma\ll1\quad\text{and}\quad\left(\frac{b}{k}\right)^{1/10}\ll2\sqrt{bk}+\gamma.\label{eq:asymptoticregime}
\end{equation}
These conditions arise from the search for a parameter regime where minimizers
satisfy
\begin{equation}\label{eq:nearlystrainfree-nearlyflat}
\varepsilon\approx0\quad\text{and}\quad w\approx0
\end{equation}
so that they are nearly strain-free and nearly flat. One expects this to hold if stretching and substrate forces dominate the response. As with isometric embeddings, there
exist infinitely many nearly strain-free displacements to any neighborhood of the plane. (We will construct such displacements later on. See also  \cite{lewicka2017convex} for the case $\varepsilon =0$.)
With so much freedom, it is reasonable to think of minimizing the
bending, substrate, and surface energies while treating \prettyref{eq:nearlystrainfree-nearlyflat} as a ``constraint''; this is
an instance of the ``Gauss\textendash Euler elastica'' variational
principle proposed recently in \cite{davidovitch2019geometrically}.
Following the line of reasoning there, one predicts the minimum energy to scale $\sim2\sqrt{bk}$
or $\gamma$, whichever is the larger. The typical values of the
stretching and substrate forces are $\sim1$ and $k$, so we are
lead to take  $2\sqrt{bk}+\gamma\ll1$ and $k$ as in the first part of \prettyref{eq:asymptoticregime}.

The last part of \prettyref{eq:asymptoticregime} is harder
to explain. It comes from the fact that in order to justify the claim that minimizers satisfy \prettyref{eq:nearlystrainfree-nearlyflat}, 
we must be able to prove the existence of in- and out-of-plane displacements satisfying
\[
\int_{\Omega}|\varepsilon|^{2}\ll \min\,E_{b,k,\gamma}\quad\text{and}\quad k\int_{\Omega}|w|^{2}\ll \min\,E_{b,k,\gamma}
\]
in a regime where the minimum energy is expected to scale $\sim2\sqrt{bk}+\gamma$. Furthermore, as we
intend to prove a $\Gamma$-convergence result, we must accomplish
this for any possible limiting in-plane displacement $u_{\text{eff}}$.
This is not a straightforward task, and it becomes all the more difficult
(perhaps eventually impossible) as $2\sqrt{bk}+\gamma\to0$. Our choice to impose the second part of \prettyref{eq:asymptoticregime}  arises from the 
details of our herringbone-based recovery sequences. See the discussion
following \prettyref{thm:gamma-lim} for more on this and \prettyref{sec:Gamma_limsup}
for the details.

\subsubsection{Functions of bounded deformation and bounded Hessian\label{subsec:Bounded-deformation-maps}}

The goals of this paper are to obtain and analyze the $\Gamma$-limit of $\frac{1}{2\sqrt{bk}+\gamma}E_{b,k,\gamma}$ in the parameter regime \prettyref{eq:asymptoticregime}. To this end, we make use of
the spaces of \emph{bounded deformation} and \emph{bounded Hessian}
functions
\begin{align*}
BD(\Omega) =\left\{ u\in L^{1}(\Omega;\mathbb{R}^{2}):e(u)\in\mathcal{M}(\Omega;\text{Sym}_{2})\right\} \quad\text{and}\quad
HB(\Omega) =\left\{ \varphi\in L^{1}(\Omega):\nabla\nabla\varphi\in\mathcal{M}(\Omega;\text{Sym}_{2})\right\} 
\end{align*}
where $\mathcal{M}(\Omega;\text{Sym}_{2})$ is the space of finite,
symmetric $2$-by-$2$ matrix-valued Radon measures on the given reference domain $\Omega\subset \mathbb{R}^2$. As these spaces 
 may not be immediately familiar to all, we recall their basic
properties and refer to \cite{demengel1989compactness,temam2018mathematical,temam1980functions} for more details.
The reader wishing to skip forward to our results should go to \prettyref{subsec:statementofresults}.

The spaces $BD(\Omega)$ and $HB(\Omega)$ are Banach spaces under the norms 
\[
||u||_{BD(\Omega)}=||u||_{L^{1}(\Omega)}+\int_{\Omega}|e(u)|_{1}\quad \text{and} \quad
||\varphi||_{HB(\Omega)}=||\varphi||_{L^{1}(\Omega)}+\int_{\Omega}|\nabla\nabla\varphi|_{1},
\]
where we define
\[
\int_{\Omega}|\mu|_{1}=\sup_{\substack{\sigma\in C_{c}(\Omega;\text{Sym}_{2})\\
|\sigma(x)|_{\infty}\leq1\ \forall\,x\in\Omega
}
}\,\int_{\Omega}\left\langle \sigma,\mu\right\rangle 
\]
for  $\mu\in\mathcal{M}(\Omega;\text{Sym}_{2})$. Although the norm $\int_{\Omega}|\mu|_{1}$ is equivalent to the more common total variation one $\sum_{ij}|\mu_{ij}|(\Omega)$, we use the former simply because  it appears in our results.
Note $\int_{\Omega}|\mu|_{1}=\text{tr}\,\mu(\Omega)$ if
$\mu\geq0$.
The natural injections $BD(\Omega)\hookrightarrow L^{2}(\Omega;\mathbb{R}^{2})$
and $HB(\Omega)\hookrightarrow C(\overline{\Omega})$ are continuous
in the strong topologies induced by the given norms. As it turns out,
$BD(\Omega)$ and $HB(\Omega)$ are dual spaces and so 
possess weak-$*$ topologies. Between these and the strong topologies
lie the so-called \emph{intermediate} topologies metrized by the distances
\[
||u-v||_{L^{1}(\Omega)}+\left|\int_{\Omega}|e(u)|_{1}-\int_{\Omega}|e(v)|_{1}\right| \quad \text{and} \quad
||\varphi-\psi||_{L^{1}(\Omega)}+\left|\int_{\Omega}|\nabla\nabla\varphi|_{1}-\int_{\Omega}|\nabla\nabla\psi|_{1}\right|.
\]
The trace maps $BD(\Omega)\to L^{1}(\partial\Omega,\mathcal{H}^{1})$,
$u\mapsto u|_{\partial\Omega}$ and $HB(\Omega)\to(C(\partial\Omega),L^{1}(\partial\Omega,\mathcal{H}^{1}))$,
$\varphi\mapsto(\varphi|_{\partial\Omega},\nabla\varphi|_{\partial\Omega})$
are intermediately continuous, and are defined by continuous extension from the intermediately dense set $C^{\infty}(\overline{\Omega};\mathbb{R}^{2})$. 
We often drop the notation $\cdot|_{\partial\Omega}$
when the meaning is clear, as in the integration-by-parts identities
\begin{equation}
\int_{\Omega}\left\langle \sigma,e(u)\right\rangle =-\int_{\Omega}\text{div}\,\sigma\cdot u\,dx+\int_{\partial\Omega}\left\langle \sigma,u\otimes\hat{\nu}\right\rangle ds,\quad
\int_{\Omega}\left\langle \sigma,\nabla\nabla\varphi\right\rangle =-\int_{\Omega}\text{div}\,\sigma\cdot\nabla\varphi\,dx+\int_{\partial\Omega}\left\langle \sigma,\nabla\varphi\otimes\hat{\nu}\right\rangle ds\label{eq:IBP1}
\end{equation}
which hold for all $\sigma\in C^{1}(\overline{\Omega};\text{Sym}_{2})$.

Finally, we introduce the quotient space $BD(\Omega)/\mathcal{R}$. Setting $\sigma=Id$
into the first identity in \prettyref{eq:IBP1} shows that $\int_{\partial\Omega}u\cdot\hat{\nu}$
is unchanged under the replacement $u\to u+r$ if $e(r)=0$.
By definition, 
\[
\mathcal{R} =\{ r\in BD(\Omega) : e(r) = 0\}
\]
is the space of \emph{linearly strain-free} maps. It consists of all maps $x\mapsto Rx+b$ where $R$ is
anti-symmetric and $b\in\mathbb{R}^{2}$. Although Korn's inequality
fails on $BD(\Omega)$, the Poincar\'e-type inequality
\[
\min_{r\in\mathcal{R}}\,||u-r||_{L^{1}(\Omega)}\lesssim_{\Omega}\int_{\Omega}|e(u)|_{1}\quad\forall\,u\in BD(\Omega)
\]
holds. Thus, $\int_{\Omega}|e(u)|_{1}$ defines a norm on the quotient space
\[
BD(\Omega)/\mathcal{R}=\left\{ u+r:u\in BD(\Omega),r\in\mathcal{R}\right\}
\]
under which it is a Banach space. 
By Banach--Alaoglu, norm-bounded subsets of $BD(\Omega)/\mathcal{R}$ are weakly-$*$ precompact. Note $u_{n}\overset{*}{\rightharpoonup}u$ weakly-$*$ in $BD(\Omega)/\mathcal{R}$ if and only if $e(u_{n})\overset{*}{\rightharpoonup}e(u)$
weakly-$*$ in $\mathcal{M}(\Omega;\text{Sym}_{2})$. In such a case, there exists $\{r_{n}\}_{n\in\mathbb{N}}\subset\mathcal{R}$
so that $u_{n}+r_{n}\to u$ strongly in $L^{1}(\Omega;\mathbb{R}^{2})$. 

\subsection{Statement and discussion of results}\label{subsec:statementofresults}

Having introduced the (non-dimensionalized) energies $E_{b,k,\gamma}$ in \prettyref{eq:non-dim_energy},
we proceed to state our results. We start in \prettyref{subsec:The-limiting-variational}
 by deriving the sought after effective energy $E_\text{eff}$ of the floating shell as the $\Gamma$-limit of the rescaled energies $\frac{1}{2\sqrt{bk} + \gamma} E_{b,k,\gamma}$. There we produce a first statement of the limiting problem in terms of the effective in-plane displacement $u_\text{eff}$. A second statement appears in \prettyref{subsec:Defect-measures} in terms of a new variable called the ``defect measure'' $\mu$. We think of it  as encoding the patterns. 
\prettyref{subsec:The-geometry-of} obtains a dual description via an ``Airy potential'' function $\varphi$, and produces a boundary value problem for optimal $\mu$ whose coefficients depend on an optimal choice of $\varphi$. Finally,  \prettyref{subsec:method-of-stable-lines}
presents our method of stable lines. 
For a short list of open questions, see \prettyref{subsec:openquestions}.

\vspace{.5em}

\paragraph*{\uline{Assumptions}. \label{par:Assumptions} \quad{}} 
Here we collect for the reader's convenience a list of assumptions that will reappear throughout. The following assumptions are basic to what we do:
\begin{subequations}\label{eq:A-basic}
\begin{gather}
\Omega\subset\mathbb{R}^2\text{ is a bounded, Lipschitz domain}\quad\text{and}\quad p\in W^{2,2}(\Omega)\label{eq:A1a} \\
\{(b_{n},k_{n},\gamma_{n})\}_{n\in\mathbb{N}}\subset(0,1]\times(0,\infty)\times[0,1]\quad\text{satisfies}\quad\frac{b_{n}}{k_{n}},\ \frac{\gamma_{n}}{k_{n}},\ 2\sqrt{b_{n}k_{n}}+\gamma_{n}\to0\quad\text{as }n\to\infty.\label{eq:A2a}
\end{gather}
\end{subequations}
Certain of our results require the following additional assumptions to hold: 
\begin{subequations}\label{eq:A-additional}
\begin{gather}
\Omega \text{ is strictly star-shaped}\quad\text{and}\quad p\in W^{2,\infty}(\Omega) \label{eq:A1b}\\ 
\frac{(b_{n}/k_{n})^{1/10}}{2\sqrt{b_{n}k_{n}}+\gamma_{n}}\to0\quad\text{as }n\to\infty. \label{eq:A2b}
\end{gather}
\end{subequations}
Unless otherwise stated, any asymptotic statement involving $b$, $k$, or $\gamma$ is understood to hold on a sequence satisfying \prettyref{eq:A2a} and \prettyref{eq:A2b}. We often mute the subscript $n$. 
Recall $\Omega$ is said to be  \emph{strictly star-shaped} if there exists $x\in\Omega$ so that for all $y\in\partial\Omega$ the open line segment from $x$ to $y$ belongs to $\Omega$. Sometimes, we make use of the hypothesis that $\Omega$ is simply connected to simplify the statements of certain results.

These and other assumptions enter at various steps in our analysis. Briefly, the situation is as follows: while for our complete $\Gamma$-convergence result  we must impose all of the assumptions in \prettyref{eq:A-basic} and \prettyref{eq:A-additional}, each of its components hold in greater generality; so do our results regarding the analysis of the limiting problems. To help the reader navigate, we have included statements at the top of \prettyref{sec:Gamma_liminf}-\prettyref{sec:examples} clarifying the set of assumptions that are needed there.

\subsubsection{The limiting area problem\label{subsec:The-limiting-variational}}

Our first result is a formula for the effective energy $E_\text{eff}$ of a weakly curved, floating shell along with the limiting (linearized) area problem it implies. 
Anticipating the minimum energy to scale $\sim2\sqrt{bk}+\gamma$, we divide by this amount and pass to the limit in the sense of $\Gamma$-convergence.
As usual, we fix the admissible set and extend the energies $E_{b,k,\gamma}:BD(\Omega)\times W^{1,2}(\Omega)\to(-\infty,\infty]$ by taking
\[
E_{b,k,\gamma}(u,w)=\begin{cases}
\prettyref{eq:non-dim_energy} & (u,w)\in W^{1,2}(\Omega)\times W^{2,2}(\Omega)\\
\infty & \text{otherwise}
\end{cases}.
\]
Define $E_{\text{eff}}:BD(\Omega)\times W^{1,2}(\Omega)\to(-\infty,\infty]$ by
\[
E_{\text{eff}}(u,w)=\begin{cases}
\int_{\Omega}\frac{1}{2}|\nabla p|^{2}\,dx-\int_{\partial\Omega}u\cdot\hat{\nu}\,ds & e(u)\leq\frac{1}{2}\nabla p\otimes\nabla p\,dx,\ w=0\\
\infty & \text{otherwise}
\end{cases}
\]
where $\mu\leq\tilde{\mu}$ if $\tilde{\mu}-\mu \in\mathcal{M}_{+}(\Omega;\text{Sym}_{2})$, the space of finite, non-negative, $\text{Sym}_{2}$-valued Radon measures on $\Omega$.   
\begin{theorem}
\label{thm:gamma-lim} Let $\Omega$, $p$, and $\{(b,k,\gamma)\}$ satisfy the assumptions \prettyref{eq:A-basic} and \prettyref{eq:A-additional}. The $\Gamma$-convergence
\[
\frac{1}{2\sqrt{bk}+\gamma}E_{b,k,\gamma}\overset{\Gamma}{\longrightarrow}E_{\emph{{eff}}}\quad\text{holds with respect to the weak-}*\text{ }BD(\Omega)/\mathcal{R}\times W^{1,2}(\Omega)\text{ topology}
\]
and the rescaled energies are equi-coercive on that space. More precisely,
we have the following results:
\begin{enumerate}
\item ($\Gamma$-liminf inequality) Given any weakly-$*$ converging 
sequence
\[
(u_{b,k,\gamma},w_{b,k,\gamma})\stackrel{*}{\rightharpoonup}(u,w)\quad\text{weakly-\ensuremath{*} in }BD(\Omega)/\mathcal{R}\times W^{1,2}(\Omega),
\]
there holds
\[
\liminf\,\frac{E_{b,k,\gamma}(u_{b,k,\gamma},w_{b,k,\gamma})}{2\sqrt{bk}+\gamma}\geq E_{\emph{{eff}}}(u,w);
\]
\item (recovery sequences) Given any $(u,w)\in BD(\Omega)\times W^{1,2}(\Omega)$,
there exists a sequence
\[
(u_{b,k,\gamma},w_{b,k,\gamma})\stackrel{*}{\rightharpoonup}(u,w)\quad\text{weakly-\ensuremath{*} in }BD(\Omega)/\mathcal{R}\times W^{1,2}(\Omega)
\]
such that
\[
\lim\,\frac{E_{b,k,\gamma}(u_{b,k,\gamma},w_{b,k,\gamma})}{2\sqrt{bk}+\gamma}=E_{\emph{{eff}}}(u,w);
\]
\item (equi-coercivity) Any sequence $\{(u_{b,k,\gamma},w_{b,k,\gamma})\}$
that satisfies
\[
\limsup\,\frac{E_{b,k,\gamma}(u_{b,k,\gamma},w_{b,k,\gamma})}{2\sqrt{bk}+\gamma}<\infty
\]
admits a sub-sequence that converges weakly-$*$ in $BD(\Omega)/\mathcal{R}\times W^{1,2}(\Omega)$.
\end{enumerate}
\end{theorem}
\begin{remark} At first glance, it may seem surprising that the space $BD$, which was originally introduced in connection with plasticity (see, e.g., \cite{temam2018mathematical}), should arise in a problem devoid of plastic effects. It can, however, be anticipated on the grounds that our energies are geometrically linear. As we expect patterns to form, it is natural that $\nabla w$  should be bounded \emph{a priori} in $L^2$. The scaling $|e(u)|\sim|\nabla w|^2$ then indicates a bound on $e(u)$ in $L^1$, implying weak-$*$ pre-compactness in $BD$ up to a linearly strain-free map. The equi-coercivity result above justifies these claims. A similar observation was made in \cite{conti2006rigorous} where $BD$ appeared in a $\Gamma$-limit analysis of clamped elastic membranes.
\end{remark}
It is well-known that $\Gamma$-convergence combined with equi-coercivity
implies the convergence of minimum values along with minimizers \cite{dalmaso1993introduction,degiorgi1975gammaconv}. Here,
we deduce that
\[
\lim\,\frac{\min\,E_{b,k,\gamma}}{2\sqrt{bk}+\gamma}=\min_{(u_{\text{eff}},w_{\text{eff}})\in BD(\Omega)\times W^{1,2}(\Omega)}\,E_{\text{eff}}(u_{\text{eff}},w_{\text{eff}}).
\]
Furthermore, the displacements $(u_{\text{eff}},w_{\text{eff}})$
appearing on the righthand side are optimal if and only if they are
the weak-$*$ limit of a sequence of \emph{almost minimizers} $\{(u_{b,k,\gamma},w_{b,k,\gamma})\}$ of $E_{b,k,\gamma}$. Such sequences satisfy
\[
E_{b,k,\gamma}(u_{b,k,\gamma},w_{b,k,\gamma})=\min\,E_{b,k,\gamma}+o(2\sqrt{bk}+\gamma)
\]
by definition. Reducing to the finite part of $E_{\text{eff}}$ yields the following
result:
\begin{corollary}
\label{cor:limiting_pblm_effectivedisplacement} Given the assumptions \prettyref{eq:A-basic} and \prettyref{eq:A-additional}, the rescaled minimum
energies satisfy
\begin{equation}
\lim\,\frac{\min\,E_{b,k,\gamma}}{2\sqrt{bk}+\gamma}=\min_{\substack{u_{\emph{eff}}\in BD(\Omega)\\
e(u_{\emph{eff}})\leq\frac{1}{2}\nabla p\otimes\nabla p\,dx
}
}\,\int_{\Omega}\frac{1}{2}|\nabla p|^{2}\,dx-\int_{\partial\Omega}u_{\emph{eff}}\cdot\hat{\nu}\,ds.\label{eq:limitingpblm}
\end{equation}
Furthermore, $(u_{\emph{eff}},w_{\emph{eff}})$ arises as the weak-$*$
$BD(\Omega)/\mathcal{R}\times W^{1,2}(\Omega)$ limit of almost minimizers of $E_{b,k,\gamma}$ if and only if $u_{\emph{eff}}$ solves the limiting
problem on the righthand side, and
$w_{\emph{eff}}=0$. 
\end{corollary}
We wish to make two remarks, on the geometric meaning of the limiting
problem just derived, and on the proof of the $\Gamma$-convergence result. 
First, we demonstrate how the limiting problem in \prettyref{cor:limiting_pblm_effectivedisplacement} can be recovered by linearizing the proposed area problem \prettyref{eq:effectiveArea}. Just as we may associate to a deformation $\Phi:S\to\mathbb{R}^{3}$
the in- and out-of-plane displacements $u$ and $w$, we may
associate to the limiting or \emph{effective} displacements $u_{\text{eff}}$
and $w_{\text{eff}}=0$ the effective deformation $\Phi_{\text{eff}}:S\to\mathbb{R}^{3}$
given by
\[
\Phi_{\text{eff}}\left(x,p(x)\right)=\left(x+u_{\text{eff}}(x),0\right),\quad x\in\Omega.
\]
Whereas the area of the undeformed mid-shell $S$ satisfies
\[
A(S)=\int_{\Omega}\sqrt{1+|\nabla p|^{2}}\,dx=|\Omega|+\int_{\Omega}\frac{1}{2}|\nabla p|^{2}\,dx+\text{h.o.t.}
\]
the area of its image under $\Phi_{\text{eff}}$ satisfies
\[
A(\Phi_{\text{eff}}(S))=\int_{S}\sqrt{\det D\Phi_{\text{eff}}^T D\Phi_\text{eff}} \, dA=|\Omega|+\int_{\Omega}\text{div}\,u_{\text{eff}}\,dx+\text{h.o.t.}
\]
to leading order in $\nabla p$ and $\nabla u_\text{eff}$. (This ignores the possibility that $\Phi_\text{eff}$ may not be one-to-one.) 
Subtracting and applying the divergence theorem yields the expansion
\begin{equation}
A(S)-A(\Phi_{\text{eff}}(S))=\int_{\Omega}\frac{1}{2}|\nabla p|^{2}\,dx-\int_{\partial\Omega}u_{\text{eff}}\cdot\hat{\nu}\,ds+\text{h.o.t.}\label{eq:Taylorexp_area}
\end{equation}
Similarly, the one-sided constraint 
\[
e(u_\text{eff}) \leq \frac{1}{2} \nabla p\otimes \nabla p \,dx 
\]
from \prettyref{eq:limitingpblm}
--- which we refer to henceforth as the statement that $u_{\text{eff}}$
is \emph{(linearly) tension-free}  ---  can be recovered by linearizing the statement that $\Phi_{\text{eff}}$ is \emph{short}, i.e.,
\[
d_{\mathbb{R}^{2}}\left(\Phi_{\text{eff}}(x),\Phi_{\text{eff}}(y)\right)\leq d_{S}\left(x,y\right)\quad\forall\,x,y \in S.
\]
In this way, the limiting problem from \prettyref{cor:limiting_pblm_effectivedisplacement} manifests as the leading order part of the geometric variational problem
\[
\min_{\substack{\Phi_{\text{eff}}:S\to\mathbb{R}^{2}\\
\text{that are short}
}
}\,A(S)-A(\Phi_{\text{eff}}(S))
\]
which asks to cover up as much area as possible with a length-shortening map of $S$ to the plane. 

We turn to discuss the key ingredients in the proof of \prettyref{thm:gamma-lim}. It requires establishing \emph{a priori} lower bounds on
$E_{b,k,\gamma}$, and verifying that they are asymptotically sharp. Behind
the $\Gamma$-liminf part is a sort of ``geometric interpolation inequality'' that quantifies the
fact that two regular enough embedded surfaces cannot be both extrinsically close and intrinsically far. Here, the
surfaces in question are those of the nearly isometrically deformed
mid-shell $\Phi(S)$, and of its projection to the plane. In terms of the displacements $u$ and $w$,
the inequality states that 
\begin{equation}
\left(\int_{\Omega}|\frac{1}{2}\Delta w|^{2}\right)^{1/2}\left(\int_{\Omega}|w|^{2}\right)^{1/2}\geq\int_{\Omega}\frac{1}{2}|\nabla p|^{2}-\int_{\partial\Omega}u\cdot\hat{\nu}+\text{h.o.t.}\label{eq:geometric_interpolation_inequality}
\end{equation}
whenever $\varepsilon\approx0$ and $w\approx0$.
On the lefthand side we see a trade-off between the linearized
mean curvature $H\approx\frac{1}{2}\Delta w$ and the out-of-plane
displacement $w$. On the righthand side we recognize from \prettyref{eq:Taylorexp_area}
the difference between the intrinsic and planar projected areas of the shell.
Taking the trace of the statement that $\varepsilon\approx0$ we see
from \prettyref{eq:strain_vK} that 
\[
\text{div}\,u+\frac{1}{2}|\nabla w|^{2}\approx\frac{1}{2}|\nabla p|^{2}.
\]
Thus, \prettyref{eq:geometric_interpolation_inequality} reminds of
the classic Gagliardo--Nirenberg interpolation inequality
\begin{equation}
C\left(\int_{\Omega}|\nabla\nabla w|^{2}\right)^{1/2}\left(\int_{\Omega}|w|^{2}\right)^{1/2}\geq\int_{\Omega}|\nabla w|^{2}+\text{h.o.t.}\label{eq:GNinequality}
\end{equation}
which holds for $w\approx 0$ and independently of the strain (see, e.g., \cite{gilbarg2001elliptic}). While \prettyref{eq:GNinequality}
implies the equi-coercivity part of \prettyref{thm:gamma-lim},
it is not strong enough to establish its $\Gamma$-liminf part.  
Thinking of replacing the full Hessian $\nabla\nabla w$ with $\Delta w$, which is justified when $\varepsilon \approx 0$, we were led
to its sharpened form \prettyref{eq:geometric_interpolation_inequality}. Though we are certainly not
the first to apply a Gagliardo--Nirenberg interpolation inequality to the study of  elastic patterns ---
such inequalities play an organizing role throughout the subject of energy-driven
pattern formation  \cite{kohn2007energy} ---
we know of only one other analysis of wrinkling in which an
optimal prefactor is known \cite{bella2015transition}. The suggestion
that the geometric interpolation inequality \prettyref{eq:geometric_interpolation_inequality}
should be used in place of \prettyref{eq:GNinequality} appears to
be new. See \prettyref{sec:Gamma_liminf} for more details.

 Much of our work is devoted to the construction of recovery sequences
verifying the optimality of our lower bounds. Given any candidate
tension-free displacement $u_{\text{eff}}$, we construct in \prettyref{sec:Gamma_limsup}
an admissible sequence $\{(u_{b,k,\gamma},w_{b,k,\gamma})\}$ converging
weakly-$*$ to $(u_{\text{eff}},0)$ and whose energy satisfies
\[
E_{b,k,\gamma}(u_{b,k,\gamma},w_{b,k,\gamma})=(2\sqrt{bk}+\gamma)\left(\int_{\Omega}\frac{1}{2}|\nabla p|^{2}-\int_{\partial\Omega}u\cdot\hat{\nu}\right)+O\left((\frac{b}{k})^{1/10}\right).
\]
The out-of-plane parts of three such constructions are depicted in
\prettyref{fig:three_constructions}. Their essential character is
given by 
\begin{equation}
w(x)=\sqrt{2\text{tr}\left\langle \varepsilon_{\text{eff}}\right\rangle }\cdot l_\text{wr}\cos\left(\frac{x\cdot\hat{\eta}_\text{herr}(x)}{l_\text{wr}}\right)\label{eq:ansatz}
\end{equation}
where we denote the \emph{effective strain} of $u_{\text{eff}}$ by 
\begin{equation}
\varepsilon_{\text{eff}}=e(u_{\text{eff}})-\frac{1}{2}\nabla p\otimes\nabla p\,dx.\label{eq:effstrain}
\end{equation}
We envision a ``piecewise
herringbone'' pattern consisting of multiple 
herringbones, one of which appears in each bold square in Panel (a) of \prettyref{fig:three_constructions}.
Herringbones are made of twinned uni-directional wrinkles superimposed on alternating bands of in-plane shear. We select them as our basic building blocks as they are highly effective at accommodating
constant bi-axial compressive strains \cite{kohn2013analysis}. Simply put, our
idea is that with enough herringbones, one should be
able accommodate any non-constant $\varepsilon_{\text{eff}}$ ---
even the measure-valued ones in \prettyref{thm:gamma-lim}. 

\begin{figure}
\centering
\subfloat[]{\includegraphics[width=0.33\paperwidth]{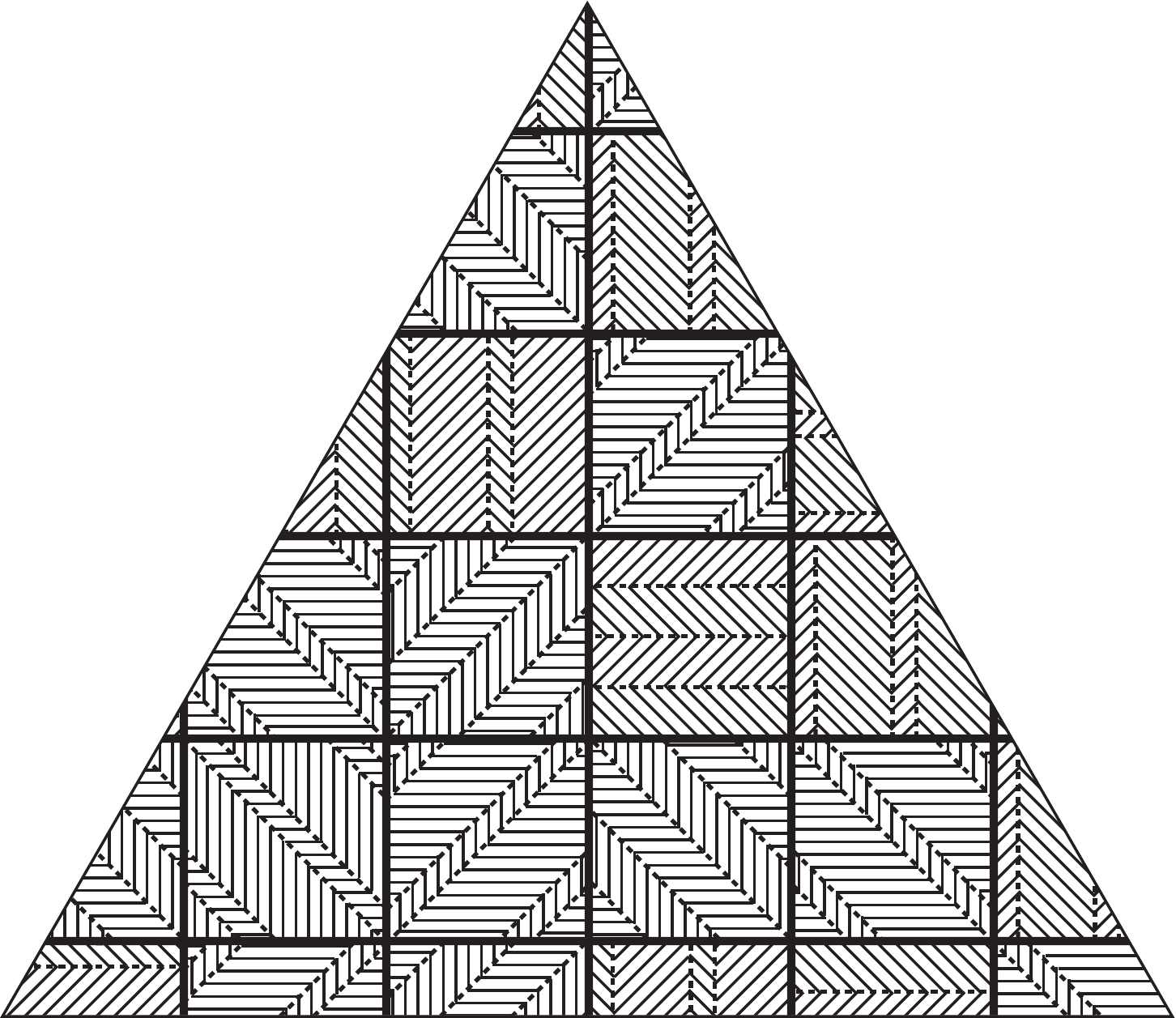}}\hspace{.10\textwidth}\subfloat[]{\includegraphics[width=0.33\paperwidth]{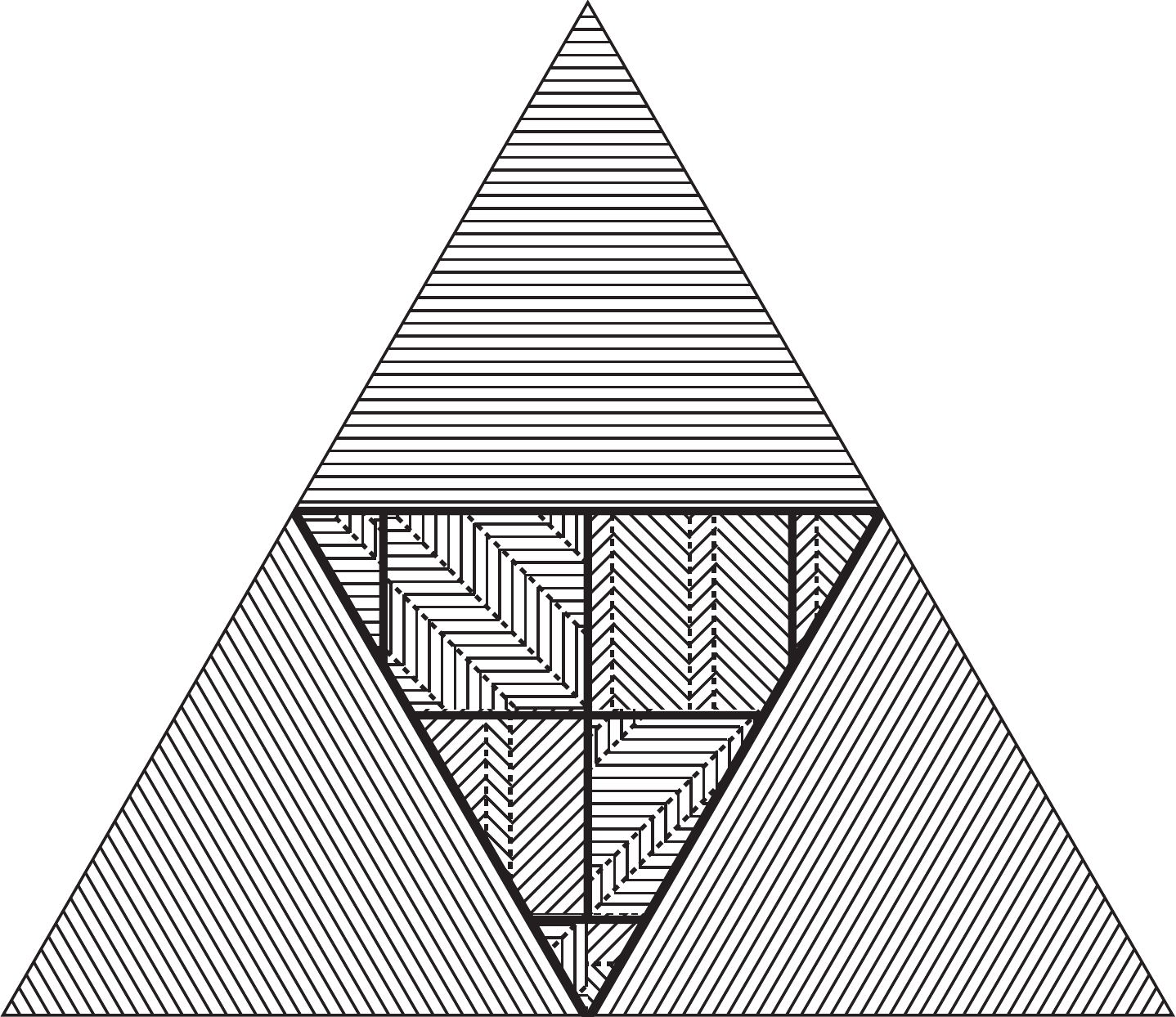}}\hspace{.10\textwidth}\subfloat[]{\includegraphics[width=0.33\paperwidth]{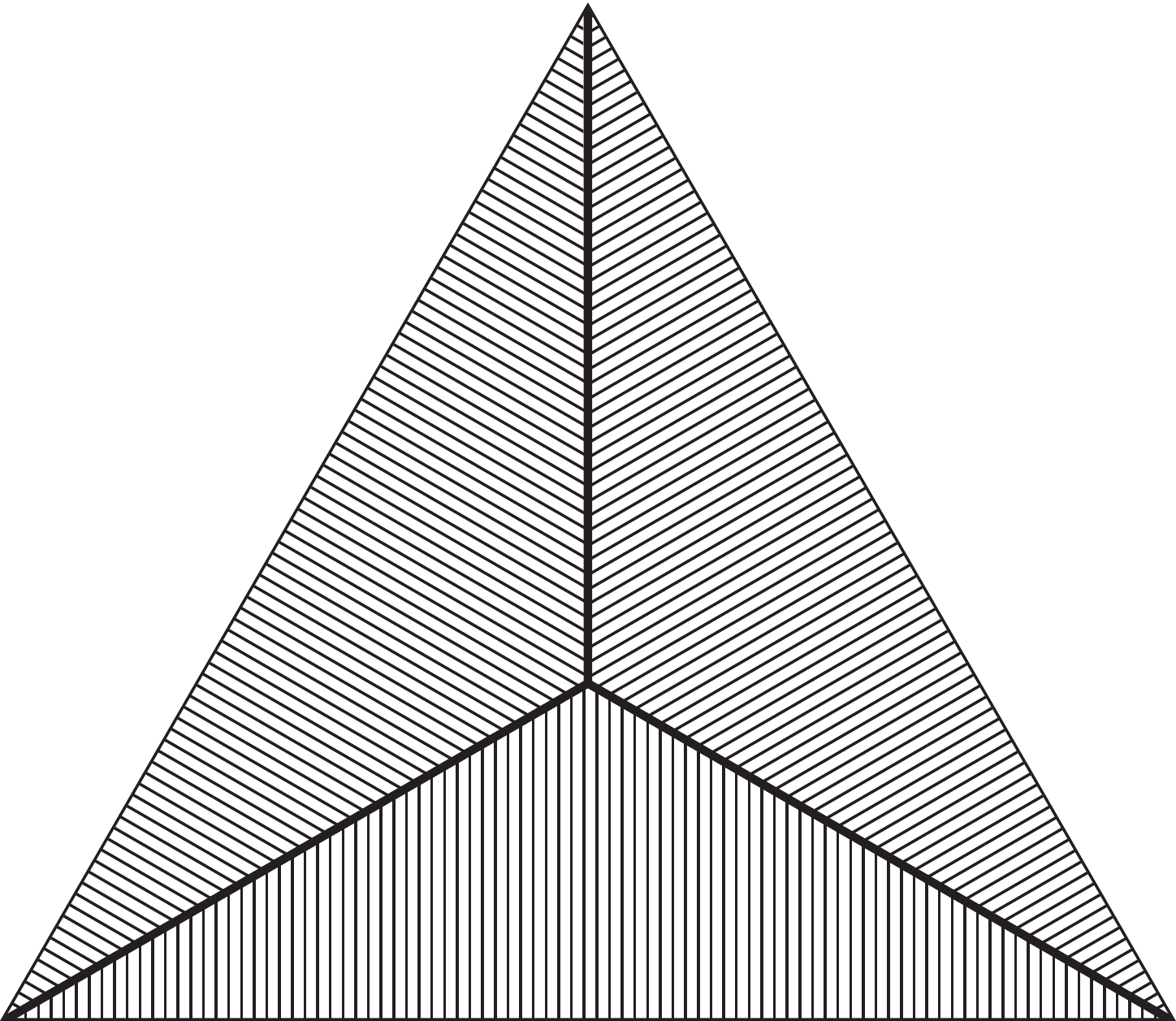}}
\caption{Three possible patterns formed by a floating triangular shell. 
Panel (a) depicts the ``piecewise herringbone'' pattern we use to
construct arbitrary recovery sequences. It consists of multiple herringbones,
one per square, each of which is made up of twinned wrinkles and alternating
in-plane shear. Panel (b) depicts an almost minimal pattern in the positively curved case, consisting of ordered, uni-directional wrinkles
and a piecewise herringbone to model its disordered part. Panel (c) depicts an almost minimal pattern in the negatively curved case. The emergence of ordered ``wrinkle domains'' such as in (b) and (c) will be shown to follow from the principle of minimum energy.\label{fig:three_constructions}}
\end{figure}

At the smallest scales, the ansatz \prettyref{eq:ansatz} features
 uni-directional wrinkles at a lengthscale $l_\text{wr}$
and in the direction of $\hat{\eta}_\text{herr}$. Other
larger lengthscales include one associated with the size of the individual
herringbones (also with the ``averaging'' operator $\left\langle \cdot\right\rangle $
through which $\varepsilon_{\text{eff}}$ will be approximated as 
piecewise constant), and one associated with the oscillations in $\hat{\eta}_\text{herr}$ representing the wrinkle twins.
Before moving on, we would like to motivate the locally sinusoidal
character of our piecewise herringbones --- which are closer to the herringbones treated in \cite{audoly2008buckling_b,kohn2013analysis} than, say, the origami-based ones in \cite{audoly2008buckling_c} --- on the grounds of energy minimization. Consider what it
takes for the two sides of the geometric interpolation inequality
\prettyref{eq:geometric_interpolation_inequality} to be nearly the
same: in its additive form, the inequality can be improved to say
that 
\[
\frac{b}{2}\int_{\Omega}|\Delta w|^{2}+\frac{k}{2}\int_{\Omega}|w|^{2}-2\sqrt{bk}\left(\int_{\Omega}\frac{1}{2}|\nabla p|^{2}-\int_{\partial\Omega}u\cdot\hat{\nu}\right)\geq\frac{1}{2}\int_{\Omega}|b^{1/2}\Delta w+k^{1/2}w|^{2}+\text{h.o.t.}
\]
so long as $\varepsilon\approx0$ and $w\approx0$. Hence, minimizers must
satisfy 
\[
-\Delta w\approx\sqrt{\frac{k}{b}}w
\]
consistent with the locally sinusoidal character of \prettyref{eq:ansatz}.
Note this also explains the choice $l_\text{wr}=(b/k)^{1/4}$ we will eventually make. 
It is well-appreciated in the literature on elastic pattern formation that such a lengthscale should  
emerge from a competition between bending and substrate effects (see, e.g.,
\cite{cerda2003geometry}). We refer to \prettyref{sec:Gamma_liminf} for more
on the geometric interpolation inequality and \prettyref{sec:Gamma_limsup}
for the details of our piecewise herringbones.

\subsubsection{Defect measures\label{subsec:Defect-measures}}

Thus far, our results have focused on the effective displacements that arise as limits of almost minimizers of $E_{b,k,\gamma}$. 
As explained in \prettyref{cor:limiting_pblm_effectivedisplacement}, these can be found by solving the  (linearized) area problem 
\begin{equation}
\min_{\substack{u_{\text{eff}}\in BD(\Omega)\\
e(u_{\text{eff}})\leq\frac{1}{2}\nabla p\otimes\nabla p\,dx
}
}\,\int_{\Omega}\frac{1}{2}|\nabla p|^{2}\,dx-\int_{\partial\Omega}u_{\text{eff}}\cdot\hat{\nu}\,ds\label{eq:linearized_area_problem}
\end{equation}
for the effective in-plane displacement $u_\text{eff}$, and recalling that $w_\text{eff}= 0$. 
In light of our previous discussion of the experiments on floating shells from \cite{aharoni2017smectic,albarran2018curvature,tobasco2020principles}, the reader may wonder whether solving \prettyref{eq:linearized_area_problem} actually recovers the observed wrinkle domains and possibly disordered parts.  
Indeed, deducing these is the goal of the rest of our results, which culminate in \prettyref{subsec:method-of-stable-lines} with our method of stable lines.
We start by rewriting \prettyref{eq:linearized_area_problem} as a minimization over the effective strain $\varepsilon_{\text{eff}}$ from \prettyref{eq:effstrain} or, as we prefer to think of it, over a quantity called the ``defect measure''. 

Defect measures are a basic tool for encoding the properties of high-frequency oscillations (and concentrations) governed by PDEs  \cite{lions1985concetration-I,lions1985concetration-II}. We define them in the present context as follows. 
Whenever a sequence $\{(u_n,w_n)\}_{n\in\mathbb{N}}$ converges
weakly-$*$ in $BD(\Omega)/\mathcal{R}\times W^{1,2}(\Omega)$ to $(u_{\text{eff}},0)$
and satisfies 
\begin{equation}\label{eq:asymptotically-strain-free}
e(u_n)+\frac{1}{2}\nabla w_n\otimes\nabla w_n\to\frac{1}{2}\nabla p\otimes\nabla p\quad\text{strongly in }L^{2}(\Omega)
\end{equation}
so that it is \emph{asymptotically strain-free}, we may associate
to it a non-negative $\text{Sym}_{2}$-valued \emph{defect
measure} 
\begin{equation}
\mu=\text{weak-\ensuremath{*}}\lim\,\nabla w_n\otimes\nabla w_n\,dx\quad\text{in }\mathcal{M}(\Omega;\text{Sym}_{2}).\label{eq:defect_measure_firstdefn}
\end{equation}
Taking limits, we deduce the important identity
\begin{equation}
e(u_{\text{eff}})+\frac{1}{2}\mu=\frac{1}{2}\nabla p\otimes\nabla p\,dx\label{eq:limiting_strain_free}
\end{equation}
which couples $\mu$ back to $u_{\text{eff}}$ thus guaranteeing it is well-defined. In particular, the limit in \prettyref{eq:defect_measure_firstdefn} holds \emph{a posteriori} since any converging sub-sequence must yield the same result (for a complete proof, see \prettyref{lem:defectmeasure}). 
Combining this with \prettyref{eq:effstrain}, we see that
\[
\mu=-2\varepsilon_{\text{eff}}\quad \text{where}\quad \varepsilon_\text{eff} = e(u_\text{eff})-\frac{1}{2}\nabla p\otimes \nabla p\,dx.
\]
Evidently, solving for the defect measure of a given sequence is tantamount to finding its effective strain.

Some examples are in order. Consider a uni-directional wrinkling pattern with lengthscale $l_\text{wr} \ll 1$ and constant direction $\hat{\eta}$. Thinking that the out-of-plane part should satisfy 
\[
w(x) = \sqrt{2} l_\text{wr}\cos\left(\frac{x\cdot\hat{\eta}}{l_{\text{wr}}}\right) \quad\text{yields the defect measure}\quad \mu = \hat{\eta}\otimes\hat{\eta}\,dx
\]
as $l_{\text{wr}} \to 0$. The same measure results for non-constant $\hat{\eta}$, so long as its variations are sufficiently mild. If $\hat{\eta}$  varies rapidly, as it does for the piecewise herringbone patterns in \prettyref{eq:ansatz},  $\mu$ can end up being rank two. Folds with various profiles can also be handled. Let $l_\text{f} \ll 1$ and fix $\hat{\eta}$. Taking 
\[
w(x)=\frac{l^{\frac{1}{2}}_{\text{f}}}{(2\pi)^{\frac{1}{4}}}e^{-\frac{1}{2}\left(\frac{x\cdot\hat{\eta}}{l_{\text{f}}}\right)^{2}} \quad\text{yields the defect measure}\quad \mu = \hat{\eta}\otimes\hat{\eta}\,\mathcal{H}^1\lfloor\{x:x\cdot \hat{\eta} = 0\}
\]
as $l_\text{f} \to 0$. The notation on the right indicates the restriction of the one-dimensional Hausdorff measure $\mathcal{H}^1$ to the given line.
The motifs of wrinkles and folds are ubiquitous in thin elastic sheets \cite{brau2013wrinkle,pocivavsek2008stress}. We propose to model them using defect measures in the vanishing thickness limit. 

Returning to the context of weakly curved, floating shells, we now change variables from $u_\text{eff}$ to $\mu$. We do so by identifying the set of defect measures associated to the recovery sequences from \prettyref{thm:gamma-lim}. 
In fact, all recovery sequences are asymptotically strain-free (see \prettyref{sec:Gamma_liminf}). Hence, $\mu\in \mathcal{M}(\Omega;\text{Sym}_{2})$ arises as the defect measure of a recovery sequence if and only if  \prettyref{eq:limiting_strain_free} holds for some tension-free $u_\text{eff}\in BD(\Omega)$. That $u_\text{eff}$ is tension-free is equivalent to the statement that $\mu \geq 0$.
Recall the Saint-Venant compatibility conditions
which state, for simply connected domains, that a $\text{Sym}_{2}$-valued
matrix field $m$ is a linear strain, i.e., $m=e(u)$ for some $u$ if and only if
\begin{equation}
\partial_{11}m_{22}+\partial_{22}m_{11}-2\partial_{12}m_{12}=0. \label{eq:SVcompat}
\end{equation}
That this holds in the smooth setting appears in standard references on elasticity  (see, e.g., \cite{love1944treatise}). By a straightforward
approximation argument, it also holds when $m\in\mathcal{M}$
and $u\in BD$. Denote 
\begin{equation}
\text{curl}\text{curl}\,m=\partial_{11}m_{22}+\partial_{22}m_{11}-2\partial_{12}m_{12} \label{eq:curlcurl}
\end{equation}
and observe the ``very weak Hessian'' identity 
\begin{equation}
-\frac{1}{2}\,\text{curl}\text{curl}\,\nabla w\otimes\nabla w=\det\nabla\nabla w,\label{eq:veryweak_hessian}
\end{equation}
so named as it defines $\det\nabla\nabla w$
even for $w\in W^{1,2}$ \cite{iwaniec2001concept,lewicka2017convex}. 
Combining \prettyref{eq:SVcompat} and \prettyref{eq:veryweak_hessian} yields the following fact: 
provided $\Omega$ is simply connected, there exists $u\in BD(\Omega)$
satisfying \prettyref{eq:limiting_strain_free} if and only if $\mu\in\mathcal{M}(\Omega;\text{Sym}_{2})$
satisfies
\begin{equation}
-\frac{1}{2}\text{curl}\text{curl}\,\mu=\det\nabla\nabla p\label{eq:PDEmu}
\end{equation}
in the sense of distributions. Therefore, we may exchange the set of admissible
$u_{\text{eff}}$ from \prettyref{eq:linearized_area_problem}
with the new set of admissible defect measures $\mu$ characterized
by their non-negativity and the PDE \prettyref{eq:PDEmu}. To finish the exchange, note the identity 

\[
\frac{1}{2}\int_{\Omega}|\mu|_{1}=\int_{\Omega}\frac{1}{2}|\nabla p|^{2}\,dx-\int_{\partial\Omega}u_{\text{eff}}\cdot\hat{\nu}\,ds
\]
which follows from \prettyref{eq:limiting_strain_free} upon
integrating its trace. The following result is proved:
\begin{corollary}
\label{cor:limiting_pblm_defect} Given the assumptions  \prettyref{eq:A-basic} and \prettyref{eq:A-additional}, the rescaled minimum energies satisfy
\begin{equation}
\lim\,\frac{\min\,E_{b,k,\gamma}}{2\sqrt{bk}+\gamma}=\min_{\substack{\mu\in\mathcal{M}_{+}(\Omega;\emph{Sym}_{2})\\
-\frac{1}{2}\emph{curl}\emph{curl}\,\mu=\det\nabla\nabla p
}
}\,\frac{1}{2}\int_{\Omega}|\mu|_{1}.\label{eq:limitingpblm-defect}
\end{equation}
Furthermore, $\mu$ arises as the defect measure of almost minimizers of $E_{b,k,\gamma}$ 
if and only if it solves the limiting problem on the righthand side.
\end{corollary}
The limiting problems in \prettyref{cor:limiting_pblm_effectivedisplacement}
and \prettyref{cor:limiting_pblm_defect} are two sides of the same
coin: whereas \prettyref{cor:limiting_pblm_effectivedisplacement}
determines the limiting displacement of the shell via optimal $u_{\text{eff}}$, \prettyref{cor:limiting_pblm_defect}
determines the limiting features of its patterns via optimal $\mu$. 
It should be noted that defect measures play a similar role for curvature-driven wrinkling to that of the ``wrinkling strain'' identified in \cite{pipkin1994relaxed} for tension-driven wrinkling. Both quantities specify how much material must be ``gotten rid of'' in an appropriate limit. However, $\mu$ does not derive from finding the relaxation of a fixed energy density, but rather from identifying the $\Gamma$-limit of a sequence of suitably rescaled energies. 

Before proceeding to discuss the optimizers of the limiting problems in detail, we pause to answer the question of whether the minimum energy actually scales $\sim 2\sqrt{bk} + \gamma$, under the assumptions \vpageref{par:Assumptions}. Rearranging \prettyref{eq:limitingpblm}
or \prettyref{eq:limitingpblm-defect} yields the expansion
\[
\min\,E_{b,k,\gamma}=C_{1}\cdot (2\sqrt{bk}+\gamma)+o(2\sqrt{bk}+\gamma)
\]
where $C_{1}$ is the minimum value of the limiting problems.
As \prettyref{eq:limitingpblm-defect} makes clear,
\[
C_{1}=0\quad\iff\quad\det\nabla\nabla p=0\quad\text{a.e. on }\Omega.
\]
Thus, the minimum energy scales $\sim2\sqrt{bk}+\gamma$ if and only
if the initial Gaussian curvature of the shell, which is proportional to $\det \nabla \nabla p$, is other than zero. 

\subsubsection{Convex analysis of the limiting problems\label{subsec:The-geometry-of}}

The previous results, in particular \prettyref{cor:limiting_pblm_effectivedisplacement}
and \prettyref{cor:limiting_pblm_defect}, established the role of the
limiting problems
\begin{equation}
\min_{\substack{u_{\text{eff}}\in BD(\Omega)\\
e(u_{\text{eff}})\leq\frac{1}{2}\nabla p\otimes\nabla p\,dx
}
}\,\int_{\Omega}\frac{1}{2}|\nabla p|^{2}\,dx-\int_{\partial\Omega}u_{\text{eff}}\cdot\hat{\nu}\,ds\quad\text{and}\quad\min_{\substack{\mu\in\mathcal{M}_{+}(\Omega;\text{Sym}_{2})\\
-\frac{1}{2}\text{curl}\text{curl}\,\mu=\det\nabla\nabla p
}
}\,\frac{1}{2}\int_{\Omega}|\mu|_{1}\label{eq:primal}
\end{equation}
for the leading order response of weakly curved, floating elastic shells. We turn to study their
minimizers. Each of the problems in
\prettyref{eq:primal} is convex. On general grounds,
such ``primal'' minimization problems should admit a ``dual'' maximization problem, the solutions of which are paired via ``complementary slackness'' conditions.  
What distinguishes the present discussion of convex duality from the typical example
(as in, e.g., \cite{ekeland1999convex}) is that, here, the natural pairing
will turn out to require an ``inner product'' between matrix-valued measures. The situation reminds of the duality between stress
and strain in Hencky plasticity, where similar issues arise \cite{kohn1983dual}
(see also \cite{arroyo-rabasa2017relaxation}).

Let $\rho\in C_{c}^{\infty}(B_{1})$ be non-negative and supported
on the open unit disc $B_{1}=B(0,1)$, and let $\int_{B_1} \rho\,dx = 1$. Given any $\mu\in\mathcal{M}(\Omega;\text{Sym}_{2})$, we
define its \emph{mollifications} $\{\mu_{\delta}\}_{\delta>0}\subset C^{\infty}(\overline{\Omega};\text{Sym}_{2})$
by
\begin{equation}
\mu_{\delta}(x)=\int_{\Omega}\frac{1}{\delta^{2}}\rho\left(\frac{x-y}{\delta}\right)\,d\mu(y),\quad x\in\overline{\Omega}.\label{eq:mollified_measure_defn-intro}
\end{equation}
Denote 
\begin{equation}
\nabla^{\perp}\nabla^{\perp}=\text{cof}\,\nabla\nabla=\left(\begin{array}{cc}
\partial_{22} & -\partial_{12}\\
-\partial_{12} & \partial_{11}
\end{array}\right).\label{eq:Airyoperator}
\end{equation}
\begin{theorem}
\label{thm:duality} Let $\Omega$ and $p$ satisfy the assumption \prettyref{eq:A1a} and suppose  that $\Omega$ is simply connected. 
The minimization problems in \prettyref{eq:primal}
are dual to the maximization problem
\begin{equation}
\max_{\substack{\varphi:\mathbb{R}^{2}\to\mathbb{R}\\
\varphi\text{ is convex}\\
\varphi=\frac{1}{2}|x|^{2}\text{ on }\mathbb{R}^{2}\backslash\Omega
}
}\,\int_{\Omega}(\varphi-\frac{1}{2}|x|^{2})\det\nabla\nabla p\,dx.\label{eq:dual}
\end{equation}
In particular, the optimal values in \prettyref{eq:primal} and \prettyref{eq:dual}
are the same, and admissible $\mu$ and $\varphi$ are optimal if
and only if the complementary slackness conditions
\begin{equation}
0=\lim_{\delta\to0}\,\int_{\Omega}|\left\langle \mu_{\delta},\nabla^{\perp}\nabla^{\perp}\varphi\right\rangle |\quad\text{and}\quad0=\lim_{\delta\to0}\,\int_{\partial\Omega}|\hat{\nu}\cdot [\nabla\varphi] \left\langle \hat{\tau}\otimes\hat{\tau},\mu_{\delta}\right\rangle |\,ds\label{eq:compl_slackness}
\end{equation}
hold; the same is true for admissible $u_{\emph{eff}}$ with $\nabla p\otimes\nabla p\,dx-2e(u_{\emph{eff}})$
 in place of $\mu$. 
Here, $[\nabla \varphi]$ denotes the jump in $\nabla \varphi$ across $\partial\Omega$ in the direction of $\hat{\nu}$. It equals to  $x - \nabla \varphi |_{\partial\Omega}$ where the trace is taken from $\Omega$. 
\end{theorem}
\begin{remark}\label{rem:bcs}
The admissible set in \prettyref{eq:dual} consists of all convex
extensions of $\frac{1}{2}|x|^{2}$ from $\mathbb{R}^{2}\backslash\Omega$
into $\Omega$. The use of $\mathbb{R}^{2}$ is immaterial,
as it can be replaced by any convex neighborhood of $\Omega$. In fact, $\varphi$ can be shown to be admissible if and only if it restricts to $\Omega$ as an element of $HB$ with $\nabla\nabla\varphi\geq 0$ and such that
the boundary conditions
\begin{equation}
\varphi=\frac{1}{2}|x|^{2}\quad\text{and}\quad\hat{\nu}\cdot\nabla\varphi\leq\hat{\nu}\cdot x\quad\text{at }\partial\Omega \label{eq:bdry_conditions}
\end{equation}
hold in the sense of trace. See \prettyref{lem:restriction}. Regarding traces at $\partial\Omega$, our convention will be that whenever we refer to the values of a quantity there, we mean those of its trace from $\Omega$ unless otherwise explicitly specified.
\end{remark}
\begin{remark}
\label{rem:complementary-slackness-freedom}The complementary slackness
conditions \prettyref{eq:compl_slackness} hold regardless of the choice of the kernel $\rho$ in  \prettyref{eq:mollified_measure_defn-intro}, so long as it belongs to $C_c^\infty(B_1)$, is
non-negative, and integrates to one. Other equivalent statements hold using the approximating sequences in \prettyref{prop:duality}.
The freedom to approximate $\mu$ as needed will come in handy later
on in \prettyref{sec:method_of_characteristics} when it comes time
to justify the upcoming assertions of our method of stable lines. 
Even more equivalent statements of complementary slackness can be obtained by approximating $\varphi$. We leave this to the reader.  
\end{remark}
\begin{remark} Other, more general versions of the dual problem appear in \prettyref{sec:dualitytheory}, including ones that apply when $\Omega$ is multiply connected. See \prettyref{prop:duality} and the discussion immediately thereafter.
\end{remark}
While we explain in \prettyref{sec:dualitytheory} how one
can anticipate the form of the dual problem \prettyref{eq:dual} on
general grounds --- it follows from a minimax procedure using
the divergence-free ``stress''  $\sigma=\nabla^{\perp}\nabla^{\perp}\varphi$
as a Lagrange multiplier for the tension-free constraint ---
here we demonstrate how the complementary slackness conditions \prettyref{eq:compl_slackness}
arise. As the primal problems \prettyref{eq:primal} are convex, their
solutions are completely characterized by first order optimality.
The key is an integration by parts identity
that says, roughly speaking, that
\begin{equation}
\frac{1}{2}\int_{\Omega}|\mu|_{1}-\int_{\Omega}(\varphi-\frac{1}{2}|x|^{2})\det\nabla\nabla p=\frac{1}{2}\int_{\Omega}\left\langle \nabla^{\perp}\nabla^{\perp}\varphi,\mu\right\rangle +\frac{1}{2}\int_{\partial\Omega}\hat{\nu}\cdot[\nabla\varphi]\left\langle \hat{\tau}\otimes\hat{\tau},\mu\right\rangle \label{eq:ibp_identity}
\end{equation}
whenever $\mu$ and $\varphi$ are admissible for \prettyref{eq:primal}
and \prettyref{eq:dual}. The integrands on the righthand
side are non-negative, and the lefthand side vanishes if and only if
$\mu$ and $\varphi$ are optimal. Hence, optimality should be equivalent to the complementary slackness conditions
\begin{equation}
\left\langle \nabla^{\perp}\nabla^{\perp}\varphi,\mu\right\rangle =0\quad\text{on }\Omega\quad\text{and}\quad\hat{\nu}\cdot[\nabla\varphi]\left\langle \hat{\tau}\otimes\hat{\tau},\mu\right\rangle =0\quad\text{at }\partial\Omega.\label{eq:complslackness_illposed}
\end{equation}
The only problem with this is that the terms appearing in \prettyref{eq:complslackness_illposed}
are not obviously well-defined: evaluating the first one requires making sense of an ``inner product'' between the matrix-valued measures $\mu$ and $\nabla^{\perp}\nabla^{\perp}\varphi$; evaluating the second one requires assigning boundary values to the $\hat{\tau}\hat{\tau}$-component of $\mu$. While it may be possible
to take advantage of the relationship between the formally adjoint
operators $\nabla^{\perp}\nabla^{\perp}$ and $\text{curl}\text{curl}$
to treat \prettyref{eq:complslackness_illposed} in some more intrinsic way,
 we choose to regularize instead. Integrating by parts with the mollifications $\{\mu_{\delta}\}$ from \prettyref{eq:mollified_measure_defn-intro}, we obtain \prettyref{eq:ibp_identity} upon sending $\delta\to0$. The asserted complementary slackness conditions
 follow. See \prettyref{sec:dualitytheory} for the complete proof of \prettyref{thm:duality}, as well as for a discussion of duality for general domains. 

\prettyref{thm:duality} separates the problem of determining the
overall layout of the patterns encoded by optimal $\mu$ 
from that of determining their amplitude. 
We envision a two-step procedure, where in the
first step an optimal \emph{Airy potential} $\varphi$ is found by solving the dual problem \prettyref{eq:dual}, and in the second step
the complementary slackness conditions are systematically applied.
To lighten the notation, we use \prettyref{eq:complslackness_illposed}
from now on to refer to the complementary slackness conditions \prettyref{eq:compl_slackness}
with a remark that they hold \emph{in the regularized sense}. Note the meaning of this is independent of the choice of the mollifying kernel $\rho$.
\begin{corollary}
\label{cor:two-step-program} Let $\Omega$, $p$, and $\{(b,k,\gamma)\}$ satisfy the assumptions  \prettyref{eq:A-basic} and \prettyref{eq:A-additional}, and let $\varphi$ solve the dual problem \prettyref{eq:dual}. Then $\mu \in \mathcal{M}_+(\Omega;\emph{Sym}_2)$ arises as the defect measure of a sequence of almost minimizers $\{(u_{b,k,\gamma},w_{b,k,\gamma})\}$
of $E_{b,k,\gamma}$, i.e., 
\begin{equation}\label{eq:defect-measure-interpretation}
\mu=\emph{weak-\ensuremath{*}}\lim\,\nabla w_{b,k,\gamma}\otimes\nabla w_{b,k,\gamma}\,dx\quad\text{in }\mathcal{M}(\Omega;\emph{Sym}_{2})
\end{equation}
 if and
only if
\begin{equation}
\begin{cases}
-\frac{1}{2}\emph{curl}\emph{curl}\,\mu=\det\nabla\nabla p & \text{on }\Omega\\
\left\langle \nabla^{\perp}\nabla^{\perp}\varphi,\mu\right\rangle =0 & \text{on }\Omega\\
\hat{\nu}\cdot[\nabla\varphi]\left\langle \hat{\tau}\otimes\hat{\tau},\mu\right\rangle =0 & \text{at }\partial\Omega
\end{cases}.\label{eq:measure_valued_bvp}
\end{equation}
The first equation holds in the sense of distributions, while the
second and third ones hold in the regularized sense.
\end{corollary}

\subsubsection{Stable lines\label{subsec:method-of-stable-lines}}

We come at last to our method of stable lines. This is a way to deduce from knowledge of an optimal $\varphi$ solving the dual problem \prettyref{eq:dual} that optimal $\mu$ solving the primal problem \prettyref{eq:primal} are rank one, absolutely continuous, and uniquely determined on a subset of $\Omega$. In other words, the method asserts the existence of an ``ordered'' part, where stable lines exist and the convergence of almost minimizers towards known patterns is implied. (As usual, any reference to the almost minimizers of $E_{b,k,\gamma}$ is contingent on the $\Gamma$-convergence in \prettyref{thm:gamma-lim}.) 
At its heart is an analysis of \prettyref{eq:measure_valued_bvp} as a boundary value problem for $\mu$ via the method of characteristics. As such, it is a bit difficult to describe the method in a manner that is both general and precise. The following contains only the essence of what we achieve in  \prettyref{sec:method_of_characteristics} and \prettyref{sec:examples}. See \prettyref{subsec:openquestions} for open questions that remain.

The first task is to explain what we mean by the ``stable lines'' and the ``ordered part'' of the shell. The definitions we present here are only preliminary, as they require more regularity than generally holds. More general definitions are given in \prettyref{sec:method_of_characteristics}.
Suppose, for the sake of argument, that $\mu$ and $\varphi$ are not only optimal in \prettyref{eq:primal} and \prettyref{eq:dual} but are also smooth, at least off of some small (say, Hausdorff one-dimensional) set. Then, the first complementary slackness condition in \prettyref{eq:measure_valued_bvp} implies that
\[
\nabla^{\perp}\nabla^{\perp}\varphi\perp\mu
\]
in the pointwise sense. Since $\nabla\nabla\varphi$ and $\mu$ are non-negative, it follows that the sum of their ranks is at most two.
Where $\text{rank}\,\nabla\nabla\varphi=2$, it must be that $\mu=0$. On the other hand, where $\text{rank}\,\mu=2$ we see that $\varphi$ is affine. The part where $\text{rank}\, \nabla\nabla\varphi=1$ can be said to be \emph{ordered}, as there 
\begin{equation}
\mu=\lambda\hat{\eta}\otimes\hat{\eta}\quad\text{for some }\lambda\geq0\text{ and }\hat{\eta}\in R(\nabla\nabla\varphi).\label{eq:mu_lambda}
\end{equation}
Given the interpretation of $\mu$ as a defect measure of almost minimizers in \prettyref{eq:defect-measure-interpretation}, we see that
\[
\hat{\eta}^{\perp}\cdot\nabla w_{b,k,\gamma}\to0\quad\text{strongly in }L^{2}\text{ on the ordered part}.
\]
Put another way, the  peaks and troughs of any wrinkles that persist must become asymptotically perpendicular to the unit vector field $\hat{\eta}$ throughout the ordered part. 
To help keep track of this, we propose the following geometric construction: given an optimal $\varphi$, plot its
\begin{align}
\text{\emph{stable lines} --- curves parallel to }N(\nabla\nabla\varphi)\text{ where }\text{rank}\,\nabla\nabla\varphi=1. \label{eq:stableline_prelimin}
\end{align}
In an asymptotic sense, these are the wrinkle peaks and troughs. They form domains, as is apparent in \prettyref{fig:library_of_patterns}. Naturally, one wonders if their geometry can
be described. First, let us give an argument for their existence.

\begin{figure}
\centering
\subfloat[]{\includegraphics[width=0.33\paperwidth]{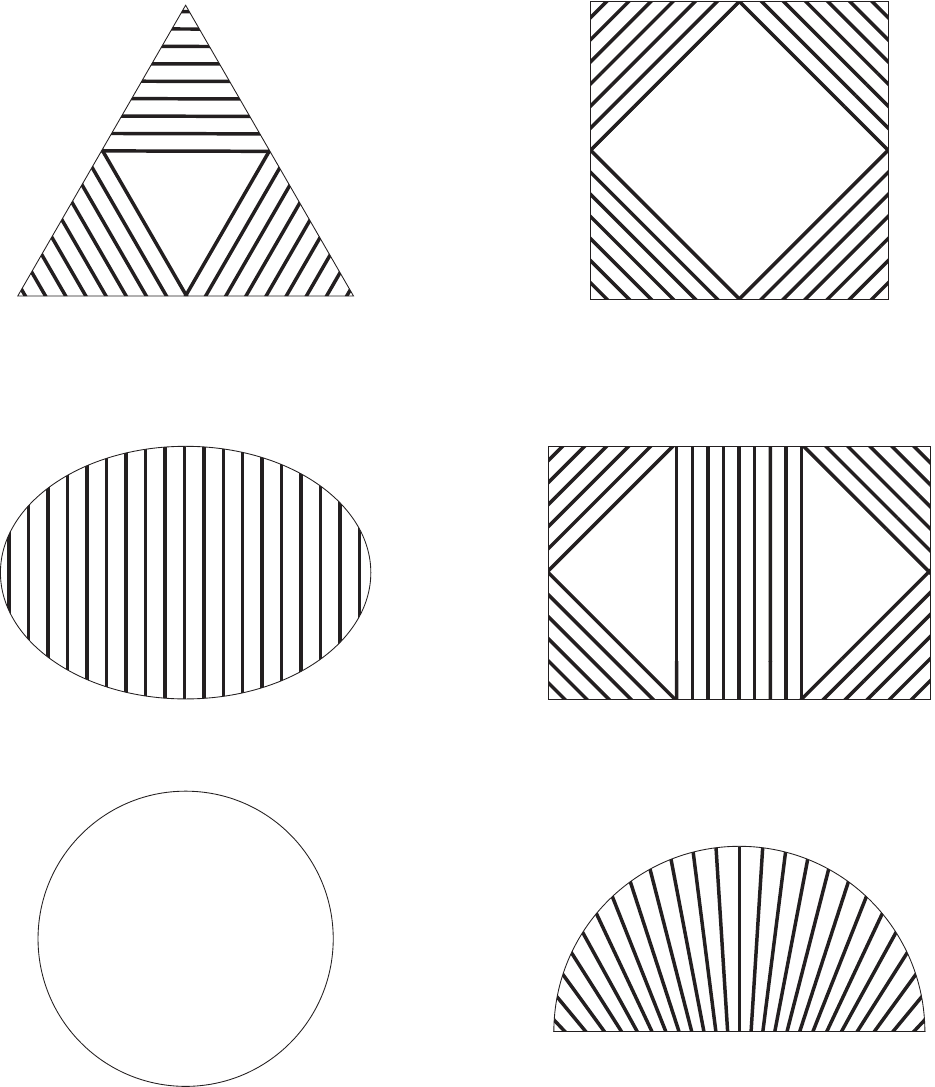}}\hspace{.12\textwidth}\subfloat[]{\includegraphics[width=0.33\paperwidth]{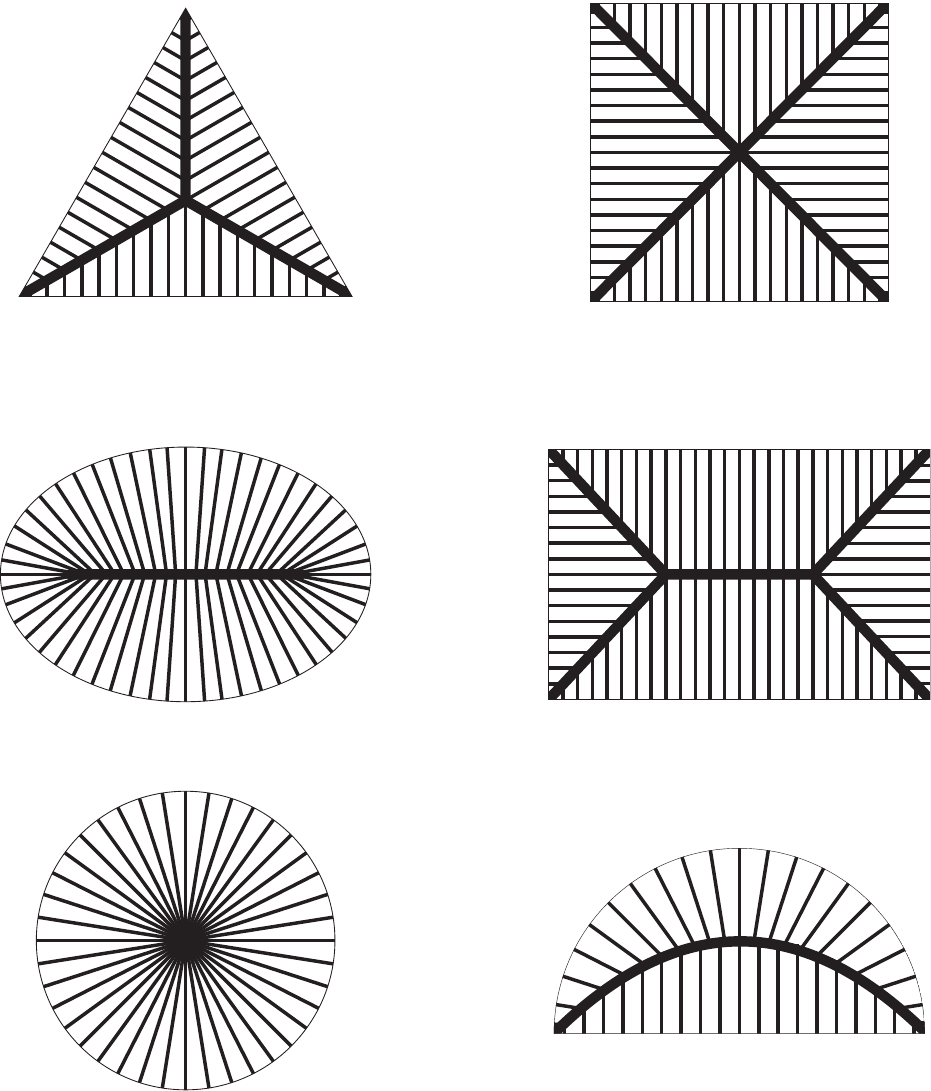}}
\caption{Optimal wrinkle patterns arrange themselves according to the plotted ``stable lines''. Panel (a) depicts the stable lines of various positively curved shells, and Panel (b) does the same but for negatively curved shells. By definition, stable lines fill out the ordered part of the shell;  
any disordered response is confined to regions absent these lines (shown as blank).
Stable lines are also the characteristic curves of a certain second order, linear PDE governing the defect measures of almost minimizers. When combined with appropriate boundary conditions, uniqueness and regularity theorems can be proved throughout the ordered part. 
\label{fig:library_of_patterns}}
\end{figure}

It is a well-known fact of differential geometry that any smooth enough
\emph{developable surface} --- which by definition has vanishing Gaussian curvature throughout its extent --- is the disjoint union of planar regions and an otherwise \emph{ruled} part consisting of line segments that extend between boundary points. Such segments define the \emph{generators} or \emph{ruling
lines} of the given surface (see, e.g., \cite{stoker1989differential,struik1988lectures}). Now where
$
\det\nabla\nabla\varphi=0
$
the Gaussian curvature of the graph of $\varphi$ vanishes, so that
it describes a developable surface. Where $\text{rank}\,\nabla\nabla\varphi=1$,
it consists of ruling lines. Upon projection to the plane we recover the desired stable lines. In fact, this argument shows a little more: any smooth curve picked out by \prettyref{eq:stableline_prelimin} is the planar projection of a ruling line. 

Evidently the layout of the stable lines, and so that of the ordered wrinkle domains they describe, is tied up with the geometry of developable surfaces. 
We note the ruled part of any (piecewise) smooth developable surface can be decomposed further into developable pieces of the following three elementary types: cylinders, whose ruling lines are parallel; cones, whose ruling lines intersect at a point; and ``tangential developables'', whose ruling lines are tangent to a space curve. 
Upon projection, we deduce the following classification of stable lines: we say that
\begin{itemize}
\item a family of stable lines is of the \emph{cylindrical} type if its consists of parallel segments;
\item a family of stable lines is of the \emph{conical} type if its segments, when extended, meet at a point; and
\item a family of stable lines is of the \emph{tangential} type if its segments, when extended, meet along a curve.
\end{itemize} 
General arrangements of stable lines are built from these.
From the twelve shells depicted in \prettyref{fig:library_of_patterns} we count seven consisting of only the cylindrical type; two with only the conical type (the positively curved half-disc and the negatively curved disc); one featuring both cylindrical and conical types (the negative half-disc); and one with only the tangential type (the negative ellipse). 

The following result is the key to \prettyref{fig:library_of_patterns}. It contains recipes for solving the dual problem \prettyref{eq:dual} when the initial Gaussian curvature is of one sign. 
\begin{proposition}
\label{prop:explicit_formulas} Let $\Omega$ and $p$ satisfy the assumption \prettyref{eq:A1a} and suppose that $\Omega$ is simply connected. If $\det\nabla\nabla p\geq0$ a.e., the dual problem is solved by the largest convex extension
$\varphi_{+}$ of $\frac{1}{2}|x|^{2}$ into $\Omega$. It satisfies
\begin{equation}
\varphi_{+}(x)=\min_{\{y_{i}\}\subset\partial\Omega}\,\sum_{i=1}^{3}\theta_{i}\frac{1}{2}|y_{i}|^{2}\quad \text{for }x\in \Omega\label{eq:convexroof_inf-1}
\end{equation}
where the minimization is taken over all pairs
and triples $\{y_{i}\}\subset\partial\Omega$ such that 
\[
x=\sum_{i}\theta_{i}y_{i}\quad\text{where }\{\theta_{i}\}\subset(0,1)\text{ satisfies }\sum_{i}\theta_{i}=1.
\]
If $\det\nabla\nabla p\leq0$ a.e., the dual problem
is solved by the smallest convex extension $\varphi_{-}$ of $\frac{1}{2}|x|^{2}$
into $\Omega$. It satisfies 
\begin{equation}
\varphi_{-}(x)=\frac{1}{2}|x|^{2}-\frac{1}{2}d_{\partial\Omega}^{2}(x)\quad\text{where}\quad d_{\partial\Omega}(x)=\min_{y\in\partial\Omega}\,|x-y|\quad \text{for } x\in \Omega.\label{eq:squareddistance}
\end{equation}
\end{proposition}
\begin{remark}\label{rem:uniqueness}
Optimal $\varphi$ are not in general unique. However, if $\det\nabla\nabla p$
is strictly positive or strictly negative a.e.,  then $\varphi_{+}$ or $\varphi_{-}$
is the unique solution of \prettyref{eq:dual}.
See \prettyref{subsec:Solution-formulas-optimal-Airys} for a proof.
\end{remark}

We just finished describing how stable lines dictate the geometry of optimal $\mu$, and how we were able to solve for the ones in \prettyref{fig:library_of_patterns}. There is a second, equally as important role played by the stable lines. We claim that they are characteristic curves along which the PDE in \prettyref{eq:measure_valued_bvp} becomes an ordinary differential equation (ODE) for the only possibly non-zero eigenvalue $\lambda$ of $\mu$ on the ordered part.
Going back to our previous assumptions of smoothness for $\mu$ and $\varphi$, and again postponing precise statements until  \prettyref{sec:method_of_characteristics}, we note the existence of a function $\varrho>0$
such that 
\[
-\frac{1}{2}\text{curl}\text{curl}(\hat{\eta}\otimes\hat{\eta}\cdot)=-\frac{1}{2\varrho}\partial_{\hat{\eta}^{\perp}}^{2}(\varrho\cdot)
\]
where $\partial_{\hat{\eta}^{\perp}}=\hat{\eta}^{\perp}\cdot\nabla$. Since by its definition $\hat{\eta}$ points perpendicularly to the stable lines, we recognize this operator as a directional derivative along their extent. 
Setting \prettyref{eq:mu_lambda} into the first part of \prettyref{eq:measure_valued_bvp}, we deduce that 
\begin{equation}
-\frac{1}{2\varrho}\partial_{\hat{\eta}^{\perp}}^{2}(\varrho\lambda)=\det\nabla\nabla p\quad\text{along the stable lines}.\label{eq:ODE_stablelines}
\end{equation}
Thus, we have arrived at a second order linear ODE for $\lambda$. Boundary data can be extracted from \prettyref{eq:measure_valued_bvp}, after which integration yields uniqueness, regularity, and even explicit solution formulas across the ordered part. 

Let us briefly comment on the sort of boundary data that can be deduced. 
In some cases, stable lines pass between boundary points, such as for the  positively curved shells in Panel (a) of \prettyref{fig:library_of_patterns}.
Provided $\varphi$ is suitably non-degenerate, the last part of \prettyref{eq:measure_valued_bvp} can be shown to imply the Dirichlet-type condition
\begin{equation}
\varrho\lambda=0\quad\text{where stable lines meet }\partial\Omega.\label{eq:Dirichlet-data}
\end{equation}
Our proof of this assumes in particular that the coefficient $\hat{\nu}\cdot[\nabla\varphi]$ is non-zero, and also that the stable lines meet $\partial\Omega$ transversely. 
Together, \prettyref{eq:ODE_stablelines} and \prettyref{eq:Dirichlet-data} constitute a family of two-point boundary value problems indexed by the stable lines.  
See \prettyref{cor:non-intersecting-chars} for a precise statement of this result.

In other cases, stable lines meet in the interior. This happens for the negatively curved shells in Panel (b) of \prettyref{fig:library_of_patterns}.
Since we expect $\nabla\nabla\varphi$ to explode where stable lines meet, the second equation in \prettyref{eq:measure_valued_bvp} should provide Dirichlet data. The first equation should yield matching conditions. Altogether we will show that
\begin{equation}
\varrho\lambda=0\quad\text{and}\quad\partial_{\hat{\eta}^{\perp}}(\varrho\lambda)=0\quad\text{where stable lines meet}, \label{eq:Cauchy-data}
\end{equation}
again subject to non-degeneracy conditions on $\varphi$ (e.g., if the stable lines meet along a curve, we assume they meet it transversely). Combining  \prettyref{eq:ODE_stablelines} and \prettyref{eq:Cauchy-data} yields a family of Cauchy problems indexed by the stable lines. 
See \prettyref{cor:chars-meet-along-curve} and \prettyref{cor:chars-meet-at-point} for precise statements of this result. 

The preceding observations suggest a significantly more general result: optimal $\mu$ should be uniquely determined wherever stable lines exist. While we do not know how to prove this for general shells, we will show in \prettyref{sec:examples} that it holds for each the shells in \prettyref{fig:library_of_patterns}, along with other related ones as well. That section also contains the proof of \prettyref{prop:explicit_formulas}. We hope our general description of the method of stable lines here and in \prettyref{sec:method_of_characteristics} helps the reader see the bigger picture behind what it implies at the level of the examples in \prettyref{sec:examples}.

\subsubsection{Open questions}\label{subsec:openquestions}

We close this introduction with a few open questions. Besides the
obvious ones regarding the extension of our results beyond the assumptions
 \vpageref{par:Assumptions} and beyond the realm of weakly
curved shells, there are some important issues that remain
regarding the method of stable lines. 

First, it is an admittedly awkward point throughout that $\nabla\nabla\varphi$
is only \emph{a priori} a measure.
For this reason, we do not yet have a generally useful definition of stable lines. 
What we lack is a regularity theory for optimal $\varphi$ solving \prettyref{eq:dual} or, failing that, a classification of developable surfaces of regularity $HB$. For now, we note that each of the examples
in \prettyref{sec:examples} enjoys the following additional regularity: optimal $\varphi$ are smooth off of a singular set of finite length. This is more than enough to justify our approach. More generally, building off of the theory of $W^{2,2}$ developable surfaces in \cite{hornung2011approximation,hornung2011fine,pakzad2004sobolev} we show how to make sense of it where $\varphi$ is (locally) $W^{2,2}$. See \prettyref{sec:method_of_characteristics}.

Second, we wonder if optimal $\mu$ solving \prettyref{eq:primal} are unique under the condition that there exists an optimal $\varphi$ that is nowhere  affine. \prettyref{conj:uniqueness_neg} at the very end gives a concrete version of this question for negatively curved shells. Our reasoning is simply that stable lines (suitably defined) should be characteristic curves for the boundary value problem \prettyref{eq:measure_valued_bvp}, and that the given hypothesis on $\varphi$ should imply their density in $\Omega$. This should lead to the uniqueness of $\mu$. Though we do not have a general theorem to this effect, we have achieved it in the context of several examples, including the positively curved ellipse and half-disc in Panel (a) of \prettyref{fig:library_of_patterns}, as well as each of the negatively curved shells in Panel (b). See \prettyref{subsec:Positively-curved-shells} and \prettyref{subsec:Negatively-curved-shells}. 

Conversely, we wonder if there must exist infinitely many optimal $\mu$
provided there exists a region on which optimal $\varphi$ are affine. Only when $\varphi$ is affine on all of $\Omega$ have we shown this to
be true --- see \prettyref{ex:positively-curved-disc} for the highly degenerate case of a positively curved disc. When combined with our $\Gamma$-convergence results, the existence of infinitely many optimal $\mu$ implies the existence of infinitely many almost minimizing
sequences for $E_{b,k,\gamma}$, a situation that could perhaps explain the disorder seen in 
 ultrathin shells \cite{tobasco2020principles}. Whether
this disorder arises from an overall flatness of the energy landscape,
or instead to a prevalence of local minimizers remains to be understood.

\subsection{Outline of the paper}

The remainder of the paper establishes the results outlined above. \prettyref{sec:Gamma_liminf}
covers the $\Gamma$-liminf and equi-coercivity parts of \prettyref{thm:gamma-lim},
while \prettyref{sec:Gamma_limsup} handles the recovery sequence part. \prettyref{sec:dualitytheory} establishes \prettyref{thm:duality}.
\prettyref{sec:method_of_characteristics} discusses the method of
stable lines. Finally, \prettyref{sec:examples} proves \prettyref{prop:explicit_formulas} and presents the details behind the patterns sketched in \prettyref{fig:library_of_patterns}. Since \prettyref{cor:limiting_pblm_effectivedisplacement}-\prettyref{cor:two-step-program} follow more or less immediately from the theorems as above, we do not repeat their proofs below.

\subsection{Notation}

We use big $O$ and little $o$ notation as well as their abbreviations
$\lesssim$ and $\ll$. We write $f=o(g)$ and $f\ll g$
 to mean that the functions $f$ and $g$ satisfy $\frac{f}{g}\to0$ in a relevant limit, and $f=O(g)$ and $f\lesssim g$ to mean that there exists a constant
$C>0$ such that $f\leq Cg$. If $C=C(\alpha)$ we indicate this using a subscript, as in $f\lesssim_{\alpha}g$.  
We write $f\sim g$ to mean that $f\lesssim g$ and $g\lesssim f$. We abbreviate $f\vee g=\max\{f,g\}$ and $f\wedge g=\min\{f,g\}$.

Dots and angle brackets denote the Euclidean vector and Frobenius matrix inner
products $x\cdot y=\sum_{i}x_{i}y_{i}$ and $\langle A,B \rangle = \sum_{ij}A_{ij}B_{ij}$.
Single lines without subscripts denote Euclidean and Frobenius
norms. The open Euclidean ball centered at $x$ with radius $r$ is $B(x,r)=B_r(x)=\{y: |x-y|<r\}$. The shortest Euclidean distance from $x$ to a set $S$ is $d(x,S)=d_S(x) = \inf_{y\in S} |x-y|$. We use the matrix norms
\[
|A|=|A|_{2}=\sqrt{\sum_{ij}|A_{ij}|^{2}},\quad|A|_{1}=\sum_{ij}|A_{ij}|,\quad\text{and}\quad|A|_{\infty}=\max_{ij}\,|A_{ij}|
\]
throughout. Double lines $||\cdot||$ are reserved for function space norms. 

Regarding function spaces, we use $C^{k}(X)$ and $\text{Lip}(X)$ to mean the spaces of real-valued, $k$-times differentiable and Lipschitz continuous functions on some appropriate domain or metric space $X$.  
Subscripts are used as normal, with $b$ for uniformly bounded functions and $c$ for compactly supported ones. Semi-colons indicate ranges other than $\mathbb{R}$. For instance, $C_b(\Omega;\text{Sym}_{d})$ indicates the space of continuous, uniformly bounded, symmetric $d$-by-$d$ matrix-valued functions on $\Omega$. The Sobolev spaces $W^{k,r}(\Omega)$ and their local versions $W^{k,r}_\text{loc}(\Omega)$ are defined as usual. We take the convention of referring to the (unique) continuous representative of a  function that is a.e.\ defined, provided it exists.  
See \prettyref{subsec:Bounded-deformation-maps} for $BD(\Omega)$ and $HB(\Omega)$. 

Regarding measures, we write $\mathcal{M}(X)$ to mean the space of finite, real-valued Radon measures on some locally compact Hausdorff space $X$. Semi-colons indicate values in vector spaces other than $\mathbb{R}$. The subscript $+$ indicates non-negativity. Given $\sigma\in C_{b}(X;\text{Sym}_{d})$ and $\mu\in\mathcal{M}(X;\text{Sym}_{d})$,
we denote their Frobenius inner product and its integral on $X$ by 
\[
\left\langle \sigma,\mu\right\rangle =\sum_{ij}\sigma_{ij}\mu_{ij}\quad\text{and}\quad\int_{X}\left\langle \sigma,\mu\right\rangle =\left\langle \sigma,\mu\right\rangle (X).
\]
Given a Borel measurable set $S$, we write $\mu\lfloor S$ to
mean the restriction of $\mu$ to $S$ defined by $\mu\lfloor S(\cdot)=\mu(S\cap \cdot)$. 
The two-dimensional Lebesgue and one-dimensional Hausdorff measures $\mathcal{L}^{2}$  and $\mathcal{H}^{1}$ appear throughout. We often use the notations $dx$ and $ds$. 
When a property is stated simply as holding ``a.e.'',
we mean that it holds a.e.\ with respect to Lebesgue unless the situation dictates otherwise. 
We denote $|S|=\mathcal{L}^{2}(S)$. 

Finally, by a ``curve'' we mean a homeomorphic copy of an open interval $I\subset \mathbb{R}$, i.e., its image under  
a continuous and one-to-one map. Such a map ``parameterizes'' the given curve. Any further regularity will be specified, e.g., a smooth curve is one that admits a $C^\infty$ parameterization. Given a Lipschitz curve $\Gamma\subset \mathbb{R}^2$, its tangent lines $T_{s}\Gamma$ are defined for $\mathcal{H}^{1}$-a.e.\ $s\in\Gamma$, along with a choice of unit tangent and unit normal vector $\hat{\tau}_{\Gamma}$ and $\hat{\nu}_{\Gamma}$ satisfying
\[
\hat{\tau}_{\Gamma}(s)||T_{s}\Gamma\quad\text{and}\quad\hat{\nu}_{\Gamma}(s)\perp T_{s}\Gamma\quad\text{for }\mathcal{H}^{1}\text{-a.e. }s\in \Gamma.
\]
We refer to the unit tangent and outwards-pointing unit normal at $\partial\Omega$ simply as $\hat{\tau}$ and $\hat{\nu}$. By convention, $\hat{\tau}=\hat{\nu}^{\perp}$ where $\perp$ denotes counterclockwise rotation by $\pi/2$.

\section{\emph{A priori} lower bounds and tension-free limits\label{sec:Gamma_liminf}}

This section establishes the equi-coercivity and $\Gamma$-liminf 
parts of \prettyref{thm:gamma-lim}. These results do not rely on the full set of assumptions listed \vpageref{par:Assumptions}, and instead make use of only the basic ones in \prettyref{eq:A-basic}. For the reader's convenience, we recall these assumptions in the formal statement of what we prove. 
\begin{proposition}
\label{prop:gam_liminf} (equi-coercivity and $\Gamma$-liminf inequality)
Suppose 
\begin{equation}\label{eq:assumptions-Prop2.1}
\Omega\text{ is bounded and Lipschitz},\quad p\in W^{2,2}(\Omega),\quad\text{and}\quad
\frac{b}{k},\ \frac{\gamma}{k},\ 2\sqrt{bk}+\gamma\ll1
\end{equation}
and let the sequence $\{(u_{b,k,\gamma},w_{b,k,\gamma})\}$
satisfy 
\begin{equation}
\limsup\,\frac{E_{b,k,\gamma}(u_{b,k,\gamma},w_{b,k,\gamma})}{2\sqrt{bk}+\gamma}<\infty.\label{eq:energybd}
\end{equation}
Then the following statements hold:
\begin{enumerate}
\item $\{(u_{b,k,\gamma},w_{b,k,\gamma})\}$ is weakly-$*$
pre-compact in $BD(\Omega)/\mathcal{R}\times W^{1,2}(\Omega)$;
\item each of its weak-$*$ limit points are of the form $(u_{\emph{eff}},0)$
where $u_{\emph{eff}}$ is tension-free, i.e.,
\[
e(u_{\emph{eff}})\leq\frac{1}{2}\nabla p\otimes\nabla p\,dx.
\]
\end{enumerate}
Moreover, if
\[
(u_{b,k,\gamma},w_{b,k,\gamma})\stackrel{*}{\rightharpoonup}(u_{\emph{eff}},0)\quad\text{weakly-}*\text{ in }BD(\Omega)/\mathcal{R}\times W^{1,2}(\Omega)
\]
then
\[
\liminf\,\frac{E_{b,k,\gamma}(u_{b,k,\gamma},w_{b,k,\gamma})}{2\sqrt{bk}+\gamma}\geq\int_{\Omega}\frac{1}{2}|\nabla p|^{2}\,dx-\int_{\partial\Omega}u_{\emph{eff}}\cdot\hat{\nu}\,ds.
\]
\end{proposition}
\begin{remark}
In the course of proving this result, we will show
that any sequence that obeys \prettyref{eq:energybd}
is asymptotically strain-free in that  \prettyref{eq:asymptotically-strain-free} holds --- this follows in particular from the first inequality in \prettyref{lem:aprioribds}.
\prettyref{lem:defectmeasure} then explains how a unique defect measure $\mu$ can be associated to any asymptotically strain-free and weakly-$*$ converging sequence. Together, these facts justify the introduction of defect measures in \prettyref{subsec:Defect-measures}. 
\end{remark}

We begin with a list of \emph{a priori} bounds. Recall the definition of the strain $\varepsilon$ from \prettyref{eq:strain_vK}.
\begin{lemma}
\label{lem:aprioribds} The inequalities
\begin{align*}
\int_{\Omega}|\varepsilon|^{2} & \lesssim_{\Omega}E_{b,k,\gamma}+\gamma^{2},\quad\int_{\Omega}|e(u)|\lesssim_{\Omega,p}1+\sqrt{E_{b,k,\gamma}}+\frac{E_{b,k,\gamma}}{\sqrt{bk}\vee\gamma}\\
\int_{\Omega}|w|^{2} & \lesssim_\Omega \frac{E_{b,k,\gamma}+\gamma^2}{k},
\quad\int_{\Omega}|\nabla w|^{2}\lesssim_{\Omega,p}1+\frac{E_{b,k,\gamma}}{\sqrt{bk}\vee\gamma},\quad\int_{\Omega}|\nabla\nabla w|^{2}\lesssim_{\Omega,p}1+\frac{E_{b,k,\gamma}+\gamma^2}{b}
\end{align*}
hold for all $0<b\leq k$ and $0\leq\gamma\leq1$.
\end{lemma}
\begin{proof}
First, we add a suitable constant to $E_{b,k,\gamma}$ to make
the result non-negative. Integrating the pointwise identity
\begin{align*}
\frac{1}{2}|\varepsilon-\gamma Id|^{2} & =\frac{1}{2}|\varepsilon|^{2}-\left\langle \varepsilon,\gamma Id\right\rangle +\frac{1}{2}|\gamma Id|^{2}=\frac{1}{2}\left|\varepsilon\right|^{2}-\gamma\text{tr}\,\varepsilon+\gamma^{2}\\
 & =\frac{1}{2}|\varepsilon|^{2}+\gamma\left(\frac{1}{2}|\nabla p|^{2}-\text{div}\,u\right)-\frac{\gamma}{2}|\nabla w|^{2}+\gamma^{2}
\end{align*}
and applying the divergence theorem, there results
\[
\frac{1}{2}\int_{\Omega}|\varepsilon-\gamma Id|^{2}+\frac{\gamma}{2}\int_{\Omega}|\nabla w|^{2}=\frac{1}{2}\int_{\Omega}|\varepsilon|^{2}+\gamma\left(\int_{\Omega}\frac{1}{2}|\nabla p|^{2}-\int_{\partial\Omega}u\cdot\hat{\nu}\right)+\gamma^{2}|\Omega|.
\]
Therefore, 
\begin{equation}
\tilde{E}_{b,k,\gamma}=E_{b,k,\gamma}+\gamma^{2}|\Omega|=\frac{1}{2}\int_{\Omega}|\varepsilon-\gamma Id|^{2}+\frac{\gamma}{2}\int_{\Omega}|\nabla w|^{2}+\frac{b}{2}\int_{\Omega}|\nabla\nabla w-\nabla\nabla p|^{2}+\frac{k}{2}\int_{\Omega}|w|^{2}\geq0.\label{eq:shifted_energy}
\end{equation}
Being a sum of squares, $\tilde{E}_{b,k,\gamma}$ easily admits lower bounds. Bounds on the original energy follow.

We proceed to prove the inequalities from the claim. It follows directly from
\prettyref{eq:shifted_energy} that
\[
\int_{\Omega}|w|^{2}\lesssim\frac{\tilde{E}_{b,k,\gamma}}{k}
\]
and so the third inequality holds. The first and last inequalities are just as easily shown. 
Using the triangle inequality and \prettyref{eq:shifted_energy} we see that
\begin{align}
\int_{\Omega}|\varepsilon|^{2} & \lesssim\int_{\Omega}|\varepsilon-\gamma Id|^{2}+\gamma^{2}|\Omega|\lesssim\tilde{E}_{b,k,\gamma}+\gamma^{2}|\Omega|\label{eq:strain-bound},\\
\int_{\Omega}|\nabla\nabla w|^{2} & \lesssim\int_{\Omega}|\nabla\nabla w-\nabla\nabla p|^{2}+\int_{\Omega}|\nabla\nabla p|^{2}\lesssim\frac{\tilde{E}_{b,k,\gamma}}{b}+||\nabla\nabla p||_{L^{2}}^{2}.\label{eq:second-deriv-bound}
\end{align}
The first and last inequalities follow. We turn now to control $\nabla w$ and $e(u)$.

Two separate arguments yield bounds on $\nabla w$, depending on whether $\gamma \geq \sqrt{bk}$ or not. The inequality
\begin{equation}
||\nabla w||_{L^{2}}^{2}\lesssim\frac{\tilde{E}_{b,k,\gamma}}{\gamma}\label{eq:gamma-lowerbd}
\end{equation}
follows directly from \prettyref{eq:shifted_energy} in any case. At the same time, we can interpolate between the bending and substrate terms to obtain another bound. Using the triangle inequality with \prettyref{eq:shifted_energy} as we did in the proof of \prettyref{eq:second-deriv-bound}, we note that
\begin{equation}
\tilde{E}_{b,k,\gamma}+b ||\nabla\nabla p||_{L^{2}}^{2} \gtrsim b||\nabla\nabla w||_{L^{2}}^{2}+k||w||_{L^{2}}^{2}
 \gtrsim\sqrt{bk}||\nabla\nabla w||_{L^{2}}||w||_{L^{2}}+k||w||_{L^{2}}^{2}\label{eq:interpolation_step}
\end{equation} 
by an elementary Young's inequality. Recall the Gagliardo--Nirenberg interpolation inequality
\begin{equation}
||\nabla w||_{L^{2}(\Omega)}\lesssim||\nabla\nabla w||_{L^{2}(\Omega)}^{1/2}||w||_{L^{2}(\Omega)}^{1/2}+C(\Omega)||w||_{L^{2}(\Omega)}\label{eq:GagliardoNirenberg}
\end{equation}
which holds for all $w\in W^{2,2}(\Omega)$ \cite{gilbarg2001elliptic}.
Since by hypothesis $b\leq k$, it follows from \prettyref{eq:interpolation_step}
and \prettyref{eq:GagliardoNirenberg} that
\begin{equation}
||\nabla w||_{L^{2}}^{2}\lesssim_{\Omega}\frac{\sqrt{bk}||\nabla\nabla w||_{L^{2}}||w||_{L^{2}}+k||w||_{L^{2}}^{2}}{\sqrt{bk}}\lesssim\frac{\tilde{E}_{b,k,\gamma}}{\sqrt{bk}}+||\nabla\nabla p||_{L^{2}}^{2}.\label{eq:bk-lowerbd}
\end{equation}
Combining \prettyref{eq:gamma-lowerbd} and \prettyref{eq:bk-lowerbd}
yields the fourth inequality from the claim.

Finally, we handle $e(u)$. By the definition of $\varepsilon$ and H{\"o}lder's inequality,
\begin{align*}
\int_{\Omega}|e(u)| & \leq\int_{\Omega}|\varepsilon|+\int_{\Omega}|\frac{1}{2}\nabla w\otimes\nabla w|+\int_{\Omega}|\frac{1}{2}\nabla p\otimes\nabla p|\lesssim |\Omega|^{1/2} ||\varepsilon||_{L^{2}}+||\nabla w||_{L^{2}}^{2}+||\nabla p||_{L^{2}}^{2}\\
 & \lesssim_\Omega \sqrt{\tilde{E}_{b,k,\gamma}+\gamma^{2}}+\frac{\tilde{E}_{b,k,\gamma}}{\sqrt{bk}\vee \gamma}+||\nabla\nabla p||_{L^{2}}^{2}+||\nabla p||_{L^{2}}^{2}
\end{align*}
where in the last line we used \prettyref{eq:strain-bound}, \prettyref{eq:gamma-lowerbd}, and \prettyref{eq:bk-lowerbd}.
The remaining inequality follows.
\qed\end{proof}
Next, we verify that the weak-$*$ limits of asymptotically strain-free
sequences are tension-free. At the same time, we justify the notion
of defect measures introduced in \prettyref{subsec:Defect-measures}. 
\begin{lemma}
\label{lem:defectmeasure} Let 
\[
(u_{n},w_{n})\overset{*}{\rightharpoonup}(u,0)\quad\text{weakly-\ensuremath{*} in }BD(\Omega)/\mathcal{R}\times W^{1,2}(\Omega)
\]
and suppose it is asymptotically strain-free in that
\[
e(u_{n})+\frac{1}{2}\nabla w_{n}\otimes\nabla w_{n}\to\frac{1}{2}\nabla p\otimes\nabla p\quad\text{stongly in }L^{2}(\Omega;\emph{Sym}_{2}).
\]
Then $\{\nabla w_{n}\otimes\nabla w_{n}\,dx\}$ converges weakly-$*$
in $\mathcal{M}(\Omega;\emph{Sym}_{2})$ to a non-negative, $\emph{Sym}_{2}$-valued Radon measure $\mu$ called the
defect measure of the given sequence. The defect measure
satisfies
\begin{equation}
e(u)+\frac{1}{2}\mu=\frac{1}{2}\nabla p\otimes\nabla p\,dx.\label{eq:defect_eqn-1}
\end{equation}
As a result, the limiting in-plane displacement $u$ must be
tension-free.
\end{lemma}
\begin{proof}
Evidently,
\[
\nabla w_{n}\otimes\nabla w_{n}\,dx=2\left(e(u_{n})+\frac{1}{2}\nabla w_{n}\otimes\nabla w_{n}\,dx\right)-2e(u_{n})\quad\forall\,n.
\]
On the righthand side, we see the difference between a sequence converging
strongly to $\nabla p\otimes\nabla p\,dx$,
and another converging weakly-$*$ to $2e(u)$. Therefore, the lefthand side converges weakly-$*$. Passing to the limit and rearranging yields \prettyref{eq:defect_eqn-1}. 
Non-negativity is preserved by weak-$*$ convergence. Therefore, $\mu\geq0$
and the result is proved.
\qed\end{proof}

At this point, we have enough to deduce the first part of \prettyref{prop:gam_liminf}
on compactness. In order to prove the second part 
on \emph{a priori} lower bounds, we must identify the optimal
prefactor in the bound on $\nabla w$ from \prettyref{lem:aprioribds}.
That bound was a consequence of the Gagliardo--Nirenberg interpolation
inequality \prettyref{eq:GagliardoNirenberg} so, naturally, we seek
a sharpened version of it now.
\begin{lemma}
\label{lem:sharp_GN} Let $\chi\in C_{c}^{\infty}(\Omega)$ satisfy
$0\leq\chi\leq1$. Then for all $b,k>0$ and $w\in W^{2,2}(\Omega)$
we have that
\begin{align*}
b\int_{\Omega}|\nabla\nabla w|^{2}+k\int_{\Omega}|w|^{2} & \geq2\sqrt{bk}\int_{\Omega}|\nabla w|^{2}\chi+\int_{\Omega}|b^{1/2}\Delta w+k^{1/2}w|^{2}\chi\\
 & \qquad-2\sqrt{bk}||\nabla\chi||_{L^{\infty}(\Omega)}||w||_{L^{2}(\Omega)}||\nabla w||_{L^{2}(\Omega)}-b||\nabla\nabla\chi||_{L^{\infty}(\Omega)}||\nabla w||_{L^{2}(\Omega)}^{2}.
\end{align*}
\end{lemma}
\begin{remark}
Eventually, we shall apply this to sequences $\{w_{n}\}$ converging
weakly-$*$ to zero in $W^{1,2}(\Omega)$, under the condition that
$b\ll k$. Dividing by $\sqrt{bk}$ we see that the terms appearing
on the second line above behave as errors. These arise, respectively, from estimates on $\text{div}(w\nabla w)$
and $\det\nabla\nabla w$ in suitable negative norms.
\end{remark}
\begin{remark}
Following up on the previous remark, we note that if $\varepsilon \approx 0$ then $\det \nabla\nabla w\approx \det \nabla \nabla p$ as a result of the very weak Hessian identity \prettyref{eq:veryweak_hessian}, the Saint-Venant compatibility conditions \prettyref{eq:SVcompat}, and the definition of the strain in \prettyref{eq:strain_vK}. Hence, $|\nabla \nabla w|\approx |\Delta w|$ explaining the appearance of the mean curvature $H\approx\frac{1}{2}\Delta w$ on the lefthand side of the geometric interpolation inequality \prettyref{eq:geometric_interpolation_inequality} in the introduction.
\end{remark}
\begin{proof}
Note the pointwise identities 
\begin{align}
|\nabla\nabla w|^{2} & =|\Delta w|^{2}-2\det\nabla\nabla w\label{eq:identityone}\\
b|\Delta w|^{2}+k|w|^{2} & =2\sqrt{bk}\left(|\nabla w|^{2}-\text{div}(w\nabla w)\right)+|b^{1/2}\Delta w+k^{1/2}w|^{2}\label{eq:identitytwo}
\end{align}
as well as the very weak Hessian identity \prettyref{eq:veryweak_hessian}.
Now let $\chi$ be as in the statement. Testing the first identity
\prettyref{eq:identityone} against $\chi$ and integrating by parts
using \prettyref{eq:veryweak_hessian}, we obtain that
\begin{align*}
\int_{\Omega}|\nabla\nabla w|^{2}\chi & =\int_{\Omega}|\Delta w|^{2}\chi-2\det\nabla\nabla w\chi=\int_{\Omega}|\Delta w|^{2}\chi+\left\langle \nabla w\otimes\nabla w,\nabla^{\perp}\nabla^{\perp}\chi\right\rangle \\
 & \geq\int_{\Omega}|\Delta w|^{2}\chi-||\nabla\nabla\chi||_{L^{\infty}}||\nabla w||_{L^{2}}^{2}.
\end{align*}
Testing the second identity \prettyref{eq:identitytwo} against $\chi$
and integrating by parts, there follows
\begin{align*}
\int_{\Omega}\left(b|\Delta w|^{2}+k|w|^{2}\right)\chi & =2\sqrt{bk}\int_{\Omega}|\nabla w|^{2}\chi+w\nabla w\cdot\nabla\chi+\int_{\Omega}|b^{1/2}\Delta w+k^{1/2}w|^{2}\chi\\
 & \geq2\sqrt{bk}\int_{\Omega}|\nabla w|^{2}\chi+\int_{\Omega}|b^{1/2}\Delta w+k^{1/2}w|^{2}\chi-2\sqrt{bk}||\nabla\chi||_{L^{\infty}}||w||_{L^{2}}||\nabla w||_{L^{2}}.
\end{align*}
Combining these and using that $0\leq\chi\leq1$, we deduce that
\begin{align*}
b\int_{\Omega}|\nabla\nabla w|^{2}+k\int_{\Omega}|w|^{2} & \geq\int_{\Omega}\left(b|\nabla\nabla w|^{2}+k|w|^{2}\right)\chi\\
 & \geq\int_{\Omega}\left(b|\Delta w|^{2}+k|w|^{2}\right)\chi-b||\nabla\nabla\chi||_{L^{\infty}}||\nabla w||_{L^{2}}^{2}\\
 & \geq2\sqrt{bk}\int_{\Omega}|\nabla w|^{2}\chi+\int_{\Omega}|b^{1/2}\Delta w-k^{1/2}w|^{2}\chi\\
 & \qquad-2\sqrt{bk}||\nabla\chi||_{L^{\infty}}||w||_{L^{2}}||\nabla w||_{L^{2}}-b||\nabla\nabla\chi||_{L^{\infty}}||\nabla w||_{L^{2}}^{2}.
\end{align*}
This completes the proof.
\qed\end{proof}
We are ready to prove the $\Gamma$-liminf and equi-coercivity parts
of \prettyref{thm:gamma-lim}. 

\vspace{1em}
\noindent\emph{Proof of \prettyref{prop:gam_liminf}} First, consider
an admissible sequence $\{(u_{b,k,\gamma},w_{b,k,\gamma})\}$ whose
energy satisfies
\begin{equation}
E_{b,k,\gamma}(u_{b,k,\gamma},w_{b,k,\gamma})\lesssim \sqrt{bk}\vee\gamma\ll1\label{eq:energy_bd}
\end{equation}
but is otherwise arbitrary. We must prove that it is weakly-$*$ pre-compact
and identify its limit points. According to \prettyref{lem:aprioribds}
and the assumptions on the parameters in \prettyref{eq:assumptions-Prop2.1},
any sequence satisfying \prettyref{eq:energy_bd} enjoys the estimates
\begin{align}
\int_{\Omega}|e(u_{b,k,\gamma})+\frac{1}{2}\nabla w_{b,k,\gamma}\otimes\nabla w_{b,k,\gamma}-\frac{1}{2}\nabla p\otimes\nabla p|^{2} & \lesssim_{\Omega}\sqrt{bk}\vee\gamma+\gamma^{2}\ll1,\label{eq:bound1}\\
\int_{\Omega}|e(u_{b,k,\gamma})| & \lesssim_{\Omega,p}1+\sqrt{\sqrt{bk}\vee\gamma}+\frac{\sqrt{bk}\vee\gamma}{\sqrt{bk}\vee\gamma}\lesssim1,\label{eq:bound2}\\
\int_{\Omega}|w_{b,k,\gamma}|^{2} & \lesssim_\Omega\frac{\sqrt{bk}\vee\gamma+\gamma^2}{k} \lesssim \sqrt{\frac{b}{k}}\vee\frac{\gamma}{k} \ll1,\label{eq:bound3}\\
\int_{\Omega}|\nabla w_{b,k,\gamma}|^{2} & \lesssim_{\Omega,p}1+\frac{\sqrt{bk}\vee\gamma}{\sqrt{bk}\vee\gamma}\lesssim1,\label{eq:bound4}\\
\int_{\Omega}|\nabla\nabla w_{b,k,\gamma}|^{2} &  \lesssim_{\Omega,p}1+\frac{\sqrt{bk}\vee\gamma+\gamma^2}{b} \lesssim \sqrt{\frac{k}{b}}\vee \frac{\gamma}{b}.\label{eq:bound5}
\end{align}
The first estimate shows that $\{(u_{b,k,\gamma},w_{b,k,\gamma})\}$
is asymptotically strain-free. The second one proves that $u_{b,k,\gamma}$ remains uniformly
bounded in $BD/\mathcal{R}$. The third and fourth estimates show that $w_{b,k,\gamma}$ remains uniformly bounded in $W^{1,2}$ and converges
to zero strongly in $L^{2}$. In combination with the Banach--Alaoglu
theorem, these imply that $\{(u_{b,k,\gamma},w_{b,k,\gamma})\}$ is
weakly-$*$ pre-compact in $BD/\mathcal{R}\times W^{1,2}$. Applying
\prettyref{lem:defectmeasure} we learn that its limit points are
of the form $(u_{\text{eff}},0)$ where the in-plane part $u_{\text{eff}}$
is tension-free. The first part of \prettyref{prop:gam_liminf}
is proved.

Now we consider an admissible sequence satisfying 
\[
(u_{b,k,\gamma},w_{b,k,\gamma})\stackrel{*}{\rightharpoonup}(u_{\text{eff}},0)\quad\text{weakly-}*\text{ in }BD/\mathcal{R}\times W^{1,2}
\]
and for which the bound \prettyref{eq:energy_bd} holds. In terms of the defect
measure $\mu$ from \prettyref{lem:defectmeasure}, we must prove
that
\begin{equation}
\liminf\,\frac{E_{b,k,\gamma}(u_{b,k,\gamma},w_{b,k,\gamma})}{2\sqrt{bk}+\gamma}\geq\frac{1}{2}\int_{\Omega}|\mu|_{1}.\label{eq:defectmeasure_LB}
\end{equation}
Indeed, according to \prettyref{eq:defect_eqn-1} and the first integration
by parts identity in \prettyref{eq:IBP1}, there holds
\[
\frac{1}{2}\int_{\Omega}|\mu|_{1}=\frac{1}{2}\int_{\Omega}\left\langle Id,\mu\right\rangle =\int_{\Omega}\left\langle Id,\frac{1}{2}\nabla p\otimes\nabla p\,dx-e(u_{\text{eff}})\right\rangle =\int_{\Omega}\frac{1}{2}|\nabla p|^{2}-\int_{\partial\Omega}u_{\text{eff}}\cdot\hat{\nu}\,ds.
\]
Our plan is to pass to the limit in the bending and substrate terms
from $E_{b,k,\gamma}$ using the sharp Gagliardo--Nirenberg
inequality from \prettyref{lem:sharp_GN}. Passing to the limit in the surface energy will present
no additional difficulties.

Consider the bending term. Due to \prettyref{eq:bound5} and our assumption from \prettyref{eq:assumptions-Prop2.1} that $b\ll k$, it satisfies
\begin{align*}
b\int_{\Omega}|\nabla\nabla w_{b,k,\gamma}-\nabla\nabla p|^{2} & \geq b\int_{\Omega}|\nabla\nabla w_{b,k,\gamma}|^{2}-2b||\nabla\nabla w_{b,k,\gamma}||_{L^{2}}||\nabla\nabla p||_{L^{2}}-b||\nabla\nabla p||_{L^{2}}^{2}\\
 & \geq b\int_{\Omega}|\nabla\nabla w_{b,k,\gamma}|^{2}-o(\sqrt{bk}).
\end{align*}
Fix a cutoff function $\chi\in C_{c}^{\infty}(\Omega)$ that satisfies
$0\leq\chi\leq1$ but is otherwise arbitrary. Using the sharp Gagliardo--Nirenberg
inequality from \prettyref{lem:sharp_GN} and the bound just obtained,
we conclude that 
\begin{align*}
b\int_{\Omega}|\nabla\nabla w_{b,k,\gamma}-\nabla\nabla p|^{2}+k\int_{\Omega}|w_{b,k,\gamma}|^{2} & \geq2\sqrt{bk}\int_{\Omega}|\nabla w_{b,k,\gamma}|^{2}\chi-2\sqrt{bk}||\nabla\chi||_{L^{\infty}}||w||_{L^{2}}||\nabla w||_{L^{2}}\\
 & \qquad-b||\nabla\nabla\chi||_{L^{\infty}}||\nabla w||_{L^{2}}^{2}-o(\sqrt{bk})\\
 & \geq2\sqrt{bk}\int_{\Omega}|\nabla w_{b,k,\gamma}|^{2}\chi-o(\sqrt{bk})
\end{align*}
by \prettyref{eq:bound3} and \prettyref{eq:bound4}. Combining this with the definition \prettyref{eq:non-dim_energy} of the energy and \prettyref{eq:bound1}, there results
\begin{align*}
E_{b,k,\gamma}(u_{b,k,\gamma},w_{b,k,\gamma}) & \geq\frac{b}{2}\int_{\Omega}|\nabla\nabla w_{b,k,\gamma}-\nabla\nabla p|^{2}+\frac{k}{2}\int_{\Omega}|w_{b,k,\gamma}|^{2}+\gamma\int_{\Omega}\frac{1}{2}|\nabla p|^{2}-\text{div}\,u_{b,k,\gamma}\\
 & \geq(2\sqrt{bk}+\gamma)\int_{\Omega}\left\langle \chi Id,\frac{1}{2}\nabla w_{b,k,\gamma}\otimes\nabla w_{b,k,\gamma}\right\rangle -o(\sqrt{bk}\vee\gamma).
\end{align*}
Dividing by $2\sqrt{bk}+\gamma$ and passing to the limit via the
defect measure $\mu$ from \prettyref{lem:defectmeasure} yields
\[
\liminf\,\frac{E_{b,k,\gamma}(u_{b,k,\gamma},w_{b,k,\gamma})}{2\sqrt{bk}+\gamma}\geq \lim\, \int_{\Omega}\left\langle \chi Id,\frac{1}{2}\nabla w_{b,k,\gamma}\otimes\nabla w_{b,k,\gamma}\right\rangle
= \frac{1}{2}\int_{\Omega}\left\langle \chi Id,\mu\right\rangle .
\]
Letting $\chi\uparrow1$ and noting that $\mu\geq0$ we obtain \prettyref{eq:defectmeasure_LB}.
The proof is complete.\qed

\section{The piecewise herringbone construction \label{sec:Gamma_limsup}}

The previous section established the $\Gamma$-liminf and equi-coercivity
parts of \prettyref{thm:gamma-lim}. 
Here, we complete its proof by producing the required recovery sequences. 
Note to do so we will need to make use of all but one of the assumptions \vpageref{par:Assumptions}.
The main result of this section is as follows:
\begin{proposition}
\label{prop:recovery} (recovery sequences)
Suppose 
\begin{equation}\label{eq:assumptions-Prop3.1}
\Omega\text{ is bounded, Lipschitz, and strictly star-shaped},\quad p\in W^{2,\infty}(\Omega),
\quad\text{and }\quad\left(\frac{b}{k}\right)^{1/10}\ll2\sqrt{bk}+\gamma\ll1
\end{equation}
and let $u_{\emph{eff}}\in BD(\Omega)$ be tension-free, meaning that
\[
e(u_{\emph{eff}})\leq\frac{1}{2}\nabla p\otimes\nabla p\,dx.
\]
Then there exists
\[
(u_{b,k,\gamma},w_{b,k,\gamma})\overset{*}{\rightharpoonup}(u_{\emph{eff}},0)\quad\text{weakly-\ensuremath{*} in }BD(\Omega)\times W^{1,2}(\Omega) 
\]
such that
\begin{equation}
\lim\,\frac{E_{b,k,\gamma}(u_{b,k,\gamma},w_{b,k,\gamma})}{2\sqrt{bk}+\gamma}=\int_{\Omega}\frac{1}{2}|\nabla p|^{2}\,dx-\int_{\partial\Omega}u_{\emph{eff}}\cdot\hat{\nu}\,ds.\label{eq:limit_UB}
\end{equation}
\end{proposition}
\begin{remark}
The missing hypothesis is that $\gamma \ll k$ (see \prettyref{eq:A2a}). While the form of the $\Gamma$-limit does depend on this hypothesis --- as can be anticipated from the discussion surrounding \prettyref{eq:nearlystrainfree-nearlyflat} --- it is not necessary here as the recovery sequences we describe do not depend on $\gamma$. Nevertheless, we include $\gamma$ as a subscript to keep the notation consistent, and also to remind that other recovery sequences may very well depend on all three parameters. 
\end{remark}
Our proof of \prettyref{prop:recovery} centers around the notion
of the \emph{target defect measure} $\mu$ specified by $u_{\text{eff}}$.
Following the discussion in \prettyref{subsec:Defect-measures}
in the introduction, we note that any recovery sequence must satisfy
\begin{gather*}
\nabla w_{b,k,\gamma}\otimes\nabla w_{b,k,\gamma}\,dx\stackrel{*}{\rightharpoonup}\mu\quad\text{weakly-\ensuremath{*} in }\mathcal{M}(\Omega;\text{Sym}_{2})\\
\text{where}\quad\mu=-2\varepsilon_{\text{eff}}\quad\text{and}\quad\varepsilon_{\text{eff}}=e(u_{\text{eff}})-\frac{1}{2}\nabla p\otimes\nabla p\,dx.
\end{gather*}
We think of $\frac{1}{2}\mu$ as a ``misfit'' to be alleviated by some well-chosen pattern. 
Note $\mu \geq 0$ as $u_{\text{eff}}$ is tension-free.

The proof proceeds in three steps. First, we reduce to the case where $\mu$ is Lipschitz and strictly positive. The key lemma is in \prettyref{subsec:smooth-approx}, where we show how to approximate tension-free $u\in BD$ with $u\in C^{\infty}$
that are \emph{uniformly tension-free}, meaning that
\begin{equation}
e(u)\leq\frac{1}{2}\nabla p\otimes\nabla p-\lambda Id \quad\text{for some }\lambda >0.\label{eq:uniformly_tensionfree}
\end{equation}
Though our proof of this result relies crucially on the supposed strict star-shapedness of $\Omega$, we wonder whether it holds in greater generality.  
It is not difficult to understand why we would like $\mu$ to be Lipschitz, as it can then be approximated
by a piecewise constant target defect $\left\langle \mu\right\rangle $
obtained from averaging $\mu$ on a suitable lattice of squares (shown in \prettyref{fig:disc_herringbone} in bold). On the other hand, we pass from $\mu\geq0$ to $\mu>0$ only because it shortens the proof. Note this does not preclude the possibility that optimal $\mu$ --- i.e., those minimizing the righthand side of \prettyref{eq:limit_UB}, see also  \prettyref{eq:limitingpblm-defect} --- may turn out to be rank one or even to vanish somewhere.

\begin{figure}[tbh]
\centering
\includegraphics[width=0.4\paperwidth]{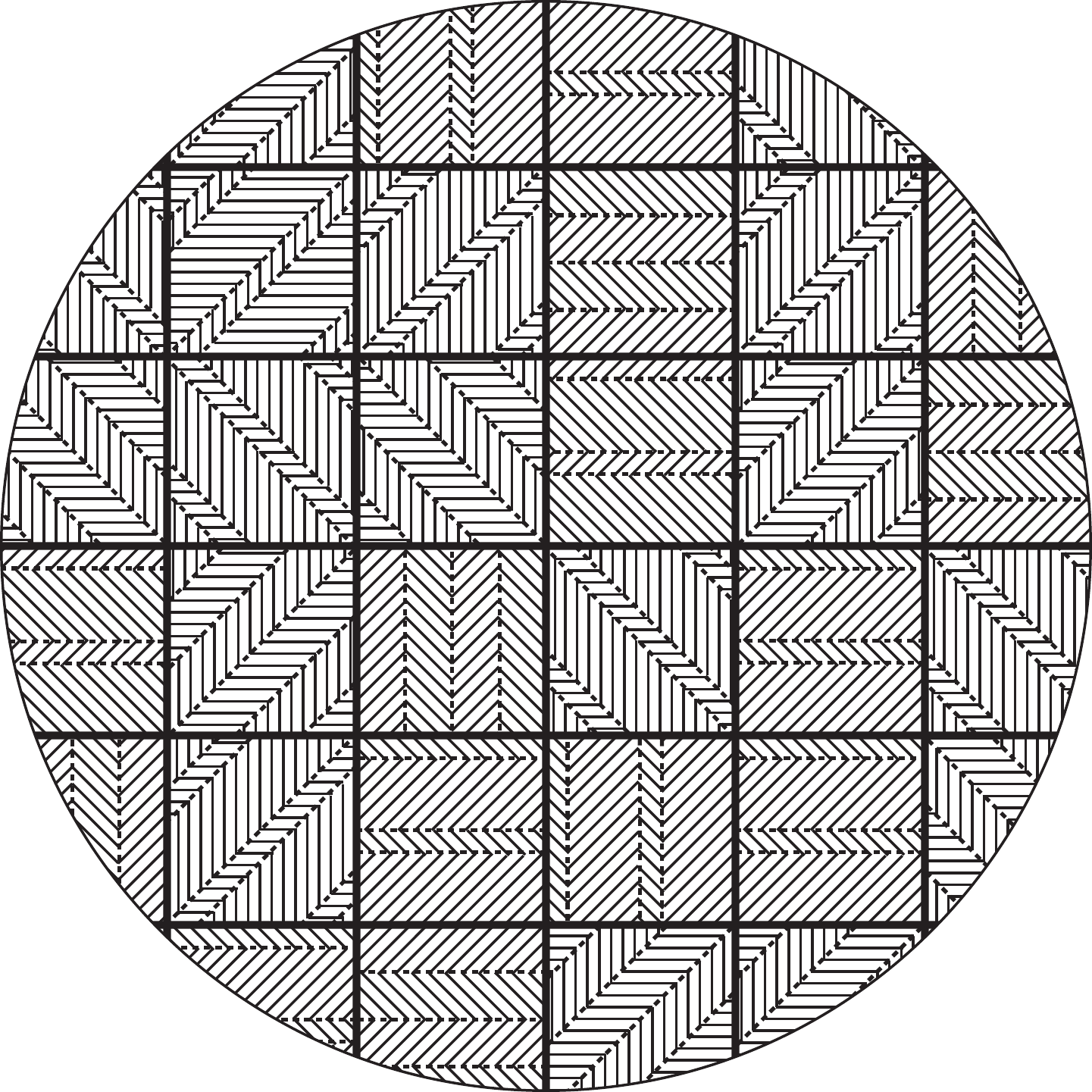}\caption{A ``piecewise herringbone'' pattern representative of the ones we use to construct recovery sequences. Herringbones adapted to constant target defects occupy
individual squares. Each herringbone consists of twinned uni-directional
wrinkles and bands of in-plane shear. Wrinkles are indicated by thin lines, and dashed lines indicate ``internal
walls'' across which their direction rapidly varies. Bold lines
indicate ``external walls'' separating the herringbones.
The number of squares, the number of twins, and the width of the walls
will be optimized.\label{fig:disc_herringbone}}
\end{figure}

The second step zooms into the squares where $\mu \approx \langle \mu \rangle $.
\prettyref{subsec:herringbone-patterns} produces a two-scale
wrinkling pattern known alternatively as the ``herringbone'', ``chevron'', or ``zigzag'' one, and which can be adapted to any constant target defect (e.g., the ones prescribed by $\langle \mu \rangle$).
Such patterns  occur naturally in bi-axially compressed sheets whose displacements are suppressed \cite{cai2011periodic,chen2004family,huang2004evolution,huang2005nonlinear}. 
We were inspired by their analyses in \cite{audoly2008buckling_b,audoly2008buckling_c,kohn2013analysis}, the last of which
comes the closest to what we do here. That reference identifies the scaling law of the minimum energy in a model favoring herringbones. In \prettyref{subsec:herringbone-patterns}, we sharpen this result with a version of the herringbone ansatz whose energy is optimal to leading order.

Finally, \prettyref{subsec:piecewise-herringbone-patterns}  assembles the individually herringboned squares into the
``piecewise herringbone''  pattern depicted
in \prettyref{fig:disc_herringbone}, and estimates its energy. 
We discuss ``walls'' of two types where the direction of wrinkling can rapidly change: ``internal walls'' that reside within the squares, and ``external walls'' at the interfaces between neighboring squares.
The total cost of the walls increases in proportion to their area. It scales, in particular, with the total number of squares. On the other hand, the cost associated with the approximation $\mu\approx\left\langle \mu\right\rangle $
decays with an increasing number of squares. Balancing these, we eventually
deduce that the excess energy implicit in \prettyref{eq:limit_UB} can be made $\lesssim(b/k)^{1/10}$ by using $\sim(k/b)^{1/10}$
squares --- see \prettyref{cor:optimized_ph_pattern}
for the details. \prettyref{subsec:recovery-sequences} concludes with  the formal proof of \prettyref{prop:recovery}.

\subsection{Smooth approximation of tension-free displacements\label{subsec:smooth-approx}}

We start by showing how to approximate tension-free displacements
by smooth and uniformly tension-free ones. It will be important later on that we  work in a topology
for which the functional on the righthand side of \prettyref{eq:limit_UB} is continuous.
Although the trace map $u\mapsto u|_{\partial\Omega}$ fails to be continuous in the weak-$*$ topology on $BD$, it is continuous
in the intermediate topology there. See \prettyref{subsec:Bounded-deformation-maps} for more details.
\begin{lemma}
\label{lem:smoothapproximation} Let $\Omega\subset \mathbb{R}^2$ be bounded, Lipschitz, and strictly star-shaped, and let $p\in W^{2,\infty}(\Omega)$. 
The set of smooth and uniformly tension-free displacements is intermediately dense in the tension-free ones. That is, given any tension-free $u\in BD(\Omega)$,
there exists $\{u_{n}\}_{n\in\mathbb{N}}\subset C^{\infty}(\overline{\Omega};\mathbb{R}^{2})$
satisfying \prettyref{eq:uniformly_tensionfree} such that
\[
u_{n}\to u\quad\text{strongly in }L^{1}(\Omega)\quad\text{and}\quad\int_{\Omega}|e(u_{n})|_{1}\to\int_{\Omega}|e(u)|_{1}\quad\text{as }n\to\infty.
\]
\end{lemma}
\begin{proof}
After a translation we can take $\Omega$ to be strictly
star-shaped with respect to the origin. Also, since $u$
is the intermediate limit of $u_{\lambda}=u-\lambda x$ as $\lambda\to0$,
it suffices to prove the result for displacements that are uniformly
tension-free. So let $u\in BD(\Omega)$ be uniformly tension-free and let
$\lambda>0$ be as in \prettyref{eq:uniformly_tensionfree}. We 
construct the desired approximations $\{u_{n}\}$ via a two step
process involving dilation and mollification. 

First, we dilate: given $\tau\in(0,1)$, set $\Omega_{\tau}=\frac{1}{\tau}\Omega$
and let $u_{\tau}:\Omega_{\tau}\to\mathbb{R}^{2}$ and $p_{\tau}:\Omega_{\tau}\to\mathbb{R}$
be given by
\begin{equation}
u_{\tau}(x)=\frac{1}{\tau}u(\tau x)\quad\text{and}\quad p_{\tau}(x)=\frac{1}{\tau}p(\tau x)\quad\text{for }x\in\Omega_{\tau}.\label{eq:defn_u_p_shifted}
\end{equation}
Since $e(u_{\tau})=e(u)(\tau\cdot)$ it follows from \prettyref{eq:uniformly_tensionfree}
that
\begin{equation}
e(u_{\tau})\leq(\frac{1}{2}\nabla p_{\tau}\otimes\nabla p_{\tau}-\lambda Id)\,dx\quad\text{on }\Omega_{\tau}.\label{eq:inequality_shifted}
\end{equation}
Next, we mollify: fix  $\rho\in C_{c}^{\infty}(B_{1})$ with $\rho \geq 0$ and $\int_{B_1} \rho\,dx = 1$, and denote by $(\cdot)_{\delta}$ the standard mollification
\[
(f)_{\delta}(x)=\int_{\mathbb{R}^{2}}\frac{1}{\delta^{2}}\rho(\frac{x-y}{\delta})f(y)\,dy\quad\text{for }\delta>0.
\]
Recall  $\Omega$ was taken to be strictly star-shaped with respect
to $0$. Thus, there exists $c_{0}(\Omega)>0$ such that 
\begin{equation}
0<\delta<c_{0}(\Omega)(1-\tau)\implies\overline{\Omega}+B_{\delta}\subset\Omega_{\tau}.\label{eq:inclusion}
\end{equation}
So long as $\tau$ and $\delta$ satisfy \prettyref{eq:inclusion},
we may define $u_{\tau,\delta}:\overline{\Omega}\to\mathbb{R}^{2}$
by writing
\[
u_{\tau,\delta}(x)=(u_{\tau})_{\delta}(x),\quad x \in \overline{\Omega}.
\]
Evidently, these are smooth. We proceed to take $\delta\to0$ and $\tau\to1$.

We claim that $u_{\tau,\delta}$ is uniformly tension-free so long
as $\tau$ is close enough to one and  $\delta$ is sufficiently small. 
To see this, note it follows from \prettyref{eq:inequality_shifted}
and our choice to take $\rho\geq0$ that 
\[
e(u_{\tau,\delta})=\left(e(u_{\tau})\right)_{\delta}\leq\left(\frac{1}{2}\nabla p_{\tau}\otimes\nabla p_{\tau}-\lambda Id\right)_{\delta}=\frac{1}{2}\left(\nabla p_{\tau}\otimes\nabla p_{\tau}\right)_{\delta}-\lambda Id\quad\text{on }\overline{\Omega}.
\]
Recalling the definition of $p_{\tau}$ from \prettyref{eq:defn_u_p_shifted}
and applying the triangle inequality, we see that 
\begin{align*}
\left|\left(\nabla p_{\tau}\otimes\nabla p_{\tau}\right)_{\delta}(x)-\nabla p\otimes\nabla p(x)\right| 
 & \leq\int\rho(y)\left|\nabla p\otimes\nabla p(\tau(x-\delta y))-\nabla p\otimes\nabla p(x)\right|\,dy\\
 & \leq||\nabla p||_{L^{\infty}(\Omega)}||\nabla\nabla p||_{L^{\infty}(\Omega)}(|\tau-1||x|+|\tau||\delta|)\lesssim_{\Omega,p}|\tau-1|+|\delta|
\end{align*}
for  $x\in\overline{\Omega}$. Taking $\tau$ close enough to one and $\delta$ sufficiently small ensures that
\[
e(u_{\tau,\delta})\leq\frac{1}{2}\nabla p\otimes\nabla p-\frac{\lambda}{2}Id\quad\text{on }\overline{\Omega}
\]
as desired. 

It remains to choose sequences $\tau_{n}\to1$ and $\delta_{n}\to0$
such that
\[
u_{\tau_{n},\delta_{n}}\to u\ \text{strongly in }L^{1}(\Omega)\quad\text{and}\quad\int_{\Omega}|e(u_{\tau_{n},\delta_{n}})|_{1}\to\int_{\Omega}|e(u)|_{1}\quad\text{as }n\to\infty.
\]
The desired $L^{1}$-convergence holds in any case. For the second
convergence, note that 
\[
e(u_{\tau,\delta})\stackrel{*}{\rightharpoonup}e(u_{\tau})\quad\text{weakly-\ensuremath{*} in }\mathcal{M}(\Omega_{\tau};\text{Sym}_{2})\quad\text{as }\delta\to0
\]
for each fixed $\tau\in(0,1)$. As $|e(u)|(\Omega)<\infty$,
there exist at most countably many $\tau$ for which $|e(u_{\tau})|(\partial\Omega)\neq0$.
Therefore, we can find $\tau_{n}\to1$ and $\delta_{n}\to0$ such
that
\[
\lim_{n\to\infty}\,\int_{\Omega}|e(u_{\tau_{n},\delta_{n}})|_{1}=\lim_{n\to\infty}\lim_{\delta\to0}\,\int_{\Omega}|e(u_{\tau_{n},\delta})|_{1}=\lim_{n\to\infty}\,\int_{\Omega}|e(u_{\tau_{n}})|_{1}=\int_{\Omega}|e(u)|_{1}.
\]
Taking $n$ large enough and setting $u_{n}=u_{\tau_{n},\delta_{n}}$
proves the  result. 
\qed\end{proof}
\prettyref{lem:smoothapproximation} allows us to restrict the proof
of \prettyref{prop:recovery} to $u_{\text{eff}}$ that are smooth
and uniformly tension-free, in which case the target defect can be
given the pointwise definition
\begin{equation}
\mu(x)=\nabla p\otimes\nabla p(x)-2e(u_{\text{eff}})(x)\quad\forall\,x\in\Omega.\label{eq:targetdefect}
\end{equation}
Note $\mu>0$ uniformly on $\Omega$. Thanks to our assumption that $p\in W^{2,\infty}$ so that $\nabla p\in\text{Lip}$, we see that $\mu\in\text{Lip}$.

We now begin the process of constructing admissible displacements
 satisfying
\[
u\approx u_{\text{eff}},\quad w\approx0,\quad\text{and}\quad e(u)+\frac{1}{2}\nabla w\otimes\nabla w\approx\frac{1}{2}\nabla p\otimes\nabla p
\]
and whose energy $E_{b,k,\gamma}$ is nearly minimized. Consider the
change of variables $u \to u_{\text{eff}}+v$.
Since $e(\cdot)$ is linear, the energy depends on $(v,w)$ as
\begin{align*}
E_{b,k,\gamma}(u_{\text{eff}}+v,w) & =\frac{1}{2}\int_{\Omega}|e(v)+\frac{1}{2}\nabla w\otimes\nabla w-\frac{1}{2}\mu|^{2}+\frac{b}{2}\int_{\Omega}|\nabla\nabla w-\nabla\nabla p|^{2}+\frac{k}{2}\int_{\Omega}|w|^{2}\\
 & \qquad+\gamma\left(\int_{\Omega}\frac{1}{2}|\nabla p|^{2}-\int_{\partial\Omega}u_{\text{eff}}\cdot\hat{\nu}\right)-\gamma\int_{\partial\Omega}v\cdot\hat{\nu}
\end{align*}
where we have introduced $\mu$ from \prettyref{eq:targetdefect}
into the stretching term. We treat the simplest case where $\mu$
is constant in \prettyref{subsec:herringbone-patterns}, and then
proceed to discuss more general $\mu$ in \prettyref{subsec:piecewise-herringbone-patterns}.

\subsection{Herringbone patterns adapted to constant defect\label{subsec:herringbone-patterns}} 

Let $Q\subset\mathbb{R}^{2}$ be a square and consider the case of a constant target defect $\mu\in\text{Sym}_{2}$ where $\mu>0$. 
Here, we describe a family of displacements
\[
\{(v_{\text{herr}},w_{\text{herr}})\}\subset W^{1,\infty}(Q;\mathbb{R}^{2})\times W^{2,\infty}(Q)
\]
adapted to $\mu$ in that
\[
v_{\text{herr}}\approx 0,\quad w_{\text{herr}}\approx 0,\quad\text{and}\quad e(v_\text{herr})+\frac{1}{2}\nabla w_\text{herr}\otimes\nabla w_\text{herr} \approx \frac{1}{2}\mu
\]
and that can be made to have nearly minimal energy. \prettyref{fig:herringbone-patterns} depicts  the herringbone patterns we intend to construct. Solid lines indicate wrinkle peaks and troughs. Their direction alternates in twin pairs, in tandem with bands of alternating in-plane shear. 
The ``area fraction'' referred
to there is set by the parameter
\begin{equation}
\theta=\frac{\lambda_{1}}{\lambda_{1}+\lambda_{2}}\in(0,\frac{1}{2}]\quad\text{where }0<\lambda_{1}\leq\lambda_{2}\text{ are the eigenvalues of } \mu.\label{eq:area_fraction}
\end{equation}
Panel (a) depicts the isotropic case $\theta=\frac{1}{2}$ in which $\mu$ is a multiple of the identity. 
Panel (b) shows an anisotropic case where $\theta\in(0,\frac{1}{2})$.
Sending $\theta \to 0$ recovers uni-directional wrinkles as in Panel (c).

\begin{figure}[tb]
\centering
\subfloat[]{\includegraphics[width=0.2\paperwidth]{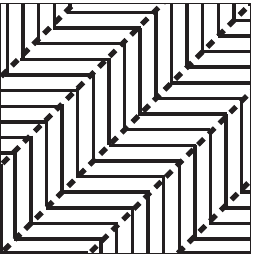}}\hspace{.06\textwidth}\subfloat[]{\includegraphics[width=0.2\paperwidth]{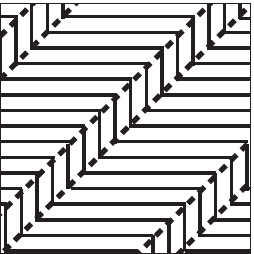}}\hspace{.06\textwidth}\subfloat[]{\includegraphics[width=0.2\paperwidth]{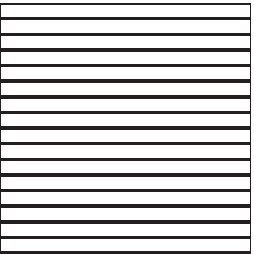}}\caption{Herringbone patterns with variable area fractions adapted to constant
defects. Solid lines depict wrinkle peaks and troughs, while dashed
lines indicate the presence of ``internal walls'' across which the
wrinkling direction changes. Panel (a) depicts an ``isotropic'' herringbone for isotropic defect. 
Panel (b) shows an ``anisotropic'' herringbone for
anisotropic defect. Panel (c) depicts uni-directional wrinkles arising for defect of rank
one. \label{fig:herringbone-patterns}}
\end{figure}

Our herringbones will be parameterized by 
\begin{equation}
l_\text{wr}\in(0,\infty),\quad l_\text{sh}\in(0,\infty),\quad\text{and}\quad\delta_\text{int}\in(0,\frac{1}{2}\theta l_\text{sh}).\label{eq:parameters_herringbone_reqs}
\end{equation}
The first parameter $l_\text{wr}$ sets the lengthscale
of the wrinkles. The second parameter $l_\text{sh}$ sets the magnitude
of the in-plane shear. There is an energetic
cost associated to changing the direction of wrinkling, and $\delta_\text{int}$
sets the thickness of the associated ``internal walls''. 
(Such walls are internal in the sense that they lie
within the herringbone, as opposed to the ``external walls'' introduced in \prettyref{subsec:piecewise-herringbone-patterns}.) 
The energy estimates obtained in this section apply so long as \prettyref{eq:parameters_herringbone_reqs} holds. However, it will be convenient going forward to keep in mind the special case where
\[
\left(\frac{b}{k}\right)^{1/4}=l_\text{wr}\ll l_\text{sh}\ll\text{diam}\,Q\quad\text{and}\quad l_\text{wr}\lesssim\delta_\text{int}\lesssim_{\mu}l_\text{sh}
\]
in which case there holds
\[
E_{b,k,\gamma}(v_{\text{herr}},w_{\text{herr}})=(2\sqrt{bk}+\gamma)\cdot \frac{1}{2}\text{tr}\,\mu|Q|+O(\frac{\delta_\text{int}}{l_\text{sh}}).
\]
The error term is due to the internal walls. Combining this estimate with the \emph{a priori} lower bounds from \prettyref{sec:Gamma_liminf}, we see that herringbones for which $\frac{\delta_\text{int}}{l_\text{sh}}\ll2\sqrt{bk}+\gamma$ are optimal at leading order. We turn to construct a general herringbone now, and to estimate its energy.

\subsubsection{Constructing the herringbone\label{subsec:Constructing-the-herringbone}}

Decompose the target defect as
\[
\mu=\lambda_{1}\hat{\eta}_{1}\otimes\hat{\eta}_{1}+\lambda_{2}\hat{\eta}_{2}\otimes\hat{\eta}_{2}
\]
where $\{\hat{\eta}_{1},\hat{\eta}_{2}\}$ are orthonormal eigenvectors
corresponding to the eigenvalues $\{\lambda_{1},\lambda_{2}\}$.

\vspace{.5em}

\paragraph*{\uline{Step 1: produce bands of alternating in-plane shear at
scale \mbox{$l_\emph{sh}$}}. \quad}

We start by introducing an in-plane displacement to transform the target
defect from rank two to rank one. Define $v_\text{sh}:\mathbb{R}^{2}\to\mathbb{R}^{2}$
by 
\begin{equation}
v_\text{sh}(x)=\sqrt{2}l_\text{sh}A\left(\frac{x\cdot(\hat{\eta}_{2}-\hat{\eta}_{1})}{\sqrt{2}l_\text{sh}}\right)(\hat{\eta}_{2}+\hat{\eta}_{1}),\quad x\in\mathbb{R}^{2}\label{eq:shearbands}
\end{equation}
where $A:\mathbb{R}\to\mathbb{R}$ is the one-periodic extension of
\[
A(t)=\begin{cases}
\frac{\lambda_{2}}{2}t & 0\leq t<\theta\\
\frac{\lambda_{2}}{2}\theta-\frac{\lambda_{1}}{2}(t-\theta) & \theta\leq t\leq1
\end{cases},\quad t\in[0,1].
\]
It follows from the definition of $\theta$ in \prettyref{eq:area_fraction}
that $A$ is Lipschitz. Indeed, $A'(t)$ equals to $\frac{\lambda_{2}}{2}$ for $t\in (0,\theta)$  and $-\frac{\lambda_{1}}{2}$ for $t\in (\theta,1)$, so that it integrates to zero.

Now as
\begin{equation}
\nabla v_\text{sh}=A'\left(\frac{x\cdot(\hat{\eta}_{2}-\hat{\eta}_{1})}{\sqrt{2}l_\text{sh}}\right)(\hat{\eta}_{2}+\hat{\eta}_{1})\otimes(\hat{\eta}_{2}-\hat{\eta}_{1})\label{eq:shearbands_gradient}
\end{equation}
we see that
\[
e(v_\text{sh})-\frac{1}{2}\mu 
  =\left[A'\left(\frac{x\cdot(\hat{\eta}_{2}-\hat{\eta}_{1})}{\sqrt{2}l_\text{sh}}\right)-\frac{\lambda_{2}}{2}\right]\hat{\eta}_{2}\otimes\hat{\eta}_{2}-\left[A'\left(\frac{x\cdot(\hat{\eta}_{2}-\hat{\eta}_{1})}{\sqrt{2}l_\text{sh}}\right)+\frac{\lambda_{1}}{2}\right]\hat{\eta}_{1}\otimes\hat{\eta}_{1}.
\]
Recalling that $A'$ is alternatively equal to $\frac{\lambda_{2}}{2}$
or $-\frac{\lambda_{1}}{2}$, we deduce that
\begin{equation}
e(v_\text{sh})-\frac{1}{2}\mu=-\frac{1}{2}\text{tr}\,\mu\cdot\hat{\eta}_\text{herr}\otimes\hat{\eta}_\text{herr}\quad\text{on }\mathbb{R}^{2}\label{eq:leftover_strain}
\end{equation}
where the unit vector field $\hat{\eta}_\text{herr}:\mathbb{R}^{2}\to S^{1}$
satisfies
\begin{equation}
\hat{\eta}_\text{herr}(x)=\begin{cases}
\hat{\eta}_{1} & 0\leq x\cdot\frac{\hat{\eta}_{2}-\hat{\eta}_{1}}{\sqrt{2}}<\theta l_\text{sh}\\
\hat{\eta}_{2} & \theta l_\text{sh}\leq x\cdot\frac{\hat{\eta}_{2}-\hat{\eta}_{1}}{\sqrt{2}}\leq1
\end{cases}\quad\text{when}\quad x\cdot\frac{\hat{\eta}_{2}-\hat{\eta}_{1}}{\sqrt{2}}\in[0,1]\label{eq:herringbone_unitvectorfield}
\end{equation}
and is otherwise periodic. Thus $v_\text{sh}$ transforms $\mu$ into a
defect which is piecewise constant and rank one.

In the next step, we introduce uni-directional wrinkles in the direction
of $\hat{\eta}_\text{herr}$. Note the jump set of $\hat{\eta}_\text{herr}$ is 
\begin{equation}
J_{\hat{\eta}_\text{herr}}=\left\{ x\in\mathbb{R}^{2}:x\cdot\frac{\hat{\eta}_{2}-\hat{\eta}_{1}}{\sqrt{2}}=0+l_\text{sh}\mathbb{Z}\right\} \cup\left\{ x\in\mathbb{R}^{2}: x\cdot \frac{\hat{\eta}_{2}-\hat{\eta}_{1}}{\sqrt{2}} =\theta l_\text{sh}+l_\text{sh}\mathbb{Z}\right\}.\label{eq:jumpset}
\end{equation}
It consists of (countably many) parallel lines at distances $\theta l_\text{sh}$
and $(1-\theta)l_\text{sh}$ apart. The pointwise estimates
\begin{equation}
||v_\text{sh}||_{L^{\infty}}\lesssim\text{tr}\,\mu\cdot l_\text{sh}\quad\text{and}\quad||\nabla v_\text{sh}||_{L^{\infty}}\lesssim\text{tr}\,\mu\label{eq:pointwise_shear_estimates}
\end{equation}
follow from \prettyref{eq:shearbands} and \prettyref{eq:shearbands_gradient}.

\vspace{.5em}

\paragraph*{\uline{Step 2: superimpose twin bands of wrinkles at scale \mbox{$l_\emph{wr}$}}. \quad}

In this step we construct uni-directional wrinkles to alleviate the strain
left over from Step 1. Define $v_\text{wr}:\mathbb{R}^{2}\to\mathbb{R}^{2}$
and $w_\text{wr}:\mathbb{R}^{2}\to\mathbb{R}$ by 
\begin{equation}
v_\text{wr}(x)=\frac{1}{2}\text{tr}\,\mu\cdot l_\text{wr}V\left(\frac{x\cdot\hat{\eta}_\text{herr}(x)}{l_\text{wr}}\right)\hat{\eta}_\text{herr}(x)\quad\text{and}\quad w_\text{wr}(x)=\sqrt{\text{tr}\,\mu}\cdot l_\text{wr}W\left(\frac{x\cdot\hat{\eta}_\text{herr}(x)}{l_\text{wr}}\right),\quad x\in\mathbb{R}^{2}.\label{eq:1dwrinkles}
\end{equation}
Here, $W:\mathbb{R}\to\mathbb{R}$ is given by
\[
W(t)=\sqrt{2}\cos(t),\quad t\in\mathbb{R}
\]
and $V:\mathbb{R}\to\mathbb{R}$ is the unique $2\pi$-periodic
solution of
\[
V'(t)+|W'(t)|^{2}=1\quad\forall\,t\in\mathbb{R},\quad\text{with}\quad V(0)=0.
\]
Such a solution exists as $\fint_{0}^{2\pi}|W'|^{2}=1$.

Evidently, there holds
\begin{equation}
\nabla v_\text{wr}=\frac{1}{2}\text{tr}\,\mu\cdot V'\left(\frac{x\cdot\hat{\eta}_\text{herr}}{l_\text{wr}}\right)\hat{\eta}_\text{herr}\otimes\hat{\eta}_\text{herr}\quad\text{and}\quad\nabla w_\text{wr}=\sqrt{\text{tr}\,\mu}\cdot W'\left(\frac{x\cdot\hat{\eta}_\text{herr}}{l_\text{wr}}\right)\hat{\eta}_\text{herr}\quad\text{on }\mathbb{R}^{2}\backslash J_{\hat{\eta}_\text{herr}}\label{eq:1dwrinkles_gradient}
\end{equation}
so that 
\begin{equation}
e(v_\text{wr})+\frac{1}{2}\nabla w_\text{wr}\otimes\nabla w_\text{wr}=\frac{1}{2}\text{tr}\,\mu\cdot\hat{\eta}_\text{herr}\otimes\hat{\eta}_\text{herr}\quad\text{on }\mathbb{R}^{2}\backslash J_{\hat{\eta}_\text{herr}}.\label{eq:wrinkledstrain}
\end{equation}
Adding up \prettyref{eq:leftover_strain} and \prettyref{eq:wrinkledstrain}
we see that 
\begin{equation}
e(v_\text{sh}+v_\text{wr})+\frac{1}{2}\nabla w_\text{wr}\otimes\nabla w_\text{wr}=\frac{1}{2}\mu\quad\text{on }\mathbb{R}^{2}\backslash J_{\hat{\eta}_\text{herr}},\label{eq:addedstrain}
\end{equation}
although as $v_\text{wr}$ and $w_\text{wr}$ may jump across $J_{\hat{\eta}_\text{herr}}$ this identity may fail to hold there.
The pointwise estimates
\begin{align}
 & ||v_\text{wr}||_{L^{\infty}}\lesssim\text{tr}\,\mu\cdot l_\text{wr},\quad||\nabla v_\text{wr}||_{L^{\infty}(\mathbb{R}^{2}\backslash J_{\hat{\eta}_\text{herr}})}\lesssim\text{tr}\,\mu,\label{eq:in-plane-wr}\\
 & ||w_\text{wr}||_{L^{\infty}}\lesssim\sqrt{\text{tr}\,\mu}\cdot l_\text{wr},\quad||\nabla w_\text{wr}||_{L^{\infty}(\mathbb{R}^{2}\backslash J_{\hat{\eta}_\text{herr}})}\lesssim\sqrt{\text{tr}\,\mu},\quad||\nabla\nabla w_\text{wr}||_{L^{\infty}(\mathbb{R}^{2}\backslash J_{\hat{\eta}_\text{herr}})}\lesssim\frac{\sqrt{\text{tr}\,\mu}}{l_\text{wr}}\label{eq:out-of-plane-wr}
\end{align}
follow from \prettyref{eq:1dwrinkles} and \prettyref{eq:1dwrinkles_gradient}.

\vspace{.5em}

\paragraph*{\uline{Step 3: join the wrinkles across internal walls at scale
\mbox{$\delta_\emph{int}$}}. \quad}

Finally, in order to ensure their bending energy is finite, we must smooth the wrinkles across the jump set $J_{\hat{\eta}_\text{herr}}$.
For simplicity, and because it will not affect the estimates 
at leading order, we use a cutoff function $\chi_\text{int}$ supported away from $J_{\hat{\eta}_\text{herr}}$ to define the
internal walls. (Our choice to use a simple cutoff here is one of the factors behind the assumption \prettyref{eq:A2b}.)

The shortest distance between two lines belonging to $J_{\hat{\eta}_{\text{herr}}}$
is $\theta l_\text{sh}$. Since by hypothesis $\delta_\text{int}<\theta l_\text{sh}$,
there exists a smooth cutoff function $\chi_\text{int}\in C^{\infty}(\mathbb{R}^{2})$
such that
\begin{itemize}
\item $0\leq\chi_\text{int}\leq1$ 
\item $\chi_\text{int}(x)=0$ if $d(x,J_{\hat{\eta}_{\text{herr}}})\leq\frac{1}{2}\delta_\text{int}$
and $\chi_\text{int}(x)=1$ if $d(x,J_{\hat{\eta}_{\text{herr}}})\geq\delta_\text{int}$,
\item $||\nabla\chi_\text{int}||_{L^{\infty}}\lesssim\frac{1}{\delta_\text{int}}$
and $||\nabla\nabla\chi_\text{int}||_{L^{\infty}}\lesssim\frac{1}{\delta_\text{int}^{2}}$ 
\end{itemize}
where the constants implicit in the above are independent of all
parameters. Let $v_{\text{herr}}:Q\to\mathbb{R}^{2}$
and $w_{\text{herr}}:Q\to\mathbb{R}$ be given by 
\[
v_{\text{herr}}(x)=v_\text{sh}(x)+v_\text{wr}(x)\cdot\chi_\text{int}(x)\quad\text{and}\quad w_{\text{herr}}(x)=w_\text{wr}(x)\cdot\chi_\text{int}(x),\quad x\in Q.
\]
This completes our construction of the herringbone. Note it follows from \prettyref{eq:addedstrain} and the definition of $\chi_\text{int}$ that
\begin{equation}
e(v_{\text{herr}})+\frac{1}{2}\nabla w_{\text{herr}}\otimes\nabla w_{\text{herr}}=\frac{1}{2}\mu\quad\text{on }d(\cdot,J_{\hat{\eta}_\text{herr}})\geq\delta_\text{int}.\label{eq:vanishing-strain-bulk}
\end{equation}
The pointwise estimates 
\begin{align}
 & ||v_\text{herr}||_{L^{\infty}}\lesssim\text{tr}\,\mu\cdot l_\text{sh}\left(1\vee\frac{l_\text{wr}}{l_\text{sh}}\right),\quad||\nabla v_\text{herr}||_{L^{\infty}}\lesssim\text{tr}\,\mu\cdot\left(1\vee\frac{l_\text{wr}}{\delta_\text{int}}\right)\label{eq:pointwise_estimates_herr_v}\\
 & ||w_\text{herr}||_{L^{\infty}}\lesssim\sqrt{\text{tr}\,\mu}\cdot l_\text{wr},\quad||\nabla w_\text{herr}||_{L^{\infty}}\lesssim\sqrt{\text{tr}\,\mu}\cdot\left(1\vee\frac{l_\text{wr}}{\delta_\text{int}}\right)\label{eq:pointwise_estimates_herr_w1}\\
 & ||\nabla\nabla w_\text{herr}||_{L^{\infty}}\lesssim\frac{\sqrt{\text{tr}\,\mu}}{l_\text{wr}}\cdot\left(1\vee\frac{l_\text{wr}^{2}}{\delta_\text{int}^{2}}\right)\label{eq:pointwise_estimates_herr_w2}
\end{align}
carry over from \prettyref{eq:pointwise_shear_estimates}, \prettyref{eq:in-plane-wr},
\prettyref{eq:out-of-plane-wr}, and the properties of $\chi_\text{int}$
listed above.

\subsubsection{Energy estimates for the herringbone\label{subsec:Energy-estimates-for-herringbone}}

We turn to estimate the energy of the herringbones defined
above. It will be convenient to decompose $Q$ into its ``wall''
and ``bulk'' regions given by
\begin{equation}
Q_\text{wall}=\left\{ x\in Q:d(x,J_{\hat{\eta}_{\text{herr}}})<\delta_\text{int}\right\} \quad\text{and}\quad Q_\text{bulk}=Q\backslash Q_\text{wall}.\label{eq:Qwall}
\end{equation}
Define
\begin{align}
a_{0}(l_\text{sh},l_\text{wr};\mu,Q) & =\text{tr}\,\mu\cdot\int_{Q}|W\left(\frac{x\cdot\hat{\eta}_\text{herr}(x)}{l_\text{wr}}\right)|^{2}\,dx,\label{eq:a_0}\\
a_{1}(l_\text{sh},l_\text{wr},\delta_\text{int};\mu,Q) & =\left(1\vee\frac{l_\text{wr}^{4}}{\delta_\text{int}^{4}}\right)|Q_\text{wall}|.\label{eq:a_1}
\end{align}

\begin{lemma}
\label{lem:herringbone_estimates} 
Let $Q$ be a square, let $\mu \in \emph{Sym}_2$ have $\mu > 0$, and let $l_\emph{wr}$, $l_\emph{sh}$, and $\delta_\emph{in}$ satisfy \prettyref{eq:parameters_herringbone_reqs}. The herringbones constructed in \prettyref{subsec:Constructing-the-herringbone} obey the following
 estimates:
\begin{itemize}
\item the stretching energy satisfies
\[
\int_{Q}|e(v_\emph{herr})+\frac{1}{2}\nabla w_\emph{herr}\otimes\nabla w_\emph{herr}-\frac{1}{2}\mu|^{2}\lesssim|\mu|^{2}a_{1}(\mu,Q);
\]
\item the bending energy satisfies 
\[
\int_{Q}|\nabla\nabla w_\emph{herr}|^{2}\leq\frac{a_{0}(\mu,Q)+C|\mu|a_{1}(\mu,Q)}{l_\emph{wr}^{2}}
\]
where the constant $C$ is independent of all parameters;
\item the substrate energy satisfies
\[
\int_{Q}|w_\emph{herr}|^{2}\leq a_{0}(\mu,Q)l_\emph{wr}^{2}.
\]
\end{itemize}
\end{lemma}
\begin{proof}
We start by estimating the stretching energy, which requires estimating
the strain $\varepsilon$. Recall from \prettyref{eq:vanishing-strain-bulk}
that
\begin{equation}
\varepsilon=e(v_\text{herr})+\frac{1}{2}\nabla w_\text{herr}\otimes\nabla w_\text{herr}-\frac{1}{2}\mu=0\quad\text{on }Q_\text{bulk}.\label{eq:bulk_strain_est}
\end{equation}
To handle the wall region, we apply the
pointwise estimates from \prettyref{eq:pointwise_estimates_herr_v}
and \prettyref{eq:pointwise_estimates_herr_w1} to get that
\begin{equation}
|\varepsilon|\leq|\nabla v_\text{herr}|+\frac{1}{2}|\nabla w_\text{herr}|^{2}+\frac{1}{2}|\mu|\lesssim|\mu|\cdot(1\vee\frac{l_\text{wr}^{2}}{\delta_\text{int}^{2}})\quad\text{on }Q_\text{wall}.\label{eq:wall_strain_est}
\end{equation}
Combining \prettyref{eq:bulk_strain_est} and \prettyref{eq:wall_strain_est} yields that
\[
\int_{Q}|\varepsilon|^{2} =\left(\int_{Q_\text{wall}}+\int_{Q_\text{bulk}}\right)|\varepsilon|^{2} 
  \lesssim|\mu|^{2}\left(1\vee\frac{l_\text{wr}^{4}}{\delta_\text{int}^{4}}\right)|Q_\text{wall}|=|\mu|^{2}a_{1}
\]
according to \prettyref{eq:a_1}. This proves the desired estimate
on the stretching energy.

Next we estimate the bending energy, being careful to keep track of the important prefactors. Since $w_\text{herr}=w_\text{wr}$ in the bulk region
and $\hat{\eta}_\text{herr}$ is locally constant there, we see from
\prettyref{eq:1dwrinkles} that
\begin{equation}
\nabla\nabla w_\text{herr}=\frac{\sqrt{\text{tr}\,\mu}}{l_\text{wr}}W''\left(\frac{x\cdot\hat{\eta}_\text{herr}}{l_\text{wr}}\right)\hat{\eta}_\text{herr}\otimes\hat{\eta}_\text{herr}\quad\text{on }Q_\text{bulk}.\label{eq:bulk_bending_est}
\end{equation}
On the other hand, it follows from the last estimate in \prettyref{eq:pointwise_estimates_herr_w2}
that 
\begin{equation}
|\nabla\nabla w_\text{herr}|\lesssim\frac{\sqrt{|\mu|}}{l_\text{wr}}\left(1\vee\frac{l_\text{wr}^{2}}{\delta_\text{int}^{2}}\right)\quad\text{on }Q_\text{wall}.\label{eq:wall_bending_est}
\end{equation}
Using \prettyref{eq:bulk_bending_est} and \prettyref{eq:wall_bending_est}
and the fact that $|W|=|W''|$ we deduce that
\begin{align*}
\int_{Q}|\nabla\nabla w_\text{herr}|^{2} & =\left(\int_{Q_\text{bulk}}+\int_{Q_\text{wall}}\right)|\nabla\nabla w_\text{herr}|^{2}\\
 & \leq\frac{\text{tr}\,\mu}{l_\text{wr}^{2}}\int_{Q}\left|W\left(\frac{x\cdot\hat{\eta}_\text{herr}}{l_\text{wr}}\right)\right|^{2}+C\frac{|\mu|}{l_\text{wr}^{2}}\left(1\vee\frac{l_\text{wr}^{4}}{\delta_\text{int}^{4}}\right)|Q_\text{wall}|=\frac{a_{0}+C|\mu|a_{1}}{l_\text{wr}^{2}}
\end{align*}
by the definitions of $a_{0}$ and $a_{1}$ in \prettyref{eq:a_0}
and \prettyref{eq:a_1}. 

We finish with the substrate energy. Evidently, there holds
\[
|w_\text{herr}|\leq|w_\text{wr}|=\sqrt{\text{tr}\,\mu}\cdot l_\text{wr}\left|W\left(\frac{x\cdot\hat{\eta}_\text{herr}}{l_\text{wr}}\right)\right|\quad\text{on }Q
\]
given that $\chi_\text{int}\leq1$ and due to the out-of-plane part of
\prettyref{eq:1dwrinkles}. It follows that
\[
\int_{Q}|w_\text{herr}|^{2}\leq\int_{Q}|w_\text{wr}|^{2}=\text{tr}\,\mu\cdot l_\text{wr}^{2}\int_{Q}\left|W\left(\frac{x\cdot\hat{\eta}_\text{herr}}{l_\text{wr}}\right)\right|^{2}=l_\text{wr}^{2}a_{0}
\]
as desired. \qed\end{proof}
Next, we estimate the quantities $a_{0}$ and $a_{1}$ defined in
\prettyref{eq:a_0} and \prettyref{eq:a_1}.
\begin{lemma}
\label{lem:constant_estimates} We have the estimates
\begin{align*}
\left|a_{0}(\mu,Q)-\emph{tr}\,\mu\cdot|Q|\right| & \lesssim\emph{tr}\,\mu\cdot\frac{l_\emph{wr}}{l_\emph{sh}}\left(1\vee\frac{l_\emph{sh}}{\emph{diam}\,Q}\right)|Q|,\\
a_{1}(\mu,Q) & \lesssim\frac{\delta_\emph{int}}{l_\emph{sh}}\left(1\vee\frac{l_\emph{wr}^{4}}{\delta_\emph{int}^{4}}\right)\left(1\vee\frac{l_\emph{sh}}{\emph{diam}\,Q}\right)|Q|.
\end{align*}
\end{lemma}
\begin{proof}
We start with $a_{1}$. Recall the definitions of $J_{\hat{\eta}_{\text{herr}}}$ and $Q_\text{wall}$ from  \prettyref{eq:jumpset} and \prettyref{eq:Qwall}. The former consists of parallel lines
at distances $\theta l_\text{sh}$ and $(1-\theta)l_\text{sh}$ apart, the total
number of which intersecting $Q$ is $\lesssim\frac{\text{diam}\,Q}{l_\text{sh}}\vee1$.
Each such line contributes area $\lesssim\delta_\text{int}\cdot\text{diam}\,Q$
to $Q_\text{wall}$. Hence,
\begin{equation}
|Q_\text{wall}|\lesssim\left(\frac{\text{diam}\,Q}{l_\text{sh}}\vee1\right)\cdot\left(\delta_\text{int}\cdot\text{diam}\,Q\right)=\frac{\delta_\text{int}}{l_\text{sh}}\left(1\vee\frac{l_\text{sh}}{\text{diam}\,Q}\right)|Q|.\label{eq:Q_wall_estimate}
\end{equation}
It follows from \prettyref{eq:a_1}
that 
\[
a_{1}\lesssim\frac{\delta_\text{int}}{l_\text{sh}}\left(1\vee\frac{l_\text{wr}^{4}}{\delta_\text{int}^{4}}\right)\left(1\vee\frac{l_\text{sh}}{\text{diam}\,Q}\right)|Q|.
\]
as required.

We turn to estimate $a_{0}$. First, we claim that 
\begin{equation}
\left|\int_{S}W^{2}\left(\frac{x\cdot\hat{\eta}}{l}\right)\,dx-|S|\right|\lesssim\mathcal{H}^{1}(\partial S)l\quad\forall\,l\in(0,\infty)\label{eq:translation_bd}
\end{equation}
whenever $\hat{\eta}\in S^{1}$ and $S\subset\mathbb{R}^{2}$ is a 
bounded, measurable set. The constant implicit in  \prettyref{eq:translation_bd} is independent
of $l$, $\hat{\eta}$, and $S$. To prove it, begin by noting that $W^{2}(\cdot)+W^{2}(\cdot+\frac{\pi}{2})=2$ from which it follows that
\begin{equation}
\int_{S}W^{2}\left(\frac{x\cdot\hat{\eta}}{l}\right)\,dx+\int_{S}W^{2}\left(\frac{(x+\frac{\pi}{2}l\hat{\eta})\cdot\hat{\eta}}{l}\right)\,dx=2|S|.\label{eq:sum_law}
\end{equation}
On the other hand, a change of variables shows that
\begin{equation}
\left|\int_{S}W^{2}\left(\frac{(x+\frac{\pi}{2}l\hat{\eta})\cdot\hat{\eta}}{l}\right)-\int_{S}W^{2}\left(\frac{x\cdot\hat{\eta}}{l}\right)\right|=\left|\left(\int_{S+\frac{\pi}{2}l\hat{\eta}}-\int_{S}\right)W^{2}\left(\frac{x\cdot\hat{\eta}}{l}\right)\right|\lesssim\left|S\triangle\left(S+\frac{\pi}{2}l\hat{\eta}\right)\right|.\label{eq:diff_est}
\end{equation}
To control the righthand side we use that
\begin{equation}
\left|S\triangle(S+l\hat{\eta})\right|\leq\mathcal{H}^{1}(\partial S)l\quad\forall\,l\in(0,\infty), \label{eq:geometric_fact}
\end{equation}
which is a direct consequence of  \cite[Theorem 3]{schymura2014upper}. Applying \prettyref{eq:sum_law}-\prettyref{eq:geometric_fact} proves  \prettyref{eq:translation_bd}.

With the estimate \prettyref{eq:translation_bd} in hand we can easily
handle $a_{0}$. From its definition in \prettyref{eq:herringbone_unitvectorfield},
we see that $\hat{\eta}_\text{herr}$ takes on only the values $\hat{\eta}_{1}$
and $\hat{\eta}_{2}$. Decompose $Q$ according to
\[
Q=S_{1}\cup S_{2}\quad\text{where}\quad S_{i}=\left\{ x\in Q:\hat{\eta}_{\text{herr}}=\hat{\eta}_{i}\right\} ,\quad i=1,2.
\]
As in the proof of \prettyref{eq:Q_wall_estimate}, we note that
\begin{equation}
\mathcal{H}^{1}(\partial S_{i})\lesssim\left(\frac{\text{diam}\,Q}{l_\text{sh}}\vee1\right)\cdot\text{diam}\,Q=\frac{1}{l_\text{sh}}\left(1\vee\frac{l_\text{sh}}{\text{diam}\,Q}\right)|Q|\label{eq:perimeter_bd}
\end{equation}
for $i=1,2$. Hence,
\begin{align*}
\left|a_{0}-\text{tr}\,\mu\cdot|Q|\right| & =\text{tr}\,\mu\cdot\left|\int_{Q}W^{2}\left(\frac{x\cdot\hat{\eta}_\text{herr}}{l_\text{wr}}\right)-|Q|\right|\leq\text{tr}\,\mu\cdot\sum_{i=1}^{2}\left|\int_{S_{i}}W^{2}\left(\frac{x\cdot\hat{\eta}_{i}}{l_\text{wr}}\right)-|S_{i}|\right|\\
 & \lesssim\text{tr}\,\mu\cdot\sum_{i=1}^{2}\mathcal{H}^{1}(\partial S_{i})l_\text{wr}\lesssim\text{tr}\,\mu\cdot\frac{l_\text{wr}}{l_\text{sh}}\left(1\vee\frac{l_\text{sh}}{\text{diam}\,Q}\right)|Q|
\end{align*}
where in the second line we applied \prettyref{eq:translation_bd}
and \prettyref{eq:perimeter_bd}.
\qed\end{proof}

\subsection{Piecewise herringbone patterns adapted to variable defect\label{subsec:piecewise-herringbone-patterns}}

We return to $\Omega\subset\mathbb{R}^{2}$ which for our present purposes must only be a bounded and Lipschitz domain, and
consider a target defect $\mu:\Omega\to\text{Sym}_{2}$ that is positive definite and Lipschitz continuous.
Our task is to construct a family of displacements
\[
\{(v_{\text{p.h.}},w_{\text{p.h.}})\}\subset W^{1,\infty}(\Omega;\mathbb{R}^{2})\times W^{2,\infty}(\Omega)
\]
adapted to $\mu$ in that 
\[
v_{\text{p.h.}}\approx 0,\quad w_{\text{p.h.}}\approx 0,\quad\text{and}\quad e(v_\text{p.h.})+\frac{1}{2}\nabla w_\text{p.h.}\otimes\nabla w_\text{p.h.} \approx \frac{1}{2}\mu
\] 
and whose energy can be made optimal at leading order. After approximating $\mu$ by a piecewise constant target defect $\langle \mu \rangle$ defined on a lattice of squares, we piece together a well-chosen family of herringbones from \prettyref{subsec:herringbone-patterns} to form our ``piecewise herringbone'' pattern. The result is depicted in \prettyref{fig:disc_herringbone} (see also Panel (a) in \prettyref{fig:three_constructions}).

Our piecewise herringbones will be parameterized by
\begin{equation}
l_\text{wr}\in(0,\infty),\quad l_\text{sh}\in(0,\infty),\quad l_\text{avg}\in(0,\infty),\quad\delta_\text{int}\in(0,\frac{1}{4}\frac{\lambda}{\Lambda}l_\text{sh}),\quad\text{and}\quad\delta_\text{ext}\in(0,\frac{1}{2}l_\text{avg})\label{eq:parameters_ph_herringbone_reqs}
\end{equation}
where $\lambda,\Lambda\in(0,\infty)$ satisfy
\begin{equation}
\lambda Id\leq\mu(x)\leq\Lambda Id\quad\forall\,x\in\Omega.\label{eq:uniformly_elliptic}
\end{equation}
The parameters $l_\text{wr}$, $l_\text{sh}$, and $\delta_\text{int}$ should already
be familiar from \prettyref{subsec:herringbone-patterns}:
these set the lengthscales of the wrinkles, the in-plane
shear, and the internal walls of the herringbones. The first new parameter
$l_\text{avg}$  gives the ``averaging'' lengthscale across which we treat $\mu$ as if it were constant. 
It will be proportional to the diameter of the herringboned squares. 
The parameter $\delta_\text{ext}$ sets the thickness of the ``external walls'' between neighboring squares. 
We shall construct a piecewise herringbone for any choice of parameters satisfying \prettyref{eq:parameters_ph_herringbone_reqs}. However, in anticipation of the optimization that is to come, we note that ones for which
\[
\left(\frac{b}{k}\right)^{1/4}=l_\text{wr}\ll l_\text{sh}\ll l_\text{avg}\ll\text{diam}\,\Omega,\quad l_\text{wr}\lesssim\delta_\text{int}\lesssim_{\mu}l_\text{sh},\quad\text{and}\quad l_\text{sh}\lesssim\delta_\text{ext}\lesssim l_\text{avg}
\]
satisfy
\[
E_{b,k,\gamma}(v_{\text{p.h.}},w_{\text{p.h.}})=(2\sqrt{bk}+\gamma)\cdot \frac{1}{2}\int_{\Omega}\text{tr}\,\mu\,dx+O(l_\text{avg}^{2}\vee\frac{\delta_\text{int}}{l_\text{sh}}\vee\frac{\delta_\text{ext}}{l_\text{avg}}).
\]
The error term accounts for the cost of the approximation $\mu\approx \langle \mu \rangle$ as well as that of the walls. When it is negligible, our piecewise herringbones are optimal at leading order (again, the requisite lower bound is contained in the results of \prettyref{sec:Gamma_liminf}). Minimizing over the free parameters maximizes the range of this result --- see  \prettyref{cor:optimized_ph_pattern} for the details. We turn now to construct a general piecewise herringbone, and to estimate its energy.

\subsubsection{Constructing the piecewise herringbone\label{subsec:Constructing-the-piecewise-herringbone}}

\vspace{.5em}

\paragraph*{\uline{Step 1: assemble an \mbox{$l_\text{avg}$}-by-\mbox{$l_\text{avg}$}
lattice of herringbones}. \quad}

Define the squares
\[
Q_{\alpha}=\alpha+(0,l_\text{avg})^{2}\quad\forall\,\alpha\in\mathbb{Z}^{2}
\]
and let the index set $\mathcal{I}$ be the smallest subset of $\mathbb{Z}^{2}$
with the property that 
\[
\Omega\subset\cup_{\alpha\in\mathcal{I}}\overline{Q_{\alpha}}.
\]
Define the locally averaged target defect $\langle \mu \rangle$ whose value on the $\alpha$th square is given by
\begin{equation}
\mu_{\alpha}=\fint_{\Omega\cap Q_{\alpha}}\mu(x)\,dx,\quad\alpha\in\mathcal{I}.\label{eq:averaging}
\end{equation}
We produce a family of herringbone constructions $\{(v_{\text{herr}}^{\alpha},w_{\text{herr}}^{\alpha})\}_{\alpha\in\mathcal{I}}$
using the results of \prettyref{subsec:herringbone-patterns}:
given $\alpha\in\mathcal{I}$, we define $v_{\text{herr}}^{\alpha}:Q_{\alpha}\to\mathbb{R}^{2}$
and $w_{\text{herr}}^{\alpha}:Q_{\alpha}\to\mathbb{R}$ following
the procedure from \prettyref{subsec:Constructing-the-herringbone}
with $\mu_{\alpha}$ as the target defect and $l_\text{sh}$, $l_\text{wr}$,
and $\delta_\text{int}$ as above. (We take the parameters to be
independent of $\alpha$ for ease of exposition, and as it will not
affect the estimates at leading order.) Copying over the
pointwise bounds \prettyref{eq:pointwise_estimates_herr_v} and \prettyref{eq:pointwise_estimates_herr_w1}, we note that
\begin{align}
 & ||v_\text{herr}^{\alpha}||_{L^{\infty}}\lesssim\text{tr}\,\mu_{\alpha}\cdot l_\text{sh}\left(1\vee\frac{l_\text{wr}}{l_\text{sh}}\right),\quad||\nabla v_\text{herr}^{\alpha}||_{L^{\infty}}\lesssim\text{tr}\,\mu_{\alpha}\cdot(1\vee\frac{l_\text{wr}}{\delta_\text{int}})\label{eq:pointwise_estimates_herr_v-1}\\
 & ||w_\text{herr}^{\alpha}||_{L^{\infty}}\lesssim\sqrt{\text{tr}\,\mu_{\alpha}}\cdot l_\text{wr},\quad||\nabla w_\text{herr}^{\alpha}||_{L^{\infty}}\lesssim\sqrt{\text{tr}\,\mu_{\alpha}}\cdot\left(1\vee\frac{l_\text{wr}}{\delta_\text{int}}\right),\quad||\nabla\nabla w_\text{herr}^{\alpha}||_{L^{\infty}}\lesssim\frac{\sqrt{\text{tr}\,\mu_{\alpha}}}{l_\text{wr}}\cdot\left(1\vee\frac{l_\text{wr}^{2}}{\delta_\text{int}^{2}}\right)\label{eq:pointwise_estimates_herr_w-1}
\end{align}
with constants independent of $\alpha\in\mathcal{I}$.

Before proceeding to the next step of the construction, let us quickly verify
that the parameters $l_\text{sh}$, $l_\text{wr}$, and $\delta_\text{int}$ are
indeed admissible for use in \prettyref{subsec:herringbone-patterns}.
According to \prettyref{eq:parameters_herringbone_reqs}, we must
check that 
\begin{equation}
l_\text{sh}\in(0,\infty),\quad l_\text{wr}\in(0,\infty),\quad\text{and}\quad\delta_\text{int}\in(0,\frac{1}{2}\theta_{\alpha}l_\text{sh})\quad\forall\,\alpha\in\mathcal{I} \label{eq:parameters_herringbone_reqs-1}
\end{equation}
where 
\[
\theta_{\alpha}=\frac{\lambda_{1}^{\alpha}}{\lambda_{1}^{\alpha}+\lambda_{2}^{\alpha}}\quad\text{and}\quad 0<\lambda_1^\alpha \leq \lambda_2^\alpha\text{ are the eigenvalues of }\mu_\alpha.
\]
It follows from \prettyref{eq:uniformly_elliptic}
and \prettyref{eq:averaging} that $\theta_\alpha \in (\frac{\lambda}{2\Lambda},\frac{\Lambda}{2\lambda})$. 
As by hypothesis $\delta_\text{int}<\frac{1}{4}\frac{\lambda}{\Lambda}l_\text{sh},$
we conclude that \prettyref{eq:parameters_herringbone_reqs-1} holds.

\vspace{.5em}

\paragraph*{\uline{Step 2: join the herringbones across external walls at
scale \mbox{$\delta_\text{ext}$}}. \quad}

The next step is to join the herringbones obtained in the previous
step into a single, globally defined piecewise herringbone pattern.
We employ a family of smooth cutoff functions supported away from
$\cup_{\alpha\in\mathcal{I}}\partial Q_{\alpha}$ to define the external walls.
Since by hypothesis $\delta_\text{ext}<\frac{1}{2}l_\text{avg}$, there exists
a family of smooth cutoff functions $\{\chi_\text{ext}^{\alpha}\}_{\alpha\in\mathcal{I}}$
such that
\begin{itemize}
\item $\chi_\text{ext}^{\alpha}\in C_{c}^{\infty}(Q_{\alpha})$ and $0\leq\chi_{\alpha}\leq1$,
\item $\chi_\text{ext}^{\alpha}(x)=0$ if $d(x,\partial Q_{\alpha})\leq\frac{1}{2}\delta_\text{ext}$
and $\chi_\text{ext}^{\alpha}(x)=1$ if $d(x,\partial Q_{\alpha})\geq\delta_\text{ext}$,
\item $||\nabla\chi_\text{ext}^{\alpha}||_{L^{\infty}}\lesssim\frac{1}{\delta_\text{ext}}$,
$||\nabla\nabla\chi_\text{ext}^{\alpha}||_{L^{\infty}}\lesssim\frac{1}{\delta_\text{ext}}$.
\end{itemize}
The constants implicit in the above are independent of all parameters (including $\alpha$). Finally, we define $v_{\text{p.h.}}:\Omega\to\mathbb{R}^{2}$
and $w_{\text{p.h.}}:\Omega\to\mathbb{R}$ by
\[
v_{\text{p.h.}}=\sum_{\alpha\in\mathcal{I}}v_{\text{herr}}^{\alpha}\chi_\text{ext}^{\alpha}\quad\text{and}\quad w_{\text{p.h.}}=\sum_{\alpha\in\mathcal{I}}w_{\text{herr}}^{\alpha}\chi_\text{ext}^{\alpha}.
\]
This completes our construction of the piecewise herringbone.
Note the pointwise estimates
\begin{align}
 & ||v_{\text{p.h.}}||_{L^{\infty}}\lesssim||\mu||_{L^{\infty}}\cdot l_\text{sh}\left(1\vee\frac{l_\text{wr}}{l_\text{sh}}\right),\quad||\nabla v_{\text{p.h.}}||_{L^{\infty}}\lesssim||\mu||_{L^{\infty}}\cdot\left(1\vee\frac{l_\text{sh}}{\delta_\text{ext}}\vee\frac{l_\text{wr}}{\delta_\text{ext}}\vee\frac{l_\text{wr}}{\delta_\text{int}}\right),\label{eq:pointwise_ph_v}\\
 & ||w_{\text{p.h.}}||_{L^{\infty}}\lesssim\sqrt{||\mu||_{L^{\infty}}}\cdot l_\text{wr},\quad||\nabla w_{\text{p.h.}}||_{L^{\infty}}\lesssim\sqrt{||\mu||_{L^{\infty}}}\cdot\left(1\vee\frac{l_\text{wr}}{\delta_\text{ext}}\vee\frac{l_\text{wr}}{\delta_\text{int}}\right),\label{eq:pointwise_ph_w1}\\
 & ||\nabla\nabla w_{\text{p.h.}}||_{L^{\infty}}\lesssim\frac{\sqrt{||\mu||_{L^{\infty}}}}{l_\text{wr}}\cdot\left(1\vee\frac{l_\text{wr}^{2}}{\delta_\text{ext}^{2}}\vee\frac{l_\text{wr}^{2}}{\delta_\text{int}^{2}}\right).
 \label{eq:pointwise_ph_w2}
\end{align}
These follow from \prettyref{eq:pointwise_estimates_herr_v-1}, \prettyref{eq:pointwise_estimates_herr_w-1},
and the properties of $\{\chi_\text{ext}^{\alpha}\}$ listed above.

\subsubsection{Energy estimates for the piecewise herringbone}

Here we estimate the energy of the piecewise herringbones just defined.
Decompose $\Omega$ into its ``wall'' and ``bulk'' regions given by
\begin{equation}
\Omega_\text{wall}=\left\{ x\in\Omega:d(x,\cup_{\alpha}\partial Q_{\alpha})<\delta_\text{ext}\right\} \quad\text{and}\quad\Omega_\text{bulk}=\Omega\backslash\Omega_\text{wall}\label{eq:Omega_wall}
\end{equation}
and define the quantities
\begin{align}
A_{0}(l_\text{avg};\mu) & =\sum_{\alpha\in\mathcal{I}}\text{tr}\,\mu_{\alpha}\cdot|Q_{\alpha}|\label{eq:A_0}\\
A_{1}(l_\text{avg},l_\text{sh},l_\text{wr},\delta_\text{int},\delta_\text{ext};\mu) & =\left(1\vee\frac{l_\text{sh}^{2}}{\delta_\text{ext}^{2}}\vee\frac{l_\text{wr}^{4}}{\delta_\text{ext}^{4}}\vee\frac{l_\text{wr}^{4}}{\delta_\text{int}^{4}}\right)|\Omega_\text{wall}|+\frac{\delta_\text{int}}{l_\text{sh}}\left(1\vee\frac{l_\text{wr}^{4}}{\delta_\text{int}^{4}}\right)\left(1\vee\frac{l_\text{sh}}{l_\text{avg}}\right)|\cup_{\alpha\in\mathcal{I}}Q_{\alpha}|\label{eq:A_1}\\
A_{2}(l_\text{avg}) & =l_\text{avg}^{2}|\Omega|.\label{eq:A_2}
\end{align}

\begin{lemma}
\label{lem:PHestimates} Let $\Omega$ be a bounded and Lipschitz, and let $\mu:\Omega\to\emph{Sym}_{2}$ be positive definite and Lipschitz continuous. Let $l_\emph{wr}$, $l_\emph{sh}$, $l_\emph{avg}$, $\delta_\emph{int}$, and $\delta_\emph{ext}$ satisfy \prettyref{eq:parameters_ph_herringbone_reqs}. 
The piecewise herringbones constructed in
\prettyref{subsec:Constructing-the-piecewise-herringbone} satisfy
the following estimates:
\begin{itemize}
\item the stretching energy satisfies
\[
\int_{\Omega}|e(v_{\emph{p.h.}})+\frac{1}{2}\nabla w_{\emph{p.h.}}\otimes\nabla w_{\emph{p.h.}}-\frac{1}{2}\mu|^{2}\lesssim||\mu||_{L^{\infty}}^{2}A_{1}+||\nabla\mu||_{L^{\infty}}^{2}A_{2};
\]
\item the bending energy satisfies 
\[
\int_{\Omega}|\nabla\nabla w_{\emph{p.h.}}|^{2}\leq\frac{A_{0}+C||\mu||_{L^{\infty}}A_{1}}{l_\emph{wr}^{2}}
\]
where the constant $C$ is independent of all parameters;
\item the substrate energy satisfies
\[
\int_{\Omega}|w_{\emph{p.h.}}|^{2}\leq\left(A_{0}+C||\mu||_{L^{\infty}}A_{1}\right)l_\emph{wr}^{2}.
\]
\end{itemize}
\end{lemma}
\begin{proof}
We begin with the stretching energy. Introduce the strains
\[
\varepsilon=e(v_\text{p.h.})+\frac{1}{2}\nabla w_\text{p.h.}\otimes\nabla w_\text{p.h.}-\frac{1}{2}\mu,\quad\text{and}\quad\varepsilon_{\alpha}=e(v_{\text{herr}}^{\alpha})+\frac{1}{2}\nabla w_{\text{herr}}^{\alpha}\otimes\nabla w_{\text{herr}}^{\alpha}-\frac{1}{2}\mu_{\alpha}\quad \text{for }\alpha\in\mathcal{I}.
\]
Using the definition of the cutoff function
$\chi_\text{ext}^{\alpha}$ we find that
\[
\varepsilon=\varepsilon_{\alpha}+\frac{1}{2}(\mu_{\alpha}-\mu)\quad\text{on }\Omega_\text{bulk}\cap Q_{\alpha},
\]
hence by the triangle inequality
\begin{equation}
|\varepsilon|\leq|\varepsilon_{\alpha}|+\frac{1}{2}|\mu_{\alpha}-\mu|\lesssim|\varepsilon_{\alpha}|+||\nabla\mu||_{L^{\infty}}l_\text{avg}\quad\text{on }\Omega_\text{bulk}\cap Q_{\alpha}.\label{eq:bulk_strain_est-1}
\end{equation}
On the other hand, the pointwise estimates \prettyref{eq:pointwise_ph_v}-\prettyref{eq:pointwise_ph_w1}
imply that
\begin{equation}
|\varepsilon|\leq|\nabla v_\text{p.h.}|+\frac{1}{2}|\nabla w_\text{p.h.}|^{2}+\frac{1}{2}|\mu|\lesssim||\mu||_{L^{\infty}}\left(1\vee\frac{l_\text{sh}}{\delta_\text{ext}}\vee\frac{l_\text{wr}^{2}}{\delta_\text{ext}^{2}}\vee\frac{l_\text{wr}^{2}}{\delta_\text{int}^{2}}\right)\quad\text{on }\Omega_\text{wall}.\label{eq:wall_strain_est-1}
\end{equation}
Applying \prettyref{eq:bulk_strain_est-1}, \prettyref{eq:wall_strain_est-1},
and the stretching part of \prettyref{lem:herringbone_estimates}
we deduce that 
\begin{align*}
\int_{\Omega}|\varepsilon|^{2} & =\left(\int_{\Omega_\text{wall}}+\int_{\Omega_\text{bulk}}\right)|\varepsilon|^{2}=\int_{\Omega_\text{wall}}|\varepsilon|^{2}+\sum_{\alpha\in\mathcal{I}}\int_{\Omega_\text{bulk}\cap Q_{\alpha}}|\varepsilon|^{2}\\
 & \lesssim||\mu||_{L^{\infty}}^{2}\left(1\vee\frac{l_\text{sh}^{2}}{\delta_\text{ext}^{2}}\vee\frac{l_\text{wr}^{4}}{\delta_\text{ext}^{4}}\vee\frac{l_\text{wr}^{4}}{\delta_\text{int}^{4}}\right)|\Omega_\text{wall}|+\sum_{\alpha\in\mathcal{I}}\left[||\mu||_{L^{\infty}}^{2}a_{1}(\mu_{\alpha},Q_{\alpha})+||\nabla\mu||_{L^{\infty}}^{2}l_\text{avg}^{2}|\Omega_\text{bulk}\cap Q_{\alpha}|\right]\\
 & \lesssim||\mu||_{L^{\infty}}^{2}\left[\left(1\vee\frac{l_\text{sh}^{2}}{\delta_\text{ext}^{2}}\vee\frac{l_\text{wr}^{4}}{\delta_\text{ext}^{4}}\vee\frac{l_\text{wr}^{4}}{\delta_\text{int}^{4}}\right)|\Omega_\text{wall}|+\frac{\delta_\text{int}}{l_\text{sh}}\left(1\vee\frac{l_\text{wr}^{4}}{\delta_\text{int}^{4}}\right)\left(1\vee\frac{l_\text{sh}}{l_\text{avg}}\right)|\cup_{\alpha\in\mathcal{I}}Q_{\alpha}|\right]+||\nabla\mu||_{L^{\infty}}^{2}l_\text{avg}^{2}|\Omega|\\
 & \lesssim||\mu||_{L^{\infty}}^{2}A_{1}+||\nabla\mu||_{L^{\infty}}^{2}A_{2}
\end{align*}
according to the definitions of $A_{1}$ and $A_{2}$ in \prettyref{eq:A_1}
and \prettyref{eq:A_2}. Note we used the estimate
\begin{equation}
\sum_{\alpha\in\mathcal{I}}a_{1}(\mu_{\alpha},Q_{\alpha})\lesssim\frac{\delta_\text{int}}{l_\text{sh}}\left(1\vee\frac{l_\text{wr}^{4}}{\delta_\text{int}^{4}}\right)\left(1\vee\frac{l_\text{sh}}{l_\text{avg}}\right)|\cup_{\alpha\in\mathcal{I}}Q_{\alpha}|\label{eq:sum_a_1}
\end{equation}
in the third line, which follows from \prettyref{lem:constant_estimates}
and the fact that $\text{diam}\,Q_{\alpha}\sim l_\text{avg}$ uniformly in $\alpha$. The desired estimate on the stretching energy is proved.

We turn to estimate the bending energy. Note that 
\begin{equation}
w_{\text{p.h.}}=w_{\text{herr}}^{\alpha}\quad\text{on }\Omega_\text{bulk}\cap Q_{\alpha}\label{eq:bulk_out-of-plane}
\end{equation}
by the definition of $\chi_\text{ext}^{\alpha}$,
while the pointwise estimate
\begin{equation}
|\nabla\nabla w_{\text{p.h.}}|\lesssim\frac{\sqrt{||\mu||_{L^{\infty}}}}{l_\text{wr}}\left(1\vee\frac{l_\text{wr}^{2}}{\delta_\text{ext}^{2}}\vee\frac{l_\text{wr}^{2}}{\delta_\text{int}^{2}}\right)\quad\text{on }\Omega_\text{wall}\label{eq:wall_out-of-plane}
\end{equation}
follows from \prettyref{eq:pointwise_ph_w2}. Note also that 
\begin{equation}
\sum_{\alpha\in\mathcal{I}}a_{0}(\mu_{\alpha},Q_{\alpha})\leq\sum_{\alpha}\text{tr}\,\mu_{\alpha}\cdot|Q_{\alpha}|+C||\mu||_{L^{\infty}}\frac{l_\text{wr}}{l_\text{sh}}\left(1\vee\frac{l_\text{sh}}{l_\text{avg}}\right)|\cup_{\alpha\in\mathcal{I}}Q_{\alpha}|\label{eq:sum_a_0}
\end{equation}
as a result of \prettyref{lem:constant_estimates}. Using \prettyref{eq:bulk_out-of-plane},
\prettyref{eq:wall_out-of-plane} and the bending part of \prettyref{lem:herringbone_estimates}
we deduce that
\begin{align*}
\int_{\Omega}|\nabla\nabla &w_\text{p.h.}|^{2}  =\left(\int_{\Omega_\text{wall}}+\int_{\Omega_\text{bulk}}\right)|\nabla\nabla w_\text{p.h.}|^{2}=\int_{\Omega_\text{wall}}|\nabla\nabla w_\text{p.h.}|^{2}+\sum_{\alpha\in\mathcal{I}}\int_{\Omega_\text{bulk}\cap Q_{\alpha}}|\nabla\nabla w_{\text{herr}}^{\alpha}|^{2}\\
 & \leq C\frac{||\mu||_{L^{\infty}}}{l_\text{wr}^{2}}\left(1\vee\frac{l_\text{wr}^{4}}{\delta_\text{ext}^{4}}\vee\frac{l_\text{wr}^{4}}{\delta_\text{int}^{4}}\right)|\Omega_\text{wall}|+\sum_{\alpha\in\mathcal{I}}\frac{a_{0}(\mu_{\alpha},Q_{\alpha})+C||\mu||_{L^{\infty}}a_{1}(\mu_{\alpha},Q_{\alpha})}{l_\text{wr}^{2}}\\
 & \leq\frac{1}{l_\text{wr}^{2}}\sum_{\alpha\in\mathcal{I}}\text{tr}\,\mu_{\alpha}\cdot|Q_{\alpha}|+C\frac{||\mu||_{L^{\infty}}}{l_\text{wr}^{2}}\left[\left(1\vee\frac{l_\text{wr}^{4}}{\delta_\text{ext}^{4}}\vee\frac{l_\text{wr}^{4}}{\delta_\text{int}^{4}}\right)|\Omega_\text{wall}|+\frac{\delta_\text{int}}{l_\text{sh}}\left(1\vee\frac{l_\text{wr}^{4}}{\delta_\text{int}^{4}}\right)\left(1\vee\frac{l_\text{sh}}{l_\text{avg}}\right)|\cup_{\alpha\in\mathcal{I}}Q_{\alpha}|\right]\\
 & \leq\frac{A_{0}+C||\mu||_{L^{\infty}}A_{1}}{l_\text{wr}^{2}}
\end{align*}
where we applied \prettyref{eq:sum_a_1} and \prettyref{eq:sum_a_0} to pass from the second line to the third. 

We finish with the substrate energy. Note that
\begin{equation}
|w_{\text{p.h.}}|\leq|w_{\text{herr}}^{\alpha}|\quad\text{on }\Omega\cap Q_{\alpha}\label{eq:out-of-plane-lessthan}
\end{equation}
as there always holds $\chi_\text{ext}^{\alpha}\leq1$. Using \prettyref{eq:out-of-plane-lessthan}
and the substrate part of \prettyref{lem:herringbone_estimates}, there follows
\begin{align*}
\int_{\Omega}|w_\text{p.h.}|^{2} & \leq\sum_{\alpha\in\mathcal{I}}\int_{\Omega\cap Q_{\alpha}}|w_{\text{herr}}^{\alpha}|^{2}\leq\sum_{\alpha\in\mathcal{I}}a_{0}(\mu_{\alpha},Q_{\alpha})l_\text{wr}^{2}\\
 & \leq l_\text{wr}^{2}\sum_{\alpha}\text{tr}\,\mu_{\alpha}\cdot|Q_{\alpha}|+Cl_\text{wr}^{2}||\mu||_{L^{\infty}}\frac{l_\text{wr}}{l_\text{sh}}\left(1\vee\frac{l_\text{sh}}{l_\text{avg}}\right)|\cup_{\alpha\in\mathcal{I}}Q_{\alpha}|\\
 & \leq\left(A_{0}+C||\mu||_{L^{\infty}}A_{1}\right)l_\text{wr}^{2}.
\end{align*}
We used \prettyref{eq:sum_a_0} to pass to
the second line, and  the definitions of $A_{0}$ and $A_{1}$ from
\prettyref{eq:A_0} and \prettyref{eq:A_1} at the end. 
\qed\end{proof}

Next, we identify an energetically optimal version of the piecewise
herringbone by minimizing over the free parameters $l_\text{avg}$,
$l_\text{sh}$, $l_\text{wr}$, $\delta_\text{int}$, and $\delta_\text{ext}$ from \prettyref{eq:parameters_ph_herringbone_reqs}.
To simplify the presentation, and as it turns out to be consistent
with optimality, we shall impose the additional constraints
\begin{equation}
l_\text{wr}\ll l_\text{sh}\ll l_\text{avg}\ll\text{diam}\,\Omega,\quad l_\text{wr}\lesssim\delta_\text{int},\quad\text{and}\quad l_\text{sh}\lesssim\delta_\text{ext}\label{eq:additional_constraints}
\end{equation}
in what follows. We require the asymptotic behavior of the quantities $A_0$, $A_1$, and $A_2$ from \prettyref{eq:A_0}-\prettyref{eq:A_2}. 
\begin{lemma}
\label{lem:Aparameter_estimates}We have that
\[
A_{0}(\mu)\to\int_{\Omega}\emph{tr}\,\mu\,dx,\quad A_{1}(\mu)\lesssim\frac{\delta_\emph{int}}{l_\emph{sh}}|\Omega|+\frac{\delta_\emph{ext}}{l_\emph{avg}}(\emph{diam}\,\Omega)^{2},\quad \emph{and}\quad A_{2}(\mu)=l_\emph{avg}^{2}|\Omega|
\]
in any limit satisfying \prettyref{eq:additional_constraints}.
\end{lemma}
\begin{proof}
The claim regarding $A_{0}$ follows from its definition, since 
\[
\cup_{\alpha\in\mathcal{I}}Q_{\alpha}\to\Omega\quad\text{as }l_\text{avg}\to0.
\]
The claim regarding $A_{2}$ is clear. Now we address $A_{1}$. First, note that
\[
|\cup_{\alpha\in\mathcal{I}}\overline{Q_{\alpha}}|\lesssim|\Omega|
\]
for all small enough $l_\text{avg}$. 
Now recall the definition of $\Omega_\text{wall}$ from \prettyref{eq:Omega_wall}.
Each square $Q_{\alpha}$ has perimeter $\lesssim l_\text{avg}$, and their
$\delta_\text{ext}$-thickenings have area $\lesssim\delta_\text{ext}\cdot l_\text{avg}$.
The total number of squares is eventually $\lesssim\left(\frac{\text{diam}\,\Omega}{l_\text{avg}}\right)^{2}$. Hence,
\[
|\Omega_\text{wall}|\lesssim\left(1\vee\frac{(\text{diam}\,\Omega)^{2}}{l_\text{avg}^{2}}\right)\cdot\left(\delta_\text{ext}\cdot l_\text{avg}\right)=\frac{\delta_\text{ext}}{l_\text{avg}}\left(1\vee\frac{l_\text{avg}^{2}}{(\text{diam}\,\Omega)^{2}}\right)\cdot\left(\text{diam}\,\Omega\right)^{2}.
\]
Setting these estimates into \prettyref{eq:A_1} and appealing to \prettyref{eq:additional_constraints}
we see that 
\[
A_{1}\lesssim|\Omega_\text{wall}|+\frac{\delta_\text{int}}{l_\text{sh}}|\cup_{\alpha\in\mathcal{I}}Q_{\alpha}|\lesssim\frac{\delta_\text{ext}}{l_\text{avg}}(\text{diam}\,\Omega)^{2}+\frac{\delta_\text{int}}{l_\text{sh}}|\Omega|.
\]
\qed\end{proof}
We are ready to optimize over the piecewise herringbone patterns adapted to $\mu$. Given \prettyref{eq:additional_constraints}, the estimates 
\begin{align}
\int_{\Omega}|e(v_{\text{p.h.}})+\frac{1}{2}\nabla w_{\text{p.h.}}\otimes\nabla w_{\text{p.h.}}-\frac{1}{2}\mu|^{2} & \lesssim_{\Omega,\mu}l_\text{avg}^{2}\vee\frac{\delta_\text{int}}{l_\text{sh}}\vee\frac{\delta_\text{ext}}{l_\text{avg}},\label{eq:stretching_estimate_beingoptimized}\\
\frac{b}{2}\int_{\Omega}|\nabla\nabla w_{\text{p.h.}}|^{2}+\frac{k}{2}\int_{\Omega}|w_{\text{p.h.}}|^{2} & \leq\left(\frac{b}{2}\frac{1}{l_\text{wr}^{2}}+\frac{k}{2}l_\text{wr}^{2}\right)\left(\int_{\Omega}\text{tr}\,\mu+o(1)+C(\Omega,\mu)\frac{\delta_\text{int}}{l_\text{sh}}\vee\frac{\delta_\text{ext}}{l_\text{avg}}\right)\label{eq:bending_estimate_beingoptimized}
\end{align}
follow from 
\prettyref{lem:PHestimates} and \prettyref{lem:Aparameter_estimates}. Balancing
the dominant terms in \prettyref{eq:stretching_estimate_beingoptimized}
and \prettyref{eq:bending_estimate_beingoptimized} yields that
\[
\frac{b}{l_\text{wr}^{2}}=kl_\text{wr}^{2}\quad\text{and}\quad l_\text{avg}^{2}\sim\frac{\delta_\text{int}}{l_\text{sh}}\sim\frac{\delta_\text{ext}}{l_\text{avg}},
\]
while saturating the last two constraints from \prettyref{eq:additional_constraints}
to minimize the energy that results yields that
\[
\frac{\delta_\text{int}}{l_\text{wr}}\sim\frac{\delta_\text{ext}}{l_\text{sh}}\sim1.
\]
These five relations underlie optimal choices for the five free parameters. 
Using them in \prettyref{eq:stretching_estimate_beingoptimized} and \prettyref{eq:bending_estimate_beingoptimized}
and recalling the pointwise estimates \prettyref{eq:pointwise_ph_v}-\prettyref{eq:pointwise_ph_w2}
we conclude the following result: 
\begin{corollary}
\label{cor:optimized_ph_pattern} Let $\Omega$ be bounded and Lipschitz, and let $\mu:\Omega \to \emph{Sym}_{2}$ be positive definite and Lipschitz continuous. 
Let $\{(v_{\emph{p.h.}},w_{\emph{p.h.}})\}$ be a sequence of piecewise herringbones as constructed in \prettyref{subsec:Constructing-the-piecewise-herringbone}, 
and suppose their parameters from \prettyref{eq:parameters_ph_herringbone_reqs} satisfy 
\[
l_\emph{wr}=\left(\frac{b}{k}\right)^{1/4}\ll\emph{diam}\,\Omega,\quad l_\emph{avg}\sim l_\emph{wr}^{1/5},\quad l_\emph{sh}\sim l_\emph{wr}^{1/2}l_\emph{avg}^{1/2},\quad\delta_\emph{int}\sim l_\emph{wr},\quad \emph{and}\quad\delta_\emph{ext}\sim l_\emph{sh}.
\]
Such a sequence satisfies the energy estimates
\begin{align*}
\int_{\Omega}|e(v_{\emph{p.h.}})+\frac{1}{2}\nabla w_{\emph{p.h.}}\otimes\nabla w_{\emph{p.h.}}-\frac{1}{2}\mu|^{2} & \lesssim_{\Omega,\mu}\left(\frac{b}{k}\right)^{1/10},\\
\frac{b}{2}\int_{\Omega}|\nabla\nabla w_{\emph{p.h.}}|^{2}+\frac{k}{2}\int_{\Omega}|w_{\emph{p.h.}}|^{2} & \leq\sqrt{bk}\cdot\int_{\Omega}\emph{tr}\,\mu\,dx+o(\sqrt{bk})
\end{align*}
as well as the pointwise estimates
\begin{align}
 & ||v_{\emph{p.h.}}||_{L^{\infty}}\lesssim_{\mu}\left(\frac{b}{k}\right)^{3/20},\quad||\nabla v_{\emph{p.h.}}||_{L^{\infty}}\lesssim_{\mu}1,\label{eq:pointwise_estimate_final}\\
 & ||w_{\emph{p.h.}}||_{L^{\infty}}\lesssim_{\mu}\left(\frac{b}{k}\right)^{1/4},\quad||\nabla w_{\emph{p.h.}}||_{L^{\infty}}\lesssim_{\mu}1,\quad||\nabla\nabla w_{\emph{p.h.}}||_{L^{\infty}}\lesssim_{\mu}\left(\frac{k}{b}\right)^{1/4}.\label{eq:pointwise_estimate_final2}
\end{align}
\end{corollary}

\subsection{Recovery sequences\label{subsec:recovery-sequences}}

We are finally ready to prove \prettyref{prop:recovery}. We take
for granted the results of \prettyref{subsec:smooth-approx}-\prettyref{subsec:piecewise-herringbone-patterns}.

\vspace{1em}
\noindent\emph{Proof of \prettyref{prop:recovery}}
Let $u_{\text{eff}}\in BD(\Omega)$
be tension-free. By the result of \prettyref{prop:gam_liminf},
it suffices to construct 
\begin{equation}
(u_{b,k,\gamma},u_{b,k,\gamma})\overset{*}{\rightharpoonup}(u_{\text{eff}},0)\quad\text{weakly-}*\text{ in }BD(\Omega)\times W^{1,2}(\Omega)\label{eq:desiredresult_wkcvg}
\end{equation}
such that
\begin{equation}
\limsup\,\frac{E_{b,k,\gamma}(u_{b,k,\gamma},w_{b,k,\gamma})}{2\sqrt{bk}+\gamma}\leq\int_{\Omega}\frac{1}{2}|\nabla p|^{2}-\int_{\partial\Omega}u_{\text{eff}}\cdot\hat{\nu}\,ds.\label{eq:desiredresult_lowdef}
\end{equation}
We begin by applying the results of \prettyref{subsec:smooth-approx}
to reduce $u_{\text{eff}}$ that are smooth and uniformly tension-free.
Due to \prettyref{lem:smoothapproximation}, there exist 
uniformly tension-free $\{u_{n}\}_{n\in\mathbb{N}}\subset C^{\infty}(\overline{\Omega};\mathbb{R}^{2})$
converging to $u_{\text{eff}}$ in the intermediate sense. In particular,
\[
u_{n}\stackrel{*}{\rightharpoonup}u_{\text{eff}}\quad\text{weakly-\ensuremath{*} in }BD(\Omega)\quad\text{and}\quad\int_{\partial\Omega}u_{n}\cdot\hat{\nu}\,ds\to\int_{\partial\Omega}u_{\text{eff}}\cdot\hat{\nu}\,ds\quad\text{as }n\to\infty.
\]
Suppose now that for each fixed $n$ we can produce a recovery sequence $\{(u_{n,m},w_{n,m})\}_{m\in\mathbb{N}}$
for $(u_n,0)$, i.e., a sequence satisfying the analog of \prettyref{eq:desiredresult_wkcvg} and \prettyref{eq:desiredresult_lowdef} but with $u_\text{eff}$ replaced by $u_n$. 
Then, a straightforward diagonalization argument produces a recovery sequence for $(u_\text{eff},0)$.
Thus, it suffices to achieve \prettyref{eq:desiredresult_wkcvg} and \prettyref{eq:desiredresult_lowdef} for  $u_{\text{eff}}$ that are smooth and uniformly tension-free. We do so via the piecewise herringbone patterns from \prettyref{subsec:piecewise-herringbone-patterns}.

Fix some $u_{\text{eff}}\in C^\infty(\overline{\Omega};\mathbb{R}^2)$ that is uniformly tension-free. Introduce the (pointwise-defined) target defect
\begin{equation}
\mu(x)=\nabla p\otimes\nabla p(x)-2e(u_{\text{eff}})(x)\quad x\in\Omega\label{eq:pointwise_target_defect}
\end{equation}
and note it is Lipschitz as
\[
||\nabla\mu||_{L^{\infty}}\lesssim||\nabla\nabla u_{\text{eff}}||_{L^{\infty}}\vee||\nabla\nabla p||_{L^{\infty}}||\nabla p||_{L^{\infty}}<\infty.
\]
It is also uniformly positive definite. Therefore, we may apply
the results of \prettyref{subsec:piecewise-herringbone-patterns}
to obtain a family of piecewise herringbones $\{(v_{\text{p.h.}},w_{\text{p.h.}})\}$
indexed by  $l_\text{avg}$, $l_\text{sh}$, $l_\text{wr}$, $\delta_\text{int}$,
and $\delta_\text{ext}$ and that are adapted to $\mu$. Guided by \prettyref{cor:optimized_ph_pattern},
we choose these parameters to depend on $b$ and $k$ (and not on $\gamma$) as
follows: we take
\[
l_\text{wr}=\left(\frac{b}{k}\right)^{1/4},\quad l_\text{sh}=l_\text{wr}^{1/2}l_\text{avg}^{1/2},\quad l_\text{avg}=l_\text{wr}^{1/5},\quad\delta_\text{int}=l_\text{wr},\quad\text{and}\quad\delta_\text{ext}=l_\text{sh}
\]
noting that these define a valid piecewise herringbone pattern, according to \prettyref{eq:parameters_ph_herringbone_reqs},  so long as 
\[
\delta_\text{int} < \frac{1}{4} \frac{\lambda}{\Lambda}l_\text{sh}\quad\text{and}\quad  \delta_\text{ext} < \frac{1}{2}l_\text{avg}
\]
where $\lambda,\Lambda \in (0,\infty)$ are as in \prettyref{eq:uniformly_elliptic}. 
Equivalently, we must arrange that $l_{wr}^{2/5}<\frac{1}{4}\frac{\lambda}{\Lambda}$. Note the quantities $\lambda$ and $\Lambda$ are fixed by $\mu$, and hence by $p$ and $u_\text{eff}$. Of course,  $l_\text{wr} \ll 1$ within the  given parameter regime in \prettyref{eq:assumptions-Prop3.1}. Therefore, the required inequalities are eventually satisfied. 

All that remains is to assemble the estimates from \prettyref{cor:optimized_ph_pattern} to prove \prettyref{eq:desiredresult_wkcvg} and \prettyref{eq:desiredresult_lowdef}. Calling 
\[
u_{b,k,\gamma}=u_{\text{eff}}+v_{\text{p.h.}}\quad\text{and}\quad w_{b,k,\gamma}=w_{\text{p.h.}}
\]
we see from \prettyref{eq:pointwise_estimate_final} and \prettyref{eq:pointwise_estimate_final2}
that the desired convergence \prettyref{eq:desiredresult_wkcvg}
holds. Using the formula \prettyref{eq:non-dim_energy} for the energy, the definition of the target defect in \prettyref{eq:pointwise_target_defect}, and the rest of the estimates in the corollary we conclude that
\begin{align*}
 E_{b,k,\gamma}(u_{b,k,\gamma},w_{b,k,\gamma}) &\leq\frac{1}{2}\int_{\Omega}|e(v_{\text{p.h.}})+\frac{1}{2}\nabla w_{\text{p.h.}}\otimes\nabla w_{\text{p.h.}}-\frac{1}{2}\mu|^{2}+ \frac{b}{2}\int_{\Omega}|\nabla\nabla w_{\text{p.h.}}|^{2}+\frac{k}{2}\int_{\Omega}|w_{\text{p.h.}}|^{2}\\
 & \quad\qquad+\gamma\left(\int_{\Omega}\frac{1}{2}|\nabla p|^{2}-\int_{\partial\Omega}u_{\text{eff}}\cdot\hat{\nu}\right)+\gamma||v_{\text{p.h.}}||_{L^{1}(\partial\Omega)}+C(p)\frac{b}{2}\left(||\nabla\nabla w_{\text{p.h.}}||_{L^{2}}+1\right) \\
 & \leq(2\sqrt{bk}+\gamma)\cdot\int_{\Omega}\frac{1}{2}\text{tr}\,\mu+o(\sqrt{bk})+C(\Omega,p,u_{\text{eff}})\left((\frac{b}{k})^{1/10}+b^{3/4}k^{1/4}+\gamma(\frac{b}{k})^{3/20}\right)\\
 & =(2\sqrt{bk}+\gamma)\cdot\int_{\Omega}\frac{1}{2}\text{tr}\,\mu+o(2\sqrt{bk}+\gamma)
\end{align*}
by the definition of the parameter regime in \prettyref{eq:assumptions-Prop3.1}.
Note in the second line we applied the identity
\[
\int_{\Omega}\frac{1}{2}\text{tr}\,\mu(x)\,dx=\int_{\Omega}\frac{1}{2}|\nabla p|^{2}\,dx-\int_{\partial\Omega}u_{\text{eff}}\cdot\hat{\nu}\,ds
\]
which follows from \prettyref{eq:pointwise_target_defect}.
The desired inequality \prettyref{eq:desiredresult_lowdef} is proved.  \qed
\vspace{1em}

Together, \prettyref{prop:gam_liminf} and \prettyref{prop:recovery}
prove \prettyref{thm:gamma-lim}. The rest of the results in
\prettyref{subsec:The-limiting-variational} and \prettyref{subsec:Defect-measures} follow as explained there.

\section{Convex analysis of the limiting problems \label{sec:dualitytheory}}

\prettyref{sec:Gamma_liminf} and \prettyref{sec:Gamma_limsup} established
the role of the limiting minimization problems 
\begin{equation}
\min_{\substack{u_{\text{eff}}\in BD(\Omega)\\
e(u_{\text{eff}})\leq\frac{1}{2}\nabla p\otimes\nabla p\,dx
}
}\,\int_{\Omega}\frac{1}{2}|\nabla p|^{2}\,dx-\int_{\partial\Omega}u_{\text{eff}}\cdot\hat{\nu}\,ds\quad\text{and}\quad\min_{\substack{\mu\in\mathcal{M}_{+}(\Omega;\text{Sym}_{2})\\
-\frac{1}{2}\text{curl}\text{curl}\,\mu=\det\nabla\nabla p
}
}\,\frac{1}{2}\int_{\Omega}|\mu|_{1}\label{eq:primal-1}
\end{equation}
in the asymptotic analysis of the energy $E_{b,k,\gamma}$. In particular, we showed
under the assumptions \vpageref{par:Assumptions} that
\[
\min\,E_{b,k,\gamma}=C_{1}\cdot(2\sqrt{bk}+\gamma)+o(2\sqrt{bk}+\gamma)
\]
where $C_{1}$ is the common minimum value of the limiting problems in \prettyref{eq:primal-1}. We also established via $\Gamma$-convergence a correspondence between the almost minimizers of $E_{b,k,\gamma}$, and optimal $u_{\text{eff}}$ and $\mu$ solving these problems.  The fact that their optimal values are the same follows from the change of variables 
\[
e(u_\text{eff})+\frac{1}{2}\mu = \frac{1}{2}\nabla p\otimes\nabla p\,dx
\]
and the Saint-Venant compatibility conditions for simply connected domains. We refer the reader to
 \prettyref{subsec:The-limiting-variational}
and \prettyref{subsec:Defect-measures} for further discussion of these results.

The remainder of this paper is devoted to the analysis of the limiting problems, and in particular to proving the results from   \prettyref{subsec:The-geometry-of}
and \prettyref{subsec:method-of-stable-lines}. The present section contains, amongst other things, a proof of \prettyref{thm:duality}: we establish
the asserted duality between the ``primal'' problems in \prettyref{eq:primal-1}
and their ``dual'' problem
\begin{equation}
\max_{\substack{\varphi:\mathbb{R}^{2}\to\mathbb{R}\\
\varphi\text{ is convex}\\
\varphi=\frac{1}{2}|x|^{2}\text{ on }\mathbb{R}^{2}\backslash\Omega
}
}\,\int_{\Omega}(\varphi-\frac{1}{2}|x|^{2})\det\nabla\nabla p\,dx\label{eq:dual-1}
\end{equation}
posed over the given admissible Airy potentials
$\varphi$ (our choice of terminology will soon become clear). 
This duality holds under the basic assumptions from \prettyref{eq:A1a} and when $\Omega$ is simply connected. Actually, the methods developed here extend with little additional effort to general domains, even as the form of the dual problem changes. The choice of primal must be addressed. Since the (linearized) area problem appearing on the left of \prettyref{eq:primal-1} is the more general of the two, we take it to be our primal in what follows. We do so also because we expect that it should extend as the $\Gamma$-limit of $\frac{1}{2\sqrt{bk}+\gamma}E_{b,k,\gamma}$ for general domains. We proceed to state its dual.

First, we define a certain linear functional $L$. Consider the vector space of functions $a:\mathbb{R}^2\to\mathbb{R}$ that are locally affine exterior to $\Omega$, i.e., that satisfy
\[
\nabla\nabla a = 0\quad \text{on }\mathbb{R}^2\backslash\overline{\Omega}
\]
in the pointwise sense, and define 
\begin{equation}
L(a)= \int_{\mathbb{R}^2\backslash \Omega} a \det \nabla \nabla \overline{p} \,dx \quad\text{using some }\overline{p}\in W^{2,2}(\mathbb{R}^2)\cap C_c(\mathbb{R}^2) \text{ with } \overline{p}|_\Omega = p|_\Omega.\label{eq:defn-of-L}
\end{equation} 
That this is independent of the choice of extension $\overline{p}$ of $p$ follows from the very weak Hessian identity 
\begin{equation}
-\frac{1}{2}\text{curl}\text{curl}\,\nabla \overline{p} \otimes \nabla \overline{p} = \det \nabla \nabla \overline{p}, \label{eq:veryweak_hessian-duality}
\end{equation}
as will be explained in detail later on in \prettyref{lem:well-definedL}. Evidently, $L(a)$ depends only on the values taken on by $a$ exterior to $\Omega$. Given that $a$ is locally affine there, we can think of it as a sort of ``boundary integral'' term. Of course, if $p$ is regular enough, this can be justified using the divergence theorem along with \prettyref{eq:veryweak_hessian-duality}. 

We come now to our general duality result. Recall the formally adjoint operators $\text{curl}\text{curl}$ and $\nabla^{\perp}\nabla^{\perp}$ from  \prettyref{eq:curlcurl} and \prettyref{eq:Airyoperator}.
Recall also that a sequence of measures is said to converge \emph{narrowly}
in $\mathcal{M}(\Omega;\text{Sym}_{2})$ if their integrals against
arbitrary elements of $C_{b}(\Omega;\text{Sym}_{2})$ converge.

\begin{proposition} \label{prop:duality} Let $\Omega$ be bounded and Lipschitz and let $p\in W^{2,2}(\Omega)$. We have the duality
\begin{equation}
\min_{\substack{u_{\emph{eff}}\in BD(\Omega)\\
e(u_{\emph{eff}})\leq\frac{1}{2}\nabla p\otimes\nabla p\,dx
}
}\,\int_{\Omega}\frac{1}{2}|\nabla p|^{2}\,dx-\int_{\partial\Omega}u_{\emph{eff}}\cdot\hat{\nu}\,ds=\max_{\substack{\varphi:\mathbb{R}^{2}\to\mathbb{R}\\
\varphi\text{ is convex}\\
\nabla\nabla(\varphi-\frac{1}{2}|x|^{2})=0\text{ on }\mathbb{R}^{2}\backslash\overline{\Omega}
}
}\,\int_{\Omega}(\varphi-\frac{1}{2}|x|^{2})\det\nabla\nabla p\,dx + L(\varphi-\frac{1}{2}|x|^{2}).
\label{eq:general-duality-result}
\end{equation}
Regarding complementary slackness, the following are equivalent given $u_\emph{eff}$ and $\varphi$ admissible in the above:
\begin{enumerate}
\item $u_{\emph{eff}}$ and $\varphi$ are optimal;
\item there exists a non-negative sequence $\{\mu_{n}\}\subset C^{2}(\overline{\Omega};\emph{Sym}_{2})$
approximating $\mu=\nabla p\otimes\nabla p\,dx-2e(u_{\emph{eff}})$ in that
\begin{equation}\label{eq:needed-approxs}
\mu_{n}\,dx\to\mu\quad\text{narrowly in }\mathcal{M}(\Omega;\emph{Sym}_{2})\quad\text{and}\quad -\frac{1}{2}\emph{curl\,curl}\,\mu_{n}\,dx\stackrel{*}{\rightharpoonup}\det\nabla\nabla p\,dx\quad\text{weakly-\ensuremath{*} in }\mathcal{M}(\Omega)
\end{equation}
as $n\to\infty$, and for which
\begin{equation}
0=\lim_{n\to\infty}\,\int_{\Omega}|\left\langle \mu_{n},\nabla^{\perp}\nabla^{\perp}\varphi\right\rangle |=\lim_{n\to\infty}\,\int_{\partial\Omega}|\hat{\nu}\cdot [\nabla\varphi] \left\langle \hat{\tau}\otimes\hat{\tau},\mu_{n}\right\rangle |\,ds;\label{eq:limit-vanishes}
\end{equation}
\item the limits in \prettyref{eq:limit-vanishes} vanish for every such
sequence of approximations $\{\mu_{n}\}\subset C^{2}(\overline{\Omega};\emph{Sym}_{2})$ to the given $\mu$.
\end{enumerate}
Here, $[\nabla \varphi]$ denotes the jump in $\nabla \varphi$ across $\partial\Omega$ in the direction of $\hat{\nu}$. It is given by  $\nabla\varphi |_{\partial(\mathbb{R}^2\backslash\overline{\Omega})} - \nabla \varphi |_{\partial\Omega}$. 
\end{proposition} 

Some remarks are in order. 
First, we observe that various other statements of the dual problem from \prettyref{eq:general-duality-result}  can be produced given additional regularity for $p$. Perhaps the most illuminating
one is as follows: if $p\in W^{2,2}(\Omega)\cap C^1(\overline{\Omega})$, the dual problem can be rewritten as 
\begin{equation}
\max_{\substack{\sigma\in\mathcal{M}_{+}\left(\mathbb{R}^{2};\text{Sym}_{2}\right)\\
\text{div}\,\sigma=0\\
\sigma=Id\,dx\ \text{on }\mathbb{R}^{2}\backslash\overline{\Omega}
}
}\,\int_{\overline{\Omega}}\left\langle \frac{1}{2}\nabla p\otimes\nabla p,Id\,dx-\sigma\right\rangle .\label{eq:relaxed-stress-formulation}
\end{equation}
As \prettyref{eq:relaxed-stress-formulation} is never used in this paper, we leave the details of its proof to the interested reader, and simply remark that it proceeds via the usual change of variables between an Airy potential $\varphi$ and its induced ``Airy stress'' 
\[
\sigma=\nabla^{\perp}\nabla^{\perp}\varphi
\]
with the only slight complication being that, here, $\sigma$ is considered to be a measure. 
In fact, the first step in our proof of \prettyref{eq:general-duality-result} will be to obtain an ill-posed version of \prettyref{eq:relaxed-stress-formulation} in which $\sigma$ is taken to be continuous, and for which a maximizer is not guaranteed. The second step is to relax this \emph{ad hoc} regularity constraint.

Our next two remarks concern the proof of \prettyref{thm:duality}: we claim that \prettyref{prop:duality} reduces to the stated duality between \prettyref{eq:primal-1} and \prettyref{eq:dual-1} in the case that $\Omega$ is simply connected. 
To see this, first note that the functional on the righthand side of \prettyref{eq:general-duality-result} is invariant under the addition of any affine function to $\varphi$. Indeed, it follows from its definition and the very weak Hessian identity \prettyref{eq:veryweak_hessian-duality} that
\begin{equation}\label{eq:L-affine}
L(a) = - \int_\Omega a \det \nabla \nabla p \,dx\quad\text{whenever } a\text{ is affine}.
\end{equation}
Now if $\Omega$ is simply connected it has only one exterior component, and hence any locally affine function on $\mathbb{R}^2\backslash\overline{\Omega}$ extends automatically as an affine function on $\mathbb{R}^2$. It follows that we may take $\varphi = \frac{1}{2}|x|^2$ on $\mathbb{R}^2\backslash{\Omega}$ without changing the result of \prettyref{eq:general-duality-result}, in which case $L=0$ and the original dual problem \prettyref{eq:dual-1} results.

We finish by showing how the complementary slackness conditions from \prettyref{thm:duality} follow from the general ones established here. The fact is that the mollifications
$\{\mu_{\delta}\}_{\delta>0}$ from \prettyref{eq:mollified_measure_defn-intro}
 approximate the given $\mu$  in the sense of \prettyref{eq:needed-approxs}. This is a direct consequence of \prettyref{lem:mollification_of_measures} and the identity 
\[
-\frac{1}{2}\text{curl}\text{curl}\,\mu = \det \nabla \nabla p.
\]
\prettyref{prop:duality} therefore implies that $u_{\text{eff}}$
and $\varphi$ are optimal if and only if 
\[
\left\langle \nabla^{\perp}\nabla^{\perp}\varphi,\mu\right\rangle =0\quad\text{in }\Omega\quad\text{and}\quad\hat{\nu}\cdot[\nabla\varphi]\left\langle \hat{\tau}\otimes\hat{\tau},\mu\right\rangle =0\quad\text{at }\partial\Omega
\]
in the regularized sense (meaning that \prettyref{eq:compl_slackness}
holds). Since the primal problems in \prettyref{eq:primal-1} are equivalent when $\Omega$ is simply connected,  \prettyref{thm:duality} follows from these remarks. 

The remainder of this section proves \prettyref{prop:duality}. \prettyref{subsec:duality-and-relaxation}
establishes \prettyref{eq:general-duality-result}. It is there that we explain how to anticipate the form of the general dual problem via a minimax
procedure. \prettyref{subsec:complementary-slackness}  proves the
complementary slackness conditions by establishing the integration
by parts identity
\begin{equation}
\frac{1}{2}\int_{\Omega}|\mu|_{1}-\int_{\Omega}(\varphi-\frac{1}{2}|x|^{2})\det\nabla\nabla p -L(\varphi-\frac{1}{2}|x|^{2}) = \lim_{n\to\infty}\,\frac{1}{2}\int_{\Omega}\left\langle \mu_{n},\nabla^{\perp}\nabla^{\perp}\varphi\right\rangle +\frac{1}{2}\int_{\partial\Omega}\hat{\nu}\cdot[\nabla\varphi]\left\langle \hat{\tau}\otimes\hat{\tau},\mu_{n}\right\rangle \label{eq:duality-gap}
\end{equation}
whenever $u_{\text{eff}}$ and $\varphi$ are admissible and $\{\mu_{n}\}_{n\in\mathbb{N}}$
approximates $\mu=\nabla p\otimes\nabla p\,dx-2e(u_{\text{eff}})$ in the sense of \prettyref{eq:needed-approxs}. Together, these complete the proof of the general duality result. They also
lay the groundwork for \prettyref{sec:method_of_characteristics}
where we make precise our method of stable lines.

\subsection{The dual problem\label{subsec:duality-and-relaxation}}

We begin by proving \prettyref{eq:general-duality-result}. The first step is to introduce a 
Lagrange multiplier $\sigma$ for the tension-free constraint, and apply a minimax
procedure to identify the dual problem it should solve. 
Given $u\in BD(\Omega)$, observe that 
\[
e(u)\leq\frac{1}{2}\nabla p\otimes\nabla p\,dx\quad\iff\quad\int_{\Omega}\left\langle \sigma,\frac{1}{2}\nabla p\otimes\nabla p\,dx-e(u)\right\rangle \geq0\quad\forall\,\sigma\in C(\overline{\Omega};\text{Sym}_{2})\text{ with }\sigma\geq0.
\]
The primal problem on the lefthand side of \prettyref{eq:general-duality-result} can therefore be rewritten as 
\[
\min_{\substack{u\in BD(\Omega)\\
e(u)\leq\frac{1}{2}\nabla p\otimes\nabla p\,dx
}
}\,\int_{\Omega}\frac{1}{2}|\nabla p|^{2}\,dx-\int_{\partial\Omega}u\cdot\hat{\nu}\,ds=\inf_{u\in BD(\Omega)}\sup_{\substack{\sigma\in C(\overline{\Omega};\text{Sym}_{2})\\
\sigma\geq0
}
}\,\int_{\Omega}\left\langle Id-\sigma,\frac{1}{2}\nabla p\otimes\nabla p\,dx-e(u)\right\rangle .
\]
Now to identify its dual, we reverse the order of operations
between $\inf$ and $\sup$. We do so informally at first,
and then again with a rigorous proof in \prettyref{lem:strong_duality}.
Let $\sigma\in C(\overline{\Omega};\text{Sym}_{2})$. By the
divergence theorem, 
\[
\int_{\Omega}\left\langle Id-\sigma,e(u)\right\rangle =0\quad\forall\,u\in BD(\Omega)\quad\iff\quad\text{div}\,\sigma=0\quad\text{on }\Omega\quad\text{and}\quad\sigma\hat{\nu}=\hat{\nu}\quad\text{at }\partial\Omega.
\]
The first condition on the righthand side is that $\sigma$ is weakly divergence-free. The second condition holds where the outwards-pointing unit normal $\hat{\nu}$ is well-defined. It follows that 
\[
\sup_{\substack{\sigma\in C(\overline{\Omega};\text{Sym}_{2})\\
\sigma\geq0
}
}\inf_{u\in BD(\Omega)}\,\int_{\Omega}\left\langle Id-\sigma,\frac{1}{2}\nabla p\otimes\nabla p\,dx-e(u)\right\rangle =\sup_{\substack{\sigma\in C(\overline{\Omega};\text{Sym}_{2})\\
\sigma\geq0\text{ and }\text{div}\,\sigma=0\\
\sigma\hat{\nu}=\hat{\nu}\ \text{at }\partial\Omega
}
}\,\int_{\Omega}\left\langle Id-\sigma,\frac{1}{2}\nabla p\otimes\nabla p\right\rangle \,dx
\]
and on the right appears our candidate dual.

The following result justifies the manipulations above.

\begin{lemma} \label{lem:strong_duality} There holds
\begin{equation}
\min_{\substack{u\in BD(\Omega)\\
e(u)\leq\frac{1}{2}\nabla p\otimes\nabla p\,dx
}
}\,\int_{\Omega}\frac{1}{2}|\nabla p|^{2}\,dx-\int_{\partial\Omega}u\cdot\hat{\nu}\,ds=\sup_{\substack{\sigma\in C(\overline{\Omega};\emph{Sym}_{2})\\
\sigma\geq0\text{ and }\emph{div}\,\sigma=0\\
\sigma\hat{\nu}=\hat{\nu}\ \text{at }\partial\Omega
}
}\,\int_{\Omega}\left\langle Id-\sigma,\frac{1}{2}\nabla p\otimes\nabla p\right\rangle \,dx.
\label{eq:first_duality}
\end{equation}
Equality continues to hold when the boundary condition is replaced by the more restrictive one 
\[
\sigma = Id \quad\text{at } \partial\Omega.
\]
\end{lemma} 
\begin{proof}  
Although the asserted equalities are not yet clear, the inequality
\[
\min_{\substack{u\in BD(\Omega)\\
e(u)\leq\frac{1}{2}\nabla p\otimes\nabla p\,dx
}
}\,\int_{\Omega}\frac{1}{2}|\nabla p|^{2}\,dx-\int_{\partial\Omega}u\cdot\hat{\nu}\,ds\geq\sup_{\substack{\sigma\in C(\overline{\Omega};\text{Sym}_{2})\\
\sigma\geq0\text{ and }\text{div}\,\sigma=0\\
\sigma\hat{\nu} = \hat{\nu}\ \text{at }\partial\Omega
}
}\,\int_{\Omega}\left\langle Id-\sigma,\frac{1}{2}\nabla p\otimes\nabla p\right\rangle \,dx 
\]
does follow directly from the minimax argument (recall the $\inf \sup$ of a functional is never less than its $\sup \inf$). 
Eliminating the common term $\int_{\Omega}\frac{1}{2}|\nabla p|^{2}$,
making the change of variables $\sigma=Id-\zeta$, and applying a
straightforward inclusion argument, we see it suffices to check that
\begin{equation}
\max_{\substack{u\in BD(\Omega)\\
e(u)\leq\frac{1}{2}\nabla p\otimes\nabla p\,dx
}
}\,\int_{\partial\Omega}u\cdot\hat{\nu}\,ds=\inf_{\substack{\zeta\in C_{c}(\Omega;\text{Sym}_{2})\\
Id\geq\zeta\text{ and }\text{div}\,\zeta=0
}
}\,\int_{\Omega}\left\langle Id-\zeta,\frac{1}{2}\nabla p\otimes\nabla p\right\rangle \,dx. \label{eq:claim_duality}
\end{equation}  
This can be done via the Fenchel--Rockafeller minimax theorem (see, e.g., \cite[Theorem 1.12]{brezis2011functional}), as we explain.

Introduce the vector spaces 
\[
E=C_{c}(\Omega;\text{Sym}_{2})\quad\text{and}\quad E^{*}=\mathcal{M}(\Omega;\text{Sym}_{2})
\]
and equip them with the uniform and dual norms, respectively. By
the Riez--Markov theorem, $E^{*}$ is the topological dual of $E$.
Define the functionals $\Phi,\Psi:E\to(-\infty,\infty]$ by 
\[
\Phi(\zeta)=\begin{cases}
\int_{\Omega}\left\langle Id-\zeta,\frac{1}{2}\nabla p\otimes\nabla p\right\rangle \,dx & Id\geq\zeta\\
\infty & \text{otherwise}
\end{cases}\quad\text{and}\quad\Psi(\zeta)=\begin{cases}
0 & \text{div}\,\zeta=0\\
\infty & \text{otherwise}
\end{cases}.
\]
Since $\frac{1}{2}\nabla p\otimes\nabla p\,dx\in E^{*}$, and as zero
is bounded uniformly away from $Id$, the functional $\Phi$ is finite
and continuous at $\zeta=0$. Evidently, $\Psi(0)<\infty$. Thus,
by the Fenchel--Rockafeller minimax theorem, 
\begin{equation}
\max_{\varepsilon\in E^{*}}\,-\Phi^{*}(-\varepsilon)-\Psi^{*}(\varepsilon)=\inf_{\zeta\in E}\,\Phi(\zeta)+\Psi(\zeta)\label{eq:Fenchel_Rockafellar}
\end{equation}
where the Legendre transforms $\Phi^{*},\Psi^{*}:E^{*}\to(-\infty,\infty]$
appearing on the lefthand side are given for $\varepsilon \in E^*$ by 
\[
\Phi^{*}(\varepsilon)=\sup_{\zeta\in E}\,\int_{\Omega}\langle\zeta,\varepsilon\rangle-\Phi(\zeta)\quad\text{and}\quad\Psi^{*}(\varepsilon)=\sup_{\zeta\in E}\,\int_{\Omega}\langle\zeta,\varepsilon\rangle-\Psi(\zeta).
\]
To finish, we must deduce from \prettyref{eq:Fenchel_Rockafellar}
the desired equality \prettyref{eq:claim_duality}. 

It is clear from the definitions that
\[
\inf_{\zeta\in E}\,\Phi(\zeta)+\Psi(\zeta)=\inf_{\substack{\zeta\in C_{c}(\Omega;\text{Sym}_{2})\\
Id\geq\zeta\text{ and }\text{div}\,\zeta=0
}
}\,\int_{\Omega}\left\langle Id-\zeta,\frac{1}{2}\nabla p\otimes\nabla p\right\rangle \,dx.
\]
Thus, the righthand sides of \prettyref{eq:claim_duality} and \prettyref{eq:Fenchel_Rockafellar} agree. To check their lefthand sides,
we must compute the Legendre transforms of $\Phi$ and $\Psi$. Given any $\varepsilon\in E^{*}$, we claim
that 
\begin{equation}
\Psi^{*}(\varepsilon)=\sup_{\substack{\zeta\in E\\
\text{div}\,\zeta=0
}
}\,\int_{\Omega}\left\langle \zeta,\varepsilon\right\rangle =\begin{cases}
0 & \exists\,u\in BD(\Omega)\ \text{s.t.}\ \varepsilon=e(u)\\
\infty & \text{otherwise}
\end{cases}. \label{eq:Legendretransform-Psi}
\end{equation}
The first equality is clear. To see the second, note by the divergence theorem (the first identity in \prettyref{eq:IBP1})
that 
\[
\int_{\Omega}\left\langle \zeta,e(u)\right\rangle =0
\]
whenever $u\in BD(\Omega)$ and $\zeta\in E\cap C_{c}^{1}(\Omega;\text{Sym}_{2})$ is divergence-free. By density, it holds for $\zeta\in E$
that are weakly divergence-free. On the other hand, suppose $\varepsilon\in E^{*}$
but that there does not exist $u\in BD(\Omega)$ such that $\varepsilon=e(u)$.
According to \cite[Proposition 2.1 and Theorem 2.1]{temam1980functions},
there exists a divergence-free $\zeta\in C_{c}^{\infty}(\Omega;\text{Sym}_{2})$
for which 
\[
\int_{\Omega}\left\langle \zeta,\varepsilon\right\rangle \neq0.
\]
Making the replacement $\zeta\to\lambda\zeta$ and sending $\lambda\to\infty$
or $-\infty$ as necessary, we deduce \prettyref{eq:Legendretransform-Psi}.

Finally, we compute the Legendre transform of $\Phi$. Given $u\in BD(\Omega)$,
we see using its definition that 
\begin{align*}
\Phi^{*}(-e(u)) & =\sup_{\substack{\zeta\in E\\
Id\geq\zeta
}
}\,\int_{\Omega}\left\langle \zeta,-e(u)\right\rangle -\int_{\Omega}\left\langle Id-\zeta,\frac{1}{2}\nabla p\otimes\nabla p\right\rangle \,dx\\
 & =\sup_{\substack{\zeta\in E\\
Id\geq\zeta
}
}\,\int_{\Omega}\left\langle \zeta,\frac{1}{2}\nabla p\otimes\nabla p\,dx-e(u)\right\rangle -\int_{\Omega}\left\langle Id,\frac{1}{2}\nabla p\otimes\nabla p\right\rangle \,dx\\
 & =\begin{cases}
\int_{\Omega}\left\langle Id,\frac{1}{2}\nabla p\otimes\nabla p\,dx-e(u)\right\rangle -\int_{\Omega}\left\langle Id,\frac{1}{2}\nabla p\otimes\nabla p\right\rangle \,dx & e(u)\leq\frac{1}{2}\nabla p\otimes\nabla p\,dx\\
\infty & \text{otherwise}
\end{cases}\\
 & =\begin{cases}
-\int_{\partial\Omega}u\cdot\hat{\nu}\,ds & e(u)\leq\frac{1}{2}\nabla p\otimes\nabla p\,dx\\
\infty & \text{otherwise}
\end{cases}.
\end{align*}
Combining this with \prettyref{eq:Legendretransform-Psi} proves that
\[
\max_{\varepsilon\in E^{*}}\,-\Phi^{*}(-\varepsilon)-\Psi^{*}(\varepsilon)=\max_{u\in BD(\Omega)}\,-\Phi^{*}(-e(u))=\max_{\substack{u\in BD(\Omega)\\
e(u)\leq\frac{1}{2}\nabla p\otimes\nabla p\,dx
}
}\,\int_{\partial\Omega}u\cdot\hat{\nu}\,ds.
\]
We conclude that the lefthand sides of 
\prettyref{eq:claim_duality} and \prettyref{eq:Fenchel_Rockafellar} are the same, and with this the proof is complete. \qed\end{proof} 

\prettyref{lem:strong_duality} is a good start, but we much prefer to identify a version of the dual problem for which maximizers
are guaranteed. The basic issue is that, while the admissible
$\sigma$ in \prettyref{eq:first_duality} satisfy 
\[
\int_{\Omega}|\sigma|_{1}=\int_{\Omega}\left\langle \nabla x,\sigma\right\rangle = 
\int_{\partial\Omega}x\cdot\hat{\nu}=\int_\Omega \langle \nabla x , Id \rangle = 2|\Omega|
\]
so that they are bounded \emph{a priori }in $L^{1}$, no similar
control on $\nabla\sigma$ is available (even as it is trace-free). So, while the admissible set in \prettyref{eq:first_duality}
is pre-compact in the weak-$*$ topology induced by the injection
\[
C(\overline{\Omega};\text{Sym}_{2})\to\mathcal{M}(\overline{\Omega};\text{Sym}_{2}),\ \sigma\mapsto\sigma\,dx,
\]
it is not closed. Evidently, the boundary conditions $\sigma \hat{\nu} = \hat{\nu}$ and \emph{ad hoc} regularity hypothesis that $\sigma \in C$ must be relaxed. Taking into account the low regularity of $\nabla p$, which is not necessarily continuous at the present level of generality, we find it convenient to do so by changing variables to the anticipated Airy
potentials $\varphi$.

It is well-known that if $\sigma$ is $\text{Sym}_{2}$-valued and
divergence-free, there exists a scalar-valued function $\varphi$ such
that 
\[
\sigma = \nabla^\perp\nabla^\perp \varphi
\]
on any simply connected domain. (See \prettyref{eq:Airyoperator}
for the notation $\nabla^{\perp}\nabla^{\perp}$.) Such functions $\varphi$ are usually referred to in the literature as ``Airy potentials''
or ``Airy stress functions'', and the divergence-free
fields $\nabla^{\perp}\nabla^{\perp}\varphi$ they generate are known as ``Airy stresses''. Note we need not rule out the possibility that $\Omega$ is multiply connected at this stage. This is because the required change of variables from $\sigma$ to $\varphi$ can be carried out on $\mathbb{R}^2$. Indeed, the boundary conditions in \prettyref{eq:first_duality} ensure that $\sigma$ can be extended to $\mathbb{R}^2$ in a divergence-free way. 
To prepare, we record some useful properties of the functional $L$ from \prettyref{eq:defn-of-L}.
\begin{lemma}\label{lem:well-definedL} 
The functional $L$ is well-defined on the vector space of functions $a:\mathbb{R}^2\to\mathbb{R}$ that are locally affine exterior to $\Omega$. 
It is linear and continuous in any norm for which the restriction map $a\mapsto a |_{\mathbb{R}^2\backslash\overline{\Omega}}$ is continuous.
\end{lemma}
\begin{proof} 
To check that $L$ is well-defined, we must show that the integral in \prettyref{eq:defn-of-L} does not depend on the choice of extension $\overline{p}$. 
That is, we must prove that
\[
\int_{\mathbb{R}^{2}\backslash\Omega}a\det\nabla\nabla \overline{p}_1=\int_{\mathbb{R}^{2}\backslash\Omega}a\det\nabla\nabla \overline{p}_2
\]
whenever $\overline{p}_1,\overline{p}_2\in W^{2,2}(\mathbb{R}^{2})\cap C_c(\mathbb{R}^2)$ satisfy $\overline{p}_1=\overline{p}_2$ on $\Omega$ and  $\nabla \nabla a =0$ on $\mathbb{R}^2\backslash\overline{\Omega}$.
By density, it suffices to take $\overline{p}_1$ and $\overline{p}_2$ to be smooth. Note we can also take $a$ to be smooth as every locally affine function on $\mathbb{R}^2\backslash\overline{\Omega}$ admits a smooth extension to $\mathbb{R}^{2}$. Testing the  identity \prettyref{eq:veryweak_hessian-duality} against $a$, integrating by parts twice, and subtracting yields that
\[
\int_{\mathbb{R}^{2}}a\left(\det\nabla\nabla \overline{p}_1-\det\nabla\nabla \overline{p}_2\right)=-\frac{1}{2}\int_{\mathbb{R}^{2}}\left\langle \nabla^{\perp}\nabla^{\perp}a,\nabla \overline{p}_1\otimes\nabla \overline{p}_1-\nabla \overline{p}_2\otimes\nabla \overline{p}_2\right\rangle.
\]
The integrand on the right vanishes a.e.\ by our hypotheses. For the same reason, the integral on the left may be taken over $\mathbb{R}^2\backslash\Omega$. The desired equality is proved.

Looking back to \prettyref{eq:defn-of-L}, we see that $L$ is a linear functional of $a |_{\mathbb{R}^2\backslash\overline{\Omega}}$. Since by hypothesis $\Omega$ is a bounded, Lipschitz domain, it has finitely many exterior components, i.e., 
\[
\mathbb{R}^2\backslash\overline{\Omega} = \cup_{i=1}^N E_i\quad\text{where the sets }E_i\text{ are open and disjoint}.
\]
That $a$ is locally affine exterior to $\Omega$ is equivalent to the existence of $\{m_{i}\}_{i=1}^{N}\subset\mathbb{R}^{2}$
and $\{b_{i}\}_{i=1}^{N}\subset\mathbb{R}$ such that
\[
a=m_{i}\cdot x+b_{i}\quad\text{on }E_i,\quad\text{for }i=1,\dots, N.
\]
Quotienting out by the equivalence relation that $a_1\sim a_2$ if $a_1|_{\mathbb{R}^2\backslash\overline{\Omega}} = a_2|_{\mathbb{R}^2\backslash\overline{\Omega}}$, there results a finite dimensional vector space on which $L$ is well-defined. The stated continuity now follows from the elementary facts that every linear function of finitely many variables is continuous, and that this holds regardless of the choice of norm.  
\qed
\end{proof}
We are ready to change variables from $\sigma$ to $\varphi$. To help simplify the presentation, and as it does not affect the end result, we use the ``restricted'' set of admissible $\sigma$ from \prettyref{lem:strong_duality}.
\begin{lemma} \label{lem:correspondence} The restricted sets of admissible
stresses and Airy potentials 
\[
\left\{ \sigma\in C(\overline{\Omega};\emph{Sym}_{2}):\sigma\geq0,\ \emph{div}\,\sigma=0,\ \sigma=Id \emph{ at }\partial\Omega \right\}\quad\text{and}\quad
\left\{ \varphi\in C^{2}(\mathbb{R}^{2}):\nabla\nabla\varphi\geq0,\ \nabla\nabla\varphi=Id\ \emph{on }\mathbb{R}^{2}\backslash\Omega \right\} 
\]
are put into a many-to-one correspondence via the relation
\[
\nabla^{\perp}\nabla^{\perp}\varphi=\begin{cases}
\sigma & \text{on }\Omega\\
Id & \text{on }\mathbb{R}^{2}\backslash\Omega
\end{cases}.
\]
Under this correspondence,
\[
\int_{\Omega}\left\langle Id-\sigma,\frac{1}{2}\nabla p\otimes\nabla p\right\rangle \,dx=\int_{\Omega}(\varphi-\frac{1}{2}|x|^{2})\det\nabla\nabla p\,dx + L(\varphi-\frac{1}{2}|x|^{2}). 
\]
\end{lemma}
\begin{proof} 
The stated correspondence follows from our previous remarks on the introduction of Airy potentials and the fact that $\mathbb{R}^2$ is simply connected. 
In particular, when we extend a given $\sigma$ by setting it equal to $Id$ off of $\Omega$, the resulting $\text{Sym}_2$-valued function is continuous, non-negative, and weakly divergence-free on all of $\mathbb{R}^2$. Thus, there exists a corresponding $\varphi$, which is of course non-unique. The reverse direction is clear. 

We turn to prove the stated equality. By \prettyref{lem:well-definedL}, we may fix some compactly supported, $W^{2,2}$ extension $\overline{p}$ of $p$  in the definition of $L$. It follows from the given correspondence and the very weak Hessian identity \prettyref{eq:veryweak_hessian-duality}  that
\begin{align*}
\int_{\Omega}\left\langle Id-\sigma,\frac{1}{2}\nabla p\otimes\nabla p\right\rangle &=\int_{\Omega}\left\langle \nabla^{\perp}\nabla^{\perp}(\frac{1}{2}|x|^{2}-\varphi),\frac{1}{2}\nabla p\otimes\nabla p\right\rangle =\int_{\mathbb{R}^{2}}\left\langle \nabla^{\perp}\nabla^{\perp}(\varphi-\frac{1}{2}|x|^{2}),-\frac{1}{2}\nabla\overline{p}\otimes\nabla\overline{p}\right\rangle \\&=\int_{\mathbb{R}^{2}}(\varphi-\frac{1}{2}|x|^{2})\det\nabla\nabla \overline{p}=\int_{\Omega}(\varphi-\frac{1}{2}|x|^{2})\det\nabla\nabla p+L(\varphi-\frac{1}{2}|x|^{2})
\end{align*}
as claimed. Note we used \prettyref{eq:defn-of-L} at the end. \qed\end{proof} 
Combining \prettyref{lem:strong_duality} and \prettyref{lem:correspondence} we get that
\begin{equation}
\min_{\substack{u\in BD(\Omega)\\
e(u)\leq\frac{1}{2}\nabla p\otimes\nabla p\,dx
}
}\,\int_{\Omega}\frac{1}{2}|\nabla p|^{2}-\int_{\partial\Omega}u\cdot\hat{\nu}=\sup_{\substack{\varphi\in C^{2}(\mathbb{R}^{2})\\
\varphi\text{ is convex}\\
\nabla\nabla(\varphi-\frac{1}{2}|x|^2) = 0\text{ on }\mathbb{R}^{2}\backslash\Omega
}
}\,\int_{\Omega}(\varphi-\frac{1}{2}|x|^{2})\det\nabla\nabla p+L(\varphi-\frac{1}{2}|x|^{2}).\label{eq:intermediate_step_equality}
\end{equation}
The last step is
to do away with the \emph{ad hoc} regularity hypothesis that $\varphi \in C^2$, thereby allowing the corresponding $\sigma$ to be measure-valued. We must check that the supremum remains the same, and that it is achieved. 

\begin{lemma} \label{lem:closures} The supremum on the righthand side of \prettyref{eq:intermediate_step_equality} is equal to the maximum on the righthand side of \prettyref{eq:general-duality-result}. 
There exists an optimal $\varphi$ achieving the maximum there.  
\end{lemma} 

\begin{proof} The first part of the result follows immediately once we establish that 
\begin{equation}
\int_{\Omega}\frac{1}{2}|\nabla p|^{2}-\int_{\partial\Omega}u\cdot\hat{\nu}\geq\int_{\Omega}(\varphi-\frac{1}{2}|x|^{2})\det\nabla\nabla p+L(\varphi-\frac{1}{2}|x|^{2})\label{eq:weak_duality}
\end{equation}
whenever $u$ and $\varphi$ are admissible in \prettyref{eq:general-duality-result}. Indeed, the admissible set of $\varphi$ in \prettyref{eq:intermediate_step_equality} is a subset of that in \prettyref{eq:general-duality-result}. Enlarging an admissible set can never make the resulting supremum smaller. 
Now to prove \prettyref{eq:weak_duality}, we shall make use of the integration by parts identity \prettyref{eq:duality-gap}
introduced above and proved in \prettyref{lem:IBP-identity} below,
as well as the properties of the mollifications $\{\mu_{\delta}\}_{\delta>0}$ from \prettyref{eq:mollified_measure_defn-intro} to be proved
in \prettyref{lem:mollification_of_measures}. The reader may check that these results stand independently
of the desired inequality.

So let $u$ and $\varphi$ be admissible in \prettyref{eq:general-duality-result} and call
$\mu=\nabla p\otimes\nabla p\,dx-2e(u)$. \prettyref{lem:mollification_of_measures} and  \prettyref{lem:IBP-identity} show that 
\[
\frac{1}{2}\int_{\Omega}|\mu|_{1}-\int_{\Omega}(\varphi-\frac{1}{2}|x|^{2})\det\nabla\nabla p-L(\varphi-\frac{1}{2}|x|^{2})=\lim_{\delta\to\infty}\,\frac{1}{2}\int_{\Omega}\left\langle \mu_{\delta},\nabla^{\perp}\nabla^{\perp}\varphi\right\rangle +\frac{1}{2}\int_{\partial\Omega}\hat{\nu}\cdot[\nabla\varphi]\left\langle \hat{\tau}\otimes\hat{\tau},\mu_{\delta}\right\rangle. 
\]
The integrands on the right are non-negative by admissibility: that  $\mu_\delta \geq 0$ follows from our choice to take the kernel $\rho \geq 0$ in their definition; that $\nabla \nabla \varphi \geq 0$ in $\Omega$ and that $\hat{\nu}\cdot[\nabla\varphi] \geq 0$ at $\partial\Omega$ are  easy consequences of the convexity of $\varphi$ (we include the details of this in \prettyref{lem:restriction} below).
Applying
the divergence theorem from \prettyref{eq:IBP1} and
the fact that $\mu\geq0$, we find that 
\[
\frac{1}{2}\int_{\Omega}|\mu|_{1}=\int_{\Omega}\left\langle Id,\frac{1}{2}\mu\right\rangle =\int_{\Omega}\left\langle Id,\frac{1}{2}\nabla p\otimes\nabla p\,dx-e(u)\right\rangle =\int_{\Omega}\frac{1}{2}|\nabla p|^{2}-\int_{\partial\Omega}u\cdot\hat{\nu}.
\]
The inequality \prettyref{eq:weak_duality} is proved. Thus,  the optimal values of the maximization problems in \prettyref{eq:general-duality-result} and  \prettyref{eq:intermediate_step_equality} must be the same, regardless of whether or not they admit any solutions.

To finish, we must verify the existence of a maximizer for the dual problem in \prettyref{eq:general-duality-result}. We apply the direct method. Let $B$ be a ball of finite radius that contains $\Omega$, and note it suffices
to search for a maximizer in the subset
\begin{equation}\label{eq:max-subset}
\left\{ \varphi\in C(\mathbb{R}^{2}):\varphi\text{ is convex},\ \nabla\nabla(\varphi-\frac{1}{2}|x|^2) = 0\text{ on }\mathbb{R}^{2}\backslash\overline{\Omega}\right\} \cap \left\{ \varphi=\frac{1}{2}|x|^{2}\text{ on }\mathbb{R}^{2}\backslash B\right\}.
\end{equation}
Indeed, the functional being maximized in \prettyref{eq:general-duality-result} is unchanged under the addition of any affine function to $\varphi$, as was noted in the paragraph surrounding \prettyref{eq:L-affine}. Thus, we can take $\varphi = \frac{1}{2}|x|^2$ on the unbounded component of $\mathbb{R}^2\backslash\overline{\Omega}$. Now observe that \prettyref{eq:max-subset} is compact
in the uniform norm topology: it is closed, and using the
\emph{a priori} bounds 
\[
||\nabla\varphi||_{L^{\infty}(B)}\leq1\quad\text{and}\quad||\varphi||_{L^{\infty}(B)}\leq\max_{x\in\overline{B}}\,\frac{1}{2}|x|^{2}+\text{diam}\,B
\]
we deduce from Arzel\`a--Ascoli that it is pre-compact.
That the functional in \prettyref{eq:general-duality-result} is uniformly continuous follows from our standing assumptions that $\Omega$ is bounded and Lipschitz and that $p\in W^{2,2}(\Omega)$. In particular, the uniform continuity of $L(\varphi-\frac{1}{2}|x|^2)$ follows from \prettyref{lem:well-definedL}. The existence of a maximizing $\varphi$ is proved.
\qed\end{proof}
We are ready to prove \prettyref{eq:general-duality-result}.

\vspace{1em}
\noindent\emph{Proof of the equality part of \prettyref{prop:duality}} 
Combining \prettyref{lem:strong_duality}-\prettyref{lem:closures} yields the
string of equalities 
\begin{multline*}
\min_{\substack{u\in BD(\Omega)\\
e(u)\leq\frac{1}{2}\nabla p\otimes\nabla p\,dx
}
}\,\int_{\Omega}\frac{1}{2}|\nabla p|^{2}-\int_{\partial\Omega}u\cdot\hat{\nu} 
= \sup_{\substack{\sigma\in C(\overline{\Omega};\text{Sym}_{2})\\
\sigma\geq0\text{ and }\text{div}\,\sigma=0\\
\sigma\hat{\nu}=\hat{\nu}\ \text{at }\partial\Omega
}
}\,\int_{\Omega}\left\langle Id-\sigma,\frac{1}{2}\nabla p\otimes\nabla p\right\rangle 
 \\ =\sup_{\substack{\varphi\in C^{2}(\mathbb{R}^{2})\\
\varphi\text{ is convex}\\
\nabla\nabla(\varphi-\frac{1}{2}|x|^2) = 0\text{ on }\mathbb{R}^{2}\backslash\Omega
}
}\,\int_{\Omega}(\varphi-\frac{1}{2}|x|^{2})\det\nabla\nabla p+L=\max_{\substack{\varphi:\mathbb{R}^{2}\to\mathbb{R}\\
\varphi\text{ is convex}\\
\nabla\nabla(\varphi-\frac{1}{2}|x|^2) = 0\text{ on }\mathbb{R}^{2}\backslash\overline{\Omega}
}
}\,\int_{\Omega}(\varphi-\frac{1}{2}|x|^{2})\det\nabla\nabla p+L
\end{multline*}
where we have abbreviated the argument $\varphi-\frac{1}{2}|x|^2$ of $L$. The proof of \prettyref{eq:general-duality-result} is complete. \qed
\vspace{1em}

Before moving on to the complementary slackness part of \prettyref{prop:duality},
we pause to point out that the admissible Airy potentials from \prettyref{eq:general-duality-result} can
be described using boundary conditions. This was stated 
in \prettyref{rem:bcs} in the context of a simply connected domain (where we took $a=0$). Note we also make use of it later on below. 

\begin{lemma} \label{lem:restriction} Upon restriction to $\Omega$,
the admissible set of Airy potentials in \prettyref{eq:general-duality-result}
can be equivalently described as those $\varphi\in HB(\Omega)$ for
which 
\begin{equation}
\nabla\nabla\varphi\geq0\quad\text{on }\Omega\label{eq:non-negative-Hessian}
\end{equation}
and such that
\begin{equation}
\varphi=\frac{1}{2}|x|^{2}+a\quad\text{and}\quad\hat{\nu}\cdot\nabla\varphi\leq\hat{\nu}\cdot(x+\nabla a)\quad\text{at }\partial\Omega\label{eq:boundary-conditions-monotone}
\end{equation}
for some $a:\mathbb{R}^{2}\to\mathbb{R}$ that is locally affine exterior to $\Omega$.
These boundary conditions are understood in the sense of trace, i.e., the values of $\varphi$ and $\nabla\varphi$ at $\partial\Omega$ are taken from $\Omega$,
while those of $a$ and $\nabla a$ at $\partial\Omega$  are taken from $\mathbb{R}^{2}\backslash\overline{\Omega}$. 
\end{lemma}

\begin{proof} The result follows from the identity
\begin{equation}
\nabla\nabla\varphi=\nabla\nabla\varphi\lfloor\Omega+Id\,dx\lfloor\mathbb{R}^{2}\backslash\overline{\Omega}+\hat{\nu}\cdot[\nabla\varphi]\hat{\nu}\otimes\hat{\nu}\,\mathcal{H}^{1}\lfloor\partial\Omega\quad\text{on }\mathbb{R}^{2}\label{eq:distributional_Hessian_identity}
\end{equation}
which holds for all $\varphi\in HB_{\text{loc}}(\mathbb{R}^{2})$
such that $\varphi-\frac{1}{2}|x|^{2}$ is locally affine on $\mathbb{R}^{2}\backslash\overline{\Omega}$.
Indeed, if $\varphi\in HB(\Omega)$ satisfies \prettyref{eq:non-negative-Hessian}
and \prettyref{eq:boundary-conditions-monotone} for some $a$ as in the statement,
 its extension by $\frac{1}{2}|x|^{2}+a$ belongs to $HB_{\text{loc}}(\mathbb{R}^{2})$ and obeys \prettyref{eq:distributional_Hessian_identity}.
Its Hessian is non-negative, and so it is convex. Therefore it is admissible in  \prettyref{eq:general-duality-result}. 

On the other hand, if $\varphi$ is admissible then it is a convex extension of $\frac{1}{2}|x|^{2}+a$ from $\mathbb{R}^{2}\backslash\overline{\Omega}$
to $\mathbb{R}^2$ for some $a$ that is locally affine on $\mathbb{R}^2\backslash\overline{\Omega}$. It belongs to $HB_{\text{loc}}(\mathbb{R}^{2})$ and satisfies $\nabla\nabla \varphi \geq 0$ on $\mathbb{R}^2$. It restricts to an element of $HB(\Omega)$ with boundary trace equal
to $\frac{1}{2}|x|^{2}+a$, so the first part of \prettyref{eq:boundary-conditions-monotone} holds. Testing \prettyref{eq:distributional_Hessian_identity} at $\partial\Omega$ and using that 
\[
[\nabla\varphi]=x+\nabla a|_{\partial(\mathbb{R}^2\backslash\overline{\Omega})} -\nabla\varphi|_{\partial\Omega}\quad\text{at }\partial\Omega
\]
yields the rest of \prettyref{eq:boundary-conditions-monotone}. Testing it at $\Omega$ proves  \prettyref{eq:non-negative-Hessian}. 
\qed\end{proof}

\subsection{Complementary slackness conditions\label{subsec:complementary-slackness} }

It remains to prove the complementary slackness part of \prettyref{prop:duality}.
First, we verify that the mollification procedure from \prettyref{eq:mollified_measure_defn-intro}
can be used to generate the approximations referred to there.

\begin{lemma} \label{lem:mollification_of_measures} Let $\mu\in\mathcal{M}(\Omega;\emph{Sym}_{2})$
be such that $\emph{curl\,curl}\,\mu\in\mathcal{M}(\Omega)$. Its mollifications $\{\mu_{\delta}\}_{\delta>0}$ from \prettyref{eq:mollified_measure_defn-intro} belong to $C^{\infty}(\overline{\Omega};\emph{Sym}_{2})$
and converge  to $\mu$ in the following sense:
\[
\mu_{\delta}\,dx\to\mu\quad\text{narrowly in }\mathcal{M}(\Omega;\emph{Sym}_{2})\quad\text{and}\quad\emph{curl\,curl}\,\mu_{\delta}\,dx\stackrel{*}{\rightharpoonup}\emph{curl\,curl}\,\mu\quad\text{weakly-\ensuremath{*} in }\mathcal{M}(\Omega)
\]
as $\delta \to 0$. If in addition $\mu\geq0$, then $\mu_{\delta}(x)\geq0$ for all $x\in\overline{\Omega}$
and $\delta>0$. \end{lemma} 

\begin{proof} The last conclusion regarding non-negativity follows
from our assumption that the mollifying kernel $\rho\geq0$. We establish
the claimed convergences now. Let $\sigma\in C_{b}(\Omega;\text{Sym}_{2})$.
Fixing $y\in\Omega$, we see that 
\[
\int_{\Omega}\rho\left(\frac{1}{\delta}(x-y)\right)\sigma(x)\,\frac{dx}{\delta^{2}}\to\sigma(y)\quad\text{as }\delta\to0
\]
and also that 
\[
\left|\int_{\Omega}\rho\left(\frac{1}{\delta}(x-y)\right)\sigma(x)\,\frac{dx}{\delta^{2}}\right| \leq\int\indicator{\Omega}(y+\delta x)\rho(x)|\sigma(y+\delta x)|\,dx \leq||\sigma||_{L^{\infty}(\Omega)}
\]
for all $\delta>0$. Therefore, 
\begin{align*}
\int_{\Omega}\left\langle \sigma,\mu_{\delta}\right\rangle \,dx & =\int_{\Omega}\left\langle \sigma(x),\int_{\Omega}\frac{1}{\delta^{2}}\rho\left(\frac{1}{\delta}(x-y)\right)d\mu(y)\right\rangle \,dx\\
 & =\int_{\Omega}\left\langle \int_{\Omega}\rho\left(\frac{1}{\delta}(x-y)\right)\sigma(x)\,\frac{dx}{\delta^{2}},d\mu(y)\right\rangle \to\int_{\Omega}\left\langle \sigma,\mu\right\rangle 
\end{align*}
by the bounded convergence theorem. As $\sigma$ was arbitrary, we
conclude the narrow convergence of $\mu_{\delta}\,dx$ to $\mu$.

Now we show the weak-$*$ convergence of $\text{curl}\text{curl}\,\mu_{\delta}\,dx$ to $\text{curl}\text{curl}\,\mu$. Let $\chi\in C_{c}^{\infty}(\Omega)$.
For small enough $\delta>0$, we can apply the fact that $\text{curl}\text{curl}$ and $\nabla^\perp\nabla^\perp$ are formally adjoint along with Fubini's theorem to write that 
\begin{align*}
 \int_{\Omega}\chi\text{curl}\text{curl}\,\mu_{\delta}\,dx 
 & =\int_{\Omega}\left\langle \int_{\Omega}\frac{1}{\delta^{2}}\rho\left(\frac{1}{\delta}(x-y)\right)\nabla^{\perp}\nabla^{\perp}\chi(x)\,dx,\mu(y)\right\rangle \\
 & =\int_{\Omega}\left\langle \int_{B_{1}}\rho\left(x\right)\nabla^{\perp}\nabla^{\perp}\chi(\cdot+\delta x)\,dx,\mu\right\rangle  =\int_{\Omega}\left\langle \nabla^{\perp}\nabla^{\perp}\int_{B_{1}}\rho(x)\chi(\cdot+\delta x)\,dx,\mu\right\rangle .
\end{align*}
Using that $\text{curl\,curl}\,\mu\in\mathcal{M}(\Omega)$, there
follows 
\[
\int_{\Omega}\left\langle \nabla^{\perp}\nabla^{\perp}\int_{B_{1}}\rho(x)\chi(\cdot+\delta x)\,dx,\mu\right\rangle =\int_{\Omega}\left[\int_{B_{1}}\rho(x)\chi(y+\delta x)\,dx\right]d\text{curl\,curl}\,\mu(y)\to\int_{\Omega}\chi d\text{curl\,curl}\,\mu
\]
as $\delta\to0$. The proof is complete. \qed\end{proof} 
Next, we establish the integration by parts identity   \prettyref{eq:duality-gap}. 

\begin{lemma} \label{lem:IBP-identity} Let $u$ and
$\varphi$ be admissible in \prettyref{eq:general-duality-result}
and suppose the sequence 
$\{\mu_{n}\}_{n\in\mathbb{N}}\subset C^{2}(\overline{\Omega};\emph{\ensuremath{\text{Sym}_{2}}})$ converges to $\mu=\nabla p\otimes\nabla p\,dx-2e(u)$ in the sense of \prettyref{eq:needed-approxs}.
Then, 
\[
\frac{1}{2}\int_{\Omega}|\mu|_{1}-\int_{\Omega}(\varphi-\frac{1}{2}|x|^{2})\det\nabla\nabla p\,dx -L(\varphi-\frac{1}{2}|x|^{2})=\lim_{n\to\infty}\,\frac{1}{2}\int_{\Omega}\left\langle \mu_{n},\nabla^{\perp}\nabla^{\perp}\varphi\right\rangle +\frac{1}{2}\int_{\partial\Omega}\hat{\nu}\cdot[\nabla\varphi]\left\langle \hat{\tau}\otimes\hat{\tau},\mu_{n}\right\rangle ds.
\] 
\end{lemma}
\begin{proof} Since $\varphi - \frac{1}{2} |x|^2$ is locally affine on $\mathbb{R}^2\backslash\overline{\Omega}$, we can find $a\in C^\infty(\mathbb{R}^{2})$ such that $\varphi-\frac{1}{2}|x|^{2}=a$ there. 
Integrating by parts twice via
\prettyref{eq:IBP1} and recalling that the unit tangent and outwards-pointing unit normal vectors to $\partial\Omega$ were taken to satisfy $\hat{\tau}=\hat{\nu}^{\perp}$, we deduce that 
\begin{align*}
 & \int_{\Omega}\left\langle \mu_{n},\nabla^{\perp}\nabla^{\perp}(\varphi-\frac{1}{2}|x|^{2}-a)\right\rangle \\
 & \qquad=\int_{\Omega}(\varphi-\frac{1}{2}|x|^{2}-a)\text{curl}\text{curl}\,\mu_{n}-\int_{\partial\Omega}\text{curl}\,\mu_{n}\cdot\hat{\tau}(\varphi-\frac{1}{2}|x|^{2}-a)+\int_{\partial\Omega}\left\langle \mu_{n},\nabla^{\perp}(\varphi-\frac{1}{2}|x|^{2}-a)\otimes\hat{\tau}\right\rangle \\
 & \qquad=\int_{\Omega}(\varphi-\frac{1}{2}|x|^{2}-a)\text{curl}\text{curl}\,\mu_{n}-\int_{\partial\Omega}\hat{\nu}\cdot \left[\nabla\varphi\right] \left\langle \mu_{n},\hat{\tau}\otimes\hat{\tau}\right\rangle 
\end{align*}
where in the last line we used that $[\nabla \varphi] = x + \nabla a - \nabla \varphi |_{\partial\Omega}$ points normally to $\partial\Omega$ a.e. Thus,
\[
\int_{\Omega}\left\langle \mu_{n},\nabla^{\perp}\nabla^{\perp}\varphi\right\rangle +\int_{\partial\Omega}\hat{\nu}\cdot\left[\nabla\varphi\right]\left\langle \mu_{n},\hat{\tau}\otimes\hat{\tau}\right\rangle =\int_{\Omega}\left\langle Id,\mu_{n}\right\rangle +\int_{\Omega}(\varphi-\frac{1}{2}|x|^{2}-a)\text{curl}\text{curl}\,\mu_{n}+\int_{\Omega}\left\langle \mu_{n},\nabla^{\perp}\nabla^{\perp}a\right\rangle 
\]
for each $n$. Taking $n\to\infty$ and using the given approximation properties we deduce that
\begin{multline}
\lim_{n\to\infty}\,\frac{1}{2}\int_{\Omega}\left\langle \mu_{n},\nabla^{\perp}\nabla^{\perp}\varphi\right\rangle +\frac{1}{2}\int_{\partial\Omega}\hat{\nu}\cdot[\nabla\varphi]\left\langle \hat{\tau}\otimes\hat{\tau},\mu_{n}\right\rangle \\
=\frac{1}{2}\int_{\Omega}|\mu|_{1}-\int_{\Omega}(\varphi-\frac{1}{2}|x|^{2}-a)\det\nabla\nabla p+\frac{1}{2}\int_{\Omega}\left\langle \nabla^{\perp}\nabla^{\perp}a,\mu\right\rangle. \label{eq:limiting-statement-step1}
\end{multline}
To finish, we must rewrite the righthand side above using the definitions of $\mu$ and $L$.

Recall we took $a\in C^\infty (\mathbb{R}^2)$ and to equal to $\varphi - \frac{1}{2}|x|^2$ on $\mathbb{R}^2\backslash\overline{\Omega}$. Observe that
\begin{equation}
\int_{\Omega}\left\langle \nabla^{\perp}\nabla^{\perp}a,e(u)\right\rangle =0.\label{eq:limiting-statement-step2}
\end{equation}
Indeed, we can find a compactly supported, $BD$ extension of $u$ to $\mathbb{R}^2$, and then as $\nabla^{\perp}\nabla^{\perp}a$ is divergence-free and vanishes outside of $\Omega$, the desired identity follows by the divergence theorem.
Next, we claim that
\begin{equation}
L(a) = - \int_\Omega a \det \nabla \nabla p + \left\langle \nabla^\perp \nabla^\perp a , \frac{1}{2}\nabla p \otimes \nabla p\right\rangle.\label{eq:limiting-statement-step3}
\end{equation}
To prove it, introduce a compactly supported, $W^{2,2}$ extension $\overline{p}$ of $p$ from $\Omega$ to $\mathbb{R}^2$, and test the very weak Hessian identity \prettyref{eq:veryweak_hessian-duality} against $a$. The result is that
\[
\int_{\mathbb{R}^2} a \det \nabla \nabla \overline{p} = - \int_{\mathbb{R}^2} \left\langle \nabla^\perp \nabla^\perp a , \frac{1}{2}\nabla \overline{p}\otimes \overline{p}\right\rangle.
\] 
Breaking up the integral on the left to be over $\Omega$ and $\mathbb{R}^2\backslash\Omega$, and using the definition \prettyref{eq:defn-of-L} of $L$, there follows \prettyref{eq:limiting-statement-step3}. 
Combining \prettyref{eq:limiting-statement-step1}-\prettyref{eq:limiting-statement-step3} with our choice to call $\mu=\nabla p\otimes\nabla p\,dx-2e(u)
$ finishes the proof. 
\qed\end{proof} 
We are finally ready to complete the proof of \prettyref{prop:duality}.

\vspace{1em}
\noindent\emph{Proof of the complementary slackness part of \prettyref{prop:duality}}
Let $u_{\text{eff}}$ and $\varphi$ be admissible in \prettyref{eq:general-duality-result}, and let
$\{\mu_{n}\}_{n\in\mathbb{N}}$ be non-negative and approximate $\mu=\nabla p\otimes\nabla p\,dx-2e(u_{\text{eff}})$ in the sense of \prettyref{eq:needed-approxs}. Recall from \prettyref{lem:IBP-identity} that 
\begin{equation}
\frac{1}{2}\int_{\Omega}|\mu|_{1}-\int_{\Omega}(\varphi-\frac{1}{2}|x|^{2})\det\nabla\nabla p-L(\varphi-\frac{1}{2}|x|^{2})=\lim_{n\to\infty}\,\frac{1}{2}\int_{\Omega}\left\langle \mu_{n},\nabla^{\perp}\nabla^{\perp}\varphi\right\rangle +\frac{1}{2}\int_{\partial\Omega}\hat{\nu}\cdot[\nabla\varphi]\left\langle \hat{\tau}\otimes\hat{\tau},\mu_{n}\right\rangle .\label{eq:IBP_applied}
\end{equation}
Since $\varphi$ is convex we find, just as in the proof of \prettyref{lem:restriction}, that
\[
\nabla\nabla\varphi\geq0 \quad\text{on } \Omega \quad\text{and}\quad \hat{\nu}\cdot[\nabla\varphi] \geq 0 \quad\text{at } \partial\Omega.
\]
So, the integrals on the righthand side of  \prettyref{eq:IBP_applied} are non-negative and limit to zero if and only if its lefthand side vanishes. At the same time, due to \prettyref{eq:general-duality-result}, the lefthand side vanishes if and only if $u_{\text{eff}}$ and $\varphi$ are optimal. It remains to produce an example
of an approximating sequence $\{\mu_{n}\}$. Using \prettyref{lem:mollification_of_measures} and the fact that 
\[
-\frac{1}{2}\text{curl}\text{curl}\,\mu = \det \nabla \nabla p,
\] 
we see the mollification procedure from \prettyref{eq:mollified_measure_defn-intro} approximates $\mu$ in the desired sense.
\qed
\vspace{1em}

It is natural to wonder if there is some more intrinsic way of stating
the complementary slackness conditions, i.e., one that does not make use of
 \emph{ad hoc} regularizations. The crux of the issue is that one must make
sense of the ``Frobenius inner product'' between two $\text{Sym}_{2}$-valued
Radon measures, one of which is divergence-free and the other of which
has its curlcurl controlled. Consider, for instance, how to define $\left\langle \nabla^{\perp}\nabla^{\perp}\varphi,e(u)\right\rangle $ when $\varphi\in HB$ and $u\in BD$. If for some reason we knew that
$\varphi\in C^{1}$, we could fall back on the identity 
\[
\left\langle \nabla^{\perp}\nabla^{\perp}\varphi,e(u)\right\rangle =2\text{curl}\left(e(u)\nabla^{\perp}\varphi\right)-\text{curl}\text{curl}\left(e(u)\varphi\right)
\]
to define the product on the lefthand side as the distribution 
on the right. This sort of approach  goes
back at least to \cite{kohn1983dual}. Unfortunately, it is not the
case that every optimal Airy potential is $C^{1}$. 
Lacking a successful distributional approach, we have simply opted to use
regularizations instead. For a related discussion see \cite{arroyo-rabasa2017relaxation}
(however, the functionals there do not appear to allow for one-sided
constraints).

\section{Wrinkle patterns by the method of stable lines\label{sec:method_of_characteristics}}

We continue our study of the limiting minimization problems in \prettyref{eq:primal-1}.
\prettyref{sec:dualitytheory} identified various versions of the dual problem, along with complementary slackness conditions satisfied by optimal primal-dual pairs. There we established a general duality result, applicable even in situations where we do not yet know the $\Gamma$-limit of the rescaled energies $\frac{1}{2\sqrt{bk}+\gamma}E_{b,k,\gamma}$ (such as when $\Omega$ is not strictly star-shaped). Our results thus far can be summarized as follows: under the full set of assumptions \vpageref{par:Assumptions} of the introduction, 
$\mu\in\mathcal{M}_+(\Omega;\text{Sym}_2)$ arises as the defect measure of an almost minimizing sequence for $E_{b,k,\gamma}$ if and only if it satisfies
\begin{equation}
\begin{cases}
-\frac{1}{2}\text{curl}\text{curl}\,\mu=\det\nabla\nabla p & \text{on }\Omega\\
\left\langle \nabla^{\perp}\nabla^{\perp}\varphi,\mu\right\rangle =0 & \text{on }\Omega\\
\hat{\nu}\cdot[\nabla\varphi]\left\langle \hat{\tau}\otimes\hat{\tau},\mu\right\rangle =0 & \text{at }\partial\Omega
\end{cases}\label{eq:wrinkling-PDEs}
\end{equation}
where $\varphi$ solves the dual problem \prettyref{eq:dual-1}. The first equation
holds in the sense of distributions, while the second and third ones
hold in the regularized sense, i.e.,
\begin{equation}
\lim_{\delta\to0}\,\int_{\Omega}|\left\langle \mu_{\delta},\nabla^{\perp}\nabla^{\perp}\varphi\right\rangle |=0\quad\text{and}\quad\lim_{\delta\to0}\,\int_{\partial\Omega}|\hat{\nu}\cdot[\nabla\varphi]\left\langle \hat{\tau}\otimes\hat{\tau},\mu_{\delta}\right\rangle |\,ds=0\label{eq:complementary_slackness_weaksense}
\end{equation}
where $\{\mu_{\delta}\}_{\delta>0}$ are the mollifications in \prettyref{eq:mollified_measure_defn-intro}.
Moreover, the same system \prettyref{eq:wrinkling-PDEs} applies even when only the basic assumptions from \prettyref{eq:A1a} hold, so long as we take $\mu = \nabla p\otimes\nabla p\,dx-2e(u_\text{eff})$ and let $u_\text{eff}$ and $\varphi$ be optimal in \prettyref{eq:general-duality-result}. Recall $[\nabla\varphi]$ denotes the jump in $\nabla\varphi$ across $\partial\Omega$ in the direction of $\hat{\nu}$.

The purpose of this section is to study \prettyref{eq:wrinkling-PDEs} as a boundary value problem for $\mu$ and, in particular, to establish the results from \prettyref{subsec:method-of-stable-lines} regarding the general formulation of our method of stable lines. Let us briefly outline what we achieve. We begin in \prettyref{subsec:Lines-of-wrinkling-classification} by defining a partition of the shell according to the structure of $\varphi$. 
Included in this partition are the ``stable lines'' and the ``ordered'' set $O$ they fill out. We show that
\[
\mu = \lambda \hat{\eta} \otimes \hat{\eta}\quad \text{on }O,\quad\text{where }\lambda\geq 0 \text{ and }\hat{\eta} \in R\left((\nabla\nabla\varphi)_\text{a.c.}\right). 
\] 
The unit vector field $\hat{\eta}$ arises as a suitable choice of normal direction to the stable lines. 
In \prettyref{subsec:wrinkles-as-chars}, we justify our assertion that the stable lines are characteristic curves for the PDE
\[
-\frac{1}{2}\text{curl}\text{curl}(\hat{\eta}\otimes\hat{\eta}\lambda) = \det \nabla \nabla p\quad\text{on }O
\] 
implied by \prettyref{eq:wrinkling-PDEs}. We do so by producing ODEs for the absolutely continuous and singular parts of $\lambda$, which hold on (a.e.) stable line. 
 Finally, in \prettyref{subsec:three-solution-formulas} we show how to derive appropriate boundary data depending on the layout of the stable lines, and how to apply them to conclude (partial) uniqueness,  regularity, and explicit solution formulas for $\lambda$ and $\mu$. The reader wishing to see concrete examples should go forward to \prettyref{sec:examples}, keeping in mind that we make repeated use of \prettyref{cor:non-intersecting-chars}-\prettyref{cor:chars-meet-at-point} there. 
 
A word on assumptions is required:  
throughout this section, we require that $\Omega$ is bounded and Lipschitz and that $p\in W^{2,2}(\Omega)$. We take $\mu\in \mathcal{M}_+(\Omega;\text{Sym}_2)$, let $\varphi:\mathbb{R}^2\to\mathbb{R}$ be convex, and assume that they satisfy \prettyref{eq:wrinkling-PDEs} (it is not necessary for our present purposes to assume that they are optimal). Importantly, we must also assume that
\begin{equation}\label{eq:special-regularity}
\text{there exists a non-empty open subset of } \Omega \text{ on which }\varphi \in W^{2,2}.
\end{equation}
As was addressed briefly in \prettyref{subsec:openquestions}, this last assumption will allow us to apply the theory of $W^{2,2}$ developable surfaces from \cite{hornung2011approximation,hornung2011fine,pakzad2004sobolev}. 
Various further assumptions on $\varphi$ will be introduced as needed below.

\subsection{Stable lines\label{subsec:Lines-of-wrinkling-classification}}

Our first task is to explain how the structure of $\varphi$ constrains
that of $\mu$ solving the system \prettyref{eq:wrinkling-PDEs}.  
Guided by the second equation there, we define a partition of $\Omega$ by writing
\begin{equation}
\Omega= \Sigma \cup F\cup O\cup U \label{eq:partition}
\end{equation}
where the sets $\Sigma$, $F$, $O$, and $U$ are as follows: 
\begin{itemize}
\item the \emph{singular} set $\Sigma$ is the smallest closed subset of $\Omega$ such that $\varphi \in W^{2,2}_\text{loc}(\Omega\backslash\Sigma)$; 
\item the \emph{flattened} set $F$ is the largest open subset of $\Omega\backslash\Sigma$
on which both of the eigenvalues of $(\nabla\nabla\varphi)_\text{a.c.}$ are locally uniformly positive a.e.; 
\item the \emph{ordered} set $O$ is the largest open subset of $\Omega\backslash\Sigma$
on which one of the eigenvalues of $(\nabla\nabla\varphi)_\text{a.c.}$ is zero a.e.\ and the other eigenvalue is locally uniformly positive a.e.; 
\item the \emph{unconstrained} set $U$ is the complement of $\Sigma\cup F\cup O$
with respect to $\Omega$. 
\end{itemize}
To be clear, we say that a function
$\zeta$ is \emph{locally uniformly positive a.e.} on a (measurable) set $A$
if for all $x\in A$ there exists $c>0$ and a relatively open neighborhood $V\subset A$ of $x$ on which $\zeta \geq c$ a.e. Note $\Sigma \neq \Omega$ due to \prettyref{eq:special-regularity}.

Next, we explain what we mean by ``stable lines''. Recall from \prettyref{subsec:method-of-stable-lines} that a curve parallel to $N(\nabla\nabla \varphi)$ throughout $O$ was (preliminarily) called a stable line.
This definition is no longer suitable in the present, measure-theoretic context. It generalizes naturally as follows:  
henceforth, we refer to a curve belonging to $O$ as a \emph{stable line} of $\varphi$ if it is a maximally contained open line segment on which $\varphi$
is affine. 
\begin{lemma} \label{lem:stablelines} Every $x\in O$ belongs to a unique stable line $\ell_{x}$ which satisfies $\partial\ell_x \subset \partial O$. 
The map $x\mapsto\ell_{x}$
is locally Lipschitz from $O$ to the projective space $\mathbb{P}^{1}$.
In particular, there exists $\hat{\eta}\in\emph{Lip}_{\emph{loc}}(O;S^{1})$
that is constant along and perpendicular to the stable lines, i.e.,
\begin{equation}
\ell_{x}=\ell_{y}\implies\hat{\eta}(x)=\hat{\eta}(y)\quad\forall\,x,y\in O\quad\text{and}\quad\hat{\eta}(x)\perp\ell_{x}\quad\forall\,x\in O.\label{eq:properties-stable-lines}
\end{equation}
\end{lemma}
\begin{remark}
In general, $\hat{\eta}$ can fail to be Lipschitz on $O$. This can happen for various reasons, such as when distinct stable lines share a common boundary point. Examples of this appear throughout Panel (b) of \prettyref{fig:library_of_patterns}. 
\end{remark}
\begin{proof} We require some facts about developable surfaces. In 
the smooth setting, a \emph{developable surface} is one whose
Gaussian curvature vanishes identically. Any such surface can be decomposed into
two disjoint parts: a \emph{ruled} part consisting of disjoint open line segments that pass between boundary points ---  the surface's \emph{ruling lines} --- and a locally planar part. 
In \cite{pakzad2004sobolev}, this decomposition is shown to hold for $W^{2,2}$ developable surfaces, i.e., ones whose Gaussian curvature vanishes a.e. By a covering argument, it holds for $W_{\text{loc}}^{2,2}$ developable surfaces as well.

These facts allow to characterize the graph of $\varphi$ over  $O$. By definition, this  is the image of the mapping $O\to\mathbb{R}^{3}$, $(x_{1},x_{2})\mapsto(x_{1},x_{2},\varphi(x))$.
Looking back to the definition of $O$ immediately after \prettyref{eq:partition}, we see that
\[
\varphi\in W_{\text{loc}}^{2,2}(O)\quad\text{and}\quad\text{rank}\,\nabla\nabla\varphi=1\quad\text{a.e. on }O.
\]
Therefore, $\varphi$ describes a $W^{2,2}_\text{loc}$ developable
surface over $O$, which consists entirely of ruling lines. Stable
lines are now easily produced: the stable line $\ell_{x}$ through $x$
arises from the  projection $(x_{1},x_{2},x_{3})\mapsto(x_{1},x_{2})$
of the unique ruling line through $(x_{1},x_{2},\varphi(x))$ to the
plane. Indeed, $\varphi$ is affine along $\ell_{x}$, and it is maximally extended in $O$. 

It remains to choose the normal direction $\hat{\eta}$. Of
course, we can take it to satisfy \prettyref{eq:properties-stable-lines}.  
That it can be chosen to be locally Lipschitz follows from the known fact \cite{hornung2011fine,kirchheim2001geometry}  that the map $x\mapsto\ell_{x}$ is locally Lipschitz from $O$ to the projective space $\mathbb{P}^1\simeq S^{1}/\left\{ \hat{n}\sim-\hat{n}\right\}$. Let us explain. Without
loss of generality, we can take $O$ to be connected after passing to
its components. In order that $\ell_{x}\perp\ell_{y}$ it must be
that $|x-y|>d_{\partial O}(x)\vee d_{\partial O}(y)$ as stable
lines never intersect. So, once we decide that 
\[
\hat{\eta}(x)\cdot\hat{\eta}(y)>0\quad\text{when}\quad |x-y|\leq d_{\partial O}(x)\vee d_{\partial O}(y)
\]
there will remain exactly two choices for $\hat{\eta}:O\to S^{1}$ (in general, the number of choices depends on the number of connected components). Fixing
$\hat{\eta}(x)$ at some $x\in O$ determines it throughout. It now follows that 
\[
\left|\hat{\eta}(x)-\hat{\eta}(y)\right|\lesssim\frac{\left|x-y\right|}{d_{\partial O}(x)\vee d_{\partial O}(y)}\quad\forall\,x,y\in O
\]
from a worst case analysis of how stable lines may meet at $\partial O$.
\qed\end{proof}
Having defined the stable lines of $\varphi$, we can now use them to characterize the structure of $\mu$.
\begin{lemma} \label{lem:structure} Let $\Omega$ be partitioned as in \prettyref{eq:partition}. 
Any solution $\mu\in\mathcal{M}_{+}(\Omega;\emph{Sym}_{2})$ of \prettyref{eq:wrinkling-PDEs}
must satisfy
\begin{equation}
\mu=0\quad\text{on }F\quad\text{and}\quad\mu=\lambda\hat{\eta}\otimes\hat{\eta}\quad\text{on }O\label{eq:rank-1-defect}
\end{equation}
for some $\lambda\in\mathcal{M}_{+}(O)$ and  $\hat{\eta}\in\emph{Lip}_{\emph{loc}}(O;S^{1})$ satisfying \prettyref{eq:properties-stable-lines}.
\end{lemma}
\begin{proof} We combine the previous result with the complementary slackness part of \prettyref{eq:wrinkling-PDEs}. In particular, we make use of the first part of \prettyref{eq:complementary_slackness_weaksense}, which implies here that 
\begin{equation}
0=\lim_{\delta\to0}\,\int_{F}\left|\left\langle \mu_{\delta},(\nabla^{\perp}\nabla^{\perp}\varphi)_\text{a.c.}\right\rangle \right|\,dx=\lim_{\delta\to0}\,\int_{O}\left|\left\langle \mu_{\delta},(\nabla^{\perp}\nabla^{\perp}\varphi)_\text{a.c.}\right\rangle \right|\,dx\label{eq:useful-cvg-to-zero}.
\end{equation}
Note it follows from their definition (as in \prettyref{lem:mollification_of_measures}) that  $\mu_{\delta}\,dx\overset{*}{\rightharpoonup}\mu$
weakly-$*$ on $F$ and $O$.

Consider the flattened set $F$. We claim that every $x\in F$ is contained in a neighborhood where $\mu=0$. Indeed, by its definition we can always find a relatively open set $V\subset F$ such that $x\in V$ and
\[
(\nabla\nabla\varphi)_\text{a.c.}\gtrsim Id \quad\text{a.e. on }V.
\]
Upon passing to the limit in the first part of \prettyref{eq:useful-cvg-to-zero},
we deduce that
\[
0=\lim_{\delta\to0}\,\int_{V}\left|\left\langle Id,\mu_{\delta}\right\rangle \right|\,dx=|\mu|(V).
\]
Every compact subset of $F$ is covered by finitely many
such neighborhoods. Hence, $\mu=0$ on $F$.

The ordered set $O$ requires a bit more care. Recall $\varphi$ is affine along its stable lines, which run perpendicularly to the vector field $\hat{\eta}\in\text{Lip}_\text{loc}(O;S^1)$ from \prettyref{lem:stablelines}. 
So, there exists $\zeta\in L^1(O;(0,\infty))\cap L_{\text{loc}}^{2}(O)$ such that 
\[
(\nabla\nabla\varphi)_\text{a.c.}=\zeta\hat{\eta}\otimes\hat{\eta}\quad\text{a.e.\ on }O.
\]
Actually, the definition of $O$ gives a bit more: for each $x\in O$ there exists a relatively open neighborhood $V\subset O$ of $x$ on which $\zeta\gtrsim1$ a.e. 
Using this, we can pass to the limit in the second part of \prettyref{eq:useful-cvg-to-zero} to find that
\[
0=\lim_{\delta\to0}\,\int_{V}\left|\left\langle \hat{\eta}^{\perp}\otimes\hat{\eta}^{\perp},\mu_{\delta}\right\rangle \right|\,dx=\left|\left\langle \hat{\eta}^{\perp}\otimes\hat{\eta}^{\perp},\mu\right\rangle \right|(V).
\]
Again as the choice of $x\in O$ was arbitrary, it follows that $\langle \hat{\eta}^{\perp}\otimes\hat{\eta}^{\perp},\mu\rangle =0$ on $O$.
\qed\end{proof} 
The possibilities for $\mu$ at $\partial O$ and $\partial F$ are less
clear, as the relevant eigenvalue(s) of $\nabla\nabla\varphi$ may
degenerate there. We leave the detailed study of this to future work, and turn to describe the role that stable lines play for solving \prettyref{eq:wrinkling-PDEs}.

\subsection{The method of characteristics\label{subsec:wrinkles-as-chars}}

The previous section described the relation between the structure of $\varphi$ and that of $\mu$ solving \prettyref{eq:wrinkling-PDEs}. 
Following  \prettyref{lem:structure}, we continue to let $\lambda\in\mathcal{M}_+(O)$ and $\hat{\eta}\in\text{Lip}_\text{loc}(O;S^1)$ satisfy 
\[
\mu = \lambda \hat{\eta}\otimes\hat{\eta}\quad\text{on }O
\]
where $\hat{\eta}$ is constant along and perpendicular to the stable lines $\{\ell_x\}$ of $\varphi$. By the first
equation in \prettyref{eq:wrinkling-PDEs},
\begin{equation}
-\frac{1}{2}\text{curl}\text{curl}\left(\hat{\eta}\otimes\hat{\eta}\,\lambda\right)=\det\nabla\nabla p\quad\text{on }O\label{eq:PDE_lambda}
\end{equation}
in the sense of distributions. We now claim that \prettyref{eq:PDE_lambda}
can be solved using the method of characteristics with stable
lines as characteristic curves. It is not difficult to understand why this ought to be the case. Denote the first and second
directional derivatives along the stable lines by 
\begin{equation}
\partial_{\hat{\eta}^{\perp}}=\hat{\eta}^{\perp}\cdot\nabla\quad\text{and}\quad\partial_{\hat{\eta}^{\perp}}^{2}=\left\langle \hat{\eta}^{\perp}\otimes\hat{\eta}^{\perp},\nabla\nabla\right\rangle .\label{eq:full_directional}
\end{equation}
Pretending for the moment that $\lambda$ and $\hat{\eta}$ are smooth (rather than belonging to $\mathcal{M}_+$ and $\text{Lip}_\text{loc}$), we apply the product rule along with the statement that $\partial_{\hat{\eta}^\perp}\hat{\eta}=0$ to write that 
\begin{equation}
\text{curl}\text{curl}\left(\hat{\eta}\otimes\hat{\eta}\lambda\right) =\partial_{\hat{\eta}^{\perp}}^{2}\lambda+\frac{2}{\varrho}\partial_{\hat{\eta}^{\perp}}\varrho\partial_{\hat{\eta}^{\perp}}\lambda+\frac{1}{\varrho}\partial_{\hat{\eta}^{\perp}}^{2}\varrho\lambda=\frac{1}{\varrho}\partial_{\hat{\eta}^{\perp}}^{2}\left(\varrho\lambda\right)\quad\text{where}\quad \partial_{\hat{\eta}^{\perp}}\varrho=\text{div}\,\hat{\eta}^{\perp}\varrho.\label{eq:rho_eqn}
\end{equation} 
Thus, the PDE \prettyref{eq:PDE_lambda} can be rewritten (at least informally, at first) as the family of ODEs 
\[
-\frac{1}{2\varrho}\partial_{\hat{\eta}^{\perp}}^{2}\left(\varrho\lambda\right)=\det\nabla\nabla p\quad\text{along the stable lines}.
\]
\prettyref{lem:disintegration} provides a rigorous version of this in the original, measure-theoretic setting of \prettyref{eq:wrinkling-PDEs}. In brief: whereas this ODE turns out to govern the absolutely continuous part of $\lambda$ along $\mathcal{H}^1$-a.e.\ stable line, its singular part is instead affine along a complementary set of stable lines.

The next few paragraphs fix the notation used in the remainder. \label{par:notation} 
First, due to the topological difficulties inherent in parameterizing the stable lines $\{\ell_x\}$ --- see \cite{hornung2011fine} for a detailed account of the related problem of parameterizing ruling lines --- we find it convenient to reduce to certain well-prepared regions of the form 
\begin{equation}\label{eq:setup-for-stable-lines}
\begin{gathered}
V=\cup_{s\in\Gamma}\ell_{s}\quad \text{where }\Gamma\subset V\text{ is a smooth curve such that}\\
\ell_s = \ell_{s'} \implies s = s' \quad\forall\, s,s'\in \Gamma\quad\text{and}\quad T_s\Gamma \not\parallel \ell_s\quad\forall\,s\in \Gamma.
\end{gathered}
\end{equation}
Recall by a \emph{smooth curve} we mean a diffeomorphic copy of an open interval $I\subset \mathbb{R}$, i.e., its image under a smooth and one-to-one map. 
(Later on, we allow $\Gamma$ to denote other, more general index sets. We note when this occurs.)  The conditions on the second line  require that the curve $\Gamma$ meets each stable line it indexes transversely and exactly once.   
Note it follows from \prettyref{lem:stablelines} that every $x\in O$ admits a neighborhood $V$ of this form. Indeed, we may simply choose $\Gamma$ to pass through $x$ and to remain approximately parallel to $\hat{\eta}$ along its extent.

Next, we introduce the notation involved with the technique of disintegration of measure. The basic facts are as follows (see, e.g.,  \cite{graf1989classification}). Given one of the regions $V=\cup_{s\in\Gamma}\ell_{s}$ from \prettyref{eq:setup-for-stable-lines}, we say that 
\begin{equation}\label{eq:pi-map}
\pi:V\to\Gamma\text{ sends }x\in V\text{ to the unique }s\in \ell_x\cap\Gamma.
\end{equation} 
Note for each $s\in \Gamma$, the fiber $\pi^{-1}(\{s\}) = \ell_s$. Given $\lambda\in\mathcal{M}(V)$ and $\vartheta \in \mathcal{M}(\Gamma)$ with $ \pi_{\#}\lambda \ll \vartheta$, there exists
a $\vartheta$-a.e.\ uniquely determined (Borel) family $\{\lambda_{s}\}_{s\in\Gamma}\subset\mathcal{M}(V)$ such that
\[
\text{supp}\,\lambda_{s}\subset\ell_{s}\quad\forall\,s\in\Gamma \quad\text{and}\quad
\int_{V}\psi\,d\lambda=\int_{\Gamma}\left[\int_{\ell_{s}}\psi\,d\lambda_{s}\right]\,d\vartheta(s)\quad\forall\,\psi\in L^{1}(V,\lambda).
\]
Here, $\pi_{\#}$ is the pushforward map through $\pi$. Thus, $\lambda$ \emph{disintegrates} into its parts $\{\lambda_{s}\}_{s\in\Gamma}$ with respect to $\pi$ and $\vartheta$, a situation we indicate by writing
\[
\lambda=\int_{\Gamma}\lambda_{s}\,d\vartheta(s).
\]
A useful example to keep in mind is the formula for the two-dimensional Lebesgue measure 
\begin{equation}
\mathcal{L}^{2}=\int_{\Gamma}\varrho\mathcal{H}^{1}\lfloor\ell_{s}\,d\mathcal{H}^{1}(s)\quad\text{on }V.\label{eq:Leb_disintegration}
\end{equation}
Note this defines the change of measure factor $\varrho:V\to(0,\infty)$ anticipated in \prettyref{eq:rho_eqn}.

Finally, we define the Sobolev spaces
$W^{k,r}(\ell_{s})$ for $k\in\mathbb{N}$ and $r\in[1,\infty]$. 
For each $\ell_s$, we say that $f\in L^{r}(\ell_{s},\mathcal{H}^{1})$ belongs to $W^{1,r}(\ell_{s})$ if there exists $g\in L^{r}(\ell_{s},\mathcal{H}^{1})$
such that 
\[
\int_{\ell_{s}}f\partial_{\hat{\eta}^{\perp}}\chi\,d\mathcal{H}^{1}=-\int_{\ell_{s}}g\chi\,d\mathcal{H}^{1}\quad\forall\,\chi\in C_{c}^{\infty}(\ell_{s}).
\]
In such a case, we write that 
\[
\partial_{\hat{\eta}^{\perp}(s)}f=g\quad\text{on }\ell_{s}
\]
and call $g$ the \emph{weak directional derivative} of $f$ in the direction of $\hat{\eta}^\perp(s)$. Thus, $\partial_{\hat{\eta}^{\perp}(s)}:W^{1,r}(\ell_{s})\to L^{r}(\ell_{s},\mathcal{H}^{1})$. 
Similarly, $W^{k,r}(\ell_{s})$
consists of all $f\in L^{r}(\ell_{s},\mathcal{H}^{1})$ whose
weak directional derivatives $\partial_{\hat{\eta}^{\perp}(s)}^{j}f$
of orders $j=1,\dots,k$ belong to $L^{r}(\ell_{s},\mathcal{H}^{1})$.
Of course, if $f$ is smooth nearby $\ell_{s}$, these derivatives
can be computed using \prettyref{eq:full_directional} along with other, analogous formulas at higher order.
Given $f\in W^{k,r}(\ell_{s})$ we define its \emph{trace} $f|_{\partial\ell_{s}}$ as usual, by continuous extension of the restriction map. As each $\ell_{s}$ is one-dimensional, $\cdot|_{\partial\ell_{s}}:W^{k,r}(\ell_{s})\to L^{\infty}(\partial\ell_{s},\mathcal{H}^{0})$.

We are ready to make precise our claim that stable lines are characteristic curves for the PDE \prettyref{eq:PDE_lambda}.

\begin{lemma} \label{lem:disintegration} 
Let $\lambda\in \mathcal{M}_+(O)$ solve \prettyref{eq:PDE_lambda}, and let $V=\cup_{s\in\Gamma}\ell_s$ and $\pi:V\to\Gamma$ be as in \prettyref{eq:setup-for-stable-lines} and   \prettyref{eq:pi-map}. Then there exist $\lambda_{\emph{a.c.}},\lambda_{\emph{sing}}:V\to[0,\infty)$ such that
\[
\lambda=\lambda_{\emph{a.c.}}\,dx+\int_{\Gamma}\lambda_{\emph{sing}}\mathcal{H}^{1}\lfloor\ell_{s}\,d\vartheta(s)\quad\text{on }V
\]
where $\vartheta$ is the singular part of $\pi_{\#}\lambda$ with
respect to $\mathcal{H}^{1}$. 
The function $\varrho\lambda_{\emph{a.c.}}$ belongs to $W^{2,1}(\ell_{s})$ and satisfies 
\begin{equation}
-\frac{1}{2\varrho}\partial_{\hat{\eta}^{\perp}(s)}^{2}(\varrho\lambda_{\emph{a.c.}})=\det\nabla\nabla p\quad\emph{on }\ell_{s}\label{eq:ODE_ac}
\end{equation}
upon restriction to $\mathcal{H}^{1}$-a.e.\ $\ell_{s}$. Likewise, the function $\lambda_{\emph{sing}}$ belongs to $W^{2,\infty}(\ell_{s})$ and
satisfies 
\begin{equation}
\partial_{\hat{\eta}^{\perp}(s)}^{2}\lambda_{\emph{sing}}=0\quad\emph{on }\ell_{s}\label{eq:ODE_sing}
\end{equation}
upon restriction to $\vartheta$-a.e.\ $\ell_{s}$.
\end{lemma} 

\begin{remark} Of course, if $p\in W^{2,2r}$ so that $\det\nabla\nabla p\in L^{r}$,
the conclusion is that $\varrho\lambda_{\text{a.c.}}\in W^{2,r}$
on $\mathcal{H}^{1}$-a.e.\ $\ell_{s}$. \end{remark} 
\begin{remark}\label{rem:unique-continuation-stable-lines} It is straightforward to check that the following unique continuation-type result holds: if $V_1 = \cup_{s\in \Gamma_1} \ell_s$ and $V_2 = \cup_{s\in \Gamma_2} \ell_s$ satisfy $\Gamma_1 \cap V_2 = \Gamma_2 \cap V_1$, the functions provided above must have $\lambda^{1}_{\text{a.c.}} = \lambda^{2}_{\text{a.c.}}$ Lebesgue a.e.\ on $V_1\cap V_2$, and $\lambda^{1}_\text{sing} = \lambda^{2}_\text{sing}$ upon restriction to $\vartheta$-a.e.\ $\ell_s$ in $V_1\cap V_2$. We use this later on in  the proof of \prettyref{cor:chars-meet-at-point}.
\end{remark}
\begin{proof} We start by asserting the existence of  $\{\lambda_{s}^{\text{a.c.}}\}_{s\in\Gamma},\{\lambda_{s}^{\text{sing}}\}_{s\in\Gamma}\subset\mathcal{M}_{+}(V)$
such that
\begin{equation}
\text{supp}\,\lambda_{s}^{\text{a.c.}},\text{supp}\,\lambda_{s}^{\text{sing}}\subset\ell_{s}\quad\forall\,s\in\Gamma\quad\text{and}\quad\lambda=\int_{\Gamma}\lambda_{s}^{\text{a.c.}}\,d\mathcal{H}^{1}(s)+\int_{\Gamma}\lambda_{s}^{\text{sing}}\,d\vartheta(s) \quad\text{on }V.\label{eq:disintegrated_twoparts}
\end{equation}
Indeed, by disintegration of measure, we can find a family 
$\{\lambda_{s}\}_{s\in\Gamma}\subset\mathcal{M}_{+}(V)$ such that
\[
\text{supp}\,\lambda_{s}\subset\ell_{s}\quad\forall\,s\in\Gamma\quad\text{and}\quad\lambda=\int_{\Gamma}\lambda_{s}\,d\pi_{\#}\lambda\quad\text{on }V.
\]
Note the Lebesgue decomposition 
\[
\pi_{\#}\lambda=\frac{d\pi_{\#}\lambda}{d\mathcal{H}^{1}}\,\mathcal{H}^{1}+\vartheta\quad\text{with}\quad \mathcal{H}^{1}\perp \vartheta.
\]
Taking
\[
\lambda_{s}^{\text{a.c.}}=\frac{d\pi_{\#}\lambda}{d\mathcal{H}^{1}}(s)\lambda_{s}\quad\text{for }\mathcal{H}^{1}\text{-a.e. }s\quad\text{and}\quad\lambda_{s}^{\text{sing}}=\lambda_{s}\quad\text{for }\vartheta\text{-a.e. }s
\]
we arrive at  \prettyref{eq:disintegrated_twoparts}.

Having disintegrated $\lambda$ into its parts $\{\lambda_{s}^{\text{a.c.}}\}_{s\in\Gamma}$, $\{\lambda_{s}^{\text{sing}}\}_{s\in\Gamma}$ we proceed to establish
the desired ODEs. We will make use
of the PDE \prettyref{eq:PDE_lambda} or, more precisely, its distributional
version 
\begin{equation}
\int_{V}-\frac{1}{2}\partial_{\hat{\eta}^{\perp}}^{2}\psi\,d\lambda=\int_{V}\psi\det\nabla\nabla p\,dx\quad\forall\,\psi\in C_{c}^{\infty}(V).\label{eq:weakcompatibility}
\end{equation}
The argument splits into two steps. The first step is to prove that  \prettyref{eq:weakcompatibility}
holds not only for the test functions above, but also for ones of the form 
\begin{equation}
\psi=\chi\Psi\circ\pi\quad\text{where}\quad\chi\in C_{c}^{\infty}(V)\text{ and }\Psi\in C_{c}^{\infty}(\Gamma)\label{eq:test_seperated}
\end{equation}
and where in place of $\partial^2_{\hat{\eta}^\perp}\psi$ we write $\partial^2_{\hat{\eta}^\perp}\chi \Psi \circ \pi$. 
To see this, fix $\chi \in C_c^\infty(V)$ and let $W\subset V$ be an open and compactly contained neighborhood of its support. 
We claim there exists a sequence $\{\pi_{k}\}_{k\in\mathbb{N}}\subset C^{\infty}(W;\Gamma)$ of smooth approximations to $\pi$ such that
\begin{equation}
\pi_{k}\to\pi,\quad\partial_{\hat{\eta}^{\perp}}\pi_{k}\to0,\quad\text{and}\quad\partial_{\hat{\eta}^{\perp}}^{2}\pi_{k}\to0\quad\text{uniformly on }W\label{eq:convergences_pi}
\end{equation}
as $k\to\infty$. Postponing their construction to \prettyref{lem:invertibility-lemma} and \prettyref{lem:stablelines-approximation}
below, we define $\{\psi_{k}\}_{k\in\mathbb{N}}\subset C_{c}^{\infty}(V)$
by 
\begin{equation}
\psi_{k}=\chi\Psi\circ\pi_{k}\label{eq:test_seperated_approx}
\end{equation}
and note using the product rule that 
\[
\partial_{\hat{\eta}^{\perp}}^{2}\psi_{k} =\partial_{\hat{\eta}^{\perp}}^{2}\chi\Psi\circ\pi_{k}+2\partial_{\hat{\eta}^{\perp}}\chi\partial_{\hat{\eta}^{\perp}}(\Psi\circ\pi_{k})+\chi\partial_{\hat{\eta}^{\perp}}^{2}(\Psi\circ\pi_{k}).
\]
Due to \prettyref{eq:convergences_pi}, only the first term on the
righthand side survives in the limit. That is,
\[
\partial_{\hat{\eta}^{\perp}}^{2}\psi_{k}\to\partial_{\hat{\eta}^{\perp}}^{2}\chi\Psi\circ\pi\quad\text{uniformly on }W.
\]
Setting \prettyref{eq:test_seperated_approx} into \prettyref{eq:weakcompatibility}
and passing to the limit finishes the first step. See \prettyref{lem:invertibility-lemma} and \prettyref{lem:stablelines-approximation} for the construction of the required approximations $\pi_k$.

We just showed that \prettyref{eq:weakcompatibility} holds for all
test functions of the form \prettyref{eq:test_seperated}. Equivalently, we have that
\begin{multline}
\int_{\Gamma}\left[\int_{\ell_{s}}\frac{1}{2}\partial_{\hat{\eta}^{\perp}}^{2}\chi\,d\lambda_{s}^{\text{a.c.}}+\int_{\ell_{s}}\chi\varrho\det\nabla\nabla p\,d\mathcal{H}^{1}\right]\Psi(s)\,d\mathcal{H}^{1}(s)+\int_{\Gamma}\left[\int_{\ell_{s}}\frac{1}{2}\partial_{\hat{\eta}^{\perp}}^{2}\chi\,d\lambda_{s}^{\text{sing}}\right]\Psi(s)\,d\vartheta(s)=0\\
\forall\,\chi\in C_{c}^{\infty}(V),\Psi\in C_{c}^{\infty}(\Gamma)\label{eq:Fubini_disintegration}
\end{multline}
by the disintegration formulas \prettyref{eq:Leb_disintegration} and  \prettyref{eq:disintegrated_twoparts}. 
The next step is to show that the bracketed terms vanish, i.e., 
\begin{equation}
\int_{\ell_{s}}-\frac{1}{2}\partial_{\hat{\eta}^{\perp}(s)}^{2}\chi\,d\lambda_{s}^{\text{a.c.}}=\int_{\ell_{s}}\chi\varrho\det\nabla\nabla p\,d\mathcal{H}^{1}\quad\text{and}\quad\int_{\ell_{s}}\partial_{\hat{\eta}^{\perp}(s)}^{2}\chi\,d\lambda_{s}^{\text{sing}}=0\quad\forall\,\chi\in C_{c}^{\infty}(\ell_{s})\label{eq:weak1}
\end{equation}
up to $\mathcal{H}^{1}$- and $\vartheta$-negligible sets. By
an extension argument it suffices to take $\chi\in C_{c}^{\infty}(V)$.
Let $\{\chi_{k}\}_{k\in\mathbb{N}}\subset C_{c}^{\infty}(V)$
be $C^{2}$-dense. Setting $\chi_{k}$ into \prettyref{eq:Fubini_disintegration}
and recalling that $\mathcal{H}^1 \perp \vartheta$, we see that
\[
\int_{\ell_{s}}-\frac{1}{2}\partial_{\hat{\eta}^{\perp}}^{2}\chi_{k}\,d\lambda_{s}^{\text{a.c.}}=\int_{\ell_{s}}\chi_{k}\varrho\det\nabla\nabla p\,d\mathcal{H}^{1}\quad\text{for }\mathcal{H}^{1}\text{-a.e. }s\quad\text{and}\quad\int_{\ell_{s}}\partial_{\hat{\eta}^{\perp}}^{2}\chi_{k}\,d\lambda_{s}^{\text{sing}}=0\quad\text{for }\vartheta\text{-a.e. }s
\]
where the exceptional sets depend on  $k$. Intersecting
over $k$ removes this dependence and yields \prettyref{eq:weak1}. In other words, we have established the ODEs 
\begin{equation}
-\frac{1}{2}\partial_{\hat{\eta}^{\perp}(s)}^{2}\lambda^{\text{a.c.}}_s=\varrho\det\nabla\nabla p\quad\text{on }\mathcal{H}^1\text{-a.e. } \ell_{s}\quad\text{and}\quad \partial_{\hat{\eta}^{\perp}(s)}^{2}\lambda^{\text{sing}}_s=0\quad\text{on }\vartheta\text{-a.e. }\ell_{s}.
\label{eq:desiredODEs-1}
\end{equation}
These hold in the sense of distributions on the specified stable lines.

The rest of the proof is more or less straightforward.
From \prettyref{eq:desiredODEs-1} we see that
$\lambda_{s}^{\text{a.c.}},\lambda_{s}^{\text{sing}}\ll\mathcal{H}^{1}\lfloor\ell_{s}$
and that their densities satisfy the same ODEs. By hypothesis, $\det\nabla\nabla p\in L^{1}(V)$. It follows
from \prettyref{eq:Leb_disintegration} and Fubini's theorem that
$\varrho\det\nabla\nabla p \in L^{1}(\ell_{s},\mathcal{H}^{1})$ and so $\frac{d\lambda_{s}^{\text{a.c.}}}{d\mathcal{H}^{1}\lfloor\ell_{s}}\in W^{2,1}(\ell_{s})$  on $\mathcal{H}^{1}$-a.e.\ $\ell_{s}$.
Evidently, $\frac{d\lambda_{s}^{\text{sing}}}{d\mathcal{H}^{1}\lfloor\ell_{s}}\in W^{2,\infty}(\ell_{s})$ on $\vartheta$-a.e.\ $\ell_{s}$ as it is affine upon restriction to those stable lines. Setting
\[
\lambda_{\text{a.c.}}=\frac{1}{\varrho}\frac{d\lambda_{s}^{\text{a.c.}}}{d\mathcal{H}^{1}\lfloor\ell_{s}}\quad\text{on }\mathcal{H}^1\text{-a.e. } \ell_s\quad\text{and}\quad\lambda_{\text{sing}}=\frac{d\lambda_{s}^{\text{sing}}}{d\mathcal{H}^{1}\lfloor\ell_{s}}\quad\text{on }\vartheta\text{-a.e. } \ell_s
\]
and using \prettyref{eq:Leb_disintegration} and \prettyref{eq:disintegrated_twoparts} once more, we conclude that 
\[
\lambda=\int_{\Gamma}\varrho\lambda_{\text{a.c.}}\mathcal{H}^{1}\lfloor\ell_{s}\,d\mathcal{H}^{1}(s)+\int_{\Gamma}\lambda_{\text{sing}}\mathcal{H}^{1}\lfloor\ell_{s}\,d\vartheta(s)=\lambda_{\text{a.c.}}\,dx+\int_{\Gamma}\lambda_{\text{sing}}\mathcal{H}^{1}\lfloor\ell_{s}\,d\vartheta(s).
\]
The desired ODEs \prettyref{eq:ODE_ac} and \prettyref{eq:ODE_sing} follow from \prettyref{eq:desiredODEs-1}.
\qed\end{proof}
Left over from the proof above is a result ensuring
that the map $\pi:V\to\Gamma$ from \prettyref{eq:pi-map}, whose
fibers are the stable lines $\{\ell_{s}\}_{s\in\Gamma}$, can be approximated
by smooth maps $\{\pi_{k}\}_{k\in\mathbb{N}}$ with fibers converging
to the stable lines (see \prettyref{eq:convergences_pi}). Similar
results appear in the proof that smooth developable surfaces are $W^{2,2}$-dense
\cite{hornung2011approximation,hornung2011fine,pakzad2004sobolev}.
There, the authors replace the surface's ruling lines with smoothly
varying ones; here, per \prettyref{lem:stablelines}, the stable lines
of $\varphi$ are the planar projection of the ruling lines of
its graph (over $O$). The main difference is the choice of topology
--- we need that $\pi_{k}$ and certain of its derivatives converge
uniformly, rather than only a.e. Nevertheless, the argument from the
references can easily be adapted to produce the desired result. We follow
\cite{hornung2011approximation,hornung2011fine}.

Let $V=\cup_{s\in\Gamma}\ell_{s}$ be as in \prettyref{eq:setup-for-stable-lines}.
The first step is to define coordinates adapted to the stable
lines. Recall we took $\Gamma$ to be a diffeomorphic copy of an open interval $I\subset \mathbb{R}$. Let $\gamma:I\to\Gamma$ be a smooth map such that 
\begin{equation}
V=\cup_{s\in I}\ell_{\gamma(s)},\quad s\mapsto\ell_{\gamma(s)}\text{ is one-to-one},\quad\gamma'\cdot\hat{\eta}\circ\gamma>0,\quad\text{and}\quad|\gamma'|=1.\label{eq:uniform-transverslity-param}
\end{equation}
Note the slight redundancy in the usage of $s$. 
In a minor modification of \cite{hornung2011approximation,hornung2011fine},
we define $\Phi_{\hat{n}}:I\times\mathbb{R}\to\mathbb{R}^{2}$ by 
\begin{equation}
\Phi_{\hat{n}}(s,t)=\gamma(s)+t\hat{n}^{\perp}(s)\label{eq:parameterization-definition}
\end{equation}
for $\hat{n}:I\to S^{1}$. Unlike the references, we do not require
that $\gamma'\parallel\hat{n}$, although we will eventually prevent
them from being perpendicular. If $\hat{n}$ is differentiable then
so is $\Phi_{\hat{n}}$, in which case 
\begin{equation}
\det\nabla\Phi_{\hat{n}}=\gamma'\cdot\hat{n}-t\kappa_{\hat{n}}\quad\text{where}\quad\kappa_{\hat{n}}=\hat{n}'\cdot\hat{n}^{\perp}.\label{eq:determinant-formula}
\end{equation}
Taking $\hat{n}=\hat{\eta}\circ\gamma$ leads to the desired coordinates. We refer to $\Phi_{\hat{\eta}}$ and $\kappa_{\hat{\eta}}$
in place of $\Phi_{\hat{\eta}\circ\gamma}$ and $\kappa_{\hat{\eta}\circ\gamma}$.
Since $\hat{\eta}$ is locally Lipschitz, $\Phi_{\hat{\eta}}$ is
as well. We claim that it admits a locally Lipschitz inverse on $V$. That it is invertible there is a clear consequence of the disjointness of the stable lines.
Note that 
\begin{equation}
\det\nabla\Phi_{\hat{\eta}}=\varrho\circ\Phi_{\hat{\eta}}>0\quad\text{a.e.\ on }\Phi_{\hat{\eta}}^{-1}(V)\label{eq:det-positive-rho}
\end{equation}
where $\varrho:V\to(0,\infty)$ is as in \prettyref{eq:Leb_disintegration}.
That $|\det\nabla\Phi_{\hat{\eta}}|=\varrho\circ\Phi_{\hat{\eta}}$
follows from the area formula for Lipschitz maps (see, e.g., \cite{maggi2012sets}). Its positivity is due to the given orientation in \prettyref{eq:uniform-transverslity-param}. Continuing, we define $t_{V}^{\pm}:I\to\mathbb{R}$
such that 
\[
t_{V}^{-}<0<t_{V}^{+}\quad\text{and}\quad\ell_{\gamma(s)}=\Phi_{\hat{\eta}}\left(\{s\}\times(t_{V}^{-}(s),t_{V}^{+}(s))\right)\quad\forall\,s\in I.
\]
Combining \prettyref{eq:determinant-formula} and \prettyref{eq:det-positive-rho}
shows that 
\begin{equation}
\frac{1}{t_{V}^{-}(s)}\leq\frac{\kappa_{\hat{\eta}}(s)}{\gamma'\cdot\hat{\eta}\circ\gamma(s)}\leq\frac{1}{t_{V}^{+}(s)}\quad\text{for a.e. }s\in I.\label{eq:good-inequalities-baseline}
\end{equation}
Hence, $\Phi_{\hat{\eta}}^{-1}\in\text{Lip}_{\text{loc}}(V;I\times\mathbb{R})$
by the Lipschitz inverse function theorem \cite{clarke1976inverse}.
For future reference, note that the functions $\pm t_{V}^{\pm}$ are
lower semi-continuous as $V$ is open; they are also bounded by its
diameter.

All this being said, we now rewrite the map $\pi:V\to\Gamma$ from
\prettyref{eq:pi-map} as 
\begin{equation}
\pi=\gamma\circ(\Phi_{\hat{\eta}}^{-1})_{1}\quad\text{where}\quad(s,t)_{1}=s.\label{eq:inverse-formula-for-pi}
\end{equation}
The plan is clear: look for a way of smoothing $\hat{\eta}$ such
that the associated maps remain invertible, at least on a given portion
of $V$. Note we avoid $\partial V$ as we do not make any assumptions
on its regularity, or on the behavior of the stable lines there (see
\cite{hornung2011approximation,hornung2011fine} for more on this
point).
\begin{lemma}
\label{lem:invertibility-lemma} Let $t^{\pm}\in C_{c}(I)$ and let
$J\subset I$ be an open interval such that 
\begin{equation}
t_{V}^{-}<t^{-}\leq0\leq t^{+}<t_{V}^{+}\quad\text{and}\quad\sup_{J}\,t^{-}<0<\inf_{J}\,t^{+}.\label{eq:invertibility-setup-times}
\end{equation}
Define the open sets 
\[
M_{t^{\pm},J}=\cup_{s\in J}\{s\}\times(t^{-}(s),t^{+}(s))\quad\text{and}\quad V_{t^{\pm},J}=\cup_{s\in J}\Phi_{\hat{\eta}}\left(\{s\}\times(t^{-}(s),t^{+}(s))\right)
\]
and let $W\subset V_{t^{\pm},J}$ be open and compactly contained.
For all $\epsilon>0$ there exists a $\delta>0$ such that if $\hat{n}:I\to S^{1}$
is Lipschitz on $J$ and satisfies 
\begin{equation}
||\hat{n}-\hat{\eta}\circ\gamma||_{L^{\infty}\left(J\right)}<\delta\quad\text{and}\quad\frac{1}{t^{-}(s)}+\epsilon\leq\frac{\kappa_{\hat{n}}(s)}{\gamma'\cdot\hat{n}(s)}\leq\frac{1}{t^{+}(s)}-\epsilon\quad\text{for a.e. }s\in J,\label{eq:invertibility-primary-assumptions}
\end{equation}
then $\Phi_{\hat{n}}$ admits an inverse on $W$ satisfying $\Phi_{\hat{n}}^{-1}\in\emph{Lip}(W;M_{t^{\pm},J})$
as well as the estimates
\[
||\Phi_{\hat{n}}^{-1}-\Phi_{\hat{\eta}}^{-1}||_{L^{\infty}(W)}\lesssim_{\Gamma,\hat{\eta},t^{\pm},J}||\hat{n}-\hat{\eta}\circ\gamma||_{L^{\infty}\left(J\right)}\quad\text{and}\quad||\nabla\Phi_{\hat{n}}^{-1}||_{L^{\infty}(W)}\lesssim_{\Gamma,\hat{\eta},t^{\pm},J}\frac{1}{\epsilon}.
\]
\end{lemma}

\begin{proof} That $\Phi_{\hat{n}}$ is onto $W$ can be checked
using the homotopy invariance of degree (see, e.g., \cite{fonseca1995degree}).
By the definitions, $\Phi_{\hat{\eta}}$ is a homeomorphism
between $M_{t^{\pm},J}$ and $V_{t^{\pm},J}$. Define the continuous
homotopy $[0,1]\to C(\overline{M_{t^{\pm},J}})$, $\theta\mapsto\Phi_{\theta}=\Phi_{(1-\theta)\hat{n}+\theta\hat{\eta}}$
from $\Phi_{0}=\Phi_{\hat{n}}$ to $\Phi_{1}=\Phi_{\hat{\eta}}$.
Note that 
\[
|\Phi_{\theta}(s,t)-\Phi_{1}(s,t)|\leq|\hat{n}(s)-\hat{\eta}\circ\gamma(s)||t|.
\]
Since by hypothesis $\overline{W}\subset V_{t^{\pm},J}$, there exists
$\delta_{0}>0$ such that $W\cap\Phi_{\theta}(\partial M_{t^{\pm},J})=\emptyset$
whenever $||\hat{n}-\hat{\eta}\circ\gamma||_{L^{\infty}(J)}<\delta_{0}$
and for all $\theta$. Applying \cite[Theorem 2.3]{fonseca1995degree}
proves that $d(\Phi_{\hat{n}},M_{t^{\pm},J},p)=d(\Phi_{\hat{\eta}},M_{t^{\pm},J},p)=1$
for all $p\in W$. Hence, 
\[
||\hat{n}-\hat{\eta}\circ\gamma||_{L^{\infty}(J)}<\delta_{0}\quad\implies\quad W\subset\Phi_{\hat{n}}(M_{t^{\pm},J})
\]
as in the first part of the claim.

Next, we show that $\Phi_{\hat{n}}$ can be made one-to-one on $M_{t^{\pm},J}$.
This part of the proof is modeled more or less directly after \cite[Section 5]{hornung2011fine}.
Introduce $\tau_{\hat{n}}:J\times J\backslash\left\{ (s,s'):s=s'\right\} \to\mathbb{R}\cup\{\infty\}$
such that 
\begin{equation}
\gamma(s)+\tau_{\hat{n}}(s,s')\hat{n}^{\perp}(s)=\gamma(s')+\tau_{\hat{n}}(s',s)\hat{n}^{\perp}(s').\label{eq:lines-intersecting-criterion}
\end{equation}
If $\hat{n}(s)\nparallel\hat{n}(s')$, $\tau_{\hat{n}}(s,s')$ gives
the travel time from $\gamma(s)$ in the direction of $\hat{n}^{\perp}(s)$
to the line containing $\gamma(s')$ and parallel to $\hat{n}^{\perp}(s')$.
If $\hat{n}(s)\parallel\hat{n}(s')$ we simply set $\tau_{\hat{n}}(s,s')=\tau_{\hat{n}}(s',s)=\infty$.
Looking back to \prettyref{eq:parameterization-definition}, we see
that $\Phi_{\hat{n}}(s,t)\neq\Phi_{\hat{n}}(s',t')$ for $(s,t),(s',t')\in M_{t^{\pm},J}$
if and only if 
\begin{equation}
\frac{1}{t^{-}(s)}\leq\frac{1}{\tau_{\hat{n}}(s,s')}\leq\frac{1}{t^{+}(s)}\quad\text{or}\quad\frac{1}{t^{-}(s')}\leq\frac{1}{\tau_{\hat{n}}(s',s)}\leq\frac{1}{t^{+}(s')}\label{eq:inequality-to-be-shown}
\end{equation}
and $s\neq s'$. We check that this holds when $||\hat{\eta}-\hat{\eta}\circ\gamma||_{L^{\infty}(J)}$
is sufficiently small.

Let $s,s'\in J$ be such that $\hat{n}(s)\nparallel\hat{n}(s')$,
and let $(s,s')$ denote the open interval with boundary points $s$
and $s'$. Of course, $(s,s')\subset J$. Dotting $\hat{n}(s')$ into
\prettyref{eq:lines-intersecting-criterion} and rearranging yields
the formula 
\begin{equation}
\frac{1}{\tau_{\hat{n}}(s,s')}=\frac{\left(\hat{n}(s')-\hat{n}(s)\right)\cdot\hat{n}^{\perp}(s)}{\left(\gamma(s')-\gamma(s)\right)\cdot\hat{n}(s')}.\label{eq:identity-on-times}
\end{equation}
The righthand side approximates $\kappa_{\hat{n}}/\gamma'\cdot\hat{n}$.
In particular, there exists $\delta_{1}>0$ and $c_{1}>0$ such that
\begin{equation}
\left|\frac{\left(\hat{n}(s')-\hat{n}(s)\right)\cdot\hat{n}^{\perp}(s)}{\left(\gamma(s')-\gamma(s)\right)\cdot\hat{n}(s')}-\fint_{(s,s')}\frac{\kappa_{\hat{n}}}{\gamma'\cdot\hat{n}}\right|\lesssim_{\Gamma,\hat{\eta},t^{\pm},J}|s-s'|\label{eq:estimate-from-above-on-curvature}
\end{equation}
if $||\hat{n}-\hat{\eta}\circ\gamma||_{L^{\infty}(J)}<\delta_{1}$
and $|s-s'|<c_{1}$. It suffices to choose $\delta_{1}$ and $c_{1}$
such that 
\begin{equation}
\gamma'(a)\cdot\hat{n}(a')\gtrsim_{\Gamma,\hat{\eta},J}1\quad\forall\,a,a'\in J\text{ with }|a-a'|<c_{1}.\label{eq:overkill-choice}
\end{equation}
This is possible by \prettyref{eq:uniform-transverslity-param} since
$J$ is compactly contained. Note we also used the bound $||\kappa_{\hat{n}}||_{L^{\infty}(J)}\lesssim_{t^{\pm},J}1$,
which follows from \prettyref{eq:invertibility-setup-times} and \prettyref{eq:invertibility-primary-assumptions}.
Since $t^{\pm}$ are continuous, these same assumptions yield $c_{2}>0$
such that 
\begin{equation}
\frac{1}{t^{-}(s)}-\frac{\epsilon}{2}\leq\fint_{(s,s')}\frac{1}{t^{-}}+\epsilon\leq\fint_{(s,s')}\frac{\kappa_{\hat{n}}}{\gamma'\cdot\hat{n}}\leq\fint_{(s,s')}\frac{1}{t^{+}}-\epsilon\leq\frac{1}{t^{+}(s)}-\frac{\epsilon}{2}\label{eq:continuity-of-times}
\end{equation}
if $|s-s'|<c_{2}$. Combining \prettyref{eq:identity-on-times}-\prettyref{eq:continuity-of-times}
with \prettyref{eq:inequality-to-be-shown}, we deduce that $\Phi_{\hat{n}}(s,t)\neq\Phi_{\hat{n}}(s',t')$
if $||\hat{n}-\hat{\eta}\circ\gamma||_{L^{\infty}(J)}<\delta_{1}$
and if $(s,t),(s',t')\in M_{t^{\pm},J}$ satisfy $0<|s-s'|<c_{1}\wedge c_{2}\wedge\frac{\epsilon}{2C}$.
The constant $C=C(\Gamma,\hat{\eta},t^{\pm},J)$ is the one implicit
in the estimate \prettyref{eq:estimate-from-above-on-curvature}.
On the other hand, since $\Phi_{\hat{\eta}}(M_{t^{\pm},J})=V_{t^{\pm},J}$,
we can write that 
\[
\frac{|s-s'|}{||\nabla\Phi_{\hat{\eta}}^{-1}||_{L^{\infty}(V_{t^{\pm},J})}}\leq|\Phi_{\hat{\eta}}(s,t)-\Phi_{\hat{\eta}}(s',t')|\leq|\Phi_{\hat{n}}(s,t)-\Phi_{\hat{n}}(s',t')|+|\hat{n}(s)-\hat{\eta}\circ\gamma(s)||t|+|\hat{n}(s')-\hat{\eta}\circ\gamma(s')||t'|
\]
on $M_{t^{\pm},J}$. Thus, there exists $\delta_{2}>0$ such that
$\Phi_{\hat{n}}(s,t)\neq\Phi_{\hat{n}}(s',t')$ if $||\hat{n}-\hat{\eta}||_{L^{\infty}(J)}<\delta_{2}$
and if $(s,t),(s',t')\in M_{t^{\pm},J}$ satisfy $|s-s'|\geq c_{1}\wedge c_{2}\wedge\frac{\epsilon}{2C}$.
The conclusion is that 
\[
||\hat{n}-\hat{\eta}\circ\gamma||_{L^{\infty}(J)}<\delta_{1}\wedge\delta_{2}\quad\implies\quad\Phi_{\hat{n}}\text{ is one-to-one on }M_{t^{\pm},J}.
\]
The first part of the claim on the invertibility of $\Phi_{\hat{n}}$
is proved.

We end with the estimates on $\Phi_{\hat{n}}^{-1}$. Recall we arranged,
by our choice of $\delta_{0}$, for the inclusion $W\subset\Phi_{\hat{n}}(M_{t^{\pm},J})$. Let $x\in W$ and produce $(s,t)\in M_{t^{\pm},J}$ with
$\Phi_{\hat{n}}(s,t)=x$. It follows that 
\[
|\Phi_{\hat{\eta}}^{-1}(x)-\Phi_{\hat{n}}^{-1}(x)|=|\Phi_{\hat{\eta}}^{-1}\circ\Phi_{\hat{n}}(s,t)-\Phi_{\hat{\eta}}^{-1}\circ\Phi_{\hat{\eta}}(s,t)|\leq||\nabla\Phi_{\hat{\eta}}^{-1}||_{L^{\infty}(V_{t^{\pm},J})}|\hat{n}(s)-\hat{\eta}\circ\gamma(s)||t|
\]
thus yielding the first estimate in the claim. Continuing, we note
that $\delta_{1}$ was chosen so that \prettyref{eq:overkill-choice}
would hold. In particular, $\gamma'\cdot\hat{n}\gtrsim_{\Gamma,\hat{\eta},J}1$
on $J$. It follows from \prettyref{eq:determinant-formula} and \prettyref{eq:invertibility-primary-assumptions}
that 
\[
\det\nabla\Phi_{\hat{n}}\geq\epsilon(t^{+}\wedge|t^{-}|)\gamma'\cdot\hat{n}\quad\text{a.e. on }M_{t^{\pm},J}.
\]
The bound on $\nabla\Phi_{\hat{n}}^{-1}$ now follows from \prettyref{eq:parameterization-definition},
\prettyref{eq:invertibility-setup-times}, and the Lipschitz inverse
function theorem.\qed\end{proof} 
Having produced a sufficient condition for the invertibility of $\Phi_{\hat{n}}$
away from $\partial V$, we can now construct the desired approximations
$\pi_{k}$ of $\pi$. 
\begin{lemma}
\label{lem:stablelines-approximation}Let $V=\cup_{s\in\Gamma}\ell_{s}$
and $\pi:V\to\Gamma$ be as in \prettyref{eq:setup-for-stable-lines}
and \prettyref{eq:pi-map}, and let $W\subset V$ be open and compactly
contained. There exists $\left\{ \pi_{k}\right\} _{k\in\mathbb{N}}\subset C^{\infty}(W;\Gamma)$
such that $\pi_{k}\to\pi$, $\partial_{\hat{\eta}^{\perp}}\pi_{k}\to0$,
and $\partial_{\hat{\eta}^{\perp}}^{2}\pi_{k}\to0$ uniformly as $k\to\infty$. 
\end{lemma}

\begin{proof} We apply \prettyref{lem:invertibility-lemma}. Carrying
over the notation from that result, we let $t^{\pm}\in C_{c}(I)$
and $J\subset I$ be an open interval for which \prettyref{eq:invertibility-setup-times} holds,
such that $\overline{W}\subset V_{t^{\pm},J}$ and
$\overline{V_{t^{\pm},J}}\subset V$. Since $t_{V}^{-}<t^{-}<0<t^{+}<t_{V}^{+}$
nearby $J$, there exists $\epsilon>0$ and a slightly larger open
interval $J_{0}\subset I$ with $\overline{J}\subset J_{0}$,
$\sup_{J_{0}}\,t^{-}<0<\inf_{J_{0}}\,t^{+}$, and such that
\begin{equation}
\frac{1}{t^{-}(s)}+2\epsilon\leq\frac{1}{t_{V}^{-}(s)}\quad\text{and}\quad\frac{1}{t_{V}^{+}(s)}\leq\frac{1}{t^{+}(s)}-2\epsilon\quad\forall\,s\in J_{0}.\label{eq:defn-of-epsilon-approximation}
\end{equation}
To be clear, we consider $t^{\pm}$, $J$, $\epsilon$ and $J_{0}$
to be fixed depending on $\Gamma$, $\hat{\eta}$, $W$, and $V$ at this stage.

Next, we let $\delta>0$ be as in \prettyref{lem:invertibility-lemma}
and produce a sequence $\{\hat{\eta}_{k}\}_{k\in\mathbb{N}}\subset C^{\infty}(I;S^{1})$
approximating $\hat{\eta}\circ \gamma$. We will verify that our sequence satisfies \prettyref{eq:invertibility-primary-assumptions} for large enough $k$, i.e., there eventually holds
\begin{equation}
||\hat{\eta}_{k}-\hat{\eta}\circ\gamma||_{L^{\infty}\left(J\right)}<\delta\quad\text{and}\quad\frac{1}{t^{-}(s)}+\epsilon\leq\frac{\kappa_{\hat{\eta}_{k}}(s)}{\gamma'\cdot\hat{\eta}_{k}(s)}\leq\frac{1}{t^{+}(s)}-\epsilon\quad\forall\,s\in J. \label{eq:needed-inequalities-to-invert}
\end{equation}
Fix $\rho\in C_{c}^{\infty}((-1,1))$ with $\rho\geq0$ and $\int_{-1}^{1}\rho\,ds=1$.
For all large enough $k\in\mathbb{N}$, define $\hat{\eta}_{k}:I\to S^1$ by taking
\begin{equation}
\hat{\eta}_{k}(s)=\frac{(\hat{\eta}\circ \gamma)_{1/k}(s)}{\left|(\hat{\eta}\circ \gamma)_{1/k}(s)\right|}\quad\text{where}\quad(\hat{\eta}\circ \gamma)_{1/k}(s)=\int_{J_{0}}k\rho\left(k(s-s')\right)\hat{\eta}\circ\gamma(s')\,ds',\quad s\in J\label{eq:direct-mollification-of-curvature}
\end{equation}
and by smoothly extending it to the rest of $I$. Its values on $I\backslash\overline{J}$ are
immaterial; we simply define it there so as to match  the lemma. Note we needed that $\left|(\hat{\eta}\circ \gamma)_{1/k}\right|>0$ on $J$ for  \prettyref{eq:direct-mollification-of-curvature} to make sense. 
This is not a problem, since for all large enough $k$ there holds
\begin{equation}
1\geq\left|(\hat{\eta}\circ \gamma)_{1/k}\right|\geq1-||\kappa_{\hat{\eta}}||_{L^{\infty}(J_{0})}\frac{2}{k}\quad\text{on }J.\label{eq:lower-bound-norm-estimate}
\end{equation}
Observe that 
\begin{equation}
\kappa_{\hat{\eta}_{k}}=\frac{(\hat{\eta}\circ \gamma)_{1/k}'\cdot(\hat{\eta}\circ \gamma)_{1/k}^{\perp}}{\left|(\hat{\eta}\circ \gamma)_{1/k}\right|^{2}}\quad\text{on }J, \quad\text{while}\quad \kappa_{\hat{\eta}}=(\hat{\eta}\circ\gamma)'\cdot(\hat{\eta}\circ\gamma)^{\perp}\quad\text{on }I.\label{eq:formula-for-kappa}
\end{equation}
When combined with the usual properties of mollification, these facts imply the estimates
\begin{equation}
||\hat{\eta}_{k}-\hat{\eta}\circ\gamma||_{L^{\infty}(J)}\lesssim||\kappa_{\hat{\eta}}||_{L^{\infty}(J_{0})}\frac{1}{k},\quad||\hat{\eta}_{k}'||_{L^{\infty}(J)}\lesssim||\kappa_{\hat{\eta}}||_{L^{\infty}(J_{0})},\quad||\hat{\eta}_{k}''||_{L^{\infty}(J)}\lesssim k||\kappa_{\hat{\eta}}||_{L^{\infty}(J_{0})}+||\kappa_{\hat{\eta}}||_{L^{\infty}(J_{0})}^{2}\label{eq:mollification-estimates-kappa}
\end{equation}
for large enough $k$. We proceed to verify \prettyref{eq:needed-inequalities-to-invert}.
That the first condition is eventually satisfied  is clear. For the
second condition, let $s\in J$ and note that
\begin{equation}
\left|\frac{\kappa_{\hat{\eta}_{k}}(s)}{\gamma'\cdot\hat{\eta}_{k}(s)}-\int_{J_{0}}k\rho\left(k(s-s')\right)\frac{\kappa_{\hat{\eta}}(s')}{\gamma'\cdot\hat{\eta}\circ\gamma(s')}\,ds'\right|\lesssim_{\Gamma,\hat{\eta},J,J_{0}}\frac{1}{k}\label{eq:building-approx-step1}
\end{equation}
for large enough $k$. Besides \prettyref{eq:lower-bound-norm-estimate}-\prettyref{eq:mollification-estimates-kappa},
the proof makes use of the lower bound $\gamma'\cdot\hat{\eta}_{k}\gtrsim_{\Gamma,\hat{\eta},J}1$
on $J$, which is eventually implied by \prettyref{eq:uniform-transverslity-param}
and the fact that $J$ is compactly contained. Note that
\begin{equation}
\frac{1}{t^{-}}+2\epsilon\leq\frac{\kappa_{\hat{\eta}}}{\gamma'\cdot\hat{\eta}\circ\gamma}\leq\frac{1}{t^{+}}-2\epsilon\quad\text{a.e. on }J_{0}\label{eq:building-approx-step2}
\end{equation}
by \prettyref{eq:good-inequalities-baseline} and \prettyref{eq:defn-of-epsilon-approximation}.
Combining \prettyref{eq:building-approx-step1}, \prettyref{eq:building-approx-step2},
our choice to take $\rho\geq0$, and the fact that $1/t^{\pm}$ are
uniformly continuous on $J_{0}$, we deduce the second condition in
\prettyref{eq:needed-inequalities-to-invert} for large enough $k$. 

The previous paragraphs checked that the hypotheses of \prettyref{lem:invertibility-lemma} hold
for the given region $W$ and for our choices of $t^{\pm}$, $J$, $\epsilon$,
and $\hat{n}=\hat{\eta}_{k}$. Note we had to take $k$ large enough so that the conditions in \prettyref{eq:needed-inequalities-to-invert} would hold. The conclusion is that
the maps $\Phi_{\hat{\eta}_{k}}:I\times\mathbb{R}\to\mathbb{R}^{2}$
from \prettyref{eq:parameterization-definition} admit inverses on
$W$ satisfying $\Phi_{\hat{\eta}_{k}}^{-1}\in\text{Lip}(W;M_{t^{\pm},J})$, again
for large enough $k$. Recall  $M_{t^{\pm},J}=\Phi_{\hat{\eta}}^{-1}(V_{t^{\pm},J})$.
The estimates 
\begin{equation}
||\Phi_{\hat{\eta}_{k}}^{-1}-\Phi_{\hat{\eta}}^{-1}||_{L^{\infty}(W)}\lesssim_{\Gamma,\hat{\eta},W,V}||\hat{\eta}_{k}-\hat{\eta}\circ\gamma||_{L^{\infty}\left(J\right)}\quad\text{and}\quad||\nabla\Phi_{\hat{\eta}_{k}}^{-1}||_{L^{\infty}(W)}\lesssim_{\Gamma,\hat{\eta},W,V}1\label{eq:deduced-estimates-from-lemma}
\end{equation}
follow directly from the ones in the lemma. We remind the reader that
$t^{\pm}$, $J$, and $\epsilon$ (and $J_{0}$) were taken to depend
 on $\Gamma$, $\hat{\eta}$, $W$, and $V$. Since $\Phi_{\hat{\eta}_{k}}$
is smooth and $|\det\nabla\Phi_{\hat{\eta}_{k}}|\gtrsim_{\Gamma,\hat{\eta},W,V}1$
on $\Phi_{\hat{\eta}_{k}}^{-1}(W)$, the inverse function theorem
gives that $\Phi_{\hat{\eta}_{k}}^{-1}\in C^{\infty}(W;M_{t^{\pm},J})$.
We are finally ready to define the desired approximations $\pi_{k}$.
In a direct analogy with the formula for $\pi$ in \prettyref{eq:inverse-formula-for-pi},
we define $\pi_{k}\in C^{\infty}(W;\Gamma)$ by 
\begin{equation}
\pi_{k}=\gamma\circ(\Phi_{\hat{\eta}_{k}}^{-1})_{1}\quad\text{where}\quad(s,t)_{1}=s\label{eq:defn-of-approx-pi_k}
\end{equation}
for all large enough $k$. The rest of the proof establishes the convergences
in the claim.

That $\pi_{k}\to\pi$ uniformly on $W$ is clear. In fact, we read
off from \prettyref{eq:inverse-formula-for-pi}, \prettyref{eq:mollification-estimates-kappa}, \prettyref{eq:deduced-estimates-from-lemma},
and \prettyref{eq:defn-of-approx-pi_k} that 
\begin{equation}
||\pi_{k}-\pi||_{L^{\infty}(W)}\lesssim_{\Gamma,\hat{\eta},W,V}\frac{1}{k},\quad||\nabla\pi_{k}||_{L^{\infty}(W)}\lesssim_{\Gamma,\hat{\eta},W,V}1,\quad||\nabla\nabla\pi_{k}||_{L^{\infty}(W)}\lesssim_{\Gamma,\hat{\eta},W,V}k.\label{eq:estimates-on-pi_k}
\end{equation}
The third inequality follows from the elementary estimates
\[
|\nabla\Phi_{\hat{\eta}_{k}}| \lesssim 1 + |t||\hat{\eta}_{k}'| \quad\text{and}\quad |\nabla\nabla\Phi_{\hat{\eta}_{k}}|\lesssim|\gamma''|+|\hat{\eta}_{k}'|+|t||\hat{\eta}_{k}''|
\]
along with the inverse function theorem. It is now convenient to consider $\hat{\eta}_{k}$ as being defined on $W$ instead of $I$. 
Abusing notation slightly, we let 
\[
\hat{\eta}_{k}(x)=\hat{\eta}_{k}\circ(\Phi_{\hat{\eta}_{k}}^{-1})_{1}(x),\quad x\in W.
\]
On the righthand side we use \prettyref{eq:direct-mollification-of-curvature},
noting that $(\Phi_{\hat{\eta}_{k}}^{-1})_{1}(W)\subset J$. It is
natural to compare against $\hat{\eta}=\hat{\eta}\circ\gamma\circ(\Phi_{\hat{\eta}}^{-1})_{1}$. 
The estimates
\begin{equation}
||\hat{\eta}_{k}-\hat{\eta}||_{L^{\infty}(W)}\lesssim_{\Gamma,\hat{\eta},W,V}\frac{1}{k}\quad\text{and}\quad||\nabla\hat{\eta}_{k}||_{L^{\infty}(W)}\lesssim_{\Gamma,\hat{\eta},W,V}1\label{eq:estimates-on-eta_k}
\end{equation}
result from \prettyref{eq:mollification-estimates-kappa} and \prettyref{eq:deduced-estimates-from-lemma}.
We define $\partial_{\eta_{k}^{\perp}}$ and $\partial_{\eta_{k}^{\perp}}^{2}$
analogously to $\partial_{\hat{\eta}^{\perp}}$ and $\partial_{\hat{\eta}^{\perp}}^{2}$
from \prettyref{eq:full_directional}. Then,
\begin{equation}
\partial_{\eta_{k}^{\perp}}\pi_{k}=0\quad\text{and}\quad\partial_{\eta_{k}^{\perp}}^{2}\pi_{k}=0.\label{eq:useful-identities-directionalderivs}
\end{equation}
These identities are consequences of the fact that 
$\partial_t (\pi_{k}\circ \Phi_{\hat{\eta}_{k}})=0$.
Indeed, $\partial_{\eta_{k}^{\perp}}$ is a directional derivative along the lines traced out by $t\mapsto\Phi_{\hat{\eta}_{k}}(s,t)$, as follows from \prettyref{eq:parameterization-definition}.

We now have all of the ingredients to prove that $\partial_{\hat{\eta}^{\perp}}\pi_{k} \to 0$
and $\partial_{\hat{\eta}^{\perp}}^{2}\pi_{k} \to 0$ uniformly
on $W$. Applying the first identity in \prettyref{eq:useful-identities-directionalderivs}
yields that
\[
\partial_{\hat{\eta}^{\perp}}\pi_{k}=\partial_{\hat{\eta}^{\perp}}\pi_{k}-\partial_{\eta_{k}^{\perp}}\pi_{k}=\partial_{(\hat{\eta}-\hat{\eta}_{k})^{\perp}}\pi_{k}.
\]
Using \prettyref{eq:estimates-on-pi_k} and \prettyref{eq:estimates-on-eta_k}
we deduce that
\[
||\partial_{\hat{\eta}^{\perp}}\pi_{k}||_{L^{\infty}(W)}\leq||\hat{\eta}-\hat{\eta}_{k}||_{L^{\infty}(W)}||\nabla\pi_{k}||_{L^{\infty}(W)}\lesssim_{\Gamma,\hat{\eta},W,V}\frac{1}{k}.
\]
Next, we differentiate the first identity in \prettyref{eq:useful-identities-directionalderivs} to see that
\[
0=\partial_{(\hat{\eta}-\hat{\eta}_{k})^{\perp}}\left(\partial_{\hat{\eta}_{k}^{\perp}}\pi_{k}\right)=\partial_{(\hat{\eta}-\hat{\eta}_{k})^{\perp}}\hat{\eta}_{k}^{\perp}\cdot\nabla\pi_{k}+\left\langle (\hat{\eta}-\hat{\eta}_{k})^{\perp}\otimes\hat{\eta}_{k}^{\perp},\nabla\nabla\right\rangle \pi_{k}.
\]
By the second identity there,
\begin{align*}
\partial_{\hat{\eta}^{\perp}}^{2}\pi_{k} & =\left\langle \hat{\eta}^{\perp}\otimes\hat{\eta}^{\perp}-\eta_{k}^{\perp}\otimes\eta_{k}^{\perp},\nabla\nabla\right\rangle \pi_{k}=\left\langle (\hat{\eta}-\hat{\eta}_{k})^{\perp}\otimes(\hat{\eta}-\hat{\eta}_{k}+2\hat{\eta}_{k})^{\perp},\nabla\nabla\right\rangle \pi_{k}\\
 & =\partial_{(\hat{\eta}-\hat{\eta}_{k})^{\perp}}^{2}\pi_{k}-2\partial_{(\hat{\eta}-\hat{\eta}_{k})^{\perp}}\hat{\eta}_{k}^{\perp}\cdot\nabla\pi_{k}.
\end{align*}
Using \prettyref{eq:estimates-on-pi_k} and \prettyref{eq:estimates-on-eta_k}
again, we find that 
\[
||\partial_{\hat{\eta}^{\perp}}^{2}\pi_{k}||_{L^{\infty}(W)}\leq||\hat{\eta}-\hat{\eta}_{k}||_{L^{\infty}(W)}^{2}||\nabla\nabla\pi_{k}||_{L^{\infty}(W)}+2||\hat{\eta}-\hat{\eta}_{k}||_{L^{\infty}(W)}||\nabla\hat{\eta}_{k}||_{L^{\infty}(W)}||\nabla\pi_{k}||_{L^{\infty}(W)}\lesssim_{\Gamma,\hat{\eta},W,V}\frac{1}{k}.
\]
The proof is complete. \qed\end{proof}

\subsection{Three solution formulas\label{subsec:three-solution-formulas}}

Ultimately, we are interested in applying the method of characteristics
to deduce (partial) uniqueness and regularity theorems, and even explicit
solution formulas for $\mu$. Doing so requires supplementing the
ODEs from \prettyref{subsec:wrinkles-as-chars} with boundary data
implied by the original system  \prettyref{eq:wrinkling-PDEs}. Different
data arise depending on the  stable lines.  
Guided by the upcoming examples in \prettyref{sec:examples},
we treat the three configurations shown in \prettyref{fig:characteristics}. 
See \prettyref{cor:non-intersecting-chars} for Panel (a), \prettyref{cor:chars-meet-along-curve} for Panel (b), and
\prettyref{cor:chars-meet-at-point} for Panel (c). We continue to use the notation from the paragraphs leading up to \prettyref{lem:disintegration}.

\begin{figure}[tb]
\centering
\subfloat[]{\includegraphics[width=0.22\paperwidth]{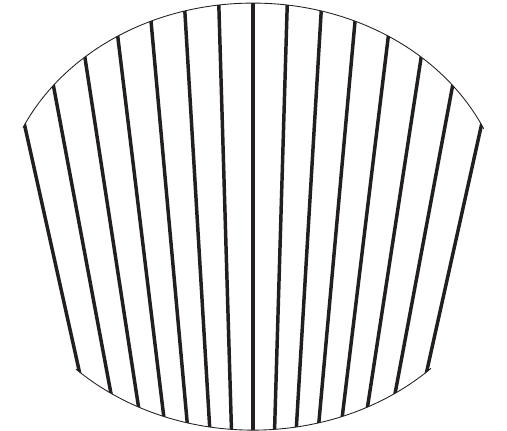}}\hspace{.06\textwidth}\subfloat[]{\includegraphics[width=0.22\paperwidth]{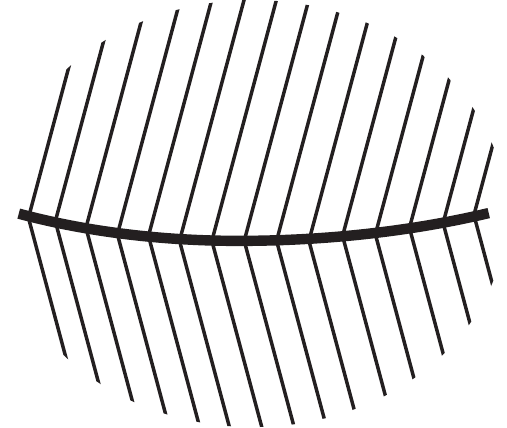}}\hspace{.06\textwidth}\subfloat[]{\includegraphics[width=0.22\paperwidth]{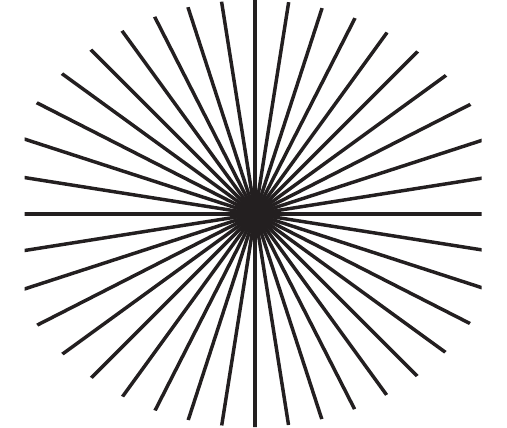}}\caption{Three configurations of stable lines. Panel (a) depicts stable
lines extending between boundary points. Panel
(b) depicts stable lines meeting along an interior curve. Panel (c) shows stable lines meeting
at a point. Given suitable non-degeneracy
conditions on $\varphi$, we prove that 
$\mu$ is uniquely determined on such lines. \label{fig:characteristics}}
\end{figure}

\paragraph*{\uline{Stable lines extending between boundary points}. \quad{}}

First, we consider Panel (a) of \prettyref{fig:characteristics}.
Recall $O$ denotes the ordered set of $\varphi$ from  \prettyref{eq:partition}, while $\{\ell_x\}$ denote its stable lines. Assume there exists an open set $V\subset O$ of the form
\begin{equation}
V=\cup_{s\in\Gamma}\ell_{s}\quad\text{where}\quad\partial\ell_{s}\subset\partial\Omega\quad\forall\,s\in\Gamma.\label{eq:extend-to-bdry-hyp1}
\end{equation}
As in \prettyref{eq:setup-for-stable-lines}, we understand $\Gamma\subset V$ to be a smooth curve (i.e., one that is diffeomorphic to an open interval) that meets each stable line it indexes transversely and exactly once. For simplicity, we suppose that
\begin{equation}\label{eq:extend-to-bdry-hyp1b}
\cup_{s\in\Gamma}\partial\ell_{s} \text{ consists of two Lipschitz curves}.
\end{equation}
By \prettyref{lem:stablelines} and the definition of $O$, there exist $\zeta\in L^1(V;(0,\infty))\cap L_{\text{loc}}^{2}(V)$ and $\hat{\eta}\in\text{Lip}_{\text{loc}}(V;S^{1})$ such that
\[
\nabla\nabla\varphi=\zeta\hat{\eta}\otimes\hat{\eta}\,dx\quad\text{on }V
\]
where $\hat{\eta}$ is constant along and perpendicular to the  stable lines. Our next assumption is that
\begin{equation}\label{eq:extend-to-bdry-Liphyp}
\hat{\eta}\in C(\overline{V}).
\end{equation}
In particular, the given stable lines are not allowed to meet at $\partial\Omega$.
Finally, we assume there exists $c>0$ such that 
\begin{equation}
\hat{\nu}\cdot[\nabla\varphi]\geq c\quad\text{and}\quad\left|\hat{\tau}\cdot\hat{\eta}|_{\partial\Omega}\right|\geq c\quad\mathcal{H}^{1}\text{-a.e. on }\cup_{s\in\Gamma}\partial\ell_{s}\label{eq:extend-to-bdry-hyp3}
\end{equation}
where $\hat{\nu}$ and $\hat{\tau}$ are the 
outwards-pointing unit normal and unit tangent vectors at $\partial\Omega$.
Note the second part of this last assumption requires that the given stable lines remain (a.e.) uniformly transverse to $\partial\Omega$.
Recall the change of measure factor $\varrho:V\to(0,\infty)$ from \prettyref{eq:Leb_disintegration}.

\begin{corollary} \label{cor:non-intersecting-chars} Suppose $\varphi$ admits some $V$ satisfying
 \prettyref{eq:extend-to-bdry-hyp1}-\prettyref{eq:extend-to-bdry-hyp3}, and let $\mu\in\mathcal{M}_{+}(\Omega;\emph{Sym}_{2})$ solve \prettyref{eq:wrinkling-PDEs}.
Then
\[
\mu=\lambda\hat{\eta}\otimes\hat{\eta}\,dx\quad\text{on }V
\]
where $\lambda:V\to[0,\infty)$ is determined as the unique weak
solution of the two-point boundary value problem 
\[
\begin{cases}
-\frac{1}{2\varrho}\partial_{\hat{\eta}^{\perp}(s)}^{2}(\varrho\lambda)=\det\nabla\nabla p & \text{on }\ell_{s}\\
\varrho\lambda=0 & \text{at }\partial\ell_{s}
\end{cases}
\]
upon restriction to $\mathcal{H}^{1}$-a.e.\ $\ell_{s}$. 
\end{corollary}

\begin{proof} Combining the second part of \prettyref{eq:rank-1-defect}
and \prettyref{lem:disintegration} yields the disintegration formula
\begin{equation}
\mu=\lambda_{\text{a.c.}}\hat{\eta}\otimes\hat{\eta}\,dx+\int_{\Gamma}\lambda_{\text{sing}}\hat{\eta}\otimes\hat{\eta}\mathcal{H}^{1}\lfloor\ell_{s}\,d\vartheta(s)\quad\text{on }V\label{eq:disintegration_1}
\end{equation}
where $\lambda_{\text{a.c.}}$ and $\lambda_{\text{sing}}$ solve
the ODEs \prettyref{eq:ODE_ac} and \prettyref{eq:ODE_sing} for $\mathcal{H}^{1}$-
and $\vartheta$-a.e.\ $s\in\Gamma$. 
Our plan is to use the complementary
slackness conditions in the original system \prettyref{eq:wrinkling-PDEs}
to deduce the boundary conditions 
\begin{equation}
\varrho\lambda_{\text{a.c.}}|_{\partial\ell_{s}}=0\quad\text{for }\mathcal{H}^{1}\text{-a.e. }s\quad\text{and}\quad\lambda_{\text{sing}}|_{\partial\ell_{s}}=0\quad\text{for }\vartheta\text{-a.e. }s.\label{eq:stable-lines-extending-bcs}
\end{equation}
It follows immediately from \prettyref{eq:ODE_sing}
and the second of these  that $\lambda_{\text{sing}}=0$.
Then, \prettyref{eq:ODE_ac} and the first boundary condition 
yield the desired characterization of $\lambda=\lambda_{\text{a.c.}}$.
The regularized formulation of the complementary slackness conditions must be applied. 
Here, we use the second part of \prettyref{eq:complementary_slackness_weaksense}, which implies that 
\begin{equation}
0=\lim_{\delta\to0}\,\int_{\cup_{s\in\Gamma}\partial\ell_{s}}|\hat{\nu}\cdot[\nabla\varphi]\left\langle \hat{\tau}\otimes\hat{\tau},\mu_{\delta}\right\rangle |\,d\mathcal{H}^{1}\label{eq:vanishinglimit}
\end{equation}
due to our assumptions \prettyref{eq:extend-to-bdry-hyp1} and \prettyref{eq:extend-to-bdry-hyp1b} (that the domain of integration is measurable follows from the second of these). 
Recall $\{\mu_{\delta}\}_{\delta>0}$
are the mollifications of $\mu$ defined in \prettyref{eq:mollified_measure_defn-intro}
using the kernel $\rho\in C_{c}^{\infty}(B_{1})$. As noted in \prettyref{rem:complementary-slackness-freedom}
and proved in \prettyref{sec:dualitytheory}, the complementary slackness
conditions hold so long as $\rho\geq 0$ and $\int_{B_1}\rho = 1$. We choose it a bit more carefully now: for the purposes of this proof, we  take $\rho >0$ on a neighborhood of zero. In fact, this same choice will also appear in the proofs of \prettyref{cor:chars-meet-along-curve} and \prettyref{cor:chars-meet-at-point}.

All this being said, we claim that the desired boundary conditions \prettyref{eq:stable-lines-extending-bcs}
hold. Using that  $\mu$ and $\rho$ are non-negative and applying  the disintegration formulas
\prettyref{eq:Leb_disintegration} and \prettyref{eq:disintegration_1}, we have by Fubini's theorem that
\begin{align*}
\mu_{\delta}(x)\geq\int_{V}\frac{1}{\delta^{2}}\rho\left(\frac{x-y}{\delta}\right)\,d\mu(y)=\int_{\Gamma}\left[\int_{\ell_{s}}\frac{1}{\delta^{2}}\rho\left(\frac{x-y}{\delta}\right)\varrho\lambda_{\text{a.c.}}\hat{\eta}\otimes\hat{\eta}(y)\,d\mathcal{H}^{1}(y)\right]\,d\mathcal{H}^{1}(s)\\
+\int_{\Gamma}\left[\int_{\ell_{s}}\frac{1}{\delta^{2}}\rho\left(\frac{x-y}{\delta}\right)\lambda_{\text{sing}}\hat{\eta}\otimes\hat{\eta}(y)\,d\mathcal{H}^{1}(y)\right]\,d\vartheta(s)
\end{align*}
for all $x\in\overline{\Omega}$. Applying this where the given stable lines meet $\partial\Omega$, we deduce the lower bound
\begin{align*}
 & \int_{\cup_{s\in\Gamma}\partial\ell_{s}}\hat{\nu}\cdot[\nabla\varphi]\left\langle \hat{\tau}\otimes\hat{\tau},\mu_{\delta}\right\rangle \,d\mathcal{H}^{1}\\
 & \qquad\geq\int_{\Gamma}\left[\int_{y\in\ell_{s}}\int_{x\in\cup_{s\in\Gamma}\partial\ell_{s}}\hat{\nu}\cdot[\nabla\varphi](x)|\hat{\tau}(x)\cdot\hat{\eta}(y)|^{2}\varrho\lambda_{\text{a.c.}}(y)\rho\left(\frac{x-y}{\delta}\right)\,\frac{d\mathcal{H}^{1}(x)d\mathcal{H}^{1}(y)}{\delta^{2}}\right]\,d\mathcal{H}^{1}(s)\\
 & \qquad\qquad+\int_{\Gamma}\left[\int_{y\in\ell_{s}}\int_{x\in\cup_{s\in\Gamma}\partial\ell_{s}}\hat{\nu}\cdot[\nabla\varphi](x)|\hat{\tau}(x)\cdot\hat{\eta}(y)|^{2}\lambda_{\text{sing}}(y)\rho\left(\frac{x-y}{\delta}\right)\,\frac{d\mathcal{H}^{1}(x)d\mathcal{H}^{1}(y)}{\delta^{2}}\right]\,d\vartheta(s).
\end{align*}
The integral on the lefthand
side  tends to zero as $\delta\to0$ by \prettyref{eq:vanishinglimit}. Using Fatou's lemma, we can pass to the limit on the
right. Recall from \prettyref{lem:disintegration} that $\varrho\lambda_{\text{a.c.}}$
and $\lambda_{\text{sing}}$ belong to $W^{2,1}(\ell_{s})$ and $W^{2,\infty}(\ell_{s})$
respectively for $\mathcal{H}^{1}$- and $\vartheta$-a.e.\ $s$.
In particular, $\varrho\lambda_{\text{a.c.}}(y)$ and $\lambda_{\text{sing}}(y)$
converge to their traces as $y\to\partial\ell_{s}$ along a.e.\ $\ell_{s}$.
Applying the hypotheses  \prettyref{eq:extend-to-bdry-Liphyp} and \prettyref{eq:extend-to-bdry-hyp3}, and making use of our choice to take $\rho >0$ nearby zero, we conclude that
\[
0=\int_{\Gamma}||\varrho\lambda_{\text{a.c.}}||_{L^{\infty}(\partial\ell_{s},\mathcal{H}^{0})}\,d\mathcal{H}^{1}(s)=\int_{\Gamma}||\lambda_{\text{sing}}||_{L^{\infty}(\partial\ell_{s},\mathcal{H}^{0})}\,d\vartheta(s)
\]
in the limit $\delta\to0$. The boundary conditions
\prettyref{eq:stable-lines-extending-bcs} are proved.
\qed\end{proof}

\paragraph*{\uline{Stable lines meeting along an interior curve}. \quad{}}

We turn to Panel (b) of \prettyref{fig:characteristics}. 
Again, recall from  \prettyref{eq:partition} that $O$ denotes the ordered set of $\varphi$ to which its stable lines $\{\ell_x\}$ belong, while $\Sigma$ is its singular set.
We now let
$V\subset O\cup \Sigma$ be an open set such that
\begin{equation}
V\backslash\Sigma=V_{-}\cup V_{+}\quad\text{where}\quad V_{\pm}=\cup_{s\in\Gamma_{\pm}}\ell_{s}\quad\text{ are disjoint}. 
\label{eq:along-sing-curve-hyp1}
\end{equation}
Here, $\Gamma_{\pm}$ are smooth curves belonging to $V_{\pm}$ that meet their indexed stable lines transversely and exactly once. 
Our second assumption is twofold: we require that
\begin{equation}\label{eq:along-sing-curve-hyp1b}
\begin{gathered}
\text{each indexed stable line }\ell_s\text{ passes between }\Sigma\text{ and }\partial V,\text{ and }
\\
\text{for all }z\in V\cap\Sigma\text{ there exist } s_\pm \in \Gamma_\pm \text{ such that } \{z\} = \partial \ell_{s_+} \cap \partial\ell_{s_-}.
\end{gathered}
\end{equation}
Looking back to \prettyref{fig:characteristics} should help make the meaning of this clear. Continuing, we assume for simplicity that \begin{equation}
V\cap\Sigma\quad\text{is a Lipschitz curve}.\label{eq:along-sing-hyp2}
\end{equation}
\prettyref{lem:stablelines} now guarantees the existence of $\zeta\in L^1(V;(0,\infty))\cap L_{\text{loc}}^{2}(V\backslash\Sigma)$
and $\hat{\eta}\in\text{Lip}_{\text{loc}}(V\backslash\Sigma;S^{1})$ such that
\[
\nabla\nabla\varphi=\zeta\hat{\eta}\otimes\hat{\eta}\,dx+\hat{\nu}_{\Sigma}\cdot[\nabla\varphi]\hat{\nu}_{\Sigma}\otimes\hat{\nu}_{\Sigma}\,\mathcal{H}^{1}\lfloor\Sigma\quad\text{on }V.
\]
Here, $\hat{\nu}_{\Sigma}$ denotes a choice of unit normal vector to $\Sigma$, which is defined a.e.\ along the portion of it belonging to $V$. We take it to point from $V_{-}$ to $V_{+}$, and we write $[\cdot]=\cdot|_{\Sigma_{+}}-\cdot|_{\Sigma_{-}}$
for the corresponding jump in a quantity where $\cdot|_{\Sigma_{\pm}}$
are the traces at $\Sigma$ from $V_{\pm}$. Our fourth assumption is that 
\begin{equation}\label{eq:along-sing-Liphyp}
\hat{\eta} \text{ restricts to each of }V_\pm\text{ as an element of }C(\overline{V_\pm}).
\end{equation}
In particular, this implies that the pair $s_\pm$  in \prettyref{eq:along-sing-curve-hyp1b} is unique, as no two stable lines on the same side of $V\cap \Sigma$ can meet there. 
Finally, we suppose there exists $c>0$ such that 
\begin{equation}
\begin{gathered}\zeta\geq c\quad\mathcal{L}^{2}\text{-a.e. on }V,\\
\hat{\nu}_{\Sigma}\cdot\left[\nabla\varphi\right]\geq c\quad\text{and}\quad\left|\hat{\tau}_{\Sigma}\cdot\hat{\eta}|_{\Sigma_{\pm}}\right|\geq c\quad\mathcal{H}^{1}\text{-a.e. on }V\cap\Sigma
\end{gathered}
\label{eq:along-sing-curve-hyp4}
\end{equation}
where $\hat{\tau}_{\Sigma}=\hat{\nu}_{\Sigma}^{\perp}$. Define the change of measure factor $\varrho:V\to(0,\infty)$ following \prettyref{eq:Leb_disintegration}.

\begin{corollary} \label{cor:chars-meet-along-curve} Suppose $\varphi$ admits some $V$ satisfying
\prettyref{eq:along-sing-curve-hyp1}-\prettyref{eq:along-sing-curve-hyp4}, 
and let $\mu\in\mathcal{M}_{+}(\Omega;\emph{Sym}_{2})$ solve \prettyref{eq:wrinkling-PDEs}.
Then 
\[
\mu=\lambda\hat{\eta}\otimes\hat{\eta}\,dx\quad\text{on }V
\]
where $\lambda:V\to[0,\infty)$ is determined as the unique weak
solution of the Cauchy problem 
\[
\begin{cases}
-\frac{1}{2\varrho}\partial_{\hat{\eta}^{\perp}(s)}^{2}(\varrho\lambda)=\det\nabla\nabla p & \text{on }\ell_{s}\\
\varrho\lambda=\partial_{\hat{\eta}^{\perp}(s)}(\varrho\lambda)=0 & \text{at }\partial\ell_{s}\cap \Sigma
\end{cases}
\]
upon restriction to $\mathcal{H}^{1}$-a.e.\ $\ell_{s}$. 
\end{corollary}

\begin{remark} We note the following curious fact: under the above hypotheses, $\det\nabla\nabla p\leq0$ a.e.\ on $V$.
Indeed, it follows from the Cauchy problem that $\varrho\lambda$
and $\det\nabla\nabla p$ take on opposite signs. Of course, $\varrho>0$
and $\lambda\geq0$.  \end{remark}

\begin{remark} In each of the examples in \prettyref{sec:examples}
it will turn out that if $\Sigma$ is not empty nor a single (smooth)
curve, it is nevertheless a tree. At its internal vertices, three or more stable lines will meet, and to achieve
the analogous result we will need to show that $\mu$ vanishes on these lines. 
This can be done using the ideas in the proof below. See \prettyref{ex:negatively-curved-convex} for more
details. \end{remark}

\begin{proof} The proof is similar in spirit to that of \prettyref{cor:non-intersecting-chars}, albeit more involved.
We start by showing that 
\begin{equation}
\mu=0\quad\text{on }V\cap\Sigma.\label{eq:invisible-jumpset}
\end{equation}
At the same time, a straightforward application of \prettyref{eq:rank-1-defect} and
\prettyref{lem:disintegration} yields the disintegration formula
\begin{equation}
\mu=\lambda_{\text{a.c.}}\hat{\eta}\otimes\hat{\eta}\,dx+\int_{\Gamma}\lambda_{\text{sing}}\hat{\eta}\otimes\hat{\eta}\mathcal{H}^{1}\lfloor\ell_{s}\,d\vartheta(s)\quad\text{on }V\backslash\Sigma\label{eq:disintegration_2}
\end{equation}
where $\lambda_{\text{a.c.}}$ and $\lambda_{\text{sing}}$ solve
the ODEs \prettyref{eq:ODE_ac} and \prettyref{eq:ODE_sing} for $\mathcal{H}^{1}$-
and $\vartheta$-a.e.\ $s$ belonging to the index set $\Gamma = \Gamma_{+}\cup\Gamma_{-}$.  Our second step will be to extract
the initial conditions 
\begin{align}
\varrho\lambda_{\text{a.c.}}|_{\partial\ell_{s}\cap \Sigma} & =\partial_{\hat{\eta}^{\perp}(s)}(\varrho\lambda_{\text{a.c.}})|_{\partial\ell_{s}\cap \Sigma}=0\quad\text{for }\mathcal{H}^{1}\text{-a.e. }s,\label{eq:zeroth_order_trace_ac-1}\\
\lambda_{\text{sing}}|_{\partial\ell_{s}\cap \Sigma} & =\partial_{\hat{\eta}^{\perp}(s)}\lambda_{\text{sing}}|_{\partial\ell_{s}\cap \Sigma}=0\quad\text{for }\vartheta\text{-a.e. }s\label{eq:zeroth_order_trace_sing-1}
\end{align}
from the first two equations in \prettyref{eq:wrinkling-PDEs}. Combining these with \prettyref{eq:ODE_ac} and \prettyref{eq:ODE_sing} proves that $\lambda_{\text{sing}}=0$, and the desired characterization of $\lambda=\lambda_{\text{a.c.}}$ follows. To accomplish these steps, we shall make use of the following consequences of the (regularized) complementary slackness conditions \prettyref{eq:complementary_slackness_weaksense}, which hold in light of the formula for $\nabla\nabla \varphi$ above: 
\begin{equation}
0=\lim_{\delta\to0}\,\int_{V\backslash\Sigma}\left|\left\langle \zeta\hat{\eta}^{\perp}\otimes\hat{\eta}^{\perp},\mu_{\delta}\right\rangle \right|\,dx=\lim_{\delta\to0}\,\int_{V\cap\Sigma}|\hat{\nu}_{\Sigma}\cdot[\nabla\varphi]\left\langle \hat{\tau}_{\Sigma}\otimes\hat{\tau}_{\Sigma},\mu_{\delta}\right\rangle |\,d\mathcal{H}^{1}\label{eq:compl_slackness_both}
\end{equation}
where $\{\mu_{\delta}\}_{\delta>0}$ are the mollified versions of $\mu$ from \prettyref{eq:mollified_measure_defn-intro}. Again, we take the kernel $\rho\in C_c^\infty(B_1)$ to satisfy $\rho > 0$ on a neighborhood of  zero (see \prettyref{rem:complementary-slackness-freedom}).

We start by proving \prettyref{eq:invisible-jumpset}. First, we note for every $x\in V$ that
\begin{equation}\label{eq:mu-lb-non-negative}
\mu_{\delta}(x)\geq\int_{V\cap\Sigma}\frac{1}{\delta^2}\rho\left(\frac{x-y}{\delta}\right)\,d\mu(y)\quad\text{and}\quad
\mu_{\delta}(x)\geq\int_{V\backslash \Sigma}\frac{1}{\delta^{2}}\rho\left(\frac{x-y}{\delta}\right)\,d\mu(y).
\end{equation}
These follow from the non-negativity of $\rho$ and $\mu$. Integrating the first of these and applying Fubini's theorem yields 
\begin{align*}
\int_{V\backslash\Sigma}\left\langle \zeta\hat{\eta}^{\perp}\otimes\hat{\eta}^{\perp},\mu_{\delta}\right\rangle \,dx & \geq\int_{V\cap\Sigma}\left\langle \int_{V\backslash\Sigma}\zeta\hat{\eta}^{\perp}\otimes\hat{\eta}^{\perp}(x)\rho\left(\frac{x-y}{\delta}\right)\,\frac{dx}{\delta^{2}},\mu(y)\right\rangle ,\\
\int_{V\cap\Sigma}\left\langle \hat{\nu}_{\Sigma}\cdot[\nabla\varphi]\hat{\tau}_{\Sigma}\otimes\hat{\tau}_{\Sigma},\mu_{\delta}\right\rangle \,d\mathcal{H}^{1} & \geq\frac{1}{\delta}\int_{V\cap\Sigma}\left\langle \int_{V\cap\Sigma}\hat{\nu}_{\Sigma}\cdot[\nabla\varphi]\hat{\tau}_{\Sigma}\otimes\hat{\tau}_{\Sigma}(x)\rho\left(\frac{x-y}{\delta}\right)\,\frac{d\mathcal{H}^{1}(x)}{\delta},\mu(y)\right\rangle.
\end{align*}
According to  \prettyref{eq:compl_slackness_both}, each of the integrals on the lefthand side tends
to zero as $\delta \to 0$. Applying  \prettyref{eq:along-sing-Liphyp} along with the first part of  \prettyref{eq:along-sing-curve-hyp4}, we can pass to the limit in the first inequality above to deduce that
\[
0=\left\langle \hat{\eta}^{\perp}\otimes\hat{\eta}^{\perp}|_{\Sigma_{\pm}},\mu\right\rangle \quad\text{on }V\cap\Sigma.
\]
Then, using that $\mu$ is  non-negative and $\text{Sym}_2$-valued, we get that
\[
\mu=\left\langle \hat{\eta}\otimes\hat{\eta}|_{\Sigma_{\pm}},\mu\right\rangle \hat{\eta}\otimes\hat{\eta}|_{\Sigma_{\pm}}\quad\text{on } V\cap\Sigma.
\] 
Now plug this into the second inequality and send $\delta \to 0$ again (this is to deal with the case where $\hat{\eta}|_{\Sigma_\pm}$ are parallel on some non-null set). Applying  \prettyref{eq:along-sing-Liphyp} along with the second and third parts of \prettyref{eq:along-sing-curve-hyp4}, and recalling that  $\rho >0$ nearby zero, there follows
\[
0=\left\langle \hat{\eta}\otimes\hat{\eta}|_{\Sigma_{\pm}},\mu\right\rangle\quad\text{on }V\cap\Sigma.
\]
Therefore $\mu = 0$ on $V\cap \Sigma$ and \prettyref{eq:invisible-jumpset} is proved.

We proceed to control $\mu$ on $V\backslash \Sigma$. As explained
above, we must establish the initial conditions \prettyref{eq:zeroth_order_trace_ac-1}
and \prettyref{eq:zeroth_order_trace_sing-1} for $\varrho\lambda_{\text{a.c.}}$
and $\lambda_{\text{sing}}$. We handle their traces first. This part of the proof can be copied almost verbatim
from that of \prettyref{cor:non-intersecting-chars}. Recall
the index set $\Gamma=\Gamma_{+}\cup\Gamma_{-}$. Using the second part of  \prettyref{eq:mu-lb-non-negative} along with  the disintegration formulas \prettyref{eq:Leb_disintegration} and \prettyref{eq:disintegration_2}
produces the lower bound
\begin{align*}
 & \int_{V\cap\Sigma}\left\langle \hat{\nu}_{\Sigma}\cdot[\nabla\varphi]\hat{\tau}_{\Sigma}\otimes\hat{\tau}_{\Sigma},\mu_{\delta}\right\rangle \,d\mathcal{H}^{1}\\
 & \geq\int_{\Gamma}\left[\int_{y\in\ell_{s}}\int_{x\in V\cap\Sigma}\hat{\nu}_{\Sigma}\cdot[\nabla\varphi](x)|\hat{\tau}_{\Sigma}(x)\cdot\hat{\eta}(y)|^{2}\varrho\lambda_{\text{a.c.}}(y)\rho\left(\frac{x-y}{\delta}\right)\,\frac{d\mathcal{H}^{1}(x)d\mathcal{H}^{1}(y)}{\delta^{2}}\right]\,d\mathcal{H}^{1}(s)\\
 & \qquad+\int_{\Gamma}\left[\int_{y\in\ell_{s}}\int_{x\in V\cap\Sigma}\hat{\nu}_{\Sigma}\cdot[\nabla\varphi](x)|\hat{\tau}_{\Sigma}(x)\cdot\hat{\eta}(y)|^{2}\lambda_{\text{sing}}(y)\rho\left(\frac{x-y}{\delta}\right)\,\frac{d\mathcal{H}^{1}(x)d\mathcal{H}^{1}(y)}{\delta^{2}}\right]\,d\vartheta(s).
\end{align*}
According to \prettyref{eq:compl_slackness_both}, the integral on the lefthand side tends to zero as $\delta \to 0$.
Applying the assumption \prettyref{eq:along-sing-Liphyp}, the second and third parts of \prettyref{eq:along-sing-curve-hyp4}, and  our choice to take $\rho >0$ nearby zero, we pass to the limit to deduce that
\[
0=\int_{\Gamma}\varrho\lambda_{\text{a.c.}}|_{\partial\ell_{s}\cap \Sigma}\,d\mathcal{H}^{1}(s)=\int_{\Gamma}\lambda_{\text{sing}}|_{\partial\ell_{s}\cap \Sigma}\,d\vartheta(s).
\]
Since the integrands are non-negative, they must vanish a.e. The first parts of \prettyref{eq:zeroth_order_trace_ac-1}
and \prettyref{eq:zeroth_order_trace_sing-1} are proved.

The next part of the proof has no analog in that of \prettyref{cor:non-intersecting-chars}: we  must 
show that $\partial_{\hat{\eta}^{\perp}}(\varrho\lambda_{\text{a.c.}})$
and $\partial_{\hat{\eta}^{\perp}}\lambda_{\text{sing}}$ vanish
at $V\cap\Sigma$.  The argument goes in two steps. The first step is to show that 
\begin{equation}
0\leq\partial_{\hat{\eta}^{\perp}(s)}(\varrho\lambda_{\text{a.c.}})|_{\partial\ell_{s}\cap \Sigma}\quad\text{for }\mathcal{H}^{1}\text{-a.e. }s\in \Gamma \quad\text{and}\quad0\leq\partial_{\hat{\eta}^{\perp}(s)}\lambda_{\text{sing}}|_{\partial\ell_{s}\cap \Sigma}\quad\text{for }\vartheta\text{-a.e. }s\in \Gamma\label{eq:non-negative_firstordertraces}
\end{equation}
using the non-negativity of $\mu$. Here, we understand that each $\ell_s$ is oriented so that it starts at $\Sigma$ and ends at $\partial V$, i.e., we take $\hat{\eta}^{\perp}$ to point away from $\Sigma$. (This is possible due to  \prettyref{eq:along-sing-curve-hyp1b}.) 
The second step is to deduce from the first equation in \prettyref{eq:wrinkling-PDEs} the matching conditions
\begin{equation}
\begin{gathered}
0=\partial_{\hat{\eta}^{\perp}(s_{+})}(\varrho\lambda_{\text{a.c.}})|_{\partial\ell_{s_{+}}\cap \Sigma}+\partial_{\hat{\eta}^{\perp}(s_{-})}(\varrho\lambda_{\text{a.c.}})|_{\partial\ell_{s_{-}}\cap \Sigma}\quad\text{and}\quad0=\partial_{\hat{\eta}^{\perp}(s_{+})}\lambda_{\text{sing}}|_{\partial\ell_{s_{+}}\cap \Sigma}+\partial_{\hat{\eta}^{\perp}(s_{-})}\lambda_{\text{sing}}|_{\partial\ell_{s_{-}}\cap \Sigma} \\ 
\text{respectively for }\mathcal{H}^{1}\text{-a.e.\ and } \vartheta\text{-a.e.\ } s_{\pm}\in\Gamma_{\pm}
\text{ such that } \partial\ell_{s_{+}}\cap\partial\ell_{s_{-}}\cap V\cap\Sigma \neq\emptyset.
\end{gathered}\label{eq:matching}
\end{equation}
Note there is a one-to-one correspondence between points $z\in V\cap\Sigma$ and pairs $s_{\pm}\in\Gamma_\pm$ satisfying $\{z\} = \partial\ell_{s_+}\cap\partial\ell_{s_-}$, due to our hypotheses  \prettyref{eq:along-sing-curve-hyp1b} and \prettyref{eq:along-sing-Liphyp}. So, the last two quantifications in \prettyref{eq:matching} make sense. Combining \prettyref{eq:non-negative_firstordertraces} and \prettyref{eq:matching} yields the remaining parts of  \prettyref{eq:zeroth_order_trace_ac-1}
and \prettyref{eq:zeroth_order_trace_sing-1}.

First, we handle \prettyref{eq:non-negative_firstordertraces}. Since $\mu\geq0$, the densities $\varrho\lambda_{\text{a.c.}}$ and
$\lambda_{\text{sing}}$ 
are non-negative. That is,
\[
0\leq\int_{\ell_{s}}\psi\varrho\lambda_{\text{a.c.}}\,d\mathcal{H}^{1}\quad\text{for }\mathcal{H}^{1}\text{-a.e. }s\quad\text{and}\quad0\leq\int_{\ell_{s}}\psi\lambda_{\text{sing}}\,d\mathcal{H}^{1}\quad\text{for }\vartheta\text{-a.e. }s
\]
whenever $\psi \geq 0$. Let $\chi\in C_{c}((1,2))$ be non-negative and integrate to one, and define $\psi_{\delta}\in C(\ell_s)$ for $\delta >0$ by
\[
\psi_{\delta}(x)=\frac{1}{|x-z|}\cdot\frac{1}{\delta}\chi\left(\frac{|x-z|}{\delta}\right)\quad\text{where}\quad \{z\} = \partial\ell_s \cap V\cap \Sigma.
\]
Recall from \prettyref{lem:disintegration} that $\varrho\lambda_{\text{a.c.}}\in W^{2,1}(\ell_s)$ and $\lambda_{\text{sing}}\in W^{2,\infty}(\ell_s)$ for $\mathcal{H}^1$- and $\vartheta$-a.e.\ $s$, respectively. Above, we proved that they vanish at a.e.\ $\partial\ell_{s}\cap\Sigma$. Therefore, 
\[
\int_{\ell_{s}}\psi_{\delta}\varrho\lambda_{\text{a.c.}}\,d\mathcal{H}^{1}\to\partial_{\hat{\eta}^{\perp}(s)}(\varrho\lambda_{\text{a.c.}})|_{\partial\ell_{s}\cap{\Sigma}}\quad\text{and}\quad\int_{\ell_{s}}\psi_{\delta}\lambda_{\text{sing}}\,d\mathcal{H}^{1}\to\partial_{\hat{\eta}^{\perp}(s)}\lambda_{\text{sing}}|_{\partial\ell_{s}\cap{\Sigma}} \quad \text{for a.e.\ } s
\]
and with this the desired inequalities \prettyref{eq:non-negative_firstordertraces} follow.

Finally, we prove the matching conditions \prettyref{eq:matching}. Testing the first part of \prettyref{eq:wrinkling-PDEs} against $\psi\in C_{c}^{\infty}(V)$ and applying the disintegration formulas  \prettyref{eq:Leb_disintegration} and \prettyref{eq:disintegration_2} yields the identity
\begin{align*}
\int_{V}\psi\det\nabla\nabla p\,dx & =\int_{V}\left\langle -\frac{1}{2}\nabla^{\perp}\nabla^{\perp}\psi,\mu\right\rangle \\
 & =\int_{\Gamma}\left[\int_{\ell_{s}}-\frac{1}{2}\partial_{\hat{\eta}^{\perp}(s)}^{2}\psi\varrho\lambda_{\text{a.c.}}\,d\mathcal{H}^{1}\right]\,d\mathcal{H}^{1}(s)+\int_{\Gamma}\left[\int_{\ell_{s}}-\frac{1}{2}\partial_{\hat{\eta}^{\perp}(s)}^{2}\psi\lambda_{\text{sing}}\,d\mathcal{H}^{1}\right]\,d\vartheta(s).
\end{align*}
Integrating by parts twice with the ODEs \prettyref{eq:ODE_ac} and \prettyref{eq:ODE_sing}, there follows
\begin{align*}
\int_{\Gamma}\left[\int_{\ell_{s}}\partial_{\hat{\eta}^{\perp}(s)}^{2}\psi\varrho\lambda_{\text{a.c.}}\,d\mathcal{H}^{1}\right]\,d\mathcal{H}^{1}(s)  &=\int_{\Gamma}\left(\psi\partial_{\hat{\eta}^{\perp}(s)}(\varrho\lambda_{\text{a.c.}})\right)|_{\partial\ell_{s}\cap{\Sigma}}\,d\mathcal{H}^{1}(s)-2\int_{V}\psi\det\nabla\nabla p\,dx,\\
\int_{\Gamma}\left[\int_{\ell_{s}}\partial_{\hat{\eta}^{\perp}(s)}^{2}\psi\lambda_{\text{sing}}\,d\mathcal{H}^{1}\right]\,d\vartheta(s)&=\int_{\Gamma}\left(\psi\partial_{\hat{\eta}^{\perp}(s)}\lambda_{\text{sing}}\right)|_{\partial\ell_{s}\cap{\Sigma}}\,d\vartheta(s).
\end{align*}
Setting these formulas into the identity above and cancelling like terms proves that
\[
0=\int_{\Gamma}\left(\psi\partial_{\hat{\eta}^{\perp}(s)}(\varrho\lambda_{\text{a.c.}})\right)|_{\partial\ell_{s}\cap{\Sigma}}\,d\mathcal{H}^{1}(s)+\int_{\Gamma}\left(\psi\partial_{\hat{\eta}^{\perp}(s)}\lambda_{\text{sing}}\right)|_{\partial\ell_{s}\cap{\Sigma}}\,d\vartheta(s)
\]
for all $\psi\in C_{c}^{\infty}(V)$.  The conditions in \prettyref{eq:matching} now follow from the correspondence between points $z\in V\cap\Sigma$ and pairs $s_\pm \in \Gamma_\pm$, and the fact that $\mathcal{H}^{1}\perp\vartheta$ on $\Gamma = \Gamma_+ \cup \Gamma_-$. 
\qed\end{proof}

\paragraph*{\uline{Stable lines meeting at an interior point}. \quad{}}

We end with the possibility in Panel (c) of \prettyref{fig:characteristics}. 
Let $O$ be the ordered set of $\varphi$ from \prettyref{eq:partition}, and suppose there exists a point $x_0\in \Omega$ and an open set $V\subset O\cup \{x_0\}$ such that
\begin{equation}
x_0 \in V\quad\text{and}\quad V\backslash\left\{ x_{0}\right\} =\cup_{s\in\Gamma}\ell_{s}\quad\text{where}\quad\ell_{s}||\hat{e}_{r}(s)\quad\forall\,s\in\Gamma.\label{eq:V-meet-at-point}
\end{equation}
We use $(r,\theta)$ to denote polar coordinates about $x_{0}$,
and $\{\hat{e}_{r},\hat{e}_{\theta}\}$ for the associated orthonormal frame. 
The set $\Gamma$ is understood to meet each stable line it indexes transversely and exactly once; we take it to be diffeomorphic to the unit circle $S^1$. It follows that
\[
\nabla\nabla\varphi=\frac{\zeta(\theta)}{r}\hat{e}_{\theta}\otimes\hat{e}_{\theta}\,dx\quad\text{on }V
\]
where $\zeta\in L^{2}((0,2\pi))$ is locally uniformly positive a.e. 
Evidently, $x_0$ belongs to the singular set $\Sigma$ of $\varphi$, as $1/r$ is not square integrable on any neighborhood of the origin. Denote $\partial_{r}=\partial_{\hat{e}_{r}}$. 
\begin{corollary} \label{cor:chars-meet-at-point} Suppose $\varphi$ admits some $V$ satisfying
\prettyref{eq:V-meet-at-point}, and let $\mu\in\mathcal{M}_{+}(\Omega;\emph{Sym}_{2})$
solve \prettyref{eq:wrinkling-PDEs}. Then
\[
\mu=\lambda\hat{e}_{\theta}\otimes\hat{e}_{\theta}\,dx\quad\text{on }V
\]
where $\lambda:V\to[0,\infty)$ is determined as the unique weak
solution of the Cauchy problem 
\[
\begin{cases}
-\frac{1}{2r}\partial_{r}^{2}(r\lambda)=\det\nabla\nabla p & \text{on }\ell_{s}\\
r\lambda=\partial_{r}(r\lambda)=0 & \text{at }\partial\ell_{s}\cap\{x_{0}\}
\end{cases}
\]
upon restriction to $\mathcal{H}^{1}$-a.e.\ $\ell_{s}$. 
\end{corollary} 

\begin{proof} We follow the same outline as the proof of \prettyref{cor:chars-meet-along-curve},
with the details being modified slightly to make up for the fact that $\nabla\nabla\varphi$ lacks a singular part. The first step will be to show that
\begin{equation}
\mu(\{x_{0}\})=0.\label{eq:invisible_core}
\end{equation}
At the same time, \prettyref{lem:structure} and \prettyref{lem:disintegration} yield the disintegration formula 
\begin{equation}
\mu=\lambda_{\text{a.c.}}\hat{e}_{\theta}\otimes\hat{e}_{\theta}\,dx+\int_{\Gamma}\lambda_{\text{sing}}\hat{e}_{\theta}\otimes\hat{e}_{\theta}\mathcal{H}^{1}\lfloor\ell_{s}\,d\vartheta(s)\quad\text{on }V\backslash\left\{ x_{0}\right\} \label{eq:disintegration_3}
\end{equation}
where $\lambda_{\text{a.c.}}$ and $\lambda_{\text{sing}}$ solve the ODEs \prettyref{eq:ODE_ac} and \prettyref{eq:ODE_sing}. 
It is here that we use the unique continuation-type result from \prettyref{rem:unique-continuation-stable-lines}. This allows to prove \prettyref{eq:disintegration_3} by finding two neighborhoods $V_i=\cup_{\Gamma_i}\ell_s$ of the form \prettyref{eq:setup-for-stable-lines} with $V\backslash\{x_0\} = V_1\cup V_2$ and $\Gamma\cap V_i = \Gamma_i$, and by applying \prettyref{lem:disintegration} to each $V_i$. That the individual disintegration formulas agree on $V_1\cap V_2$ is assured by the remark. 
The final step of the proof is to establish the initial conditions 
\begin{align}
r\lambda_{\text{a.c.}}|_{\partial\ell_{s}}(x_{0}) & =\partial_{r}(r\lambda_{\text{a.c.}})|_{\partial\ell_{s}}(x_{0})=0\quad\text{for }\mathcal{H}^{1}\text{-a.e. }s,\label{eq:zeroth_order_trace_ac-1-1}\\
\lambda_{\text{sing}}|_{\partial\ell_{s}}(x_{0}) & =\partial_{r}\lambda_{\text{sing}}|_{\partial\ell_{s}}(x_{0})=0\quad\text{for }\vartheta\text{-a.e. }s.\label{eq:zeroth_order_trace_sing-1-1}
\end{align}
We show these via the first complementary slackness
condition in \prettyref{eq:wrinkling-PDEs}, which implies that 
\begin{equation}
0=\lim_{\delta\to0}\,\int_{V}\left|\left\langle \frac{\zeta(\theta)}{r}\hat{e}_{r}\otimes\hat{e}_{r},\mu_{\delta}\right\rangle \right|\,dx.\label{eq:vanishing_mass}
\end{equation}
Note $\{\mu_{\delta}\}_{\delta>0}$ are given by \prettyref{eq:mollified_measure_defn-intro} where the mollifying kernel $\rho \in C_c^\infty(B_1)$ is non-negative and integrates to one. As in the previous two proofs, we take $\rho >0$ nearby zero (that this is possible was noted in \prettyref{rem:complementary-slackness-freedom}).

We start with the proof of \prettyref{eq:invisible_core}. Given any $x\in V$ we can write that
\[
\mu_{\delta}(x)\geq\rho\left(\frac{x-x_{0}}{\delta}\right)\mu(\{x_{0}\})
\]
since $\rho$ and $\mu$ are non-negative. It follows that
\[
\int_{V}\left\langle \frac{\zeta}{r}\hat{e}_{r}\otimes\hat{e}_{r},\mu_{\delta}\right\rangle \,dx\geq\frac{1}{\delta}\left\langle \int_{B_\delta(x_0)} \zeta \frac{\delta}{r}  \hat{e}_{r}\otimes\hat{e}_{r}(x)\rho\left(\frac{x-x_{0}}{\delta}\right)\,\frac{dx}{\delta^{2}},\mu(\{x_{0}\})\right\rangle 
\]
for all sufficiently small $\delta>0$. The integral on the lefthand side tends to zero by \prettyref{eq:vanishing_mass}. Recall $\zeta$ is locally uniformly positive a.e., i.e., every $\theta \in (0,2\pi)$ admits a neighborhood on which $\zeta \geq c(\theta)>0$ a.e. Recall also that we chose to take $\rho >0$ nearby zero. Multiplying by $\delta$ and sending $\delta \to 0$ proves that
\[
\left\langle \hat{x}\otimes\hat{x},\mu(\{x_{0}\})\right\rangle =0\quad\forall\,\hat{x}\in S^{1}.
\]
It follows that $\mu(\{x_{0}\})=0$ as claimed in \prettyref{eq:invisible_core}.

We proceed to determine $\mu$ on $V\backslash\{x_0\}$. The formula \prettyref{eq:disintegration_3} follows as indicated above. We proceed to show the initial conditions \prettyref{eq:zeroth_order_trace_ac-1-1}
and \prettyref{eq:zeroth_order_trace_sing-1-1}. First, we consider the traces of $r\lambda_{\text{a.c.}}$ and $\lambda_{\text{sing}}$.
Given $x\in V$, 
\begin{align*}
\mu_{\delta}(x)\geq\int_{V\backslash\left\{ x_{0}\right\} }\frac{1}{\delta^{2}}\rho\left(\frac{x-y}{\delta}\right)\,d\mu(y)=\int_{\Gamma}\left[\int_{y\in\ell_{s}}\frac{1}{\delta^{2}}\rho\left(\frac{x-y}{\delta}\right)\hat{e}_{\theta}\otimes\hat{e}_{\theta}(y)r\lambda_{\text{a.c.}}(y)\,d\mathcal{H}^{1}(y)\right]\,d\mathcal{H}^{1}(s)\\
+\int_{\Gamma}\left[\int_{y\in\ell_{s}}\frac{1}{\delta^{2}}\rho\left(\frac{x-y}{\delta}\right)\hat{e}_{\theta}\otimes\hat{e}_{\theta}(y)\lambda_{\text{sing}}(y)\,d\mathcal{H}^{1}(y)\right]\,d\vartheta(s)
\end{align*}
by \prettyref{eq:Leb_disintegration} and \prettyref{eq:disintegration_3}. 
It follows from Fubini's theorem that
\begin{align*}
\int_{V} & \left\langle \frac{\zeta}{r}\hat{e}_{r}\otimes\hat{e}_{r},\mu_{\delta}\right\rangle \,dx\\
 & \geq\int_{\Gamma}\left[\int_{y\in\ell_{s}}\int_{x\in B_\delta(x_0) }\zeta\frac{\delta}{r}(x)|\hat{e}_{r}(x)\cdot\hat{e}_{\theta}(y)|^{2}r\lambda_{\text{a.c.}}(y)\rho\left(\frac{x-y}{\delta}\right)\,\frac{dxd\mathcal{H}^{1}(y)}{\delta^{3}}\right]\,d\mathcal{H}^{1}(s)\\
 & \qquad+\int_{\Gamma}\left[\int_{y\in\ell_{s}}\int_{x\in B_\delta(x_0) }\zeta\frac{\delta}{r}(x)|\hat{e}_{r}(x)\cdot\hat{e}_{\theta}(y)|^{2}\lambda_{\text{sing}}(y)\rho\left(\frac{x-y}{\delta}\right)\,\frac{dxd\mathcal{H}^{1}(y)}{\delta^{3}}\right]\,d\vartheta(s)
\end{align*}
for all small enough $\delta >0$. Again, the lefthand side limits to zero by
\prettyref{eq:vanishing_mass}. Consider the terms on the right. The function $\zeta$ is locally uniformly positive a.e. Even though $\hat{e}_{r}(x)\cdot\hat{e}_{\theta}(y)=0$ when $x=y$, the
typical value of $|\hat{e}_{r}(x)\cdot\hat{e}_{\theta}(y)|$ is bounded away from zero. The kernel $\rho >0$ nearby zero by choice. Passing to the limit via Fatou's lemma proves that
\[
0=\int_{\Gamma}r\lambda_{\text{a.c.}}|_{\partial\ell_{s}}(x_{0})\,d\mathcal{H}^{1}(s)=\int_{\Gamma}\lambda_{\text{sing}}|_{\partial\ell_{s}}(x_{0})\,d\vartheta(s).
\]
As the integrands are non-negative they must vanish a.e. The first parts of \prettyref{eq:zeroth_order_trace_ac-1-1}
and \prettyref{eq:zeroth_order_trace_sing-1-1} are proved.

Finally, we must check that $\partial_{r}(r\lambda_{\text{a.c.}})$ and
$\partial_{r}\lambda_{\text{sing}}$ vanish at $x_{0}$. This last
part of the proof is directly analogous to that of \prettyref{cor:chars-meet-along-curve}.
In fact, it is so similar that we omit most of the details for brevity's
sake. First, observe that 
\begin{equation}
0\leq\partial_{r}(r\lambda_{\text{a.c.}})|_{\partial\ell_{s}}(x_{0})\quad\text{for }\mathcal{H}^{1}\text{-a.e. }s\in \Gamma\quad\text{and}\quad0\leq\partial_{r}\lambda_{\text{sing}}|_{\partial\ell_{s}}(x_{0})\quad\text{for }\vartheta\text{-a.e. }s\in \Gamma.\label{eq:non-negative_firstordertraces-1}
\end{equation}
These hold since the functions $r\lambda_{\text{a.c.}}$ and $\lambda_{\text{sing}}$ are non-negative, while their traces at $x_0$ were shown to vanish in the previous step. Second, note the matching conditions
\begin{equation}
0=\int_{\Gamma}\partial_{r}(r\lambda_{\text{a.c.}})|_{\partial\ell_{s}}(x_{0})\,d\mathcal{H}^{1}(s)=\int_{\Gamma}\partial_{r}\lambda_{\text{sing}}|_{\partial\ell_{s}}(x_{0})\,d\vartheta(s)\label{eq:matching_condition_circle}
\end{equation}
which arise from the first part of \prettyref{eq:wrinkling-PDEs} as $\hat{e}_{\theta}$ is discontinuous at $x_{0}$. Together, \prettyref{eq:non-negative_firstordertraces-1} and \prettyref{eq:matching_condition_circle} yield the remaining parts of \prettyref{eq:zeroth_order_trace_ac-1-1}
and \prettyref{eq:zeroth_order_trace_sing-1-1}. 

Let us briefly explain the proof of the matching conditions. Suppose $\psi\in C_{c}^{\infty}(V)$
has $\psi(x_{0})=1$. Combining the first part of \prettyref{eq:wrinkling-PDEs} with \prettyref{eq:Leb_disintegration} and \prettyref{eq:disintegration_3} yields that
\[
\int_{V}\psi\det\nabla\nabla p\,dx 
  =\int_{\Gamma}\left[\int_{\ell_{s}}-\frac{1}{2}\partial_{r}^{2}\psi r\lambda_{\text{a.c.}}\,d\mathcal{H}^{1}\right]\,d\mathcal{H}^{1}(s)+\int_{\Gamma}\left[\int_{\ell_{s}}-\frac{1}{2}\partial_{r}^{2}\psi\lambda_{\text{sing}}\,d\mathcal{H}^{1}\right]\,d\vartheta(s).
\]
Integrating by parts with the ODEs \prettyref{eq:ODE_ac} and \prettyref{eq:ODE_sing}
proves the identities 
\begin{align*}
\int_{\Gamma}\left[\int_{\ell_{s}}\partial_{r}^{2}\psi r\lambda_{\text{a.c.}}\,d\mathcal{H}^{1}\right]\,d\mathcal{H}^{1}(s) & =\int_{\Gamma}\partial_{r}(r\lambda_{\text{a.c.}})|_{\partial\ell_{s}}(x_{0})\,d\mathcal{H}^{1}(s)-2\int_{V}\psi\det\nabla\nabla p\,dx,\\
\int_{\Gamma}\left[\int_{\ell_{s}}\partial_{r}^{2}\psi\lambda_{\text{sing}}\,d\mathcal{H}^{1}\right]\,d\vartheta(s) & =\int_{\Gamma}\partial_{r}\lambda_{\text{sing}}|_{\partial\ell_{s}}(x_{0})\,d\vartheta(s).
\end{align*}
The desired conditions \prettyref{eq:matching_condition_circle}
 follow. \qed\end{proof}

\section{Application to shells with curvature of known sign\label{sec:examples}}

This final section combines all of our previous results to deduce the patterns seen in weakly curved, floating shells. 
In particular, we shall derive the diagrams in \prettyref{fig:library_of_patterns} and use them to demonstrate our method of stable lines. 
This should serve to complement the general presentation of the method in \prettyref{sec:method_of_characteristics}.
In addition to the basic assumptions in \prettyref{eq:A1a}, each of the examples we discuss will be subject to the simplifying hypotheses that 
\[
\Omega\text{ is simply connected}\quad\text{and}\quad\det\nabla\nabla p\text{ is a.e.\ of one sign}.
\]
Note when we refer to ``optimal'' $\mu$ we mean solutions of the primal problem(s) in \prettyref{eq:primal-1}. Equivalently, these are non-negative solutions of the boundary value problem \prettyref{eq:wrinkling-PDEs} where $\varphi$ solves the dual. Our earlier results show that, under the assumptions \vpageref{par:Assumptions}, such $\mu$ are nothing other than the defect measures of the almost minimizers of $E_{b,k,\gamma}$. As usual, any reference to almost minimizers is contingent upon the $\Gamma$-convergence in \prettyref{thm:gamma-lim}.

This section is organized as follows. We begin in \prettyref{subsec:Solution-formulas-optimal-Airys} with the general task of solving the dual problem
\begin{equation}
\max_{\substack{\varphi:\mathbb{R}^{2}\to\mathbb{R}\\
\varphi\text{ is convex}\\
\varphi=\frac{1}{2}|x|^{2}\text{ on }\mathbb{R}^{2}\backslash\Omega
}
}\,\int_{\Omega}(\varphi-\frac{1}{2}|x|^{2})\det\nabla\nabla p\,dx\label{eq:dual_pblm_apps}
\end{equation}
under the assumption that the shell is initially positively or negatively curved. We show how this boils down to finding either the largest or the smallest convex extension $\varphi_+$ or $\varphi_-$ of $\frac{1}{2}|x|^2$ into $\Omega$, and obtain two more or less explicit geometric optimization procedures for doing so. Thus, we prove \prettyref{prop:explicit_formulas}.

We then go on to the examples.  
\prettyref{subsec:Positively-curved-shells} treats various positively curved shells, including the ones depicted in Panel (a) of \prettyref{fig:library_of_patterns}. Applying \prettyref{cor:non-intersecting-chars}, we learn that optimal $\mu$ are uniquely determined on the ordered set $O$ of $\varphi_+$, and furthermore that they satisfy
\begin{equation}
\mu=\lambda\hat{\eta}\otimes\hat{\eta}\,dx\quad\text{on }O,\quad\text{where}\quad\begin{cases}
-\frac{1}{2\varrho}\partial_{\hat{\eta}^{\perp}}^{2}(\varrho\lambda)=\det\nabla\nabla p & \text{on }O\\
\varrho\lambda=0 & \text{at }\partial O\cap\partial\Omega
\end{cases}.\label{eq:conclusion_pos}
\end{equation}
Here, $\hat{\eta}$ is a suitable choice of normal to the stable lines of $\varphi_+$. These will turn out to extend between points on $\partial\Omega$ throughout the set $O$ where they are defined.

A parallel discussion of negatively curved shells is in \prettyref{subsec:Negatively-curved-shells}. 
We show how the stable lines of $\varphi_-$ follow the paths of quickest exit from $\Omega$, as indicated in Panel (b) of \prettyref{fig:library_of_patterns}. Such paths meet at the \emph{medial axis} 
\begin{equation}
M=\left\{ x\in\Omega:d_{\partial\Omega}(x)=|x-y|\text{ for multiple }y\in\partial\Omega\right\} \label{eq:medialaxis}
\end{equation}
shown in bold. Apparently, our negatively curved examples are such that their stable lines fill out the given shells. Applying \prettyref{cor:chars-meet-along-curve} or \prettyref{cor:chars-meet-at-point}, we consequently show that optimal $\mu$ are unique and that they satisfy
\begin{equation}
\mu=\lambda\nabla^\perp d_{\partial\Omega}\otimes\nabla^\perp d_{\partial\Omega}\,dx\quad\text{on }\Omega,\quad\text{where}\quad\begin{cases}
-\frac{1}{2\varrho}\partial_{\nabla d_{\partial\Omega}}^{2}(\varrho\lambda)=\det\nabla\nabla p & \text{on }\Omega\backslash\overline{M}\\
\varrho\lambda=\partial_{\nabla d_{\partial\Omega}}(\varrho\lambda)=0 & \text{at }\overline{M}
\end{cases}.\label{eq:conclusion_neg}
\end{equation} 
We close with a general conjecture on the (conditional) uniqueness of optimal $\mu$.

To be clear, the systems in \prettyref{eq:conclusion_pos} and \prettyref{eq:conclusion_neg}, and that we derive in the examples below, are only abbreviated versions of the ones implied by \prettyref{cor:non-intersecting-chars}-\prettyref{cor:chars-meet-at-point}: 
they indicate a situation where $\varrho\lambda$ restricts to $\mathcal{H}^{1}$-a.e.\ stable
line $\ell_{s}$ as the unique weak solution of an appropriate two-point
boundary value or Cauchy problem (obtained by restoring the $s$-dependences as in the corollaries). We refer to such abbreviated systems throughout.

\subsection{Optimal Airy potentials and their stable lines\label{subsec:Solution-formulas-optimal-Airys}}

We begin by proving \prettyref{prop:explicit_formulas}. Accordingly, we let $\Omega\subset \mathbb{R}^2$ be a bounded, Lipschitz domain that is simply connected, and let $p\in W^{2,2}(\Omega)$ be such that $\det \nabla \nabla p$ is a.e.\ of one sign. 
Define the functions $\varphi_\pm:\mathbb{R}^2\to\mathbb{R}$ by 
\begin{align*}
\varphi_{+}(x) & =\sup\left\{ \varphi(x):\varphi\text{ is convex on }\mathbb{R}^{2}\text{ and equals }\frac{1}{2}|\cdot|^{2}\text{ on }\mathbb{R}^{2}\backslash\Omega\right\} ,\\
\varphi_{-}(x) & =\inf\left\{ \varphi(x):\varphi\text{ is convex on }\mathbb{R}^{2}\text{ and equals }\frac{1}{2}|\cdot|^{2}\text{ on }\mathbb{R}^{2}\backslash\Omega\right\} 
\end{align*}
for all $x\in \mathbb{R}^2$. Clearly, if $\varphi$ is admissible in the dual problem \prettyref{eq:dual_pblm_apps} then 
\begin{equation}
\varphi_{-}(x)\leq\varphi(x)\leq\varphi_{+}(x)\quad\forall\,x. \label{eq:lower-upper-bds}
\end{equation}

\begin{lemma} \label{lem:explicit_formulas-1} The functions $\varphi_{+}$ and $\varphi_{-}$
are convex, and are equal to $\frac{1}{2}|x|^{2}$
on $\mathbb{R}^{2}\backslash\Omega$. Therefore, they are the largest and smallest convex extensions of $\frac{1}{2}|x|^2$ into $\Omega$. Furthermore, the formulas
\prettyref{eq:convexroof_inf-1} and \prettyref{eq:squareddistance}
hold: given $x\in\Omega$, 
\begin{equation}
\varphi_{+}(x)=\min_{\{y_{i}\}\subset\partial\Omega}\,\sum_{i=1}^{3}\theta_{i}\frac{1}{2}|y_{i}|^{2}\label{eq:convexroof_inf}
\end{equation}
where the minimization is taken over all pairs and triples $\{y_{i}\}\subset\partial\Omega$
satisfying 
\[
x=\sum_{i}\theta_{i}y_{i}\quad\text{where }\{\theta_{i}\}\subset(0,1)\text{ satisfies }\sum_{i}\theta_{i}=1;
\]
also 
\begin{equation}
\varphi_{-}(x)=\frac{1}{2}|x|^{2}-\frac{1}{2}d_{\partial\Omega}^{2}(x)\quad\text{where}\quad d_{\partial\Omega}(x)=\min_{y\in\partial\Omega}\,|x-y|.\label{eq:min-exit-time}
\end{equation}
\end{lemma} \begin{proof} The convexity of $\varphi_{+}$ is clear,
as the pointwise supremum of convex functions is convex. It is also clear that $\varphi_+ = \frac{1}{2}|x|^2$ outside of $\Omega$. The formula
\prettyref{eq:convexroof_inf} now follows from the dual characterization
of the convex envelope of a function as the infimum of convex combinations
of its graph (see, e.g., \cite[Theorem 2.35]{dacorogna2008direct}).
Indeed, we recognize from its definition that $\varphi_{+}$ is the largest convex function bounding the function $\mathscr{U}:\mathbb{R}^{2}\to\mathbb{R}$
equal to $\infty$ on $\Omega$ and $\frac{1}{2}|x|^{2}$ on $\mathbb{R}^{2}\backslash\Omega$ from below. That is, it is the convex envelope of $\mathscr{U}$. Applying the dual characterization, we get that 
\begin{align*}
\varphi_{+}(x) & =\inf\left\{ \sum_{i=1}^{3}\theta_{i}\mathscr{U}(y_{i}):y_{i}\in\mathbb{R}^{2}\text{ and }\theta_{i}\in[0,1]\text{ for }i=1,2,3,\ x=\sum_{i=1}^{3}\theta_{i}y_{i},\ \sum_{i=1}^{3}\theta_{i}=1\right\} \\
 & =\inf\left\{ \sum_{i=1}^{3}\theta_{i}\frac{1}{2}|y_{i}|^{2}:y_{i}\in\mathbb{R}^{2}\backslash\Omega\text{ and }\theta_{i}\in[0,1]\text{ for }i=1,2,3,\ x=\sum_{i=1}^{3}\theta_{i}y_{i},\ \sum_{i=1}^{3}\theta_{i}=1\right\}.
\end{align*}
The minimization can be parameterized by $\{y_{i}\}$, as once these have been chosen  $\{\theta_{i}\}$ are determined. And, as $\frac{1}{2}|x|^{2}$
is strictly convex, the minimizing $\{y_{i}\}\subset\partial\Omega$ whenever $x\in \Omega$. 
This proves \prettyref{eq:convexroof_inf}.

Now we discuss $\varphi_{-}$. We proceed in the opposite order, showing first that the function $\mathscr{L}:\mathbb{R}^{2}\to\mathbb{R}$
equal to $\frac{1}{2}|x|^{2}-\frac{1}{2}d_{\partial\Omega}^{2}$ on
$\Omega$ and $\frac{1}{2}|x|^{2}$ on $\mathbb{R}^{2}\backslash\Omega$
is convex, and then checking that it equals to $\varphi_-$.  
Note by its definition that
\[
\mathscr{L}(x)=\frac{1}{2}|x|^{2}-\frac{1}{2}\min_{y\notin \Omega}\, |x-y|^{2}=\max_{y\notin \Omega}\,x\cdot y-\frac{1}{2}|y|^{2}.
\]
Thus, $\mathscr{L}$ is convex as it is the pointwise supremum of affine functions. To finish, we must show that $\varphi\geq \mathscr{L}$ whenever $\varphi$ is
a convex extension of $\frac{1}{2}|x|^{2}$ into $\Omega$. Clearly this holds for $x\notin \Omega$, so fix some $x\in\Omega$. Let $y \in \partial\Omega$ be a point closest to $x$, and let $z$ be on the line segment from $x$ to $y$. Calling $t=|z-x|$, we obtain the lower bound
\[
\varphi\left(x+t\frac{y-x}{d_{\partial\Omega}(x)}\right)\geq\frac{(y-x)\cdot y}{d_{\partial\Omega}(x)}\left(t-d_{\partial\Omega}(x)\right)+\frac{1}{2}|y|^{2}=\frac{(y-x)\cdot y}{d_{\partial\Omega}(x)}t+\mathscr{L}(x).
\]
Note in the last step we used that $\mathscr{L}(x)=x\cdot y-\frac{1}{2}|y|^{2}$.
Setting $t=0$ yields the desired inequality $\varphi(x) \geq\mathcal{L}(x)$.
It follows from its definition that $\varphi_- = \mathcal{L}$. 
\qed\end{proof} 

\begin{corollary} The functions $\varphi_{+}$ and $\varphi_{-}$
solve the dual problem \prettyref{eq:dual_pblm_apps} respectively when $\det\nabla\nabla p\geq0$
and $\leq0$ a.e. Furthermore, if either of these inequalities is strict a.e.,
then \prettyref{eq:dual_pblm_apps} is only solved by the corresponding
$\varphi_{+}$ or $\varphi_{-}$. \end{corollary} 
\begin{proof}
We give the proof in the case that $\det \nabla \nabla p \geq 0$, as the other case is the same. Rearranging \prettyref{eq:lower-upper-bds} shows that $\varphi_+ - \varphi \geq 0$ whenever $\varphi$ is admissible for the dual. It follows that
\[
\int_\Omega (\varphi_+ - \frac{1}{2}|x|^2)\det\nabla\nabla p \geq \int_\Omega (\varphi - \frac{1}{2}|x|^2)\det\nabla\nabla p.
\]
As $\varphi_+$ is admissible by \prettyref{lem:explicit_formulas-1}, it is a maximizer. On the other hand, if $\varphi$ is a maximizer  it must be that
\[
\int_\Omega (\varphi_+ - \varphi)\det\nabla\nabla p = 0.
\] 
If in addition $\det \nabla \nabla p >0$ a.e., we see that $\varphi = \varphi_+$ so that no other maximizer exists.
\qed\end{proof}
\prettyref{prop:explicit_formulas} is proved. Note we also established the uniqueness in \prettyref{rem:uniqueness}.

\subsection{Positively curved shells\label{subsec:Positively-curved-shells}}

We are ready to solve for the patterns in
\prettyref{fig:library_of_patterns}. Here, we treat the positively curved shells shown in Panel (a), i.e., we let $p\in W^{2,2}(\Omega)$ satisfy
\[
\det\nabla\nabla p\geq0\quad\text{a.e.}
\]
We begin each example by producing 
the largest convex extension $\varphi_{+}$ of $\frac{1}{2}|x|^2$ into the given $\Omega$. 
We identify its singular, flattened, ordered, and unconstrained sets $\Sigma$, $F$, $O$, and $U$ as well as its stable lines $\{\ell_x\}$ following the definitions in \prettyref{subsec:Lines-of-wrinkling-classification}.  
Then, we show how to apply \prettyref{cor:non-intersecting-chars} to characterize optimal $\mu$. The end result is a proof that optimal $\mu$ are unique upon restriction to $O$, and that they satisfy a version of \prettyref{eq:conclusion_pos}.

Our first example is the positively curved ellipse in Panel (a) of \prettyref{fig:library_of_patterns}.
\begin{example} \label{ex:positively-curved-ellipse}(positively
curved ellipse) Let $0<b<a$ and take as the domain the ellipse 
\[
E=\left\{ (x_{1},x_{2}):\frac{x_{1}^{2}}{a^{2}}+\frac{x_{2}^{2}}{b^{2}}<1\right\} .
\]
We claim that 
\begin{equation}
\varphi_{+}(x)=\frac{1}{2}\left(b^{2}+(1-\frac{b^{2}}{a^{2}})x_{1}^{2}\right),\quad x\in E.\label{eq:Airystress_ellipse}
\end{equation}
It is straightforward to check using \prettyref{lem:restriction}
that \prettyref{eq:Airystress_ellipse} defines a convex extension
of $\frac{1}{2}|x|^{2}$ into $E$. It satisfies 
\[
\nabla\nabla\varphi_{+}=(1-\frac{b^{2}}{a^{2}})\hat{e}_{1}\otimes\hat{e}_{1}\,dx \quad\text{on }E
\]
so that its Hessian is non-negative and uniformly bounded, and it
equals to $\frac{1}{2}|x|^{2}$ at $\partial E$. Since 
\begin{equation}
\hat{\nu}\cdot[\nabla\varphi_{+}]=\frac{\left(\frac{x_{1}}{a^{2}},\frac{x_{2}}{b^{2}}\right)}{\sqrt{\frac{x_{1}^{2}}{a^{4}}+\frac{x_{2}^{2}}{b^{4}}}}\cdot\left(x_{1}-(1-\frac{b^{2}}{a^{2}})x_{1},x_{2}\right)=b^{2}\sqrt{\frac{x_{1}^{2}}{a^{4}}+\frac{x_{2}^{2}}{b^{4}}}>0\quad\text{at }\partial E\label{eq:extension_signed_ellipse}
\end{equation}
we conclude that \prettyref{eq:Airystress_ellipse}
is admissible. Now, we must verify it is the largest
convex extension. Given $x\in E$ there is a unique line containing it
and parallel to $\hat{e}_{2}$. That line intersects $\partial E$
at two points $\{(x_{1},\pm x_{2}(x_{1}))\}$, and
\[
x=\theta(x_{1},x_{2}(x_{1}))+(1-\theta)(x_{1},-x_{2}(x_{1}))\quad\text{for some }\theta\in(0,1).
\]
If $\varphi$ is any convex extension of $\frac{1}{2}|x|^{2}$ into
$E$ it follows that 
\[
\varphi(x)\leq\theta\frac{1}{2}|(x_{1},x_{2}(x_{1}))|^{2}+(1-\theta)\frac{1}{2}|(x_{1},-x_{2}(x_{1}))|^{2}=\frac{1}{2}\left(x_{1}^{2}+b^{2}(1-\frac{x_{1}^{2}}{a^{2}})\right).
\]
Thus \prettyref{eq:Airystress_ellipse} is indeed the formula for $\varphi_+$ on $E$. 

Having obtained $\varphi_{+}$, we note
 it partitions $E$ according to 
\[
O=E\quad\text{and}\quad \Sigma=F=U=\emptyset.
\]
In particular, the entire ellipse is ordered. Its stable lines $\{\ell_{x}\}$
are the lines referred to above or, rather, the portion of them within $E$. We finish by applying \prettyref{cor:non-intersecting-chars}
to characterize optimal $\mu$. 
Given any sufficiently small $\delta>0$, we claim the hypotheses \prettyref{eq:extend-to-bdry-hyp1}-\prettyref{eq:extend-to-bdry-hyp3}
 hold with $\varphi=\varphi_{+}$ and 
\[
V=\{x\in E:-a+\delta<x_{1}<a-\delta\},\quad\zeta=1-\frac{b^{2}}{a^{2}},\quad\text{and}\quad\hat{\eta}=\hat{e}_{1}.
\]
That $\hat{\nu}\cdot[\nabla\varphi_{+}]\geq c>0$ follows from \prettyref{eq:extension_signed_ellipse}.
The uniform transversality condition $|\hat{\tau}\cdot\hat{e}_1|\geq c(\delta)>0$
is clear. Note we introduced the cutoff length $\delta$ to deal with the fact that $|\hat{\tau}\cdot\hat{e}_1|\to 0$ as $x_{1}\to\pm a$. 
As the stable lines are parallel, the change of measure factor $\varrho$ from \prettyref{eq:Leb_disintegration} remains constant and non-zero along each of them; it is a function of $x_1$ whose exact form is immaterial and depends on the choice of the indexing curve $\Gamma$.
Applying
\prettyref{cor:non-intersecting-chars} and taking $\delta\to0$ we conclude that optimal $\mu$ are unique, and that they satisfy
\[
\mu=\lambda\hat{e}_{1}\otimes\hat{e}_{1}\,dx\quad\text{on }E,\quad\text{where}\quad\begin{cases}
-\frac{1}{2}\partial_{2}^{2}\lambda=\det\nabla\nabla p & \text{on }E\\
\lambda=0 & \text{at }\partial E
\end{cases}.
\]
This completes our discussion of the positively curved ellipse. 
\end{example}

Next, we consider the positively curved disc. The corresponding entry
in \prettyref{fig:library_of_patterns} is blank, indicating a complete
lack of stable lines. As a result, we will be able to construct infinitely
many optimal $\mu$. \begin{example} \label{ex:positively-curved-disc}(positively
curved disc) Let $a>0$ and take as the domain the disc 
\[
D=\left\{ (x_{1},x_{2}):x_{1}^{2}+x_{2}^{2}<a^{2}\right\} .
\]
Since $\partial D$ is a level set of $\frac{1}{2}|x|^{2}$, it follows
immediately from \prettyref{eq:convexroof_inf} that 
\[
\varphi_{+}(x)=\frac{1}{2}a^{2},\quad x\in D.
\]
Hence 
\[
\nabla\nabla\varphi_{+}=0\quad\text{on }D
\]
and the corresponding partition of $D$ degenerates into 
\[
U=D\quad\text{and}\quad\Sigma=F=O=\emptyset.
\]
The entire disc is unconstrained. The ordered set is empty and there are no stable lines.  Nevertheless, optimal $\mu$ can still be characterized as non-negative solutions of \prettyref{eq:wrinkling-PDEs}, which degenerates into the system 
\begin{equation}
\begin{cases}
-\frac{1}{2}\text{curl}\text{curl}\,\mu=\det\nabla\nabla p & \text{on }D\\
\left\langle \hat{\tau}\otimes\hat{\tau},\mu\right\rangle =0 & \text{at }\partial D
\end{cases}. \label{eq:degenerate_system_disc}
\end{equation}
Note the PDE holds in the sense of distributions, while the boundary conditions hold in the regularized sense, i.e., 
\begin{equation}\label{eq:needed-b.c.s-example-disc}
0=\lim_{\delta\to0}\,\int_{\partial D}|\left\langle \hat{\tau}\otimes\hat{\tau},\mu_{\delta}\right\rangle |\,ds
\end{equation}
where $\{\mu_\delta\}_{\delta>0}$ are the mollifications from \prettyref{eq:mollified_measure_defn-intro}. We used here  that $\hat{\nu}\cdot [\nabla \varphi_+] = a >0$ at $\partial D$.

As a boundary value problem, \prettyref{eq:degenerate_system_disc} is severely underdetermined. Here is an example of the non-uniqueness it permits: given any decomposition of $D$ into a disjoint family of open line segments with boundary points on $\partial D$, and letting $\hat{\eta}\in\text{Lip}_\text{loc}(D;S^1)$ be constant along and perpendicular to the segments, we claim that the measure 
\begin{equation}\label{eq:constructed-nonunique-solutions}
\mu=\lambda\hat{\eta}\otimes\hat{\eta}\,dx\quad\text{on }D,\quad\text{where}\quad\begin{cases}
-\frac{1}{2\varrho}\partial_{\hat{\eta}^{\perp}}^{2}(\varrho\lambda)=\det\nabla\nabla p & \text{in }D\\
\varrho\lambda=0 & \text{at }\partial D
\end{cases}
\end{equation}
is a solution of \prettyref{eq:degenerate_system_disc}. Hence, it is optimal. 
Note $\varrho:D\to(0,\infty)$ is defined via \prettyref{eq:Leb_disintegration}. 
That $\mu$ is indeed a solution can be checked using the methods of \prettyref{sec:method_of_characteristics}. Recall the system in \prettyref{eq:constructed-nonunique-solutions} stands for a family of two-point boundary value problems indexed by the given segments (just as in \prettyref{cor:non-intersecting-chars}).
The PDE in \prettyref{eq:degenerate_system_disc} now follows more or less directly from the notation introduced in \prettyref{subsec:wrinkles-as-chars} above  \prettyref{lem:disintegration}. 
Establishing the boundary conditions takes a bit more work --- our plan is to reverse the proof of \prettyref{cor:non-intersecting-chars}. 
Start by indexing the segments, which we still call $\{\ell_s\}$, with a curve $\Gamma\subset D$ such that $D = \cup_{s\in\Gamma} \ell_s$ where $s\mapsto \ell_s$ is one-to-one and $s\in \ell_s$.  It will suffice to take $\Gamma$ to be Lipschitz. 
Rewriting \prettyref{eq:needed-b.c.s-example-disc} using disintegration of measure, we must check that
\[
0=\lim_{\delta\to0}\,\int_{s\in\Gamma}\left[\int_{x\in\partial D}\int_{y\in \ell_s}\left|\hat{\tau}(x)\cdot\hat{\eta}(y)\right|^{2}\varrho\lambda (y)\rho\left(\frac{x-y}{\delta}\right)\,\frac{d\mathcal{H}^{1}(x)d\mathcal{H}^{1}(y)}{\delta^{2}}\right]d\mathcal{H}^{1}(s)
\]
where $\rho \in C_c^\infty(B_1)$ is non-negative and integrates to one. We do so via the dominated convergence theorem.

The bracketed integrals tend to zero $\mathcal{H}^{1}$-a.e.\ due to the boundary conditions in \prettyref{eq:constructed-nonunique-solutions}. 
Indeed,  $\varrho \lambda \in W^{2,1} (\ell_s)$ upon restriction to a.e.\ $\ell_s$, and the corresponding traces at $\partial \ell_s$ vanish. 
We proceed to dominate. 
Integrating the ODEs from \prettyref{eq:constructed-nonunique-solutions} along a.e.\ $\ell_s$ yields that $|\varrho\lambda(y)|\lesssim d(y,\partial \ell_s)||\varrho\det\nabla\nabla p||_{L^{1}(\ell_s,\mathcal{H}^{1})}$ for $y\in \ell_s$. 
If  $(x,y)\in\partial D \times \ell_s$ satisfies $|x-y|<\delta$, then  $d(y,\partial \ell_s)\lesssim_a (\delta/\mathcal{H}^{1}(\ell_{s}))\wedge\mathcal{H}^{1}(\ell_{s})$ and  $\left|\hat{\tau}(x)\cdot\hat{\eta}(y)\right|\lesssim_a  \mathcal{H}^{1}(\ell_{s})\vee \delta$. Hence,
\[
\int_{\substack{(x,y)\in\partial D\times\ell_{s}\\
|x-y|<\delta
}
}\left|\hat{\tau}(x)\cdot\hat{\eta}(y)\right|^{2}\varrho\lambda(y)\,\frac{d\mathcal{H}^{1}(x)d\mathcal{H}^{1}(y)}{\delta^{2}}
\lesssim_a ||\varrho\det\nabla\nabla p||_{L^{1}(\ell_{s},\mathcal{H}^{1})}\quad\text{for a.e.\ } s.
\]
The righthand side is integrable by \prettyref{eq:Leb_disintegration} since $p\in W^{2,2}$. Sending $\delta \to 0$ completes the proof.

Even though the  segments used above may remind of stable lines, we prefer not to call them as such. For one, the decomposition $D = \cup_{s\in\Gamma} \ell_s$ is not unique. Each such decomposition gives rise to a different optimal $\mu$ and, correspondingly, to a different sequence of almost minimizers of $E_{b,k,\gamma}$ under the assumptions \vpageref{par:Assumptions}. In some asymptotic sense, this is the opposite of stability.
Also, these are not the only optimal $\mu$. In particular, the set of solutions of \prettyref{eq:degenerate_system_disc} is convex. 
Taking convex combinations of the measures in \prettyref{eq:constructed-nonunique-solutions}, we deduce the existence of optimal $\mu$ that are everywhere rank two. The corresponding almost minimizers feature two-dimensional patterns instead of one-dimensional, wrinkling-like ones. 
We wonder if the disordered positively curved discs from \cite{tobasco2020principles}
can be understood using suitable solutions of \prettyref{eq:degenerate_system_disc}. If no such $\mu$ represents the observed patterns, it would neccessarily follow that they cannot be modeled as almost minimizers of $E_{b,k,\gamma}$, as least in the parameter regime \prettyref{eq:asymptoticregime}.  
\end{example}

The previous examples set the extremes: whereas a positively curved ellipse is totally
ordered, a positively curved disc is totally unconstrained (save for boundary data).
Our next two examples sit somewhere in-between. They address the triangle, square, and rectangle from Panel (a) of \prettyref{fig:library_of_patterns}.
What distinguishes the former shapes from the latter is the fact that
triangles and squares admit inscribed circles, whereas rectangles
do not. \begin{example} \label{ex:positivelycurvedtangentialpoly}(positively
curved tangential polygons) A \emph{tangential polygon} is a convex
polygon that admits an inscribed circle, known as its \emph{incircle}.
Every regular polygon is tangential; more generally, a convex polygon
is tangential if and only if its angle bisectors intersect at a distinguished
point. This point is called the \emph{incenter}, being the center
of the incircle just defined.  Given a tangential polygon $P$, we
call its \emph{contact polygon} $P'$ the convex polygon whose vertices are the points
of contact of the incircle with $P$. Thus, $P$ decomposes into a disjoint union of its contact polygon and finitely many leftover isosceles triangles, one for each vertex. 

Now let $P$ be a tangential polygon with vertices $a_{1},\dots,a_{n}\in\mathbb{R}^{2}$
and interior angles $\alpha_{1},\dots,\alpha_{n}$. 
Let its incenter be at the origin, and call the radius of its incircle $a$. Let $P'$ be the contact polygon  defined
above. The remainder $P\backslash P'$ divides into $n$ isosceles triangles, which
we label as $T_{i}$ for $i=1,\dots,n$. The labeling is such that
the $i$th vertex of the original polygon $a_{i}$ is a vertex of
the $i$th triangle $T_{i}$. For use in what follows,
we take $P'$ to be closed (relative to $P$) and
let each $T_{i}$ be open. After a fairly straightforward but somewhat
lengthy argument, one finds that 
\[
\varphi_{+}(x)=\begin{cases}
\frac{1}{2}a^{2} & x\in P'\\
\frac{1}{2}\left(\left(x\cdot\hat{a}_{i}\right)^{2}+\tan^{2}(\frac{\alpha_{i}}{2})\left(|a_{i}|-x\cdot\hat{a}_{i}\right)^{2}\right) & x\in T_{i},\ i=1,\dots,n
\end{cases}
\]
for $x\in P$. From this it follows that 
\[
\nabla\nabla\varphi_{+}=\sum_{i=1}^{n}\left(1+\tan^{2}(\frac{\alpha_{i}}{2})\right)\hat{a}_{i}\otimes\hat{a}_{i}\indicator{T_{i}}\,dx\quad\text{on }P.
\]
The stated absolute continuity follows from the fact that $P'$ is the contact polygon of $P$.
Regarding the partition implied by $\varphi_{+}$, we find that 
\[
O=\cup_{i}T_{i},\quad U=P',\quad\text{and}\quad \Sigma=F=\emptyset.
\]
The triangles $T_{i}$ are ordered, whereas the contact polygon $P'$
is unconstrained. The stable lines belonging to the $i$th triangle $T_{i}$ lie perpendicular
to $\hat{a}_{i}$ and extend from $\partial P$ to $\partial P$.

Having identified $\varphi_{+}$, we proceed to apply
\prettyref{cor:non-intersecting-chars} to characterize optimal $\mu$.
Note \prettyref{eq:extend-to-bdry-hyp1}-\prettyref{eq:extend-to-bdry-Liphyp}
holds with $\varphi=\varphi_{+}$ and 
\[
V=T_{i},\quad\zeta=1+\tan^{2}(\frac{\alpha_{i}}{2}),\quad\text{and}\quad\hat{\eta}=\hat{a}_{i}
\]
for $i=1,\dots,n$. The uniform transversality condition from \prettyref{eq:extend-to-bdry-hyp3}
is easily checked, as  $|\hat{\tau}\cdot\hat{a}_{i}|\geq c>0$.
Since 
\[
\nabla\varphi_{+}=x\cdot\hat{a}_{i}\hat{a}_{i}-\tan^{2}(\frac{\alpha_{i}}{2})\left(|a_{i}|-x\cdot\hat{a}_{i}\right)\hat{a}_{i},\quad x\in T_{i}
\]
we see that 
\[
\hat{\nu}\cdot[\nabla\varphi_{+}]=x\cdot\hat{a}_{i}^{\perp}\hat{\nu}\cdot\hat{a}_{i}^{\perp}+\tan^{2}(\frac{\alpha_{i}}{2})\left(|a_{i}|-x\cdot\hat{a}_{i}\right)\hat{\nu}\cdot\hat{a}_{i}\geq0\quad\text{at }\partial T_{i}\cap\partial P.
\]
Evidently, the first part of \prettyref{eq:extend-to-bdry-hyp3}
fails for $V$ as it allows for $x\to a_{i}$.
However, this is easy to fix: as in \prettyref{ex:positively-curved-ellipse},
we can introduce a small cutoff length $\delta>0$ and modify $V$ such that
$\hat{\nu}\cdot[\nabla\varphi_{+}]\geq c(\delta)>0$. Note $\varrho$ is constant along the stable lines as they remain parallel within each $T_i$.
Applying \prettyref{cor:non-intersecting-chars} and taking $\delta\to0$ yields that 
\[
\mu=\lambda_{i}\hat{a}_{i}\otimes\hat{a}_{i}\,dx\quad\text{on }T_{i},\quad\text{where}\quad\begin{cases}
-\frac{1}{2}\partial_{\hat{a}_{i}^{\perp}}^{2}\lambda_{i}=\det\nabla\nabla p & \text{on }T_{i}\\
\lambda=0 & \text{at }\partial T_{i}\cap\partial P
\end{cases}\quad\text{for }i=1,\dots,n.
\]
In particular, any two optimal $\mu$  agree upon restriction
to $\cup_{i}T_{i}$. Much less is known at present regarding $\mu$ on $P'$. \end{example} Most polygons
do not admit an inscribed circle, i.e., they fail to be tangential.
Nevertheless, the arguments appearing in the previous example can
be adapted to handle a more general case. \begin{example} \label{ex:positively-curved-rectangle}(positively
curved rectangle) Let $0<b<a$ and consider the rectangle 
\[
R=\left\{ (x_{1},x_{2}):-a<x_{1}<a,-b<x_{2}<b\right\} .
\]
Although $R$ does not admit an inscribed circle, there does exist
a one-parameter family of maximally contained circles given by 
\[
C_{t}=\left\{ (x_{1},x_{2}):(x_{1}-t)^{2}+x_{2}^{2}=b^{2}\right\} ,\quad-(b-a)\leq t\leq b-a.
\]
The left and rightmost circles 
\[
C_{\text{l}}=C_{-(b-a)}\quad\text{and}\quad C_{\text{r}}=C_{b-a}
\]
play a role analogous to that of the incircle above. Both $C_{\text{l}}$
and $C_{\text{r}}$ touch $R$ at exactly three points: call the triangles
formed by these points $T_{\text{l}}$ and $T_{\text{r}}$. The remainder
is made up of four $45-45-90$ triangles $T_{\text{nw}}$, $T_{\text{sw}}$,
$T_{\text{se}}$, and $T_{\text{ne}}$ and one sub-rectangle $R_{\text{c}}$.
The subscripts $nw$ and so on indicate location as on a compass.
It will probably be helpful to look at \prettyref{fig:library_of_patterns}.
There, $T_{\text{l}}$ and $T_{\text{r}}$ are indicated in blank,
whereas the remaining triangles $T_{\text{nw}}$, $T_{\text{sw}}$,
$T_{\text{se}}$, $T_{\text{ne}}$ and sub-rectangle $R_{\text{c}}$
are drawn with stripes. For use with what follows,
we take each of these to be open with the exception of $T_{\text{l}}$ and
$T_{\text{r}}$, which we take to be closed relative to $R$.

All this being said, we claim that the largest convex extension $\varphi_{+}$
of $\frac{1}{2}|x|^{2}$ into $R$ is given by solving 
\begin{equation}
\partial_{(1,-1)}^{2}\varphi_{+}=0\quad\text{on }T_{\text{ne}}\cup T_{\text{sw}},\quad\partial_{(1,1)}^{2}\varphi_{+}=0\quad\text{on }T_{\text{se}}\cup T_{\text{nw}},\quad\partial_{2}^{2}\varphi_{+}=0\quad\text{on }R_{\text{c}},\quad\nabla\nabla\varphi_{+}=0\quad\text{on }T_{\text{l}}^{\circ}\cup T_{\text{r}}^{\circ}
\label{eq:PDEconvexext}
\end{equation}
with the boundary data 
\begin{equation}
\varphi_{+}=\frac{1}{2}|x|^{2}\quad\text{at }\partial R.\label{eq:PDEconvexext-bdrydata}
\end{equation}
The function defined by \prettyref{eq:PDEconvexext} and \prettyref{eq:PDEconvexext-bdrydata}
is piecewise quadratic, and is affine upon restriction to $T_{\text{l}}$
and $T_{\text{r}}$. Clearly, it is a convex extension of $\frac{1}{2}|x|^{2}$
into $R$. To see it is the largest one, we must show it yields
an upper bound on any other convex extension $\varphi$. Evidently
$\varphi\leq\varphi_{+}$ on $T_{\text{l}}$ and $T_{\text{r}}$ since they are equal at their vertices. 
Now let $x\in R\backslash(T_{\text{l}}\cup T_{\text{r}})$ and consider the largest open line segment containing $x$ on which $\varphi_{+}$
is affine (see the corresponding entry
in \prettyref{fig:library_of_patterns}). The boundary of each
such segment consists of two points $y_{1},y_{2}\in\partial R$. By
convexity, 
\[
\varphi(x)\leq\theta\frac{1}{2}|y_{1}|^{2}+(1-\theta)\frac{1}{2}|y_{2}|^{2}=\varphi_{+}(x)
\]
where $x=\theta y_{1}+(1-\theta)y_{2}$. Hence, \prettyref{eq:PDEconvexext}
and \prettyref{eq:PDEconvexext-bdrydata} indeed define the largest
convex extension of $\frac{1}{2}|x|^{2}$ into $R$.

Moving on, we see that 
\[
\nabla\nabla\varphi_{+}=(1,1)\otimes(1,1)\indicator{T_{\text{ne}}\cup T_{\text{sw}}}\,dx+(1,-1)\otimes(1,-1)\indicator{T_{\text{se}}\cup T_{\text{nw}}}\,dx+(1,0)\otimes(1,0)\indicator{R_{\text{c}}}\,dx\quad\text{on }R.
\]
Hence, $R$ is partitioned by $\varphi_{+}$ according as 
\[
O=\left(\cup_{\substack{\alpha \in \{n,s\}\\ \beta \in \{e,w\}}}T_{\alpha\beta}\right)\cup R_{\text{c}},\quad U=T_{\text{l}}\cup T_{\text{r}},\quad\text{and}\quad \Sigma=F=\emptyset.
\]
The stable lines $\{\ell_{x}\}$ are parallel to $(1,-1)$ on $T_{\text{ne}}\cup T_{\text{sw}}$,
$(1,1)$ on $T_{\text{se}}\cup T_{\text{nw}}$, and $(0,1)$ on $R_{\text{c}}$, 
and they extend between pairs of boundary points. 
So, we can apply \prettyref{cor:non-intersecting-chars} to characterize optimal
$\mu$ on $O$. We leave the details to the reader this time, and simply point out that $\hat{\nu}\cdot[\nabla\varphi_{+}]\geq0$
at $\partial R$, and that this bound degenerates only at the vertices of $R$. Uniform transversality is clear; also, $\varrho$ is constant along each stable line as they are parallel within the connected components of $O$. The conclusion of \prettyref{cor:non-intersecting-chars} is 
that 
\begin{align*}
 & \mu=\lambda_{\text{nesw}}(\frac{1}{\sqrt{2}},\frac{1}{\sqrt{2}})\otimes(\frac{1}{\sqrt{2}},\frac{1}{\sqrt{2}})\,dx\quad\text{on }T_{\text{ne}}\cup T_{\text{sw}},\quad\ \ \mu=\lambda_{\text{senw}}(\frac{1}{\sqrt{2}},\frac{-1}{\sqrt{2}})\otimes(\frac{1}{\sqrt{2}},\frac{-1}{\sqrt{2}})\,dx\quad\text{on }T_{\text{se}}\cup T_{\text{nw}},\\
 & \text{and}\quad\ \ \mu=\lambda_{\text{c}}(1,0)\otimes(1,0)\,dx\quad\text{on }R_{\text{c}},
\end{align*}
where
\begin{align*}
 & \begin{cases}
-\frac{1}{2}\partial_{(\frac{1}{\sqrt{2}},-\frac{1}{\sqrt{2}})}^{2}\lambda_{\text{nesw}}=\det\nabla\nabla p & \text{in }T_{\text{ne}}\cup T_{\text{sw}}\\
\lambda_{\text{nesw}}=0 & \text{at }\partial(T_{\text{ne}}\cup T_{\text{sw}})\cap\partial R
\end{cases},\quad\ \ \begin{cases}
-\frac{1}{2}\partial_{(\frac{1}{\sqrt{2}},\frac{1}{\sqrt{2}})}^{2}\lambda_{\text{senw}}=\det\nabla\nabla p & \text{in }T_{\text{se}}\cup T_{\text{nw}}\\
\lambda_{\text{senw}}=0 & \text{at }\partial(T_{\text{se}}\cup T_{\text{nw}})\cap\partial R
\end{cases},\\
 & \text{and}\quad\ \ \begin{cases}
-\frac{1}{2}\partial_{2}^{2}\lambda_{c}=\det\nabla\nabla p & \text{in }R_{\text{c}}\\
\lambda_{\text{c}}=0 & \text{at }\partial R_{\text{c}}\cap\partial R
\end{cases}.
\end{align*}
Optimal $\mu$ are uniquely determined on $T_{\text{ne}}\cup T_{\text{sw}}$,
$T_{\text{se}}\cup T_{\text{nw}}$, and $R_{\text{c}}$. \end{example}

In the previous examples, stable lines ended up being
parallel within each connected component of $O$. Our last positively curved example
exhibits non-parallel stable lines. It is the half-disc from Panel (a) of \prettyref{fig:library_of_patterns}. \begin{example} \label{ex:positively-curved-half-disc}(positively
curved half-disc) Consider a disc with radius $a>0$ and center $(0,a)$,
and let 
\[
D_{+}=\left\{ (x_{1},x_{2}):x_{1}^{2}+(x_{2}-a)^{2}<a^{2},x_{2}>a\right\} .
\]
Let $(r,\theta)$ denote polar coordinates about $0$. Given $x\in D_{+}$,
consider the ray parallel to $\hat{e}_r(x)$ that passes through $x$ and begins at the origin.
This ray intersects $\partial D_{+}$ at two points, which we label as $p(x)$ and
$q(x)$ where $|p|<|q|$.  
Note there always holds $|p||q|=2a^{2}$. 
Using this, it is easy to check the identities 
\begin{equation}
|p|=\frac{a}{\sin\theta}\quad\text{and}\quad|q|=2a\sin\theta\label{eq:usebelow}
\end{equation}
which will come in handy below.

We now claim that 
\begin{equation}
\varphi_{+}(r,\theta)=\frac{1}{2}(|p|+|q|)r-\frac{1}{2}|p||q|\quad\text{on }D_{+}.\label{eq:convexroof_halfdisc}
\end{equation}
As we show at the end of this example, 
\[
\nabla\nabla\varphi_{+}=\frac{a}{r\sin^{3}\theta}\hat{e}_{\theta}\otimes\hat{e}_{\theta}\,dx \quad\text{on }D_{+}
\]
so that the partition of $D_{+}$ is given simply by 
\[
O=D_{+}\quad\text{and}\quad \Sigma=F=U=\emptyset.
\]
That is, the half-disc is totally ordered. Its stable lines  $\{\ell_{x}\}$ are
the portions of the rays described above within $D_+$. To identify $\mu$ we apply \prettyref{cor:non-intersecting-chars}
with $\varphi=\varphi_{+}$ and with 
\[
V=D_{+},\quad\zeta=\frac{a}{r\sin^{3}\theta},\quad\text{and}\quad\hat{\eta}=\hat{e}_{\theta}.
\]
The hypotheses \prettyref{eq:extend-to-bdry-hyp1}-\prettyref{eq:extend-to-bdry-Liphyp} are not difficult to check, but again we must be careful about \prettyref{eq:extend-to-bdry-hyp3}.
Note that 
\[
\nabla\varphi_{+}=\frac{1}{2}(|p|+|q|)\hat{e}_{r}.
\]
Hence, 
\[
\hat{\nu}\cdot[\nabla\varphi_{+}]=-\hat{e}_{2}\cdot((x_{1},a)-\frac{1}{2}(|p|+|q|)\hat{e}_{r})=\frac{1}{2}(|q|-|p|)\sin\theta\geq0
\]
at the bottom part of $\partial D_{+}$, while at the top part there holds 
\[
\hat{\nu}\cdot[\nabla\varphi_{+}]=(\frac{x_{1}}{a},\frac{x_{2}}{a}-1)\cdot\left(r-\frac{1}{2}(|p|+|q|)\right)(\cos\theta,\sin\theta)=\frac{1}{2}\left(|q|-|p|\right)\sin\theta\geq0.
\]
Both inequalities are strict away from the corners where $|p|=|q|$. Thus, we can  introduce a cutoff
length $\delta>0$ to be sent to zero as in the other examples. We get that $\hat{\nu}\cdot[\nabla\varphi_{+}]\geq c(\delta)>0$.
The uniform transversality condition $|\hat{\tau}\cdot\hat{e}_{\theta}|\geq c>0$
holds. This time, the change of measure factor from \prettyref{eq:Leb_disintegration} satisfies $\varrho=c_\Gamma(\theta)r$ where the exact form of $c_\Gamma$ depends on the choice of indexing curve $\Gamma$. This is because the stable lines follow rays. 
Applying \prettyref{cor:non-intersecting-chars} and taking $\delta\to0$ proves that optimal $\mu$ are unique, and  that they satisfy
\[
\mu=\lambda\hat{e}_{\theta}\otimes\hat{e}_{\theta}\,dx\quad\text{on }D_{+},\quad\text{where}\quad\begin{cases}
-\frac{1}{2r}\partial_{r}^{2}(r\lambda)=\det\nabla\nabla p & \text{in }D_{+}\\
\lambda=0 & \text{at }\partial D_{+}
\end{cases}
\]
for the positively curved half-disc.

We finish by proving that \prettyref{eq:convexroof_halfdisc}
is indeed the largest convex extension of $\frac{1}{2}|x|^{2}$ into
$D_{+}$. Making use of convexity along the rays, one easily concludes
that $\varphi\leq\varphi_{+}$ as in the other examples. Here, we
focus on checking that the given $\varphi_{+}$ is a convex extension after all. We apply \prettyref{lem:restriction}.
First, note that $\varphi_{+}=\frac{1}{2}|x|^{2}$ for $x\in\partial D_{+}$.
In particular, when $r=|p|$ it follows from \prettyref{eq:convexroof_halfdisc}
that 
\[
\varphi_{+}(r,\theta)=\frac{1}{2}(|p|+|q|)|p|-\frac{1}{2}|p||q|=\frac{1}{2}|p|^{2}=\frac{1}{2}r^{2},
\]
and similarly for $r=|q|$. That $\hat{\nu}\cdot[\nabla\varphi_{+}]\geq0$
at $\partial D_{+}$ was shown above. Finally, we check that
$\nabla\nabla\varphi_{+}\geq0$ on $D^{+}$. Note that 
\[
\varphi_{+}(r,\theta)=b_{1}(\theta)r+b_{0},\quad\text{where}\quad b_{0}=-\frac{1}{2}|p||q|=-a^{2}\quad\text{and}\quad b_{1}=\frac{1}{2}(|p|+|q|)=\frac{a}{2}(2\sin\theta+\frac{1}{\sin\theta})
\]
due to \prettyref{eq:usebelow}. Differentiating twice yields that
\[
\nabla\nabla\varphi_{+}=\frac{1}{r}(b_{1}''+b_{1})\hat{e}_{\theta}\otimes\hat{e}_{\theta}=\frac{1}{r}\frac{a}{\sin^{3}\theta}\hat{e}_{\theta}\otimes\hat{e}_{\theta}\geq0.
\]
The lemma now implies that $\varphi_{+}$ is admissible
and the example is complete. \end{example}

Before moving on to the negatively curved examples, we pause to reflect on the fact that the singular set $\Sigma$ turned out to be empty in each of the examples above.
Of course, this is related to the regularity of $\varphi_+$ and, ultimately, to the shape of $\partial\Omega$. As an example of what can be proved, we note that if $\Omega$ is uniformly convex and $\partial\Omega \in C^{1,1}$, then the results of \cite{rauch1977dirichlet,trudinger1984second} on Alexandrov solutions to the Dirichlet problem $\det\nabla\nabla \varphi=0$ in $\Omega$ and $\varphi = \frac{1}{2}|x|^2$ at $\partial\Omega$ imply that $\varphi_{+}\in C^{1,1}_\text{loc}(\Omega)$. Of course,  $\Sigma = \emptyset$ in such a case.

\subsection{Negatively curved shells\label{subsec:Negatively-curved-shells}}

We turn to the patterns in Panel (b) of \prettyref{fig:library_of_patterns}.
These were drawn assuming that $p\in W^{2,2}(\Omega)$ satisfies
\[
\det\nabla\nabla p\leq0\quad\text{a.e.}
\]
For each specified shell we obtain the smallest convex extension
$\varphi_{-}$ of $\frac{1}{2}|x|^2$ into $\Omega$, which requires solving for the boundary distance function $d_{\partial\Omega}$ from \prettyref{eq:min-exit-time}. Its singular, flattened, ordered, and unconstrained sets $\Sigma$, $F$, $O$, and $U$  follow, as do  its stable lines $\{\ell_x\}$. Again, we refer to \prettyref{subsec:Lines-of-wrinkling-classification} for the relevant definitions.

Let us comment briefly on the role that the medial axis $M$ from \prettyref{eq:medialaxis} plays.
As our examples will show, the stable lines of $\varphi_-$ follow the paths of quickest exit from $\Omega$, i.e., they lie parallel to $\nabla d_{\partial\Omega}$ where they exist. 
Based on this and the formula for $\varphi_-$ in \prettyref{eq:min-exit-time}, it is reasonable to expect that
\begin{equation}
\nabla\nabla\varphi_{-}=\left(1-d_{\partial\Omega}\Delta d_{\partial\Omega}\right)\nabla^{\perp}d_{\partial\Omega}\otimes\nabla^{\perp}d_{\partial\Omega}\,dx+d_{\partial\Omega}\left|[\nabla d_{\partial\Omega}]\right|\hat{\nu}_{M}\otimes\hat{\nu}_{M}\,\mathcal{H}^{1}\lfloor M\quad\text{on }\Omega\label{eq:negative_Hessian_formula}
\end{equation}
so long as $M$ is regular enough. (For a general discussion on the regularity of the medial axis, see \cite{li2005distance}.)
We show a version of \prettyref{eq:negative_Hessian_formula} in each example below, and use it to identify the aforementioned partition and the stable lines. Finally, we prove that optimal $\mu$ are unique and that they satisfy \prettyref{eq:conclusion_neg} by applying \prettyref{cor:chars-meet-along-curve} or \prettyref{cor:chars-meet-at-point}.

We start with the negatively curved disc from Panel (b) of \prettyref{fig:library_of_patterns}. \begin{example}
\label{ex:negatively-curved-disc}(negatively curved disc) Let 
\[
D=\left\{ (x_{1},x_{2}):x_{1}^{2}+x_{2}^{2}<a^{2}\right\} 
\]
and observe its medial axis is the singleton 
\[
M=\{0\}.
\]
Using polar coordinates $(r,\theta)$, the boundary distance
function is simply $d_{\partial D}=a-r$ and hence 
\[
\varphi_{-}=ar-\frac{1}{2}a^{2}\quad\text{on }D.
\]
Differentiating yields that 
\[
\nabla\nabla\varphi_{-}=\frac{a}{r}\hat{e}_{\theta}\otimes\hat{e}_{\theta}\,dx\quad\text{on }D.
\]
Even though $\nabla\nabla\varphi_{-}\ll dx$, its density
is not square integrable on any neighborhood of $M$. Therefore, $D$ is partitioned by $\varphi_-$ according as
\[
\Sigma=\{0\},\quad O=D\backslash \{0\},\quad\text{and}\quad F=U=\emptyset.
\]
The stable lines form rays parallel to $\nabla d_{\partial D}=-\hat{e}_{r}$. 

Next, we apply \prettyref{cor:chars-meet-at-point} to identify optimal
$\mu$. Its hypothesis \prettyref{eq:V-meet-at-point} holds with 
\[
x_{0} = 0\quad\text{and}\quad V=D.
\]
We immediately conclude that the unique optimal $\mu$ for the negatively curved disc satisfies
\[
\mu=\lambda\hat{e}_{\theta}\otimes\hat{e}_{\theta}\,dx\quad\text{on }D,\quad\text{where}\quad\begin{cases}
-\frac{1}{2r}\partial_{r}^{2}(r\lambda)=\det\nabla\nabla p & \text{on }D\backslash\{0\}\\
r\lambda=\partial_{r}(r\lambda)=0 & \text{at }0
\end{cases}.
\]
In particular, the almost minimizers of $E_{b,k,\gamma}$ must exhibit an (approximately) azimuthally symmetric response, of course subject to the conditions  \vpageref{par:Assumptions} under which our $\Gamma$-convergence results hold.

Before moving on, we note that a similar result can be proved for the case of a flat disc
attached to a weakly curved spherical substrate --- a model problem
that has been the focus of much previous research, including at least
\cite{bella2017wrinkling,davidovitch2019geometrically,hohlfeld2015sheet}.
The conclusion is that optimal $\mu$ are uniquely
determined, absolutely continuous, and parallel to $\hat{e}_{\theta}\otimes\hat{e}_{\theta}$ with a density as above.
As far as we know, this yields the first mathematically rigorous proof that azimuthal wrinkling
is energetically preferred in a problem absent tensile loads. That
azimuthal wrinkling should be preferred has often been explained as a consequence of
symmetry (with the notable exception of \cite{bella2017wrinkling}
where it was derived via minimization, albeit for
a problem on the borderline between tension- and curvature-driven).
Just because a shell is naturally symmetric does
not mean that it should remain so when pressed onto a substrate, even if it has the same symmetries as the shell: indeed, \prettyref{ex:positively-curved-disc} shows that
a positively curved disc confined to the plane admits infinitely
many non-symmetric optimal $\mu$ and, correspondingly, infinitely
many non-symmetric almost minimizers. Whether or not global minimizers must
exhibit symmetry remains unknown. \end{example}

Our next example concerns the ellipse in Panel (b) of \prettyref{fig:library_of_patterns}.
In lieu of producing an exact formula for $d_{\partial\Omega}$, we
will make use of the following fact: $d_{\partial\Omega}$ is concave if
and only if $\Omega$ is convex \cite{armitage1985convexity}. Note
this is an example where the medial axis is strictly smaller than
the singular set. 
\begin{example} \label{ex:negatively-curved-ellipse}(negatively
curved ellipse) Let 
\[
E=\left\{ (x_{1},x_{2}):\frac{x_{1}^{2}}{a^{2}}+\frac{x_{2}^{2}}{b^{2}}<1\right\} 
\]
where $0<b<a$. The boundary distance function $d_{\partial E}$ is
smooth off of the closure of the medial axis 
\[
\overline{M}=\left\{ (x_{1},0):|x_{1}|\leq a(1-\frac{b^{2}}{a^{2}})\right\} ,
\]
and we find that 
\[
\nabla\nabla\varphi_{-}=\left(1-d_{\partial E}\Delta d_{\partial E}\right)\nabla^{\perp}d_{\partial E}\otimes\nabla^{\perp}d_{\partial E}\,dx+d_{\partial E}\left|[\nabla d_{\partial E}]\right|\hat{\nu}_{M}\otimes\hat{\nu}_{M}\,\mathcal{H}^{1}\lfloor M\quad\text{on }E
\]
as anticipated in \prettyref{eq:negative_Hessian_formula}. The function $\varphi_-$ partitions $E$ according as 
\[
\Sigma=\overline{M},\quad O=E\backslash\overline{M},\quad\text{and}\quad F=U=\emptyset
\]
and the stable lines $\{\ell_x\}$ run parallel to $\nabla d_{\partial E}$. We
now claim that there is a single optimal $\mu$ for the negatively
curved ellipse, and that it satisfies 
\begin{equation}
\mu=\lambda\nabla^{\perp}d_{\partial E}\otimes\nabla^{\perp}d_{\partial E}\,dx\quad\text{on }E,\quad\text{where}\quad\begin{cases}
-\frac{1}{2\varrho}\partial_{\nabla d_{\partial E}}^{2}(\varrho\lambda)=\det\nabla\nabla p & \text{on }E\backslash\overline{M}\\
\varrho\lambda=\partial_{\nabla d_{\partial E}}(\varrho\lambda)=0 & \text{at }\overline{M}
\end{cases}\label{eq:optimaldefect_neg_ellipse}
\end{equation}
and where $\varrho:E\backslash\overline{M}\to(0,\infty)$ is given
by \prettyref{eq:Leb_disintegration}. We proceed 
in two steps: first, we use \prettyref{cor:chars-meet-along-curve}
to identify $\mu$ off of the half-open line segments 
\[
L_{\pm}=\left\{ (x_{1},0)\in E:\pm x_{1}\geq a(1-\frac{b^{2}}{a^{2}})\right\} ,
\]
and then we check that $\mu=0$ on $L_{\pm}$ by a separate argument.

Recall $\ell_x$ denotes the stable line through $x$, which is given here by $\ell_x = \{\theta z + (1-\theta) y:\theta\in(0,1)\}$ for $z\in\overline{M}$ and $y\in\partial E$. Given any small enough $\delta > 0$, let
\[
M_\delta = \left\{(x_1,0):|x_1| < a(1-\frac{b^2}{a^2})-\delta\right\}
\]
and observe the assumptions \prettyref{eq:along-sing-curve-hyp1}-\prettyref{eq:along-sing-curve-hyp4}
of \prettyref{cor:chars-meet-along-curve} hold with $\varphi = \varphi_-$ and
\[
V=M_\delta \cup \{x\in E\backslash\overline{M} : \partial\ell_x \cap M_\delta \neq \emptyset \},\quad\zeta=1-d_{\partial E}\Delta d_{\partial E},\quad\hat{\eta}=\nabla^{\perp}d_{\partial E}.
\]
In particular as $E$ is convex,  $d_{\partial E}$ is
concave so that $\zeta\geq1$. Both $|[\nabla d_{\partial E}]|$ and $|\hat{\tau}_{M}\cdot\nabla^{\perp}d_{\partial E}|_{M_{\pm}}|$ are bounded below by some $c(\delta)>0$ uniformly on $M_\delta$. 
Applying \prettyref{cor:chars-meet-along-curve}
and sending $\delta\to0$ yields that 
\[
\mu=\lambda\nabla^{\perp}d_{\partial E}\otimes\nabla^{\perp}d_{\partial E}\,dx\quad\text{on }E\backslash\left\{ L_{+}\cup L_{-}\right\}, \quad\text{where}\quad\begin{cases}
-\frac{1}{2\varrho}\partial_{\nabla d_{\partial E}}^{2}(\varrho\lambda)=\det\nabla\nabla p & \text{on }E\backslash\left\{ L_{+}\cup L_{-}\right\} \\
\varrho\lambda=\partial_{\nabla d_{\partial E}}(\varrho\lambda)=0 & \text{at }M
\end{cases}.
\]
This identifies $\mu$ uniquely off of the segments $L_\pm$.

All that remains is to prove that $\mu=0$ on $L_{\pm}$. These segments are the closure, relative to $E$, of the stable lines passing between $z_\pm=(\pm a(1-\frac{b^{2}}{a^{2}}),0)$ and $y_\pm = (\pm a,0)$. 
Going back to \prettyref{lem:structure} and  \prettyref{lem:disintegration}, we see that
\[
\mu\lfloor L_{\pm}=A_{\pm}\delta_{z_{\pm}}+\lambda_{\pm}\nabla^{\perp}d_{\partial E}\otimes\nabla^{\perp}d_{\partial E}\,\mathcal{H}^{1}\lfloor L_{\pm},\quad\text{where}\quad \lambda_{\pm}(x) = c_{\pm}+\tilde{c}_{\pm}(x-z_{\pm})\cdot\nabla d_{\partial E}
\]
and where $A_{\pm}\in\text{Sym}_{2}$ and $c_{\pm},\tilde{c}_{\pm}\in\mathbb{R}$.
The measures $\mu\lfloor L_{\pm}$ are curlcurl-free in the sense
of distributions on $E$, i.e.,
\begin{align*}
0 & =\int_{E}\left\langle \nabla^{\perp}\nabla^{\perp}\psi,\mu\lfloor L_{\pm}\right\rangle =\left\langle A_{\pm},\nabla^{\perp}\nabla^{\perp}\psi(z_{\pm})\right\rangle +\int_{L_{\pm}}\lambda_{\pm}\partial_{\nabla d_{\partial E}}^{2}\psi\,d\mathcal{H}^{1}\\
 & =\left\langle A_{\pm},\nabla^{\perp}\nabla^{\perp}\psi(z_{\pm})\right\rangle +c_{\pm}\partial_{\nabla d_{\partial E}}\psi(z_{\pm})-\tilde{c}_{\pm}\psi(z_{\pm})
\end{align*}
for all $\psi\in C_{c}^{\infty}(E)$. It follows that $A_{\pm}=0$
and $c_{\pm}=\tilde{c}_{\pm}=0$. Thus, $\mu\lfloor L_{\pm}=0$ and this completes the proof of \prettyref{eq:optimaldefect_neg_ellipse}. \end{example}

Several negatively curved convex polygons appear in Panel (b) of  \prettyref{fig:library_of_patterns}.
We consider these next. As in the previous example,
the solution formulas from \prettyref{subsec:three-solution-formulas}
must  be supplemented by a separate argument  showing that $\mu$ is not supported on certain leftover stable
lines. This time, the argument involves the complementary slackness conditions from \prettyref{eq:wrinkling-PDEs}. 
\begin{example} \label{ex:negatively-curved-convex}(negatively
curved convex polygons) Let $P$ be a convex polygon with vertices
$a_{1},\dots,a_{n}\in\mathbb{R}^{2}$ labeled in counterclockwise
order and sides $S_{1}=[a_{1},a_{2}],\dots,S_{n}=[a_{n},a_{n+1}]$
where $a_{n+1}=a_{1}$. The outwards-pointing unit normal vector to
$\partial P$ takes on $n$ distinct values, which we label as
\[
\hat{\nu}_{i}=\hat{\nu}|_{S_{i}}=\frac{(a_{i}-a_{i+1})^{\perp}}{|a_{i}-a_{i+1}|},\quad i=1,\dots,n.
\]
The medial axis $M$ is a tree made up of finitely many line segments. Its complement $P\backslash M$ is the disjoint
union of $n$ (open) sub-polygons $P_{1},\dots,P_{n}$. The labels
are such that $S_{i}\subset\partial P_{i}$. Let $\hat{\tau}_{M}$
and $\hat{\nu}_{M}$ denote unit tangent and unit normal vectors to $M$. 
We take them to be locally constant off of its internal vertices.

With the notation set, we can describe $\varphi_{-}$. Observe that $d_{\partial P}=d_{S_{i}}=d(\cdot,S_{i})$
on the $i$th sub-polygon $P_{i}$. Hence, 
\[
\varphi_{-}=\frac{1}{2}|x|^{2}-\frac{1}{2}d_{S_{i}}^{2}\quad\text{on }P_{i}
\]
for $i=1,\dots,n$. As each side $S_{i}$ is a line segment, it follows
easily that 
\[
\nabla\nabla\varphi_{-}=\sum_{i=1}^{n}\hat{\nu}_{i}^{\perp}\otimes\hat{\nu}_{i}^{\perp}\indicator{P_{i}}\,dx+\sum_{1\leq i<j\leq n} d_{\partial P}|\hat{\nu}_{i}-\hat{\nu}_{j}|\hat{\nu}_{M}\otimes\hat{\nu}_{M}\,\mathcal{H}^{1}\lfloor\partial P_{i}\cap\partial P_{j}\quad\text{on }P.
\]
Thus, the stable lines run parallel to $\nabla d_{\partial P}=\hat{\nu}_{i}$
on $P_{i}$, and the original polygon $P$ is partitioned by $\varphi_{-}$ into the sets 
\[
\Sigma=M,\quad O=P\backslash M,\quad\text{and}\quad F=U=\emptyset.
\]
All this being said, we now claim that the unique
optimal $\mu$ is given by 
\begin{equation}
\mu=\sum_{i=1}^{n}\lambda_{i}\hat{\nu}_{i}^{\perp}\otimes\hat{\nu}_{i}^{\perp}\indicator{P_{i}}\,dx\quad\text{on }P,\quad\text{where}\quad\begin{cases}
-\frac{1}{2}\partial_{\hat{\nu}_{i}}^{2}\lambda_{i}=\det\nabla\nabla p & \text{on }P_{i}\\
\lambda_{i}=\partial_{\hat{\nu}_{i}}\lambda_{i}=0 & \text{at }\partial P_{i}\cap M
\end{cases}\quad\text{for }i=1,\dots,n.\label{eq:optimaldefect_neg_polygon}
\end{equation}
As in the previous example, our plan is as follows: first we apply
\prettyref{cor:chars-meet-along-curve} to identify $\mu$ away from
an exceptional one-dimensional set $L$, and then we verify separately
that $\mu=0$ on $L$.

\prettyref{cor:chars-meet-along-curve} is built to handle situations
where stable lines meet along a curve. Here, stable lines meet along
a tree --- the medial axis $M$. Its edges are line segments, its
external vertices are given by $\{a_{i}\}_{i=1}^{n}\subset\partial P$,
and we label its internal vertices as $\{z_{k}\}_{k=1}^{N}\subset P$.
It will probably be useful to look back at the medial
axis (in bold) of the triangle, square, or rectangle in  \prettyref{fig:library_of_patterns}. 
Each internal vertex $z_{k}$ belongs to the boundary of finitely many stable lines. Let $L_{k}$
be the union of $z_{k}$ and its associated stable lines, and let
$L=\cup_{k=1}^{N}L_{k}$. Now in the same manner as was done for \prettyref{ex:negatively-curved-ellipse}
(and as will be done for \prettyref{ex:negatively-curved-half-disc}
below), we can apply \prettyref{cor:chars-meet-along-curve} to deduce
that 
\[
\mu=\lambda\nabla^{\perp}d_{\partial P}\otimes\nabla^{\perp}d_{\partial P}\,dx\quad\text{on }P\backslash L,\quad\text{where}\quad\begin{cases}
-\frac{1}{2}\partial_{\nabla d_{\partial P}}^{2}\lambda=\det\nabla\nabla p & \text{on }P\backslash L\\
\lambda=\partial_{\nabla d_{\partial P}}\lambda=0 & \text{at }\partial(P\backslash L)\cap M
\end{cases}
\]
thus showing it is uniquely determined there. For brevity's
sake we leave the details of this to the reader, and simply note the relevant hypotheses can be checked using that $|\hat{\nu}_{i}-\hat{\nu}_{j}|$, $\hat{\tau}_{M}\cdot\hat{\nu}_{i}^{\perp}$,
and $\hat{\tau}_{M}\cdot\hat{\nu}_{j}^{\perp}$ are all non-zero 
at $\partial P_{i}\cap\partial P_{j}$.

It remains to show that $\mu=0$ on $L$. Consider the restriction
of $\mu$ to $L_{k}$, which we recall consists of $z_{k}$ along with the finitely many stable lines
$\{\ell_{s_{i}}\}$ having $z_{k}\in\partial\ell_{s_{i}}$.
The labeling is such that $\ell_{s_{i}}\subset P_i$, 
and we think of it as running from $z_{k}$ to the $i$th side
$S_{i}$. In particular, $\ell_{s_{i}}$
is parallel to $\hat{\nu}_{i}$. \prettyref{lem:structure}
and \prettyref{lem:disintegration} prove that
\begin{equation}
\mu\lfloor L_{k}=A\delta_{z_{k}}+\sum_{i}\lambda_{i}\hat{\nu}_{i}^{\perp}\otimes\hat{\nu}_{i}^{\perp}\,\mathcal{H}^{1}\lfloor\ell_{s_{i}},\quad\text{where}\quad\lambda_{i}(x)=c_{i}+\tilde{c}_{i}(x-z_{k})\cdot\hat{\nu}_{i}\label{eq:mu_Lk}
\end{equation}
for $A\in\text{Sym}_{2}$ and $c_{i},\tilde{c}_{i}\in\mathbb{R}$.
Since $\mu\lfloor L_{k}$ is curlcurl-free in the sense of distributions
on $P$, there holds
\[
0  =\int_{P}\left\langle \nabla^{\perp}\nabla^{\perp}\psi,\mu\lfloor L_{k}\right\rangle 
 =\left\langle A,\nabla^{\perp}\nabla^{\perp}\psi(z_{k})\right\rangle -\sum_{i}c_{i}\partial_{\hat{\nu}_{i}}\psi(z_{k})+\sum_{i}\tilde{c}_{i}\psi(z_{k})
\]
for all $\psi\in C_{c}^{\infty}(P)$. It follows immediately that
\begin{equation}
A=0,\quad \sum_{i}c_{i}\hat{\nu}_{i}=0,\quad\text{and}\quad\sum_{i}\tilde{c}_{i}=0.\label{eq:discrete-matching}
\end{equation}
So far, the argument has been more or less the same as in the previous example, and indeed we can already conclude that $\mu(\{z_{k}\})=0$. However, we cannot conclude that $\mu\lfloor L_{k}=0$ at this point. 
The trouble is that stable lines belonging to different $P_{i}$ may be parallel (e.g., for a negatively-curved rectangle). In such a case, at least two of the vectors $\hat{\nu}_i$ appearing in \prettyref{eq:discrete-matching} will be parallel, and the desired conclusion that $c_{i}=0$ will not follow.

Instead, the key is to go back to the first complementary slackness conditions in \prettyref{eq:wrinkling-PDEs}, which state here that $\left\langle\nabla^{\perp}\nabla^{\perp}\varphi_{-},\mu\right\rangle=0$ in the regularized sense. In particular, 
\begin{equation}
0=\lim_{\delta\to0}\int_{M}|d_{\partial P}[\nabla d_{\partial P}]\left\langle\hat{\tau}_{M}\otimes\hat{\tau}_{M},\mu_{\delta}\right\rangle|\,d\mathcal{H}^{1}\label{eq:vanishing-integral-useful}
\end{equation}
where $\{\mu_{\delta}\}_{\delta>0}$ are the mollifications
in \prettyref{eq:mollified_measure_defn-intro}. As noted in \prettyref{rem:complementary-slackness-freedom}, we may take the kernel
 $\rho>0$ nearby zero.
Given $x\in M$, observe using the non-negativity of $\mu$, the formula
\prettyref{eq:mu_Lk}, and Fubini's theorem that
\begin{multline*}
\int_{M}d_{\partial P}|[\nabla d_{\partial P}]|\left\langle\hat{\tau}_{M}\otimes\hat{\tau}_{M},\mu_{\delta}\right\rangle\,d\mathcal{H}^{1}\\\geq  \int_{y\in\ell_{s_{i}}}\int_{x\in M\cap\partial P_{i}}d_{\partial P}|[\nabla d_{\partial P}](x)||\hat{\tau}_{M}(x)\cdot\hat{\nu}_{i}^{\perp}|^{2}\lambda_{i}(y)\rho\left(\frac{x-y}{\delta}\right)\,\frac{d\mathcal{H}^{1}(x)d\mathcal{H}^{1}(y)}{\delta^{2}}
\end{multline*}
for each $i$. According to \prettyref{eq:vanishing-integral-useful}, the lefthand side tends to zero as $\delta\to0$.
Regarding the righthand side, note that $d_{\partial P}|[\nabla d_{\partial P}]|$
and $|\hat{\tau}_{M}\cdot\hat{\nu}_{i}^{\perp}|$ are bounded away from zero within the integral.  Also, $\lambda_{i}(y)\to c_{i}$ as $y\to z_{k}$ along $\ell_{s_{i}}$.
Using that $\rho>0$ nearby zero, we find upon sending $\delta \to 0$ that
\begin{equation}
c_{i}=0\quad\forall\, i.\label{eq:vanishing-c_is}
\end{equation}
With this, we can easily control the remaining coefficients $\{\tilde{c}_{i}\}$.  
Since $\mu\lfloor\ell_{s_{i}}\geq0$ there holds $\lambda_i \geq 0$, and then using that $\lambda_i(z_k) = c_i = 0$ we see that $\partial_{\hat{\nu}_i}\lambda_i(z_k) = \tilde{c}_{i}\geq0$.
It now follows from the third part of \prettyref{eq:discrete-matching} that
\begin{equation}
\tilde{c}_{i}=0\quad\forall\, i.\label{eq:vanishing-c_is_tildes}
\end{equation}
Looking back to \prettyref{eq:mu_Lk} once more, we see that $\mu\lfloor L_{k}=0$ for each $k$. Hence, $\mu\lfloor L=0$ and  \prettyref{eq:optimaldefect_neg_polygon} is proved.

It is no accident that these last few steps followed along the same
lines as the proofs of  \prettyref{cor:chars-meet-along-curve} and \prettyref{cor:chars-meet-at-point}. Our task was, once again, to show that $\mu$ vanishes on certain leftover stable lines. The constraints \prettyref{eq:vanishing-c_is} and \prettyref{eq:vanishing-c_is_tildes}
entered as Cauchy data analogous to, e.g., \prettyref{eq:zeroth_order_trace_sing-1} in \prettyref{cor:chars-meet-along-curve}. We imagine a similar approach may be used to control $\mu$ in various other circumstances where the formulas from \prettyref{subsec:three-solution-formulas} do not directly apply.
\end{example}

Our final example is the half-disc in Panel (b) of \prettyref{fig:library_of_patterns}.
It is the only one of our examples in which $M$ ends up being curved. 
\begin{example} \label{ex:negatively-curved-half-disc}(negatively
curved half-disc) Consider a disc of radius $a$ centered at the origin
and let 
\[
D_{+}=\left\{ (x_{1},x_{2}):x_{1}^{2}+x_{2}^{2}<a^{2},x_{2}>0\right\} .
\]
Note $d_{\partial D_{+}}$ is smooth away from the medial axis 
\[
M=\left\{ (x_{1},x_{2}):2ax_{2}=a^{2}-x_{1}^{2}, x_2 > 0\right\} 
\]
which is the unique parabolic arc passing through the corners $(\pm a,0)$
and $(0,\frac{a}{2})$. Denote the part of $D_{+}$ below $M$ by
$D_{+\text{S}}$ and the part above $M$ by $D_{+\text{N}}$. Then,
\[
\nabla\nabla\varphi_{-}=\hat{e}_{1}\otimes\hat{e}_{1}\indicator{D_{+\text{S}}}\,dx+\frac{a}{r}\hat{e}_{\theta}\otimes\hat{e}_{\theta}\indicator{D_{+\text{N}}}\,dx+d_{\partial D_{+}}\left|\hat{e}_{2}+\hat{e}_{r}\right|\hat{\nu}_{M}\otimes\hat{\nu}_{M}\,\mathcal{H}^{1}\lfloor M\quad\text{on }D_{+}.
\]
We see that $\varphi_-$ partitions $D_+$ according as
\[
\Sigma=M,\quad O=D_{+}\backslash M,\quad\text{and}\quad F=U=\emptyset.
\]
The stable lines are parallel to $\nabla d_{\partial D_{+}}=\hat{e}_{2}$
in $D_{+\text{S}}$ and $\nabla d_{\partial D_{+}}=-\hat{e}_{r}$ in
$D_{+\text{N}}$.

A straightforward application of \prettyref{cor:chars-meet-along-curve} proves
that optimal $\mu$ are unique. The conditions \prettyref{eq:along-sing-curve-hyp1}-\prettyref{eq:along-sing-curve-hyp4}
hold for $\varphi=\varphi_{-}$ and 
\[
V=D_{+},\quad\zeta=\begin{cases}
1 & x\in D_{+\text{S}}\\
\frac{a}{r} & x\in D_{+\text{N}}
\end{cases},\quad\text{and}\quad\hat{\eta}=\begin{cases}
\hat{e}_{1} & x\in D_{+\text{S}}\\
\hat{e}_{\theta} & x\in D_{+\text{N}}
\end{cases}.
\]
In particular, the medial axis is smooth and the quantities $|\hat{e}_{2}+\hat{e}_{r}|$,
$|\hat{\tau}_{M}\cdot\hat{e}_{1}|$, and $|\hat{\tau}_{M}\cdot\hat{e}_{\theta}|$
are uniformly positive there. Also, $\zeta>0$ uniformly
on $D_{+}$. Finally, we see that $\varrho=c_{S}(x_1)$ in $D_{+\text{S}}$
and $\varrho=c_{N}(\theta)r$ in $D_{+\text{N}}$ due to the fact that the stable lines describe either parallel lines or rays (the functions $c_S$ and $c_N$ depend on the index set $\Gamma$). Applying \prettyref{cor:chars-meet-along-curve}, 
we conclude that optimal $\mu$ satisfy 
\[
\mu=\lambda_{\text{S}}\hat{e}_{2}\otimes\hat{e}_{2}\indicator{D_{+\text{S}}}\,dx+\lambda_{\text{N}}\hat{e}_{\theta}\otimes\hat{e}_{\theta}\indicator{D_{+\text{N}}}\,dx\quad\text{on }D_{+}
\]
where 
\[
\begin{cases}
-\frac{1}{2}\partial_{2}^{2}\lambda_{\text{S}}=\det\nabla\nabla p & \text{in }D_{+\text{S}}\\
\lambda_{\text{S}}=\partial_{2}\lambda_{\text{S}}=0 & \text{at }M
\end{cases}\quad\text{and}\quad\begin{cases}
-\frac{1}{2r}\partial_{r}^{2}(r\lambda_{\text{N}})=\det\nabla\nabla p & \text{in }D_{+\text{N}}\\
\lambda_{\text{N}}=\partial_{r}\lambda_{\text{N}}=0 & \text{at }M
\end{cases}.
\]
These systems determine optimal $\mu$ uniquely for the negatively curved half-disc.
\end{example}

We close with a general conjecture on the uniqueness of optimal $\mu$. 
Although optimal $\mu$ turned out to be unique in each of the negatively curved examples above, the reader may yet wonder whether the convexity of $\Omega$ is crucial for this, or if it is simply an artifact of our  examples. We believe the latter is true. In fact, based on our method of stable lines, we expect optimal $\mu$ will be unique whenever there exists an optimal $\varphi$ that is nowhere locally affine (regardless of the curvature). Here is a more concrete version of this conjecture specialized to simply connected, negatively curved shells. 
Recall the medial axis $M$ from \prettyref{eq:medialaxis} and the change of measure factor $\varrho$ from \prettyref{eq:Leb_disintegration}. 
\begin{conjecture}\label{conj:uniqueness_neg} 
Suppose that $\Omega$ is simply connected and let $\det\nabla\nabla p\leq0$ a.e. Assume the paths of quickest exit from $\Omega$ do not meet at $\partial\Omega$. 
Then optimal $\mu$ are unique, and moreover satisfy
\[
\mu=\lambda\nabla^{\perp}d_{\partial\Omega}\otimes\nabla^{\perp}d_{\partial\Omega}\,dx\quad\text{on }\Omega,\quad\text{where}\quad\begin{cases}
-\frac{1}{2\varrho}\partial_{\nabla d_{\partial\Omega}}^{2}\left(\varrho\lambda\right)=\det\nabla\nabla p & \text{on }\Omega\backslash\overline{M}\\
\varrho\lambda=\partial_{\nabla d_{\partial\Omega}}\left(\varrho\lambda\right)=0 & \text{at } \overline{M}
\end{cases}.
\]
\end{conjecture}


\begin{acknowledgements}
We thank Eleni Katifori, Joseph D.\ Paulsen, Yousra Timounay, and
Desislava V.\ Todorova for sharing their experimental and numerical
results on thin and ultrathin floating shells in advance of their publication. We thank Benny Davidovitch, Charles R.\
Doering, and Robert V.\ Kohn for helpful discussions. 
\end{acknowledgements}

%
 \section*{Conflict of interest}

 The author declares that they have no conflict of interest.

\bibliographystyle{spmpsci}      
\bibliography{CurvatureDrivenWrinkling}   

\end{document}